\documentclass[12pt,reqno,twoside,a4paper]{amsart}

\makeatletter
\@namedef{subjclassname@2020}{\textup{2020} Mathematics Subject Classification}
\makeatother

\usepackage{amsmath, amsthm, amssymb, amscd, amsfonts}
\usepackage{mathrsfs}
\usepackage{color}
\usepackage[all]{xy}             % diagramas
 \usepackage{hyperref}          %  hipervinculo
\font \smallrm=cmr10 at 10truept
\font \smallsl=cmsl10 at 10truept
\font \ssmallrm=cmr10 at 9truept

\numberwithin{equation}{section}

\newcommand{\subu}[2]{{#1}_{\raise-2pt\hbox{$ \scriptstyle #2 $}}}
\newcommand{\subd}[3]{{#1}_{\raise-2pt\hbox{$ \scriptstyle #2 #3 $}}}

\newtheorem{lema}{Lemma}[subsection]
\newtheorem{theorem}[lema]{Theorem}

\newtheorem{prop}[lema]{Proposition}

\theoremstyle{definition}
\newtheorem{definition}[lema]{Definition}
\newtheorem{exa}[lema]{Example}

\newtheorem{rmk}[lema]{Remark}
\newtheorem{rmks}[lema]{Remarks}
\newtheorem{free text}[lema]{}
\newtheorem{obs}[lema]{Observation}
\newtheorem{obs's}[lema]{Observations}

\theoremstyle{remark}

\def \NN {\mathbb{N}}
\def \ZZ {\mathbb{Z}}
\def \k {\Bbbk}
\def \kh {{\Bbbk[[\hbar]]}}
\def \khp {{\Bbbk(\hskip-1,7pt(\hbar)\hskip-1,7pt)}}
\def \cF {\mathcal{F}}
\def \cL {\mathcal{L}}
\def \cK {\mathcal{K}}
\def \cG {\mathcal{G}}
\def \cR {\mathcal{R}}
\def \cHA {\mathcal{HA}}
\def \cO {\mathcal{O}}
\def \cP {\mathcal{P}}
\def \cS {\mathcal{S}}
\def \cT {\mathcal{T}}

\def \QUEA {{\mathcal{Q\hskip1ptUE\hskip-1ptA}}}
\def \QFSHA {{\mathcal{Q\hskip1ptFSHA}}}

\def \JJ {\mathfrak{J}}

\def \lieg {{\mathfrak{g}}}
\def \lieh {{\mathfrak{h}}}
\def \lieb {{\mathfrak{b}}}
\def \liem {{\mathfrak{m}}}
\def \lieso{\mathfrak{so}}
\def \lies{\mathfrak{s}}

\def \ug {{U(\hskip0,5pt\lieg)}}
\def \ugs {{U(\hskip0,5pt\lieg^*)}}
\def \fg {{F[[G\hskip1,5pt]]}}
\def \fgs {{F[[G^*\hskip0,5pt]]}}

\def \uhg{{U_\hbar(\hskip0,5pt\lieg)}}
\def \uhgs{{U_\hbar(\hskip0,5pt\lieg^*)}}
\def \Uhg{{\mathbb{U}_\hbar(\hskip0,5pt\lieg)}}
\def \uhbm{{U_\hbar(\hskip0,5pt\lieb^-)}}
\def \uhbp{{U_\hbar(\hskip0,5pt\lieb^+)}}

\def \fhg{{F_\hbar[[G\hskip1,5pt]]}}
\def \fhgs{{F_\hbar[[G^*]]}}

\def \fhbm{{F_\hbar[[B^-\hskip1,5pt]]}}
\def \fhbp{{F_\hbar[[B^+\hskip1,5pt]]}}

\def \ot{{\otimes}}
\def \otimeshat {{\,\widehat{\otimes}\,}}
\def \otimestilde {{\,\widetilde{\otimes}\,}}

\def \Ker {\textsl{Ker}}
\def \Prim {\operatorname{Prim}}
\def \id {\operatorname{id}}

\newcommand \rk{\operatorname{rk}}

\def \ad {\operatorname{ad}}
\def \Ad {\operatorname{Ad}}
\def \Hom {\textsl{Hom}}

\def \op {\operatorname{op}}
\def \uRPhg{U^{\,\cR}_{\!P,\hskip0,7pt\hbar}(\hskip0,5pt\lieg)}
\def \Rpicc {{\scriptscriptstyle \cR}}
\def \Ppicc {{\scriptscriptstyle P}}

\def \Lie {\textsl{Lie}}

\def \pf{\begin{proof}}
\def \epf{\end{proof}}

\theoremstyle{plain}

\begin{document}

{\ }
 \vskip-1pt

\title[QUANTUM GROUP DEFORMATIONS and  $ R $--(CO)MATRICES
vs.\ QUANTUM DUALITY]
%%%%%
% \title[Quantum Group Deformations and QDP]
%%%%%
 {Quantum Group Deformations  \\
  and Quantum  $ R $--(co)matrices vs.  \\
  Quantum Duality Principle}

\author[Gast{\'o}n Andr{\'e}s GARC{\'I}A\,, \ \  Fabio GAVARINI]
{Gast{\'o}n Andr{\'e}s Garc{\'\i}a${}^\flat$\,,  \ \ Fabio Gavarini$\,{}^\sharp$}

\address{\newline
 Departamento de Matem\'atica, Facultad de Ciencias Exactas
   \newline
 Universidad Nacional de La Plata   ---   CMaLP-CIC-CONICET
   \newline
 \phantom{aa} 1900 La Plata, Pcia.\ de Buenos Aires, ARGENTINA
   \newline
 \phantom{aa} E-mail: \  {\tt ggarcia@mate.unlp.edu.ar}
 \vspace*{0.7cm}
   \newline
Dipartimento di Matematica,
   \newline
 Universit\`a degli Studi di Roma ``Tor Vergata''
   \newline
 \phantom{aa} Via della ricerca scientifica 1   \, --- \,   I\,-00133 Roma, ITALY
   \newline
 \phantom{aa} E-mail: \  {\tt gavarini@mat.uniroma2.it}
   \newline
 Istituto Nazionale di Alta Matematica ``Francesco Severi'' --- GNSAGA }

\thanks{\noindent 2020 \emph{MSC:}\,  17B37, 17B62   ---
\emph{Keywords:} Quantum Groups, Quantum Enveloping Algebras.}

%%%%%
%  \date{\bf \today}
%%%%%

\begin{abstract}
 \medskip
   In this paper we describe the effect on quantum groups
   --- namely, both QUEA's and QFSHA's ---   of deformations by twist and by 2--cocycles,
   showing how such deformations affect the semiclassical limit.
                                                                       \par
   As a second, more important task, we discuss how these deformation procedures can be
   ``stretched'' to a new extent, via a formal variation of the original recipes, using
   \textit{polar twists\/}  and  \textit{polar 2--cocycles}.
   These recipes seemingly should make no sense at all, yet we prove that they actually work,
   thus providing well-defined, more general deformation procedures.
   Later on, we explain the underlying reason that motivates such a result in light of the
   Quantum Duality Principle, through which every ``polar twist/2--cocycle'' for a given
   quantum group can be seen as a standard twist/2--cocycle for another quantum group,
   associated to the original one via the appropriate Drinfeld functor.
                                                                       \par
   As a third task, we consider standard constructions involving  $ R $--(co)matrices in
   the general theory of Hopf algebras.  First we adapt them to quantum groups,
   then we show that they extend to the case of  \textsl{polar}  $ R $--(co)matrices,
   and finally we discuss how these constructions interact with the Quantum Duality Principle.
   As a byproduct, this yields new special symmetries (isomorphisms) for the underlying pair
   of dual Poisson (formal) groups that one gets by specialization.
\end{abstract}

%%%%%
 {\ } \vskip9pt
%%%

\maketitle

\tableofcontents

%%%%%%%%%%%%%%%%%%%%%%%%%%%%%%%%%%%%
%%%%%%%%% PAPER %%%%%%%%%%%%%%%%%%%%%%%

%%%%%%%%%%%% INTRODUCTION %%%%%%

\section{Introduction}  \label{sec: intro}  {\ }
 \vskip3pt
   In Hopf algebra theory, there exists a well-established theory of ``deformations'' that are
   produced via specific tools, namely  \textit{twists\/}  in one case and   \textit{2--cocycles\/}
   in the other case (the terminology is not entirely agreed upon, yet our choice of terms seems
   to be the standard one among Hopf algebraists, at least).  Given a Hopf algebra  $ H $,
   a twist for it is a suitable element  $ \, \cF \in H \otimes H \, $,  while (dually) a 2--cocycle is a
   suitable 2--form  $ \, \sigma \in {(H \otimes H)}^* \, $.  Deformation by the twist  $ \cF $
   provides  $ H $ with a new Hopf algebra structure, by modifying the coproduct
   (and the antipode) but not the product, whereas deformation by the 2--cocycle  $ \sigma $
   endows  $ H $ with yet another Hopf structure by changing the product alone
   (and the antipode) but not the coproduct.
 \vskip5pt
   Quantum groups are Hopf algebras of special type, which come in two versions: QUEAs
   (= quantized universal enveloping algebras) and QFSHAs (= quantized formal series Hopf
   algebras).
%%%
 Roughly speaking, a QUEA is a (topological) Hopf algebra  $ U_\hbar $  over the
 $ \k $--algebra  of formal power series  $ \kh $  such that
 $ \, U_0 := U_\hbar \big/ \hbar\,U_\hbar \, $  is isomorphic to  $ \ug $  for some Lie algebra
 $ \lieg \, $.  It follows then that  $ \ug $  inherits from  $ U_\hbar $  a Poisson cobracket,
 which makes it into a co-Poisson Hopf algebra, hence  $ \lieg $  bears a Lie cobracket making it
 into a Lie bialgebra.  One then says that  $ U_\hbar $  is a  \textsl{quantization\/}
 of the co-Poisson Hopf algebra  $ \ug \, $,  or just of the Lie bialgebra  $ \lieg \, $.
%%%
 Dually, a QFSHA is a (topological) Hopf algebra  $ F_\hbar $  over  $ \kh $  such that
 $ \, F_0 := F_\hbar \big/ \hbar\,F_\hbar \, $  is isomorphic to  $ \fg $
 for some formal algebraic group  $ G \, $.  Then  $ \fg $  inherits from  $ F_\hbar $  a
 Poisson bracket, which makes it into a Poisson Hopf algebra, thus  $ G $  bears a
 Poisson structure which makes it into a formal Poisson (algebraic) group.  One says then that
 $ F_\hbar $  is a  \textsl{quantization\/}  of the Poisson Hopf algebra  $ \fg \, $,  or just of the
 (formal) Poisson group  $ G \, $.
 \vskip5pt
   As a general philosophy, from any Hopf-theoretical notion   --- at the quantum level ---
   one typically infers a Lie-theoretical counterpart   --- at the semiclassical level.
   When dealing with deformations, this leads to devising suitable notions of ``twists'' and
   ``2--cocycles'' for Lie bialgebras as well as ``deformations'' (of Lie bialgebras) by them.
   In particular, a deformation by twist yields a new Lie bialgebra structure where only the Lie
   cobracket is modified, whereas deformation by 2--cocycle defines yet another, similar
   structure where only the Lie bracket is changed.
                                                                            \par
   Following this recipe, the following should hold: when one deforms (as a Hopf algebra) a
   quantization  $ U_\hbar $  of  $ \lieg $  by a twist which is trivial modulo  $ \hbar \, $,
   the outcome is a quantization of  $ \lieg' \, $,  with the latter being a deformation by twist
   (as a Lie bialgebra) of  $ \lieg \, $:  moreover, the (Lie) twist working upon  $ \lieg $
   is directly ``induced'' by the (Hopf) twist that works upon  $ U_\hbar \, $,
   namely the former (Lie) twist is the ``semiclassical limit''
%%%%%
% , in a suitable sense,
%%%%%
 of the latter (Hopf) twist.
%%%
                                                                            \par
   Dually, the following also should hold: when one deforms (as a Hopf algebra) a quantization
   $ F_\hbar $  of  $ G \, $  by a 2--cocycle which is trivial modulo  $ \hbar \, $,
   the outcome is a quantization of  $ G' \, $,  the latter being a (formal) Poisson group whose
   cotangent Lie bialgebra is a deformation by 2--cocycle of the cotangent Lie bialgebra
   $ \, \lieg^* := {\Lie(G)}^* \, $  of  $ G \, $:  moreover, the (Lie) 2--cocycle acting on  $ \lieg^* $
   is ``induced'' by the (Hopf) 2--cocycle that acts on  $ F_\hbar \, $,  namely the former (Lie)
   2--cocycle is the ``semiclassical limit''
%%%%%
% , in a suitable sense,
%%%%%
 of the latter (Hopf) 2--cocycle.
                                                                            \par
   Nevertheless, neither of the two results mentioned above seems to be published anywhere in
   literature (to the best of the authors' knowledge, say).  Therefore, as a first contribution in
   this paper we provide a full, complete statement and proof for the above sketched results,
   turning them into well-established theorems.
 \vskip5pt
   As a second step   --- our main contribution in this paper ---   we extend the notions of
   (Hopf) twist and 2--cocycle, as well as the construction of (Hopf) deformations by them,
   to a wider setup.  Namely, we introduce the notions of  \textit{polar twist\/}  for a QFSHA
   and of  \textit{polar 2--cocycle\/}  for a QUEA: roughly speaking, a polar twist for
   $ F_\hbar $  has the formal Hopf properties of a twist but has the form
   $ \, \exp\big( \hbar^{-1} \varphi \big) \, $,  while any twist (trivial modulo  $ \hbar \, $)
   looks like  $ \, \exp\big( \hbar^{+1} \phi \big) \, $   --- and similarly for the link between
   polar 2--cocycles and 2--cocycles.  Thus even the very definition of these ``polar objects'',
   at least in this form, seems to be problematic, to say the least   --- as multiplying by
   $ \, \hbar^{-1} \, $  in a  $ \kh $--module  is meaningless.  In spite of this, we show that the
   recipe defining deformations by twist, resp.\ by 2--cocycle, for a QFSHA, resp.\ for a QUEA,
   still makes sense if we replace ``twists'' with ``polar twists'', resp.\ ``2--cocycles'' with
   ``polar 2--cocycles''.  Moreover, we can describe the semiclassical limit of these
   deformations (by ``polar objects''), again in terms of deformations of Lie bialgebras by some
   (Lie) twist, resp.\ 2--cocycle, that can be read out as the semiclassical limit of the
   quantum (Hopf) polar twist, resp.\ polar 2--cocycle, that we started with.  In a nutshell, we find the perfect
   ``polar versions'' of the results above for standard quantum group deformations,
   i.e.\ those by twist or by 2--cocycle.
 \vskip5pt
   The fact that ``deformations by  \textsl{polar\/}  objects''  do make sense, albeit problematic at
   first sight, can be explained in light of the  \textit{Quantum Duality Principle (=QDP)}.
   In fact, the latter provides functorial recipes   -- via \textit{Drinfeld's functors}  ---
   which turn any QUEA into a QFSHA and any QFSHA into a QUEA.
   Then, through the QDP lens, every ``polar twist'' for a QFSHA, resp.\ every
   ``polar 2--cocycle'' for a QUEA, is just a sheer standard twist, resp.\ 2--cocycle, for the
   QUEA, resp.\ the QFSHA, obtained when applying the appropriate Drinfeld functor.
   In this way, our deformation procedures ``by polar objects'' turn out to be tightly related
   with standard ones, but applied to different quantum groups.  Nevertheless, one still has to
   prove that the (standard) deformation applied to the new quantum group can actually be
   ``adapted'' (by restriction or by extension, depending on the type of quantum group and
   Drinfeld functor involved) to the original quantum group; in fact, this still requires a detailed,
   careful analysis.
 \vskip5pt
   As a third contribution, we finally consider some constructions of morphisms that,
   in the setup of general Hopf algebra theory, are provided by  $ R $--matrices  or
   $ \varrho $--comatrices.  We apply these constructions to the case of quantum groups,
   showing that their outcome is much finer than expected from the general theory, and
   bringing to light their geometrical meaning at the semiclassical level.
   In addition, we improve those results as follows: we introduce the notions of
   \textit{polar  $ R $--matrices\/}  and  \textit{polar  $ \varrho $--comatrices}
   (much in the same spirit as with polar twists and polar 2--cocycles), and then we extend
   the construction of morphisms induced by  $ R $--matrices  and  $ \varrho $--comatrices
   to the case of polar  $ R $--matrices  and polar  $ \varrho $--comatrices, with a clear
   interpretation that once more is provided by the QDP.
 \vskip7pt
   The paper is organized as follows.
                                                               \par
   In  \S \ref{sec: QG's-QDP-deform.'s}  we quickly recall the material we work with: Lie
   bialgebras and their deformations, Hopf algebras and their deformations, quantum groups
   (in the formal setting, both as QUEAs and as QFSHAs) and the
   \textit{Quantum Duality Principle}.
                                                               \par
   In  \S \ref{sec: deform.'s-QG's}  we present the bulk of the paper   --- its main core, in a sense.  Namely, we first study deformations of QUEAs by twist and of QFSHAs by 2--cocycles, then we present the new notions of polar 2--cocycle (for a QUEA) and polar twist (for a QFSHA) and the procedures of deformations by these.  Later on, all this material is discussed again in  \S \ref{sec: deform.'s-QDP},  in light of the  \textit{Quantum Duality Principle}.
                                                               \par
   Finally, in  \S \ref{sec: (co)quasitriangular}  we study the morphisms defined through
   $ R $--matrices  or  $ \varrho $--comatri\-ces,  showing the general results
   (for any Hopf algebra) actually improve in the case of quantum groups and explaining
   their meaning at the semiclassical limit.  Moreover, we also extend those constructions and
   results to the newly minted notions of polar  $ R $--matrices  and polar  $ \varrho $--comatrices.

\vskip17pt

   \centerline{\ssmallrm ACKNOWLEDGEMENTS}
 \vskip3pt
 {\smallrm This article was partially supported by:
                                                         \par
   \quad  \textit{(1)}\,  CONICET, ANPCyT, Secyt (Argentina),
                                                         \par
   \quad  \textit{(2)}\,  INdAM research branch GNSAGA (Italy),
                                                         \par
   \quad  \textit{(3)}\,  the MIUR  {\smallsl Excellence Department Project
   MatMod@TOV (CUP E83C23000330006)\/}  awarded to the Department of Mathematics,
   University of Rome ``Tor Vergata'' (Italy),
                                                         \par
   \quad  \textit{(4)}\,  the European research network
   {\smallsl COST Action CaLISTA CA21109}.
                                                       \par
 The authors thank Benjamin Enriquez for sparking their interest in the
   topics of \S \ref{sec: (co)quasitriangular}.}

\bigskip
 \medskip

\section{Quantum groups, Quantum Duality Principle, and deformations}
\label{sec: QG's-QDP-deform.'s}

\vskip13pt

   In this section we recap the basic notions we deal with in this paper:
   Lie bialgebras, quantum groups, deformations of both, and the Quantum Duality Principle.

\medskip

\subsection{Lie bialgebras and Lie deformations}  \label{subsec: Lie-bialg's & deform's}  {\ }
 \vskip7pt
   In this subsection we recall some definitions and basic facts about Lie
   bialgebras and their deformations.  For a more detailed treatment we refer to \cite{CP},
   \cite{Mj}.
                                                        \par
Throughout the paper,  $ \k $  will be a field of characteristic zero.

\vskip9pt

\begin{free text}  \label{gen's on LbA's}
 {\bf Generalities.}  A  \textsl{Lie bialgebra\/}  is a triple
 $ \, \big(\, \lieg \, ; \, [\,\ ,\,\ ] \, , \, \delta \,\big) \, $  such that  $ \, (\lieg, [\,\ ,\,\ ]) \, $
 is a Lie algebra
over  $ \k \, $, $(\lieg, \delta)$ is a \textit{Lie coalgebra}
 with \textit{Lie cobracket\/}  $ \, \delta : \lieg \longrightarrow \lieg \wedge \lieg \, $,
 i.e.\  $ \, \delta^* : \lieg^* \wedge \lieg^* \longrightarrow \lieg^* \, $  is a Lie algebra bracket
 on  $ \lieg^* \, $),
 and the two structures are linked by the constraint that  $ \delta $  is a  $ 1 $--cocycle
for the Chevalley-Eilenberg cohomology of the Lie algebra  $ \, \big(\, \lieg \, ; \, [\,\ ,\,\ ] \,\big) \, $
 with coefficients in  $ \, \lieg \wedge \lieg \, $:
\begin{equation}  \label{eq:compat-bracket-cobracket}
   \begin{aligned}
      \delta([x,y])  &  \; = \;  \ad_x\big(\delta(y)\big) - \ad_y\big(\delta(x)\big)  \; =  \\
                     &  \; = \;  \big[x,y_{[1]}\big] \otimes y_{[2]} + y_{[1]} \otimes \big[x,y_{[2]}\big] -
                     \big[y,x_{[1]}\big] \otimes x_{[2]} - x_{[1]} \otimes \big[y,x_{[2]}\big]
   \end{aligned}
\end{equation}
 using Sweedler's-like notation  $ \, \delta(x) = x_{[1]} \otimes x_{[2]} \, $  for any
 $ \, x \in \lieg \, $.  We write also
 $ \, x \wedge y := 2^{-1} (x \otimes y - y \otimes x) \, $  and thus we identify
 $ \, \lieg \wedge \lieg \, $
 with the subspace  $ \, {(\lieg \otimes \lieg)}^{\ZZ_2} \, $.
                                                                        \par
   Finite-dimensional Lie bialgebras are  \textit{self-dual},  in the sense that
 $ \, \big(\, \lieg \, ; \, [\,\ ,\,\ ] \, , \, \delta \,\big) \, $  is a Lie bialgebra if and only if
 $ \, \big(\, \lieg^* \, ; \, \delta^* , \, {[\,\ ,\,\ ]}^* \,\big) \, $ is so; the latter is called the
 \textit{dual\/}
Lie bialgebra to  $ \, \big(\, \lieg \, ; \, [\,\ ,\,\ ] \, , \, \delta \,\big) \, $.
This also holds in the infinite-dimensional case, up to technicalities.
We denote a Lie bialgebra simply by  $ \lieg \, $,  and by $ \lieg^* $ its dual.
                                                             \par
   Given  $ \, r = r_{1} \otimes r_{2} \, $ in  $ \, \lieg \otimes \lieg \, $,  we write
   $ \; r_{2,1} := r_{2} \otimes r_{1} \; $ and
   $ \; r_{1,2} := r_{1} \otimes r_{2} \otimes 1\; $,
   $ \; r_{2,3} := 1\otimes r_{1} \otimes r_{2} \; $
   $ \; r_{1,3} := r_{1} \otimes 1\otimes r_{2}\in  \, \lieg \otimes \lieg \otimes \lieg $.
   For example, given
   $ \, s = s_{1} \otimes s_{2} \in \lieg\ot\lieg $ we define the element
\begin{align*}
[[ r , s ]] & :=  [r_{1,2},s_{1,3}] + [r_{1,2},s_{2,3}] + [r_{1,3},s_{2,3}]\\
& = \big[r_{1},s_{1}\big] \otimes r_{2} \otimes s_{2} + r_{1} \otimes
  \big[r_{2},s_{1}\big] \otimes s_{2} +
  r_{1} \otimes s_{1} \otimes \big[r_{2},s_{2}\big] .
\end{align*}
\end{free text}

\vskip9pt

\begin{free text}  \label{deformations of LbA's}
 {\bf Deformations of Lie bialgebras.}  In this work, we are mainly interested in
 two kinds of deformations, where either the Lie cobracket or the Lie bracket alone
 is deformed.  A general theory of deformations for Lie bialgebras using cohomology
 theory exists, see e.g.\  \cite{CG},  \cite{MW}, and references therein for more details.
 \vskip3pt
   Let  $ \big(\, \lieg \, ; \, [\,\ ,\,\ ] \, , \, \delta \,\big) \, $
   be a Lie bialgebra and  $ \, c \in \lieg \otimes \lieg \, $
   be such that
\begin{equation}  \label{eq: twist-cond_Lie-bialg}
   \ad_x\!\big((\delta \otimes \id)(c) + \text{c.p.} - [[ c \, , c \,]] \,\big)  \; = \; 0  \;\; ,
   \quad  \ad_x\!\big( c + c_{\,2,1} \big) \, = \, 0   \qquad   \forall \;\; x \in \lieg
\end{equation}
 where  $ \, \ad_x \, $  denotes the standard adjoint action of  $ x $  and \,c.p. means
 cyclic permutations on the tensor factors of the previous summand.
 Then,  the  map
   $ \, \delta^{\,c} : \lieg \relbar\joinrel\longrightarrow \lieg \wedge \lieg \, $  defined by
\begin{equation}  \label{eq: def_twist-delta}
  \delta^{\,c}  \, := \;  \delta - \partial(c) \;\; ,  \qquad  \text{i.e.\ \ \ }
  \delta^{\,c}(x) \, := \, \delta(x) - \ad_x(c)   \qquad \forall \;\; x \in \lieg
\end{equation}
 defines a new Lie cobracket on  $ \big(\, \lieg \, ; \, [\,\ ,\,\ ] \,\big) \, $  making
 $ \, \big(\, \lieg \, ; \, [\,\ ,\ ] \, , \, \delta^{c} \,\big) \, $  into a new Lie bialgebra
 (cf.\ \cite[Theorem 8.1.7]{Mj}).
\end{free text}

\vskip9pt

\begin{definition}  \label{def: twist-deform_Lie-bialg's}
 An element  $ \, c \in \lieg \otimes \lieg \, $  satisfying  \eqref{eq: twist-cond_Lie-bialg}
 is called a  \textit{twist\/}  of the Lie bialgebra  $ \lieg \, $,
 \,and the corresponding Lie bialgebra
 $ \; \lieg^{\,c} \, := \, \big(\, \lieg \, ; \, [\,\ ,\,\ ] \, , \, \delta^{\,c} \,\big) \; $
 is called a \textit{deformation by twist}  or  \textit{twist deformation\/}  of $ \lieg \, $.
 \hfill  $ \diamondsuit $
\end{definition}

\vskip7pt

\begin{rmk}  \label{rmk: differ-Mj-Nosotros_x_def-twist}
 To be precise, we are adopting here conventions that are slightly
 different from those in  \cite{Mj},  yet equivalent.
 Indeed, we choose to  \textsl{define\/}  the deformed Lie cobracket in
 \eqref{eq: def_twist-delta}
 via the formula  $ \; \delta^c := \delta - \partial(c) \; $, \;whereas Majid's definition is
 $ \; \delta^c := \delta + \partial(c) \; $.  This change entails that  \eqref{eq: twist-cond_Lie-bialg}  the term  $ \, - [[ c \, , c \,]] \, $  has opposite sign with respect to Majid's book.
 However, our choice of sign yields better-reading statements for our results that connect,
 via specialization, deformations by twist at the quantum level with those at
 the semi-classical level  (cf.\ Theorem \ref{thm: twist-deform-QUEA}  and
 Theorem \ref{thm: propt.'s qs-twist-deform-QFSHA});
 that is why we chose such a different option, lest we should insert an odd-looking
   sign in those results.
 The equivalence stands in that any  $ \, c \in \lieg \otimes \lieg \, $
 is a twist element in our sense if and only if  $ \, -c \, $  is a twist in Majid's sense.
\end{rmk}

\vskip7pt

   Now we introduce a deformation of the Lie bracket.
   Let  $ \big(\, \lieg \, ; \, [\,\ ,\,\ ] \, , \, \delta \,\big) \, $
   be a Lie bialgebra and
   $ \, \gamma \in \Hom_\Bbbk\!\big(\, \lieg \otimes \lieg \, , \Bbbk \,\big) \, $.
   For $\lieg$ finite-dimensional, we
   identify  $ \; \Hom_\Bbbk\!\big(\, \lieg \otimes \lieg \, , \Bbbk \,\big) \, = \,
   {( \lieg \otimes \lieg )}^* \, = \, \lieg^*\otimes \lieg^* \; $.
   There exists some technicalities in the infinite-dimensional case,
   yet the outcome is always the same.
   Dualizing the notion of twist for $\lieg^{*}$ we obtain the dual notion of $2$-cocycle
   as follows.
   Condition  \eqref{eq: twist-cond_Lie-bialg}  with  $ \lieg^* $  replacing  $ \lieg $  and
   $ \gamma $  in the role of  $ c $  yields
\begin{equation}  \label{eq: cocyc-cond_Lie-bialg}
 \begin{aligned}
   \ad_\psi\!\big(\,\partial_*(\gamma) + {[[ \gamma \, , \gamma \,]]}_* \,\big)  \, = \, 0  \;\, ,
   \;\;\;  \ad_\psi\!\big(\, \gamma + \gamma_{{}_{2,1}} \big) \, = \, 0
   \qquad   \forall \;\; \psi \in \lieg^*
 \end{aligned}
\end{equation}
 where  $ \, \gamma_{{}_{2,1}} := \gamma^{\scriptscriptstyle T} \, $ and
 $\big(\partial_*(\gamma)\!\big)(a\,,b\,,c\,) \; = \; \gamma\big([a\,,b\,]\,,c\big) \, + \, \text{c.p.} \, $.
 Similarly, $ \, {[[\,\ ,\,\ ]]}_* \, $  has the same meaning as above but with respect to  $ \lieg^* $.
 \vskip5pt
  For any  $ \gamma $  satisfying  \eqref{eq: cocyc-cond_Lie-bialg},
  the map  $ \; {[\,\ ,\ ]}_\gamma : \lieg \wedge \lieg \relbar\joinrel\longrightarrow \lieg \, $
  given by
\begin{equation}  \label{eq: def_cocyc-bracket}
  {[x,y]}_\gamma  \, := \;  [x,y] \, + \, \gamma\big(x_{[1]},y\big) \, x_{[2]} \, - \, \gamma\big(y_{[1]},x\big) \, y_{[2]}
  \qquad \qquad \forall \;\; x, y \in \lieg
\end{equation}
defines a new Lie bracket on the Lie coalgebra  $ \big(\, \lieg \, ; \, \delta \,\big) \, $  making
 $ \, \big(\, \lieg \, ; \, {[\,\ ,\ ]}_\gamma \, , \, \delta \,\big) \, $  into a new Lie bialgebra (cf.\ \cite[Exercise 8.1.8]{Mj}).

\vskip11pt

\begin{definition}  \label{def: cocyc-deform_Lie-bialg's}
 Every  $ \, \gamma \in \Hom_\Bbbk\!\big(\, \lieg \wedge \lieg \, , \Bbbk \,\big) \, $
 that obeys  \eqref{eq: cocyc-cond_Lie-bialg}  is called a  \textit{2--cocy\-cle\/}  of the Lie bialgebra  $ \lieg \, $,
 \,and the Lie bialgebra  $ \, \lieg_\gamma \, := \, \big(\, \lieg \, ; \, {[\,\ ,\ ]}_\gamma \, , \, \delta \,\big) \, $
 is called a  \textit{deformation by 2--cocycle} or  \textit{2--cocycle deformation\/})  of  $ \lieg \, $.
 \hfill  $ \diamondsuit $
\end{definition}

\vskip7pt

\begin{rmk}  \label{rmk: differ-Mj-Nosotros_x_def-2coc}
 Another observation which is dual to  Remark \ref{rmk: differ-Mj-Nosotros_x_def-twist}
 applies to our given definition of 2--cocycle and of 2--cocycle deformation,
 in comparison to the original one in  \cite{Mj}.
 Again, our notion of 2--cocycle is equivalent to Majid's because any
 $ \, \gamma \in {( \lieg \otimes \lieg )}^* \, $
 is a 2--cocycle in our sense if and only if its opposite  $ \, -\gamma \, $
 is a 2--cocycle in Majid's, and viceversa.
\end{rmk}

\vskip5pt

The following result, which is standard, formalizes the fact that
 the notions of ``twist'' and of  ``$ 2 $--cocycle''  for Lie bialgebras
 are devised to be dual to each other.

\vskip11pt

\begin{prop}  \label{prop: duality-deforms x LbA's}
 Let  $ \lieg $  be a Lie bialgebra, and  $ \lieg^* $  the dual Lie bialgebra.
 \vskip3pt
   {\it (a)}\,  Let  $ c $  be a twist for  $ \lieg \, $,  and  $ \gamma_c $
   the image of  $ c $  in $ {(\lieg \otimes \lieg)}^* $
   for the natural composed embedding
   $ \, \lieg \otimes \lieg \lhook\joinrel\relbar\joinrel\longrightarrow \lieg^{**} \otimes \lieg^{**}
   \lhook\joinrel\relbar\joinrel\longrightarrow {\big( \lieg^* \otimes \lieg^* \big)}^* \, $.
   Then  $ \gamma_c $  is a 2--cocycle  for  $ \lieg^* \, $,
   and there exists a canonical isomorphism
   $ \, {\big( \lieg^* \big)}_{\gamma_c} \cong {\big( \lieg^c \big)}^* \, $.
 \vskip3pt
   {\it (b)}\,  Let  $ \gamma $  be a 2--cocycle  for  $ \lieg \, $;
   assume that  $ \lieg $  is finite-dimensional, and let  $ c_{\,\gamma} $
   be the image of  $ \chi $  in the natural identification
   $ \, {(\lieg \otimes \lieg)}^* = \lieg^* \otimes \lieg^* \, $.
   Then  $ c_{\,\gamma} $  is a twist for  $ \lieg^* \, $,
   and there exists a canonical isomorphism
   $ \, {\big( \lieg^* \big)}^{c_{\,\gamma}} \cong {\big( \lieg_\gamma \big)}^* \, $.
\qed
\end{prop}

\medskip

\subsection{Hopf algebra deformations and  $ R $--(co)matrices}
\label{subsec: HA-deform's_&_R-(co)mat.'s}  {\ }
 \vskip7pt
   Below we recall some notions on deformations for Hopf algebras.
   Our main reference for general theory of Hopf algebras is  \cite{Ra}.
   For topological (and complete) Hopf algebras, we refer to  \cite{Ks},
   \cite{CP}  and  \cite{KS}.
 \vskip7pt
   Let us fix our notation for Hopf algebra theory (mainly standard, indeed).
   The comultiplication is denoted  $ \Delta $, the counit  $\epsilon  \, $
   and the antipode  $ \cS \, $;
   for the first, we use the Heyneman-Sweedler notation, namely  $ \, \Delta(x) = x_{(1)} \otimes x_{(2)} \, $.
   The augmentation ideal of any Hopf algebra  $ H $  is denoted by  $ \, H^+ := \Ker(\epsilon) \, $.
   In general, we denote by  $ \k $  the ground ring of our algebras, coalgebras, etc.; we identify it with its image
   $ \, \Bbbk\,1_{{}_H} \, $  through the unit map of  $ H $,  with  $ \, 1_{{}_H} \in H \, $  being the unit element.
                                                                        \par
   For a Hopf algebra $ H $ (or just bialgebra),  we write  $ H^{\text{op}} \, $,  resp.\  $ H^{\text{cop}} \, $,
   for the Hopf algebra (or bialgebra) given by taking in  $ H $  the opposite product, resp.\ coproduct.
 \vskip5pt
   There exist two standard methods to deform Hopf algebras, leading to so-called
   ``$ 2 $--cocycle  deformations'' and to ``twist deformations'': hereafter we recall
   both procedures, and their link via duality.
   They also adapt to the setup of  \textsl{topological\/}  Hopf algebras; later on,
   we apply them to quantum groups.

\vskip11pt

\begin{definition}  \label{def: twist & R-mat.'s}
 Let  $ H $  be a bialgebra (possibly topological, over some commutative ground ring),
 and let  $ \, \cF \in H \otimes H \, $.  Then:
 \vskip3pt
   \textit{(a)}\,  $ \cF $  is said to be  \textit{unitary\/}  if
\begin{equation}  \label{eq: unitary-cond.'s x twist}
  \big( \epsilon \otimes \text{id} \big)(\cF\,)  \; = \;  1  \; = \;
  \big( \text{id} \otimes \epsilon \big)(\cF\,)
\end{equation}
 \vskip3pt
   \textit{(b)}\,  $ \cF $  is called a  \textit{twist\/}  if it is  \textsl{invertible\/}  in
   $ \, H \otimes H \, $,  it is  \textsl{unitary},  and
\begin{equation}  \label{eq: twist-cond.'s}
  \cF_{1{}2} \, \big( \Delta \otimes \text{id} \big)(\cF\,)  \,\; = \;\,
  \cF_{2{}3} \, \big( \text{id} \otimes \Delta \big)(\cF\,)
\end{equation}
 \vskip3pt
   \textit{(c)}\,  $ \cF $  is called a  \textit{(quantum)  $ R $--matrix\/}
   if it is  \textsl{invertible\/}  in  $ \, H \otimes H \, $  and
\begin{equation}  \label{eq: R-mat_properties}
  \big( \Delta \otimes {id} \,\big)(\cF\,) \, = \, \cF_{13} \, \cF_{23}  \quad ,
  \qquad  \big(\, {id} \otimes \Delta \big)(\cF\,) \, = \, \cF_{13} \, \cF_{12}
\end{equation}
 \vskip3pt
   \textit{(d)}\,  $ \cF $  is called a  \textit{(quantum)  $ R $--matrix  twist\/}
   if it complies both   \textit{(b)\/}  and   \textit{(c)\/}  above
 \vskip7pt
   \textit{(e)}\,  $ \cF $  is said to be a  \textit{solution of the quantum
   Yang-Baxter equation (=QYBE)\/}  if
\begin{equation}  \label{eq: qYB-equation x F}
  \cF_{1{}2} \, \cF_{1{}3} \, \cF_{2{}3}  \; = \;  \cF_{2{}3} \, \cF_{1{}3} \, \cF_{1{}2}
\end{equation}
\end{definition}

\vskip9pt

\begin{rmks}  \label{rmks: about R-matrix twists}
 \textit{(a)}\,  If  $ H $  is a Hopf algebra (including a topological one)
 and there exists  $ \, \cF \in H \otimes H \, $  which is  \textsl{invertible\/}  and such that
\begin{equation}  \label{eq: quasi-cocommutative}
  \cF \, \Delta(x) \, \cF^{-1} = \Delta^{\!\text{op}}(x)   \qquad \qquad  \forall \;\; x \in H
\end{equation}
 then  $ H $  is said to be  \textsl{quasicocommutative}.
 If in addition  $ \cF $  obeys also  \eqref{eq: R-mat_properties},
 then  $ H $  itself is said to be  \textsl{quasitriangular}.
 Indeed, the standard notion of  ``$ R $--matrix''
 in literature usually demands the constraint  \eqref{eq: quasi-cocommutative}
 besides condition  \eqref{eq: R-mat_properties}.
                                                                                \par
   Our notion of  ``$ R $--matrix''  as in  Definition \ref{def: twist & R-mat.'s}\textit{(c)\/}
   above is given the name  ``weak  $ R $--matrix''  in  \cite{Ch},  Definition 1.1.
 \vskip3pt
   \textit{(b)}\,  Every  $ R $--matrix  in the sense of
   Definition \ref{def: twist & R-mat.'s}\textit{(c)\/}
   above is automatically  \textsl{unitary},  cf.\ \cite{Ch},  Lemma 1.2.
   Conversely, if  $ \cF $  is  \textsl{unitary\/}  and enjoys  \eqref{eq: R-mat_properties},
   then it is  \textsl{invertible\/}  too, hence it is an  $ R $--matrix.
   In short, the conditions ``invertible'' and ``unitary'' for an  $ R $--matrix
   are equivalent and interchangeable.
 \vskip3pt
   \textit{(c)}\,  If  $ \, \cR \, $  is an  $ R $--matrix  for  $ H $,
   then so is  $ \, {\big( \cR^{-1} \big)}_{2{}1} = {\big( \cR_{2{}1} \big)}^{-1} \, $;
   moreover,  $ \cR_{2{}1} $  and  $ \cR^{-1} $  are  $ R $--matrices  for  $ H^{\text{op}} $
   and  $ H^{\text{cop}} $  alike   --- see  \cite{Mj}, \cite{Ra}.
 \vskip3pt
   \textit{(d)}\,
 It follows by definitions that  \eqref{eq: twist-cond.'s}  and  \eqref{eq: R-mat_properties}
   jointly imply  \eqref{eq: qYB-equation x F},  while  \eqref{eq: R-mat_properties}  and
   \eqref{eq: qYB-equation x F}  jointly imply  \eqref{eq: twist-cond.'s}.
\end{rmks}

\vskip11pt

\begin{free text}  \label{twist-deform.'s}
 \textbf{Deformations by twist.}
 The importance of twists comes from the following construction.
 Let  $ H $  be a bialgebra (over some commutative ring  $ \k \, $),
 and let  $ \, \cF \in H \otimes H \, $  be a twist for it
 --- as in  Definition \ref{def: twist & R-mat.'s}\textit{(b)\/}  above.
 Then  $ H $  bears a second bialgebra structure, denoted  $ H^\cF $
 and called  \textit{twist deformation\/}  of the old one,
 with the old product, unit and counit, but with a new ``twisted'' coproduct
 $ \Delta^\cF $  given by
  $$  \Delta^{\!\cF}(x)  \, := \,  \cF \, \Delta(x) \, \cF^{-1}   \eqno  \forall \;\; x \in H   \qquad  $$
 If in addition  $ H $  is a Hopf algebra with antipode  $ \cS \, $,
 then this ``twisted'' bialgebra  $ H^\cF $
 is again a Hopf algebra with antipode  $ \cS^\cF $  given by
  $$  \cS^\cF(x) \, := \, v\,\cS(x)\,v^{-1}   \eqno  \forall \;\; x \in H   \qquad  $$
where  $ \, v := \sum_\cF \cS(f'_1)\,f'_2 \, $   --- with
$ \, \sum_\cF f'_1 \otimes f'_2 = \cF^{-1} \, $  ---
 is invertible in  $ H $  (see,  \cite[ \S 4.2.{\sl E}]{CP},  for further details).
 \textsl{When  $ H $  is in fact a
 \textit{topological}  bialgebra or Hopf algebra},
 then the same notions still make sense, and the related results apply again,
 up to properly reading them.
\end{free text}

\vskip7pt

   We present now the dual picture:

\vskip11pt

\begin{definition}  \label{def: 2-cocyc.'s & rho-comat.'s}
 Let  $ H $  be a bialgebra (possibly topological, over some commutative ground ring),
 and let  $ \, \sigma \in {\big( H^{\otimes 2} \big)}^* \, $.  Then:
 \vskip3pt
   \textit{(a)}\,  $ \sigma $  is said to be  \textit{unitary\/}  if
\begin{equation}  \label{eq: unitary-cond.'s x 2-cocycle}
  \sigma (a,1)  \; = \;  \epsilon(a)  \; = \;  \sigma(1,a)   \qquad \qquad  \forall \; a \in H
\end{equation}
 \vskip3pt
   \textit{(b)}\,  $ \sigma $  is called a  \textit{$ 2 $--cocycle\/}  if it is
   \textsl{(convolution) invertible\/}  in
   $ \, {\big( H^{\otimes 2} \big)}^* \, $,  it is unitary, and such that
\begin{equation}  \label{eq: 2-cocyc-cond.'s}
    \hskip-5pt
  \sigma(a_{(1)},b_{(1)}) \, \sigma(a_{(2)}b_{(2)},c)  \,\; = \;\,
  \sigma(b_{(1)},c_{(1)}) \, \sigma(a,b_{(2)}c_{(2)})
\end{equation}
 for all  $ \, a, b, c \in H \, $   --- where we abuse of notation identifying
 $ \, \sigma \in {\big( H \otimes H \big)}^* \, $  with the corresponding  $ \k $--bilinear  map
 $ \, \rho : H \times H \longrightarrow \k \, $,  and we adapt notation accordingly;
%%%%%
%  see  \cite[Sec.\ 7.1]{Mo};
%%%%%
%
 \vskip3pt
   \textit{(c)}\,  $ \sigma $  is called a  \textit{(quantum)  $ \varrho $--comatrix\/}  if
 it is  \textsl{(convolution) invertible\/}  in  $ \, {\big( H^{\otimes 2} \big)}^* \, $
   and --- for all  $ \, a, b, c \in H \, $  ---   we have
\begin{equation}  \label{eq: rho-comat_prop.'s}
  \sigma(a\,b\,,c) \, = \, \sigma\big(a\,,c_{(1)}\big) \, \sigma\big(b\,,c_{(2)}\big)  \;\; ,  \qquad
  \sigma(a\,,b\,c) \, = \, \sigma\big(a_{(1)},c\big) \, \sigma\big(a_{(2)},b\big)
\end{equation}
 \vskip3pt
   \textit{(d)}\,  $ \sigma $  is called a  \textit{(quantum)
   $ \varrho $--comatrix  $ 2 $--cocycle\/}
   if it complies with   \textit{(b)\/}  and   \textit{(c)\/};
 \vskip7pt
   \textit{(e)}\,  $ \sigma $  is said to be a  \textit{solution of the quantum
   Yang-Baxter equation (=QYBE)\/}  if
\begin{equation}  \label{eq: qYB-equation x sigma}
  \sigma_{1{}2} * \sigma_{1{}3} * \sigma_{2{}3}  \,\; = \;
  \sigma_{2{}3} * \sigma_{1{}3} * \sigma_{1{}2}
\end{equation}
 where hereafter  ``$ \, * \, $''  denotes the convolution product.
\end{definition}

\vskip9pt

\begin{rmks}  \label{rmks: about rho-comatrix 2-cocycles}
 \textit{(a)}\,  If  $ H $  is a Hopf algebra and there exists
 $ \, \sigma \in {(H \otimes H)}^* \, $  which is  \textsl{(convolution) invertible\/}  and such that
\begin{equation}  \label{eq: quasi-commutative}
  \sigma \, * \, m * \sigma^{-1}  \, = \;  m_{\,\text{op}}
\end{equation}
 then  $ H $  is said to be  \textsl{quasicommutative}.
 If in addition  $ \sigma $  obeys also  \eqref{eq: rho-comat_prop.'s},  then  $ H $
 itself is said to be  \textsl{coquasitriangular}.  Indeed, the standard notion
 of  ``$ \varrho $--comatrix'',
 or ``dual  $ R $--matrix'',  in literature usually demands
 \eqref{eq: quasi-commutative}  besides  \eqref{eq: rho-comat_prop.'s}.
                                                        \par
   Following the spirit of  \cite{Ch},  one might also use such terminology as
   ``weak  $ \varrho $--comatrix'',  or ``weak dual  $ R $--matrix''.
 \vskip3pt
   \textit{(b)}\,  Every  $ \varrho $--matrix  in the sense of
   Definition \ref{def: 2-cocyc.'s & rho-comat.'s}\textit{(c)\/}  above is  \textsl{unitary}.
   Conversely, if  $ \rho $  is  \textsl{unitary\/}  and enjoys  \eqref{eq: rho-comat_prop.'s},
   then it is  \textsl{(convolution) invertible\/}  too, hence it is a  $ \varrho $--comatrix
   (cf.\ \cite{Mj}, Lemma 2.2.2).  In short, the conditions ``invertible'' and ``unitary'' for a
   $ \varrho $--comatrix  are equivalent and interchangeable.
 \vskip3pt
   \textit{(c)}\,  Much like for  $ R $--matrices, if  $ \, \rho \, $  is a  $ \varrho $--comatrix
   for  $ H $,  then so is  $ \, {\big( \rho^{-1} \big)}_{2{}1} = {\big( \rho_{2{}1} \big)}^{-1} \, $;
   moreover,  $ \rho_{2{}1} $  and  $ \rho^{-1} $  are  $ \varrho $--comatrices  for
   $ H^{\text{op}} $  and  $ H^{\text{cop}} $  alike.
 \vskip3pt
   \textit{(d)}\,
   It follows by definitions that  \eqref{eq: 2-cocyc-cond.'s}  and
   \eqref{eq: rho-comat_prop.'s}  jointly imply  \eqref{eq: qYB-equation x sigma},
   while  \eqref{eq: rho-comat_prop.'s}  and  \eqref{eq: qYB-equation x sigma}
   jointly imply  \eqref{eq: 2-cocyc-cond.'s}.
\end{rmks}

\vskip9pt

\begin{free text}  \label{2-cocyc-deform.'s}
 \textbf{Deformations by  $ 2 $--cocycles.}
 We can use  $ 2 $--cocycles  to perform a different type of deformation.
 Let  $ H $  be a bialgebra (over some commutative ring  $ \k \, $),
 and let  $ \, \sigma \in {(H \otimes H)}^* \, $  be a $ 2 $--cocycle  for it.
 Then  $ H $  bears a second bialgebra structure, denoted  $ H_\sigma $
 and called  \textit{$ 2 $--cocycle  deformation\/}  of the old one, with the
 old coproduct, counit and unit, but with new product
 $ \, m_{\sigma} = \sigma * m * \sigma^{-1} : H \otimes H \longrightarrow H \, $
 given by
  $$
  m_{\sigma}(a,b)  \, = \,  a \cdot_{\sigma} b  \, =  \,
  \sigma(a_{(1)},b_{(1)}) \, a_{(2)} \, b_{(2)} \, \sigma^{-1}(a_{(3)},b_{(3)})
  \eqno \forall \;\, a, b \in H  \quad
  $$
 If in addition  $ H $  is a Hopf algebra with antipode  $ \cS \, $,
 then this ``deformed'' bialgebra  $ H_\sigma $  is again a Hopf algebra
 with antipode
 $ \; \cS_{\sigma}  \, $,
 \,which in detail reads
  $$  \cS_{\sigma}(a)  \, = \,  \sigma(a_{(1)},\cS(a_{(2)})) \, \cS(a_{(3)}) \,
  \sigma^{-1}(\cS(a_{(4)}),a_{(5)})
  \eqno \forall \;\, a \in H  \quad  $$
 (see  \cite{Doi}  for more details).
  \textsl{If  $ H $  is a  \textit{topological}  bialgebra or Hopf algebra},
  all this construction applies again, as well as the related results,
  up to technicalities.
\end{free text}

\vskip11pt

   The two notions of  ``$ 2 $--cocycle''  and of ``twist'',
   as well as the corresponding deformations,
   are so devised as to be dual to each other with respect to Hopf duality
   (cf.\ \cite{Mj}),  also in the setup of  \textsl{topological\/}
   Hopf algebras as with QUEA's and QFSHA's.
   The same holds for the notions of  ``$ \varrho $--comatrix''  and
   of  ``$ R $--matrix''.  All this is recorded in the following result,
   whose proof is trivial (an exercise in Hopf theory):

\vskip13pt

\begin{prop}  \label{prop: duality-deforms}
 Let  $ H $  be a Hopf algebra (possibly topological), and  $ H^* $
 its dual Hopf algebra (possibly in topological sense).
 \vskip5pt
   {\it (a)}\,  Let  $ \cF $  be a twist, resp.\ an  $ R $--matrix,  for  $ H \, $,
   and  $ \sigma_{{}_\cF} $  the image of  $ \cF $  in $ {(H \otimes H)}^* $
   for the natural embedding
   $ \, H \otimes H \lhook\joinrel\relbar\joinrel\longrightarrow H^{**} \otimes H^{**}
   \lhook\joinrel\relbar\joinrel\longrightarrow {\big( H^* \otimes H^* \big)}^* \, $.
   Then  $ \sigma_{{}_\cF} $  is a 2--cocycle, resp.\ a  $ \varrho $--comatrix,
   for  $ H^* \, $.  Moreover, in the first case there exists a canonical Hopf
   algebra isomorphism
  $ \, {\big( H^* \big)}_{\sigma_{{}_\cF}} \cong {\big( H^\cF \,\big)}^* \, $.
 \vskip3pt
   {\it (b)}\,  Let  $ \sigma $  be a 2--cocycle, resp.\ a  $ \varrho $--comatrix,
   for  $ H \, $;  assume that we have a natural identification
   $ \, {(H \otimes H)}^* = H^* \otimes H^* \, $
   (e.g., if  $ H $  is finite-dimensional), and let  $ \cF_\sigma $
   be the image of  $ \sigma \, $  in  $ \, H^* \otimes H^* \, $  via this identification.
   Then  $ \cF_\sigma $  is a twist, resp.\ an  $ R $--matrix,  for  $ H^* \, $.
   Moreover, in the first case there exists a canonical Hopf algebra isomorphism
   $ \, {\big( H^* \big)}^{\cF_\sigma} \cong {\big( H_\sigma \big)}^* \, $.
\qed
\end{prop}

\vskip11pt

\begin{free text}  \label{Hopf-morph.'s}
 \textbf{Hopf morphisms from  $ R $--matrices  and  $ \varrho \, $--comatrices.}
 Let  $ H $  be a Hopf algebra, possibly in topological sense.
 We assume that its (possibly topological) finite dual  $ H^* $
 is a Hopf algebra as well (possibly in a topological sense):
 the easiest example is when  $ H $  is defined over a field and it is finite dimensional,
 but the extension to less trivial cases is straightforward in a wide variety of situations.
                                                             \par
   Hereafter we recall some well-known constructions, somewhat shortly:
   further details can be found, e.g., in  \cite{CP},  \cite{KS}  and  \cite{Mj}.
\end{free text}

\vskip7pt

\begin{prop}  \label{prop: morph.'s from R-mat}  {\ }
 \vskip3pt
   (a)\;  Every  $ R $--matrix  $ \, \cR = \cR_1 \otimes \cR_2 \, $
   (using Sweedler's-like notation) for  $ H $
   provides two Hopf algebra morphisms
  $$
  \overleftarrow{\Phi}_{\!\cR} : H^* \!\longrightarrow H^{\text{\rm cop}}
  \;\;  \big(\, \eta \, \mapsto \, \eta(\cR_1) \, \cR_2 \,\big) \quad ,
  \qquad  \overrightarrow{\Phi}_{\!\cR} : H^* \!\longrightarrow H^{\text{\rm op}}
  \;\; \big(\, \eta \, \mapsto \, \cR_1 \, \eta(\cR_2) \,\big)
  $$
 \vskip1pt
   (b)\;
%   If an  $ R $--matrix  $ \, \cR \, $  for  $ H $  is invertible, then its inverse
%   $ \cR^{-1} $  is an  $ R $--matrix  for both  $ H^{\text{\rm cop}} $  and
%   $ H^{\text{\rm op}} $,  and  $ \, \overleftarrow{\Phi}_{\!\cR} \, $,
%   resp.\  $ \, \overrightarrow{\Phi}_{\!\cR} \, $,  is convolution invertible,
%   with convolution inverse  $ \, \overleftarrow{\Phi}_{\!\cR^{-1}} \, $,
%   resp.\  $ \, \overrightarrow{\Phi}_{\!\cR^{-1}} \, $.
If  $ \, \cR \, $  is an  $ R $--matrix  for  $ H $,  and
$ \cR^{-1} $  is its inverse, then  $ \, \overleftarrow{\Phi}_{\!\cR} \, $,
resp.\  $ \, \overrightarrow{\Phi}_{\!\cR} \, $,  is convolution invertible,
with convolution inverse  $ \, \overleftarrow{\Phi}_{\!\cR^{-1}} \, $,
resp.\  $ \, \overrightarrow{\Phi}_{\!\cR^{-1}} \, $.
\end{prop}

\begin{proof}
 \textit{(a)}\,  Consider for instance the case of  $ \overleftarrow{\Phi}_{\!\cR} \, $.
 The left-hand side, resp.\ the right-hand side, of  \eqref{eq: R-mat_properties}
 implies that it preserves multiplication, resp.\ comultiplication.
 On the other hand, from the right-hand side of  \eqref{eq: twist-cond.'s},
 the identity  $ \, \big( \epsilon \otimes \id \big)(\cR) = 1 \, $  implies that
 $ \overleftarrow{\Phi}_{\!\cR} \, $  preserves the unit, and the identity
 $ \, \big( \id \otimes \epsilon \big)(\cR) = 1 \, $  implies that it preserves the counit.
 The proof for  $ \overrightarrow{\Phi}_{\!\cR} $  is similar, namely it is left-right symmetric.
 \vskip3pt
   \textit{(b)}\,  This follows directly from definitions, through sheer bookkeeping.
\end{proof}

\vskip7pt

   The previous result has its dual counterpart, whose proof is again straightforward:

\vskip7pt

\begin{prop}  \label{prop: morph.'s from rho-comat}  {\ }
 \vskip3pt
   (a)\;  Every  $ \varrho $--comatrix  $ \, \rho \, $  for  $ H $
   provides two Hopf algebra morphisms
  $$
  \overleftarrow{\Psi}_{\!\rho} :
  H \longrightarrow {\big( H^* \big)}^{\text{\rm cop}} \, ,
  \;\; \ell \, \mapsto \, \rho(\ell \, , \, - \,)  \quad ,
  \qquad  \overrightarrow{\Psi}_{\!\rho} :
  H \longrightarrow {\big( H^* \big)}^{\text{\rm op}} \, ,  \;\; \ell \, \mapsto \,
  \rho(\, - \, , \ell \,)
  $$
 \vskip1pt
   (b)\;
%   If a  $ \varrho $--comatrix  $ \, \rho \, $  for  $ H $  is (convolution) invertible,
% then its inverse  $ \rho^{-1} $  is a  $ \varrho $--comatrix  for both
% $ H^{\text{\rm cop}} $  and  $ H^{\text{\rm op}} $,  and
% $ \, \overleftarrow{\Psi}_{\!\rho} \, $,  resp.\
% $ \, \overrightarrow{\Psi}_{\!\rho} \, $,  is convolution invertible, with
% convolution inverse  $ \, \overleftarrow{\Psi}_{\!\rho^{-1}} \, $,  resp.\
% $ \, \overrightarrow{\Psi}_{\!\rho^{-1}} \, $.
 If  $ \, \rho \, $  is a  $ \varrho $--comatrix  for  $ H $,  and  $ \rho^{-1} $
 is its (convolution) inverse, then  $ \, \overleftarrow{\Psi}_{\!\rho} \, $,  resp.\
  $ \, \overrightarrow{\Psi}_{\!\rho} \, $,  is convolution invertible, with
  convolution inverse  $ \, \overleftarrow{\Psi}_{\!\rho^{-1}} \, $,  resp.\
  $ \, \overrightarrow{\Psi}_{\!\rho^{-1}} \, $.
   \hfill  $ \square $
\end{prop}

\vskip9pt

\begin{rmk}  \label{rmk: duality morph.'s from R-mat/rho-comat}
 Inasmuch as any  $ R $--matrix,  resp.\ any  $ \varrho $--comatrix,
 for  $ H $  is a  $ \varrho $--comatrix, resp.\ an  $ R $--matrix,
 for the dual Hopf algebra  $ H^* $
 --- cf.\  Proposition \ref{prop: duality-"polar"deforms}  ---
 applying  Proposition \ref{prop: morph.'s from R-mat}  to  $ H^* $  we get
 Proposition \ref{prop: morph.'s from rho-comat},  and, conversely,
 applying  Proposition \ref{prop: morph.'s from rho-comat}  to  $ H^* $
 we get  Proposition \ref{prop: morph.'s from R-mat}.
 In the same spirit, the following result about Hopf algebras in duality
 follows from the very definitions:
\end{rmk}

\vskip7pt

\begin{prop}  \label{prop: duality-morph.'s_R-mat/rho-comat}
 Let  $ K $  and  $ \varGamma $  be two Hopf algebras
 (over the same ground ring, and possibly topological)
 that are dual to each other, say  $ \, \varGamma = K^* \, $  and
 $ \, K = \varGamma^\star \, $  for suitably defined dual functors  $ \, {(\ )}^* $
 and  $ \, {(\ )}^\star \, $.  Let also  $ \, \cR = \rho \, $  be an  $ R $--matrix  for
 $ K $  and a  $ \varrho $--comatrix  for  $ \varGamma $
 --- applying  Proposition \ref{prop: duality-deforms}.
 Then for the morphisms in  Proposition \ref{prop: morph.'s from R-mat}
 and  Proposition \ref{prop: morph.'s from rho-comat}  we have canonical
 identifications
  $$
  \overleftarrow{\Phi}_{\!\cR} \! = \! \overleftarrow{\Psi}_{\!\rho} \, : \, K^* \! =
  \varGamma \! \relbar\joinrel\longrightarrow \! {\big( \varGamma^\star \! =
  K \big)}^{\text{\rm cop}}
   \; ,  \hskip9pt
      \overrightarrow{\Phi}_{\!\cR} \! = \! \overrightarrow{\Psi}_{\!\rho} \, : \, K^* \! =
      \varGamma \! \relbar\joinrel\longrightarrow \! {\big( \varGamma^\star \! =
      K \big)}^{\text{\rm op}}  \hskip11pt   \eqno \square
      $$
\end{prop}

\medskip

\subsection{Quantum groups}  \label{subsec: QG's & QG-deform's}  {\ }
 \vskip7pt
   We recall hereafter the basic notions on quantum groups,
   in the shape of either quantized universal enveloping algebras (=QUEA's)
   or quantized formal series Hopf algebras (=QFSHA's).
   Both types of ``quantum groups'' are Hopf algebras in a topological sense,
   that we shall presently fix.

\vskip13pt

\begin{free text}  \label{prel_Q-Groups}
 \textbf{Classical and quantum preliminaries.}
 Hereafter we fix a base field  $ \Bbbk $  of characteristic zero.
 We recall the following from  \cite{CP}.
 \vskip5pt
   For any Lie algebra  $ \lieg $  over  $ \Bbbk \, $,
   its universal enveloping algebra  $ U(\lieg) $
   has a canonical structure of Hopf algebra,
   which is cocommutative and connected.
   If  $ \lieg $  is also a Lie bialgebra, with Lie cobracket
   $ \delta \, $,  then  $ \delta $  uniquely extends to define a
   \textsl{Poisson cobracket\/}
   $ \, \delta : U(\lieg) \longrightarrow U(\lieg) \otimes U(\lieg) \, $,
   just by imposing that it fulfill the co-Leibnitz identity
   $ \; \delta(x\,y) \, = \, \delta(x) \, \Delta(y) + \Delta(x) \, \delta(y) \; $.
   Conversely, if the Hopf algebra  $ U(\lieg) $  is actually even a Hopf
   \textsl{co-Poisson\/}  algebra, then its Poisson co-bracket  $ \delta $
   maps  $ \lieg $  into  $ \, \lieg \otimes \lieg \, $,
   thus yielding a Lie cobracket for  $ \lieg $  that makes the latter into a
   Lie bialgebra.
                                                                    \par
   Dually, let  $ G $  be any formal algebraic group  $ G $  over  $ \Bbbk \, $:
   by this we loosely mean that  $ G $  is the spectrum of its formal function algebra
   $ \fg \, $,  the latter being a topological Hopf algebra which is commutative and
   $ I $--adically  complete, where  $ \, I := \textsl{Ker}\,(\epsilon) \, $
   is the augmentation ideal of $ \fg \, $.  Then  $ G $  is a (formal) Poisson group
   if and only if its formal function algebra  $ \fg $  is a  \textsl{Poisson\/}
   (formal) Hopf algebra, with respect to some Poisson bracket  $ \{\,\ ,\ \} \, $.
   In this case, the cotangent space  $ \, I \big/ I^2 \, $  of  $ G $
   has a Lie bracket induced by  $ \{\,\ ,\ \} $  via
   $ \; [x,y] := \big\{ x' , y' \big\} \; \big( \text{mod\ } I^2 \,\big) \; $  for all
   $ \, x , y \in I \big/ I^2 \, $  with  $ \, x = x' \; \big( \text{mod\ } I^2 \,\big) \, $,
   $ \, y = y' \; \big( \text{mod\ } I^2 \,\big) \, $:  this makes  $ I \big/ I^2 $
   into a Lie algebra, but its dual  $ \, \lieg = \Lie\,(G) := {\big( I \big/ I^2 \,\big)}^* \, $
   is also a Lie algebra (the tangent Lie algebra to  $ G \, $)
   and the two structures are compatible, so that
   $ \, \lieg^\star := I \big/ I^2 \, $  is a  \textsl{Lie bialgebra\/}  indeed.
 \vskip5pt
   We come now to  \textsl{quantizations\/}
   of the previous co-Poisson/Poisson structures.
 \vskip5pt
   Let  $ \cT_\otimeshat $  be the category whose objects are all topological
   $ \Bbbk[[h]] $--modules  which are topologically free
 (i.e.~isomorphic to  $ V[[h]] $  for some  $ \Bbbk $--vector  space  $ V $,
 with the  $ h $--adic topology)
 and whose morphisms are the  $ \Bbbk[[h]] $--linear  maps (which are
 automatically continuous).  This is a tensor category w.r.t.~the tensor product
 $ \, T_1 \otimeshat T_2 \, $  defined to be the separated  $ h $--adic
 completion of the algebraic tensor product
 $ \, T_1 \otimes_{\Bbbk[[h]]} T_2 \, $  (for all  $ T_1 $,  $ T_2 \in \cT_\otimeshat $).
                                            \par
   Let  $ \cP_\otimestilde $  be the category whose objects are all topological
   $ \Bbbk[[h]] $--modules  isomorphic to modules of the type  $ {\Bbbk[[h]]}^E $
%
% (the Cartesian product indexed by  $ E $,  with the Tikhonov product topology)
%
 for some set  $ E \, $:  \,these are complete w.r.t.~to the weak topology
%
% , in fact they are isomorphic to the projective limit of their finite free submodules (each one taken with the  $ h $--adic  topology);
%
and whose morphisms in  $ \cP_\otimestilde $  are the  $ \Bbbk[[h]] $--linear
 continuous maps.
 This is a tensor category w.r.t.~the tensor product
 $ \, P_1 \otimestilde P_2 \, $  defined to be the completion of the
 algebraic tensor product  $ \, P_1 \otimes_{\Bbbk[[h]]} P_2 \, $
 w.r.t.~the weak topology:  therefore
 $ \, P_i \cong {\Bbbk[[h]]}^{E_i} $  ($ i = 1 $,  $ 2 $)  yields
 $ \, \, P_1 \otimestilde P_2 \cong {\Bbbk[[h]]}^{E_1 \times E_2} \, $
 (for all $ P_1 $,  $ P_2 \in \cP_\otimestilde $).
                                            \par
   Note that the objects of  $ \cT_\otimeshat $  and of  $ \cP_\otimestilde $
   are complete and separated w.r.t.~the
$ h $--adic  topology, so one has  $ \, X \cong X_0[[h]] \, $
for every such object  $ X $,  with  $ \, X_0 := X \Big/ \hbar\,X \, $.
                                            \par
   We denote by  $ \cHA_\otimeshat $  the subcategory of  $ \cT_\otimeshat $
   whose objects are all the Hopf algebras in  $ \cT_\otimeshat $
   and whose morphisms are all the Hopf algebra morphisms in $ \cT_\otimeshat $.
   Similarly, we call  $ \cHA_\otimestilde $  the subcategory
of  $ \cP_\otimestilde $  whose objects are all the Hopf algebras in
$ \cP_\otimestilde $  and whose morphisms are all the Hopf algebra morphisms in
$ \cP_\otimestilde $.  To simplify notation, we shall usually drop the subscripts
``$ \; \widehat{\ } \; $''  and  ``$ \; \widetilde{\ } \; $''  from the symbol
``$ \, \otimes \, $''.
 \vskip3pt
Finally, when dealing with any  $ \kh $--module  $ M $  by such notation as  $ \cO\big( \hbar^s \big) $
 we shall mean any (unspecified) element belonging to  $ \, \hbar^s \, M \, $,  \,for all  $ \, s , n \in \NN \, $;
 \,in other words, for any  $ \, x \in M \, $  by writing  $ \, x = \cO\big( \hbar^s \big) \, $  we mean that
 $ \, x \equiv 0 \, \big(\, \text{mod} \; \hbar^s M \,\big) \, $.
\end{free text}

\vskip3pt

   We are ready now to define quantum groups, in two different incarnations:

\vskip11pt

\begin{free text}  \label{QUEA's}
 \textbf{Quantized Universal Enveloping Algebras (=QUEA's).}
 Retain notation as in  \S \ref{prel_Q-Groups} above.
                                               \par
   A  \emph{quantized universal enveloping algebra}
   --- or QUEA in short ---   is a (topological) Hopf algebra
   $ U_\hbar $  in  $ \cHA_\otimeshat $  such that
 $ \, U_0 := U_\hbar \big/ \hbar\,U_\hbar \, $  is a connected,
 cocommutative Hopf algebra over  $ \Bbbk $   --- or, equivalently,
 $ \, U_0 \, $  is isomorphic to an enveloping algebra  $ U(\lieg) $	
 for some Lie algebra  $ \lieg \, $.  Then the formula
\begin{equation}  \label{eq: def-cobracket}
   \delta(x) \, := \, \frac{\; \Delta\big(x'\big) - \Delta^{\text{op}}\big(x'\big) \;}{\hbar} \;
   \mod \hbar\,U_\hbar^{\,\widehat{\otimes}\,2}
\end{equation}
 --- where  $ \, x' \in U_\hbar \, $  is any lift of  $ \, x \in \lieg \, $
 ---   defines a co-Poisson structure on  $ \, U_0 = U(\lieg) \, $,
 \,hence a Lie bialgebra structure on  $ \lieg \, $.
 In this case, we say that  $ U_\hbar $  is a  \emph{quantization\/}
 of the co-Poisson Hopf algebra  $ U(\lieg) \, $,  or
 (with a slight abuse of language) of the Lie bialgebra  $ \lieg \, $;  \,conversely,
 $ U(\lieg) $   --- or just  $ \lieg $  alone ---   is the  \textit{semiclassical limit\/}  of
 $ U_\hbar \, $.  We summarize it writing  $ \, \uhg := U_\hbar \, $.
 In the following, we denote by  $ \, \QUEA \, $
  the full subcategory of  $ \, \cHA_\otimeshat $  whose objects are all of the QUEAs.
\end{free text}

\vskip9pt

\begin{free text}  \label{QFSHA's}
 \textbf{Quantized Formal Series Hopf Algebras (=QFSHA's).}
 Retain again notation as in  \S \ref{prel_Q-Groups} above.
                                               \par
   A  \emph{quantized formal series Hopf algebra}
   --- or QFSHA in short ---   is a (topological) Hopf algebra  $ F_\hbar $  in
   $ \cHA_\otimestilde $  such that
 $ \, F_0 := F_\hbar \big/ \hbar\,F_\hbar \, $  is a commutative,
 $ I $--adically  complete topological Hopf algebra over  $ \Bbbk \, $,
 where  $ I $  is the augmentation ideal   --- or, equivalently,  $ \, F_0 \, $
 is isomorphic to the algebra of functions for some
 formal algebraic group  $ \fg \, $.
 Then the formula
\begin{equation}  \label{eq: def-bracket}
   \{x,y\} \, := \, \frac{\; x' \, y' \, - \, y' \, x' \;}{\;\hbar\;} \; \mod \hbar\,F_\hbar
\end{equation}
 --- where  $ \, x' , y' \in F_\hbar \, $  are lifts of  $ \, x, y \in \fg \, $  ---
 defines a Poisson bracket in  $ \fg \, $,  thus making  $ G $
 into a (formal) Poisson group.  In this case, we say that  $ F_\hbar $
 is a  \emph{quantization\/}  of the Poisson Hopf algebra  $ \fg \, $,  or
 (stretching a point) of the formal Poisson group  $ G \, $;  \,conversely,  $ \fg $
 --- or just  $ G $  alone ---   is the  \textit{semiclassical limit\/}  of  $ F_\hbar \, $.
 We summarize it writing  $ \, \fhg := F_\hbar \, $.
 In the following, we denote by  $ \, \QFSHA \, $  the full subcategory of
  $ \, \cHA_\otimestilde $  whose objects are all of the QFSHAs.
\end{free text}

\vskip9pt

\begin{free text}  \label{equiv-&-(stand)-duality}
 \textbf{Equivalence and duality between quantizations.}  If  $ H_1 $,  $ H_2 $,
 are two QUEA's, respectively two QFSHA's, we say that  \textsl{$ H_1 $
 is equivalent to  $ H_2 $},  and we write  $ \, H_1 \equiv H_2 \, $,
 if there is an isomorphism $ \, \varphi \, \colon H_1 \cong H_2 \, $
 (in  $ \QUEA $,  resp.\ in  $ \QFSHA $)
 such that  $ \, \varphi = \hbox{\it id} \mod h \, $.
 In particular, in both cases the semiclassical limit of either
 $ H_1 $  or  $ H_2 $  is the same.
 \vskip5pt
%
%%%%%
%  Clearly QUEA's and QFSHA's form categories, denoted  $ \QUEA $
% and  $ \QFSHA $  respectively; by construction, these
%%%%%
By their very construction, the categories
  $ \QUEA $  and  $ \QFSHA $  are dual to each other (w.\ r.\ to the natural,
  topological linear duality functors in both directions).
 In detail, by  \textit{dual\/}  of any  $ \, U_\hbar \in \QUEA \, $,  denoted
 $ U_\hbar^{\,*} $,  we take the set of all  $ \kh $--linear  functions from
 $ U_\hbar $  to  $ \kh $  (which are automatically continuous w.\ r.\ to the  $ \hbar $--adic
 topology): this is naturally an object in $ \QFSHA \, $.
 On the other hand, by  \textit{dual\/}  of any  $ \, F_\hbar \in \QFSHA \, $,
 denoted  $ F_\hbar^{\,\star} $,  we take the set of all maps from  $ F_\hbar $  to
 $ \kh $  that are continuous with respect to the  $ \hbar $--adic  topology  on  $ \kh $
 and to the  $ I_\hbar $--adic  topology on  $ F_\hbar \, $,
 with  $ \, I_\hbar := \hbar\,F_\hbar + \Ker\big(\epsilon_{F_\hbar}\big) \, $;
 this  $ F_\hbar^{\,\star} $  is an object in  $ \QUEA \, $.  Finally,  $ {(\ )}^* $
 and  $ {(\ )}^\star $  are contravariant functors inverse to each other
   --- cf.\ \cite{Ga1}.
\end{free text}

\vskip9pt

   We finish this part with a trivial, technical result, that we will use several times:

\vskip13pt

\begin{lema}   \label{lemma: technic-Hopf}
 Let  $ H $  be a Hopf algebra (possibly topological).
 We denote by  $ \, [\ \,,\ ] \, $  the commutator operation in  $ H $,
 and write  $ \, H^+ := \Ker\,(\epsilon) \, $.   Then:
 \vskip3pt
   (a)\,  There exists a splitting into direct sum  $ \, H = \Bbbk \oplus H^+ \, $.
   With respect to that splitting, every  $ \, z \in H \, $  uniquely splits into
   $ \, z = \epsilon(z) + z^+ \, $  with  $ \, z^+ := z - \epsilon(z) \in H^+ \, $.
 \vskip2pt
   (b)\,  For any  $ \, x, y \in H \, $  we have  $ \; [x\,,y\,] = \big[x^+\,,y^+\big] \; $
%%%%%
%  --- notation of claim (a).  Then  $ \; [H\,,H\,] \subseteq H^+ \, $,  \,and in general
% $ \, \big[ {(H^+)}^{\,r} , {(H^+)}^{\,s} \big] \subseteq {(H^+)}^{\,r+s-1} \, $  for all
% $ \, r, s \in \NN_+ \, $.
--- see (a) ---   so  $ \; [H\,,H\,] \subseteq H^+ \, $.
 \vskip2pt
   (c)\,  Assume that  $ \, H = \fhg \, $  is a QFSHA, with  $ \, J_\hbar := H^+ \, $.
   Then we have  $ \; [H\,,H\,] = \big[ J_\hbar \, , J_\hbar \,\big] \subseteq \hbar \, J_\hbar \, $,
   \,and more in general (for all  $ \, k, r_1, r_2, r_3, \dots , r_k, s \in \NN_+ \, $)
  $$
  \big[ J_\hbar^{\,r_1} , \big[ J_\hbar^{\,r_2} , \big[ J_\hbar^{\,r_3} ,
  \cdots \big[ J_\hbar^{\,r_k} , J_\hbar^{\,s} \big] \cdots \big] \big] \big]  \; \subseteq \;
  (1 - \delta_{s,0}) \, {\textstyle \prod\limits_{i=1}^k} (1 - \delta_{r_i,0}) \,
  \hbar^k \, J_\hbar^{\,r_1 + r_2 + r_3 + \cdots + r_k + s - k}
  $$
 \vskip0pt
   (d)\,  For any  $ \, z \in H \, $,  \,we have
  $$
  \Delta(z)  \; = \;  \epsilon(z) \cdot 1 \otimes 1 \, + \, z^+ \otimes 1 \, + \, 1 \otimes z^+ \, +
  \, {\big(z_{(1)}\big)}^+ \otimes {\big(z_{(2)}\big)}^+
  $$
 or also
  $$
  \Delta(z)  \; = \;  -\epsilon(z) \cdot 1 \otimes 1 \, + \, z \otimes 1 \, +
  \, 1 \otimes z \, + \, {\big(z_{(1)}\big)}^+ \otimes {\big(z_{(2)}\big)}^+
  $$
 in particular, for  $ \; \nabla := \Delta - \Delta^{\!\text{\rm op}} \, $  this yields
 $ \; \nabla(z) \in {\Ker\,(H)}^{\otimes 2} \, $,  \,with
  $$
  \nabla(z)  \; = \;  {\big(z_{(1)}\big)}^{\!+} \!\otimes\! {\big(z_{(2)}\big)}^{\!+} \! -
  {\big(z_{(2)}\big)}^{\!+} \!\otimes\! {\big(z_{(1)}\big)}^{\!+}   \eqno \qed
  $$
\end{lema}

\medskip

\subsection{The Quantum Duality Principle}  \label{subsec: QDP}  {\ }
 \vskip7pt
   We recall hereafter the main facets of the so-called ``Quantum Duality Principle'',
   which establishes an explicit equivalence between the category of QUEA's
   and that of QFSHA's (whereas linear duality provides an
   \textsl{anti\/}equivalence  instead).  Further details can be found in  \cite{Ga1}.

\medskip

\begin{definition}  \label{def: Drinfeld's functors}
 \textit{(Drinfeld's functors)}  We define Drinfeld's functors from  $ \QUEA $  to
 $ \QFSHA $  and viceversa as follows:
 \vskip3pt
   \textit{(a)}\,  Let  $ \uhg $  be any QUEA, and assume for simplicity that
   $ \lieg $  is finite-dimensio\-nal.
   Let  $ \, \iota : \kh \relbar\joinrel\longrightarrow \uhg \, $  and
   $ \, \epsilon : \uhg \relbar\joinrel\longrightarrow \kh \, $
   be its unit and counit maps; moreover, for every  $ \, n \in \NN \, $  set
   $ \, \delta_n := {(\id - \iota \circ \epsilon)}^{\otimes n} \circ \Delta^{(n-1)} \, $   --- mapping  $ \uhg $  to  $ {\uhg}^{\widehat{\otimes}\, n} $.  Then we define
 \vskip5pt
   \hfill   $ {\uhg}'  \; := \;  \Big\{\, \eta \in \uhg \,\Big|\; \delta_n(\eta) \in
   \hbar^n\,{\uhg}^{\otimes n} \;\; \forall \; n \in \NN \,\Big\} $   \hfill {\ }
 \vskip3pt
   This defines the functor  $ {(\ )}' $, from  $ \QUEA $  to  $ \QFSHA \, $,
 onto objects: then onto morphisms it is clearly defined by taking restriction.
 \vskip3pt
   \textit{(b)}\,  Let  $ \fhg $  be any QFSHA,
   and assume for simplicity that  $ G $  be finite-dimensional.
   Let  $ \, \epsilon_{\scriptscriptstyle F} : \fhg \relbar\joinrel\longrightarrow \kh \, $
   be its counit map, and consider also
   $ \, I_{\fhg} := \hbar\,\fhg + \textsl{Ker\/}(\epsilon_{\scriptscriptstyle F}) \, $.
   Then we define
 \vskip5pt
   \qquad \qquad \hfill   $ {\fhg}^\vee  \; := \;  \hbar $--adic  completion of
   $ \, \sum_{n \geq 0} \hbar^{-n} I_{\fhg}^{\;{}^{\scriptstyle n}} $   \hfill {\ }
 \vskip3pt
   This defines the functor  $ {(\ )}^\vee $, from  $ \QFSHA $  to  $ \QUEA \, $,
 onto objects: onto morphisms, we define it via scalar extension
 --- from  $ \kh $  to  $ \khp $  ---   followed by restriction and completion.
 \hfill  $ \diamondsuit $
\end{definition}

\vskip9pt

   The original recipes for these functors were given in  \cite[\S 7]{Dr};
   the corresponding proofs (that everything is well-defined, etc.) can be found in  \cite{Ga1}.
   Indeed, the overall result is very strong, involving linear duality for Lie bialgebras,
   as follows:

\vskip11pt

\begin{theorem}  \label{thm: QDP}
 \textit{(``The quantum duality principle''; cf.~\cite{Dr}, \cite{Ga1})}
 \vskip3pt
   (a)\,  The assignments  $ \, H \mapsto H' \, $  and  $ \, H \mapsto H^\vee \, $
   respectively define functors of tensor categories
   $ \, \QUEA \relbar\joinrel\longrightarrow \QFSHA \, $  and
   $ \, \QFSHA \relbar\joinrel\longrightarrow \QUEA \, $,  that are
   \textsl{inverse to each other},  thus yielding an equivalence of catefories.
 \vskip3pt
   (b)\,  For all  $ \, \uhg \in \QUEA \, $  and all  $ \, \fhg \in \QFSHA \, $  one has
  $$  {\uhg}' \Big/ h \, {\uhg}' \, = \; F[[G^*\hskip0,7pt]] \quad ,
\qquad\;  {\fhg}^\vee \Big/ h \, {\fhg}^\vee \, = \; U(\lieg^*)  $$
 \vskip-3pt
\noindent
 that is,
 if  $ \, \uhg $  is a quantization of  $ \, \ug $  then  $ \, {\uhg}' $  is a quantization of
 $ \, F[[G^*\hskip0,7pt]] \, $,  and if  $ \, \fhg $  is a quantization of  $ \, \fg $  then
 $ \, {\fhg}^\vee $  is a quantization of  $ \, U(\lieg^*) \, $.
 \vskip3pt
   (c)\,  Both Drinfeld's functors preserve equivalence, that is
   $ \, H_1 \equiv H_2 \, $  implies that  $ \, {H_1}' \equiv {H_2}' \, $  and
   $ \, {H_1}^{\!\vee} \equiv {H_2}^{\!\vee} \, $  in either case.   \hfill  $ \square $
\end{theorem}

\vskip7pt

   In a very precise sense, Drinfeld's functors are dual to each other: namely
   --- cf.\ \cite{Ga1}  ---   one has (with notation as in  \S \ref{equiv-&-(stand)-duality})
\begin{equation}   \label{eq: QDP-duality}
  {\big( {\uhg}^* \big)}^{\!\vee}  =  {\big( {\uhg}' \,\big)}^\star   \qquad  \text{and}
  \qquad   {\big( {\fhg}^\vee \,\big)}^* \! = \! {\big( {\fhg}^\star \big)}'
\end{equation}
                                                                   \par
   On the other hand, it is worth stressing a strong asymmetry between these
   functors.  Indeed, the definition of  $ {\fhg}^\vee $  is pretty  \textsl{concrete\/}
   (through an explicit generating  procedure) whereas that of  $ {\uhg}' $  is somewhat
   \textsl{implicit\/}  (it is described as the set of solution of a system of countably many
   equations), hence way more tough to work out.  \textsl{Nevertheless},
   an alternative description for  $ {\uhg}' $  is available, that we shall make use of in the
   later on, namely the following  (cf.\ \cite[Proposition 3.1.2]{Ga2}):

\vskip11pt

\begin{prop}  \label{prop: exist-lifts-x_i}
 For any\/  $ \k $--basis  $ {\{ \overline{y}_i \}}_{i \in I} $  of\/  $ \lieg \, $,
 \,there are  $ \, y_i \in \uhg \, $  such that:
 \vskip3pt
   (a)\;  $ \, \epsilon(y_i) = 0 \, $,  $ \, \big(\, y_i \!\mod \hbar\,\uhg \big) =
   \overline{y}_i \, $  and  $ \; y'_i := \hbar\,y_i \, \in \, {\uhg}' \, $  for all
   $ \, i \in I \, $;
 \vskip3pt
   (b)\;  $ \, {\uhg}' $  is the completion of the unital\/  $ \kh $--subalgebra  of
   $ \, \uhg $  generated by all the  $ x'_i $'s  with respect to its  $ I'_\hbar $--adic
   topology, where  $ I'_\hbar $  is the ideal (in that subalgebra) generated by
   $ \hbar $  and all the  $ x'_i $'s,  so that
%%%%%
  $ \; {\uhg}' = \Bbbk\big[\big[ {\{ x'_i \}}_{i \in I} \cup \{\hbar\} \big]\big] \; $.
  \hfill  $ \square $
%%%%%
%
\end{prop}

\vskip21pt

\section{Deformations of quantum groups}  \label{sec: deform.'s-QG's}

   This section is dedicated to explore the effect of deformations of quantum groups,
   either by twist or by 2-cocycle, seeting the cases of QUEA's and QFSHA's apart.

\vskip13pt

\subsection{Deformations by twist of QUEA's}  \label{subsec: twist-QUEAs}  {\ }
 \vskip7pt
   In this subsection we consider deformations by twist of QUEA's;
   in some sense, this is the easiest case.  We begin with a technical result:

\vskip5pt

\begin{lema}  \label{lemma: twist-cond's/exp ==> twist-cond's/log}
 Let  $ H $  be an  $ \hbar $--adically complete Hopf algebra over  $ \kh \, $,
 and let consider an element of the form
 $ \, \cF = \exp\!\big( \hbar\,\varphi \big) \in H \otimes H $,  with
 $ \, \varphi = \varphi_1 \otimes \varphi_2 \in H^{\otimes 2} \, $,
 \,such that  $ \, (\epsilon \otimes \id)(\cF\,) = 1 = (\id \otimes \epsilon)(\cF\,) \, $.
 Then
  $$
  \epsilon(\varphi_1) \otimes \varphi_2 \, = \, 0 \;\; ,
  \qquad  \varphi_1 \otimes \epsilon(\varphi_2) \, = \, 0 \;\; ,
  \qquad  \epsilon(\varphi_1) \otimes \epsilon(\varphi_2) \, = \, 0
  $$
 As a consequence, one can assume
 $ \, \varphi_1 = \varphi_1^+ , \, \varphi_2 = \varphi_2^+ \in \Ker(\epsilon) \, $,
 \,so  $ \, \varphi \in \Ker(\epsilon)^{\,\otimes\, 2} \, $.
\end{lema}

\begin{proof}
 By assumption one has  $ \, (\epsilon \otimes \id)(\cF\,) = 1 \, $.  Since
 $ \, \cF = \exp\!\big( \hbar\,\varphi \big) \, $,  this yields
  $$
  \displaylines{
   \hbar \, \epsilon(\varphi_1) \otimes \varphi_2  \; = \;
   (\epsilon \otimes \id)(\hbar \, \varphi)  \; = \;
   (\epsilon \otimes \id)\big(\log(\cF\,)\big)  \; =   \hfill  \cr
   = \;  (\epsilon \otimes \id)\bigg(\, {\textstyle \sum\limits_{n>0}} \,
   \frac{{(-1)}^{n-1}}{\;n\;} \, {\big( \cF - 1 \big)}^n \bigg)  \; = \;
   {\textstyle \sum\limits_{n>0}} \, \frac{{(-1)}^{n-1}}{\;n\;} \,
   {\big( (\epsilon \otimes \id)(\cF\,) - 1 \big)}^n  \; = \;  0  }
   $$
 thus  $ \, \hbar \, \epsilon(\varphi_1) \otimes \varphi_2 = 0 \, $,
 whence  $ \, \epsilon(\varphi_1) \otimes \varphi_2 = 0 \, $  as claimed.
 Similarly, the condition  $ \, (\id \otimes \epsilon)(\cF\,) = 1 \, $  implies
 $ \, \varphi_1 \otimes \epsilon(\varphi_2) = 0 \, $.
 Finally, expanding  $ \varphi_i $  as
 $ \, \varphi_i = \epsilon(\varphi_i) + \varphi_i^+ \, $  ($ \, i \in \{1\,,2\} \, $)
 --- cf.\ Lemma \ref{lemma: technic-Hopf}  ---   and using
 $ \, \epsilon(\varphi_1) \otimes \varphi_2 = 0 =
 \varphi_1 \otimes \epsilon(\varphi_2) \, $  we get
  $$  \displaylines{
   \varphi  \; = \;  \varphi_1 \otimes \varphi_2  \; = \;
%
%%%%%
% \big( \epsilon(\varphi_1) + \varphi_1^+ \big) \otimes
% \big( \epsilon(\varphi_2) + \varphi_2^+ \big)  \; = \;
%%%%%
%
 \epsilon(\varphi_1) \otimes \epsilon(\varphi_2) +
 \epsilon(\varphi_1) \otimes \varphi_2^+ + \varphi_1^+
\otimes \epsilon(\varphi_2) + \varphi_1^+ \otimes \varphi_2^+  \; =   \hfill  \cr
   = \;   \epsilon(\varphi_1) \otimes \epsilon(\varphi_2) +
   \epsilon(\varphi_1) \otimes \varphi_2 -
   \epsilon(\varphi_1) \otimes \epsilon(\varphi_2) +
   \varphi_1 \otimes \epsilon(\varphi_2) - \epsilon(\varphi_1) \otimes
   \epsilon(\varphi_2) + \varphi_1^+ \otimes \varphi_2^+  \; =  \cr
   \hfill   = \;   - \epsilon(\varphi_1) \otimes \epsilon(\varphi_2) +
   \epsilon(\varphi_1) \otimes \varphi_2 + \varphi_1 \otimes \epsilon(\varphi_2) +
   \varphi_1^+ \otimes \varphi_2^+  \; =  \cr
   \hfill   = \;   - \epsilon(\varphi_1) \otimes \epsilon(\varphi_2) + \varphi_1^+ \otimes \varphi_2^+  }  $$
 in short  $ \, \varphi = - \epsilon(\varphi_1) \otimes \epsilon(\varphi_2) +
 \varphi_1^+ \otimes \varphi_2^+ \, $.
 But  $ \; \epsilon(\varphi_1) \otimes \epsilon(\varphi_2) =
 (\id \otimes \epsilon)\big(\epsilon(\varphi_1) \otimes \varphi_2\big) = 0 \, $,
 \,hence  $ \; \varphi \, = \, \varphi_1^+ \otimes \varphi_2^+ \, \in \,
 {\Ker(\epsilon)}^{\otimes\, 2} $,  \,q.e.d.
\end{proof}

\vskip3pt

   We are now ready for our first meaningful result:

\vskip3pt

\begin{theorem}  \label{thm: twist-deform-QUEA}  {\ }
 \vskip3pt
   Let  $ \uhg $  be a QUEA over the Lie bialgebra
   $ \, \lieg = \big(\, \lieg \, ; \, [\,\ ,\ ] \, , \, \delta \,\big) \, $.
   Let  $ \, \cF \in {\uhg}^{\widehat{\otimes}\,2} \, $  be a twist for
   $ \uhg $  such that
   $ \; \cF \equiv 1 \; \Big(\, \text{\rm mod} \; \hbar \, \uhg^{\,\widehat{\otimes}\, 2} \,
   \Big) \, $;  \,then
   $ \, \kappa := \hbar^{-1} \log(\cF\,) \in {\uhg}^{\,\widehat{\otimes}\, 2} \, $,
   \,and  $ \, \cF = \exp\big( \hbar\,\kappa \big) \, $.
   Last, we set  $ \, \kappa_a := \kappa - \kappa_{2,1} \, $.  Then we have:
 \vskip5pt
   (a)\;  $ \kappa $  is antisymmetric, i.e.\  $ \, -\kappa = \kappa_{2,1} \, $,
   iff\/  $ \cF $  is orthogonal, i.e.\  $ \, \cF^{-1} = \cF_{2,1} \, $;
 \vskip5pt
   (b)\;  the element  $ \, c := \,  \overline{\kappa_a} \, = \,
   \kappa_a \; \Big(\, \text{\rm mod} \; \hbar \, \uhg^{\,\widehat{\otimes}\, 2} \,\Big) \, $
   belongs to  $ \, \lieg \otimes \lieg \, $,  \,and it is an
 \textsl{antisymmetric twist}  element
 for the Lie bialgebra  $ \lieg \, $;
 \vskip5pt
   (c)\;  the deformation  $ \, {\big( \uhg \big)}^\cF $  of  $ \, \uhg $  by the twist
   $ \cF $  is a QUEA for the Lie bialgebra
   $ \, \lieg^c = \big(\, \lieg \, ; \, [\,\ ,\ ] \, , \, \delta^{\,c} \big) \, $
   which is the deformation of  $ \, \lieg $  by the twist  $ c \, $;
   in a nutshell, we have
   $ \; {\big( \uhg \big)}^\cF \cong \, U_\hbar\big(\lieg^c\,\big) \; $.
\end{theorem}

\pf
 \textit{(a)}\,  By construction, this follows from standard identities
 for exponentials and for logarithms.
 \vskip7pt
   \textit{(b)}\,  We fix hereafter the notation  $ \, U_\hbar := \uhg \, $  and
   $ \, J_\hbar := \Ker\big(\epsilon_{U_\hbar}\big) \, $,  \, and we write
   $ \, \kappa \in U_\hbar^{\,\widehat{\otimes} 2} \, $  with Sweedler's like
   $ \sigma $--notation  $ \, \kappa = \kappa_1 \otimes \kappa_2 \, $.
%
%%%%%
% %
%  \vskip3pt
% %
%    The conditions
% %
%  $ \; \big( \epsilon \otimes \text{id} \big)(\cF\,)  =  1  =  \big( \text{id} \otimes \epsilon \big)(\cF\,) \; $
% %
%  for the twist  $ \, \cF = \exp\!\big( \hbar\,\kappa \big) \, $  yield at once
%  $ \; \epsilon(\kappa_1) \otimes \kappa_2 = 0 = \kappa_1 \otimes \epsilon(\kappa_2) \; $,
%  \;which means that we can assume (as we shall do henceforth) $ \, \kappa_1 , \kappa_2 \in J_\hbar \, $,
%  \,hence  $ \, \kappa \in J_\hbar^{\,\widehat{\otimes} 2} \, $.
%%%%%
%
 By  Lemma \ref{lemma: twist-cond's/exp ==> twist-cond's/log}
 we can assume (as we shall do henceforth) that
 $ \, \kappa_1 , \kappa_2 \in J_\hbar \, $,  \,hence
 $ \, \kappa \in J_\hbar^{\,\widehat{\otimes} 2} \, $.
 \vskip3pt
   Now we consider the identity
 $ \; \cF_{1{}2} \, \big( \Delta \otimes \text{id} \big)(\cF\,) \, = \,
 \cF_{2{}3} \, \big( \text{id} \otimes \Delta \big)(\cF\,) \; $.
 Writing  $ \, \cF = \exp\big( \hbar \, \kappa_1 \otimes \kappa_2 \big) \, $  and
 $ \, \Delta(\kappa_s) = \kappa_s^{(1)} \otimes \kappa_s^{(2)} \, (s=1,2) \, $
 this reads
  $$
  \exp\!\big( \hbar \, \kappa_1 \otimes \kappa_2 \otimes 1 \big) \,
  \exp\!\Big( \hbar \, \kappa_1^{(1)} \otimes \kappa_1^{(2)} \otimes \kappa_2 \Big)
  \; = \;
  \exp\!\big( \hbar \, 1 \otimes \kappa_1 \otimes \kappa_2 \big) \,
  \exp\!\Big( \hbar \, \kappa_1 \otimes \kappa_2^{(1)} \otimes \kappa_2^{(2)} \Big)
  $$
 Now taking  $ \hbar $--adic  expansion in both sides of this last identity, at order 0
 --- in  $ \hbar $ ---   we get  $ \, 1 \otimes 1 \otimes 1 = 1 \otimes 1 \otimes 1 \, $,
 \,hence from order 1 we get the non-trivial identity
\begin{equation}  \label{eq: first order}
  \kappa_1 \otimes \kappa_2 \otimes 1 +
  \kappa_1^{(1)} \otimes \kappa_1^{(2)} \otimes \kappa_2
  \; \underset{\hbar}{\equiv} \;  1 \otimes \kappa_1 \otimes \kappa_2 +
  \kappa_1 \otimes \kappa_2^{(1)} \otimes \kappa_2^{(2)}
\end{equation}
 where hereafter any symbol  $ \, \underset{\hbar^n}{\equiv} \, $
 means ``congruent modulo
 $ \, \hbar^n \, U_\hbar^{\,\widehat{\otimes} 3} \, $''  (for any  $ \, n \in \NN \, $).
 Then taking  \eqref{eq: first order}  modulo  $ \hbar $  we get
\begin{equation}  \label{eq: first order modulo h}
  \overline{\kappa_1} \otimes \overline{\kappa_2} \otimes 1 \, + \,
  \overline{\kappa_1}^{\,(1)} \otimes \overline{\kappa_1}^{\,(2)} \otimes \overline{\kappa_2}  \,\; = \;\,  1
  \otimes \overline{\kappa_1} \otimes \overline{\kappa_2} \, + \,
  \overline{\kappa_1} \otimes \overline{\kappa_2}^{\,(1)} \otimes \overline{\kappa_2}^{\,(2)}
\end{equation}
 where hereafter  $ \, \overline{x} := x \, \big( \text{mod} \; \hbar \big) \, $''
 and we took into account that
 $ \, \overline{\kappa_s^{(i)}} = \overline{\kappa_s}^{\,(i)} \, $  for all
 $ \, s , i \in  \{1\,,2\} \, $.
 Now,
%
%%%%%
% by construction
%%%%%
%
 $ \, \overline{\kappa_s^{(1)}} \otimes \overline{\kappa_s^{(2)}} =
 \Delta(\,\overline{\kappa_s}\,) \, $
 with  $ \; \overline{\kappa_s} \in U_\hbar \big/ \hbar \, U_\hbar \, = \, U(\lieg) \; $
 has the form
\begin{equation}  \label{eq: class-coprod}
  \overline{\kappa_s}^{\,(1)} \otimes \overline{\kappa_s}^{\,(2)}  \,\; = \;\,
  \overline{\kappa_s} \otimes 1 \, + \,
  1 \otimes \overline{\kappa_s} \, + \,
  \dot{\overline{\kappa_s}}^{\,(1)} \otimes \dot{\overline{\kappa_s}}^{\,(2)}
\end{equation}
 for some  $ \, \dot{\overline{\kappa_s}}^{\,(i)} \in \Ker\big(\epsilon_{U(\lieg)}\big) \, $
 ---  $ \, i \in \{1\,,2\} \, $  ---
 having the following property: if we denote by  $ {U(\lieg)}_n $  the  $ n $--th
 piece in the canonical filtration of  $ U(\lieg) $
 --- its coradical filtration, in Hopf theoretical language ---
 and for any  $ \, \text{x} \in {U(\lieg)}_n \setminus {U(\lieg)}_{n-1} \, $
 we set  $ \, \partial(\text{x}) := n \, $,
 \,then in  \eqref{eq: class-coprod}  we have
 $ \; \partial\Big( \dot{\overline{\kappa_s}}^{\,(i)} \Big)
 \lneqq \partial\big(\,\overline{\kappa_s}\,\big) \; $.
 Now, using  \eqref{eq: class-coprod}  to re-write
 \eqref{eq: first order modulo h}  we find,
 after cancelling out three summands on both sides, the following
\begin{equation*}
%
% \label{eq: first order (mod h) - refined}
%
  \dot{\overline{\kappa_1}}^{\,(1)} \otimes \dot{\overline{\kappa_1}}^{\,(2)} \otimes
  \overline{\kappa_2}  \,\; = \;\,
  \overline{\kappa_1} \otimes \dot{\overline{\kappa_2}}^{\,(1)}
  \otimes \dot{\overline{\kappa_2}}^{\,(2)}
\end{equation*}
 and then the condition  $ \; \partial\Big( \dot{\overline{\kappa_s}}^{\,(i)} \Big)
 \lneqq \partial\big(\,\overline{\kappa_s}\,\big) \; $  forces
 $ \, \dot{\overline{\kappa_1}}^{\,(1)} \!\otimes \dot{\overline{\kappa_1}}^{\,(2)} \! =
 0 = \dot{\overline{\kappa_2}}^{\,(1)} \!\otimes \dot{\overline{\kappa_2}}^{\,(2)} \, $.
 Thus  \eqref{eq: class-coprod}  reads
%%%
 $ \; \Delta(\overline{\kappa_s}\,) =
 \overline{\kappa_s}^{\,(1)} \otimes \overline{\kappa_s}^{\,(2)} = \,
 \overline{\kappa_s} \otimes 1 + 1 \otimes \overline{\kappa_s} \; $;
%%%
 \;this means  $ \, \overline{\kappa_s} \in \lieg \in \big( \subseteq U(\lieg) \,\big) \, $
 ---  $ \, s \in \{1\,,2\} \, $  ---
 so  $ \, \overline{\kappa} =
 \overline{\kappa_1} \otimes \overline{\kappa_2} \in \lieg \otimes \lieg \, $,
 \,hence $ \, c := \overline{\kappa_a} \in \lieg \otimes \lieg \, $.
 \vskip5pt
   Now we have to prove that  $ c $  is an
  \textsl{antisymmetric twist\/}
 for the Lie bialgebra  $ \lieg \, $.
                                                                        \par
   Keeping notation from above, since  $ \, \overline{\kappa_s} \in \lieg \, $
   we have
\begin{equation}  \label{eq: cobracket(k_s)}
  \Delta(\kappa_s)  \,\; \underset{\hbar^2}{\equiv} \;\,  \kappa_s \otimes 1 \, +
  \, 1 \otimes \kappa_s \, + \, \hbar \, \kappa_s^{\,[1]} \otimes \kappa_s^{\,[2]}
\end{equation}
 with  $ \; \overline{\kappa_s}^{\,[1]} \otimes \overline{\kappa_s}^{\,[2]} -
 \overline{\kappa_s}^{\,[2]} \otimes \overline{\kappa_s}^{\,[1]} =
 \delta(\overline{\kappa_s}) \; $
 being the Lie cobracket of  $ \, \overline{\kappa_s} \, $,  \,by assumption.
 When we plug  \eqref{eq: cobracket(k_s)}  in the  $ \hbar $--adic
 expansion of the identity
  $$
  \exp\!\big( \hbar \, \kappa_1 \otimes \kappa_2 \otimes 1 \big) \,
  \exp\!\Big( \hbar \, \kappa_1^{(1)} \otimes \kappa_1^{(2)} \kappa_2 \Big)  \; = \;
  \exp\!\big( \hbar \, 1 \otimes \kappa_1 \otimes \kappa_2 \big) \,
  \exp\!\Big( \hbar \, \kappa_1 \otimes \kappa_2^{(1)} \otimes \kappa_2^{(2)} \Big)
  $$
 we find that at order 2   --- in  $ \hbar $  ---   it implies an identity
\begin{equation}  \label{eq: second order modulo h x F}
  \hskip0pt   \overline{\kappa_1}^{\,[1]} \otimes \, \overline{\kappa_1}^{\,[2]} \otimes \,
  \overline{\kappa_2} \,
  + \, \overline{\kappa_{1,2}} \cdot \overline{\kappa_{1,3}} \, +
  \, \overline{\kappa_{1,2}} \cdot \overline{\kappa_{2,3}}  \; = \,
  \overline{\kappa_1} \otimes \, \overline{\kappa_2}^{\,[1]} \otimes \,
  \overline{\kappa_2}^{\,[2]} \, + \,
  \overline{\kappa_{2,3}} \cdot \overline{\kappa_{1,2}} \, + \,
  \overline{\kappa_{2,3}} \cdot \overline{\kappa_{1,3}}
\end{equation}
 where each  $ \overline{\kappa_{i,j}} \, $,  as usual, is the tensor in  $ \lieg^{\otimes 3} $
 which sports the  $ \kappa_1 $'s  in position  $ i \, $,  the  $ \kappa_2 $'s  in position  $ j \, $,
 \,and a (repeated) tensor factor 1 in the last remaining position.
                                                         \par
   Now let us consider  $ \, \Bbbk\big[\mathbb{S}_3\big] \, $,  \,the group algebra over
   $ \Bbbk $  of the symmetric group  $ \mathbb{S}_3 \, $,  \,the ``antisymmetrizer''
 $ \; \textsl{Alt}_{\,3} := \big( \id - (1\,2) - (2\,3) - (3\,1) + (1\,2\,3) + (3\,2\,1) \big) \; $
 in  $ \, \Bbbk\big[\mathbb{S}_3\big] \, $,  \,and the natural action of
 $ \, \Bbbk\big[\mathbb{S}_3\big] \, $  onto  $ {U(\lieg)}^{\otimes 3} \, $.
 Let  $ \, \textsl{Alt}_{\,3} \, $  act on the identity  \eqref{eq: second order modulo h x F}:
 a sheerly straightforward calculation
 shows that the outcome, using notation
 $ \, c := \overline{\kappa_a} = \overline{\kappa} - \overline{\kappa}_{2,1} \, $,
 \,eventually is
\begin{equation*}   % \label{eq: twist identity x c}
  \big( \delta \otimes \id \big)(c) \, + \, \text{c.p.} \, + \, [[c\,,c\,]]  \,\; = \;\,  0
\end{equation*}
 This means exactly that  $ \, c \, $  is a twist for the Lie bialgebra
 $ \lieg \, $,  \,as in  Definition \ref{def: twist-deform_Lie-bialg's},
 which is obviously antisymmetric (by construction), q.e.d.
 \vskip7pt
   \textit{(c)}\,  Due to the peculiar form of the twist
   --- namely, its being trivial modulo  $ \hbar $  ---
   it is easy to see that the Hopf algebra  $ {\uhg}^\cF $  is again a QUEA,
   over some bialgebra
 $ \tilde{\lieg} \, $,  i.e.\
 $ \, {\uhg}^\cF \Big/ \hbar \, {\uhg}^\cF \, = \, U\big(\tilde{\lieg}\big) \, $,
 \,and even that one has  $ \, \tilde{\lieg} = \lieg \, $  as Lie algebras.
%
%    --- see for example \cite[Proposition 9.6]{ES}.
%
 In fact, since the twist  $ \cF $  is trivial modulo  $ \hbar \, $,  we have that
 $ \, \uhg \big/ \hbar\,\uhg \, $
 and  $ \, {\uhg}^\cF \big/ \hbar \, {\uhg}^\cF \, $  are  \textsl{isomorphic as
 Hopf algebras\/};  in particular, then,
 $ {\uhg}^\cF $  itself is again a QUEA, on the same Lie algebra  $ \lieg $
 from  $ \uhg $
 but possibly inducing on $ \lieg $  a different Lie cobracket.
 Indeed, what is actually affected, a priori, is the co-Poisson structure on the
 semiclassical limit
 --- hence the Lie cobracket on  $ \lieg $  ---   which in general on
 $ \, {\uhg}^\cF \big/ \hbar \, {\uhg}^\cF \, $
 will be different from that on  $ \, \uhg \big/ \hbar\,\uhg \, $.
 \vskip3pt
   Let us compute the Lie coalgebra structure of  $ \tilde{\lieg} $
   given by  \eqref{eq: def-cobracket}.
   Given  $ \; \text{x} \in \tilde{\lieg} \; $,  \,let  $ \, x \in {\uhg}^\cF $  be any lift of
   $ \text{x} \, $:  \,using obvious notation, its twisted coproduct is
  $$
  \displaylines{
  \;\;   \Delta^{\scriptscriptstyle \!\cF}(x)  \; = \;  \cF \, \Delta(x) \, \cF^{-1}  \; =
   \;  e^{\hbar\,\kappa} \, \Big( x \otimes 1 + 1 \otimes x +
   \hbar \, {\textstyle \sum}_i \, x^{[i]}_1 \otimes x^{[i]}_2 +
   O\big(\hbar^2\big) \Big) \, e^{-\hbar\,\kappa}  \; =   \hfill  \cr
   \qquad   = \;  \big( 1 \otimes 1 + \hbar \, \kappa \big) \, \Big( x \otimes 1 +
   1 \otimes x +
   \hbar \, {\textstyle \sum}_i \, x^{[i]}_1 \otimes x^{[i]}_2 \Big) \, \big( 1 \otimes 1
   - \hbar \, \kappa \big) + O\big(\hbar^2\big)  \; =   \hfill  \cr
   \qquad \qquad   = \;  x \otimes 1 + 1 \otimes x + \hbar \, \big[\, \kappa \, ,
   \, x \otimes 1 + 1 \otimes x \,\big] +
   \hbar \, {\textstyle \sum}_i \, x^{[i]}_1 \otimes x^{[i]}_2 + O\big(\hbar^2\big)
   \; = \hfill  \cr
   \qquad \qquad \qquad   = \;  x \otimes 1 + 1 \otimes x +
   \hbar \, \Big(\, {\textstyle \sum}_i \, x^{[i]}_1 \otimes x^{[i]}_2 -
   \ad_x\!(\kappa) \Big) \, + \, O\big(\hbar^2\big)   \hfill  }
   $$
 On the other hand, the opposite twisted coproduct is
  $$
  \displaylines{
   \quad   \big( \Delta^{\scriptscriptstyle \!\cF}\big)^{\op}(x)  \; =
   \;  {(\cF\,)}_{21} \, \Delta^{\op}(x) \, {(\cF\,)}_{21}^{-1}  \; = \;
   e^{\hbar\, \kappa_{2,1}} \, \Delta^{\op}(x) \, e^{-\hbar\, \kappa_{2,1}}  \; =   \hfill  \cr
   \hfill   = \;  x \otimes 1 + 1 \otimes x +
   \hbar \, \Big(\, {\textstyle \sum}_i \, x^{[i]}_2 \otimes x^{[i]}_1 -
   \ad_x\!\big(\kappa_{2,1}\big) \Big) \, + \, O\big(\hbar^2\big)   \hfill  }
   $$
   \indent   Thus, by the very definition of the cobracket
   --- as in  \eqref{eq: def-cobracket}  ---   we have
  $$
  \delta^\cF(x)  \; := \;  \delta(x) +
  \big( \ad_x\!\big( \kappa_{2,1} - \kappa \big) \big)  \;\  (\,\text{mod} \ \hbar\,)
  \; = \;  \delta(x) - \ad_x(c)  \; =: \;  \delta^c(x)
  $$
hence  $ \, \tilde{\lieg} \, $  is the twist deformation by  $ c $
of the Lie bialgebra  $ \lieg \, $,  \,as claimed.
\epf

\vskip9pt

\begin{obs}  \label{obs: trivial deform QUEA}
 Let us point out that the twists  $ \cF $  considered in
 Theorem \ref{thm: twist-deform-QUEA}  above are those of ``trivial type'',
 as they are the identity modulo  $ \hbar \, $.  This ensures that twisting  $ \uhg $
 by such an  $ \cF $  does not affect the Hopf structure of the semiclassical limit;
 in particular, it still is of the form  $ U\big(\tilde{\lieg}\big) \, $,  with  $ \tilde{\lieg} $
 equal to  $ \lieg $  as a Lie algebra but with a different Lie coalgebra structure.
 A more general twist might be ``unfit'', i.e.\ the deformed Hopf algebra  $ {\uhg}^\cF $
 might no longer be a QUEA.
\end{obs}

\vskip7pt

   We present now a concrete example, taken from \cite{GaGa2}, where formal ``multiparameter'' QUEAs are studied in detail.

\vskip11pt

\begin{exa}  \label{example: toral twists for FoMpQUEAs}
 Let  $ \, n \in \NN_+ \, $  and  $ \, I := \{1,\ldots,n\} \, $.
 We fix a free  $ \kh $--module  $ \lieh $  of finite rank  $ t \, $,  and we pick subsets
 $ \, \Pi^\vee := {\big\{ T^+_i , T^-_i \big\}}_{i \in I} \! \subseteq \lieh \, $,
 $ \, \Pi := {\big\{ \alpha_i \big\}}_{i \in I} \subseteq \lieh^* :=
 \Hom_\kh\!\big(\, \lieh \, , \kh \big) \, $.
 Let  $ \, P \in M_n\big(\kh\big) \, $  be any  $ (n \times n) $--matrix  with entries in
 $ \kh \, $.  A \textsl{realization\/}  of  $ P \, $  over  $ \kh $  of rank  $ t $  is a triple
 $ \, \cR := \big(\, \lieh \, , \Pi \, , \Pi^\vee \,\big) \, $  where
 $ \, \alpha_j\big(\,T^+_i\big) = p_{\,ij} \, $,  $ \, \alpha_j\big(\,T^-_i\big) = p_{j\,i} \; $
 ($ \, \forall \; i, j \in I \, $),  and
 $ \, \overline{\Sigma} := {\big\{\, \overline{S}_i := 2^{-1} \big(\, T^+_i \! + T^-_i \big) \;
 (\text{\,mod\ } \hbar\,\lieh \,) \big\}}_{i \in I} \, $  is  $ \Bbbk $--linearly
 independent as a subset in  $ \, \overline{\lieh} := \lieh \Big/ \hbar\,\lieh \, $.
                                                                                            \par
  Let  $ \, A := {\big(\, a_{ij} \big)}_{i, j \in I} \in M_{n}(\Bbbk) \, $  be a
  symmetrisable generalized Cartan matrix, with associated diagonal matrix
  $ \, D := {\big(\hskip0,7pt d_i \, \delta_{ij} \big)}_{i, j \in I} \, $.  We say that a matrix
  $ \, P \in M_n(\kh) \, $  is  \textsl{of Cartan type\/}  with corresponding Cartan matrix
  $ A $  if  $ \; P_s := 2^{-1} \big( P + P^{\,\scriptscriptstyle T} \big) = DA \; $.
                                                                                            \par
   A  \textsl{formal multiparameter quantum universal enveloping algebra\/}
   (=FoMpQUEA) with multiparameter  $ P $  and realization  $ \cR $  is the unital,
   associative, topological,  $ \hbar $--adically  complete  $ \kh $--algebra  $ \uRPhg $
   generated by the  $ \kh $--submodule  $ \lieh $  and all
 $ \, E_i \, $,  $ F_i \, $  (for all  $ \, i \in I \, $),
 with relations (for all  $ \, T , T' , T'' \in \lieh \, $,  $ \, i \, , j \in I \, $)
 \vskip-9pt
\begin{equation}  \label{eq: comm-rel's_x_uPhg}
 \begin{aligned}
   T \, E_j \, - \, E_j \, T  \, = \,  +\alpha_j(T) \, E_j  \;\; ,  \qquad
   T \, F_j \, - \, F_j \, T  \; = \;  -\alpha_j(T) \, F_j   \hskip33pt  \\
   T' \, T''  \; = \;  T'' \, T'  \;\; ,  \qquad   E_i \, F_j \, - \, F_j \, E_i  \; = \;
   \delta_{i,j} \, {{\; e^{+\hbar \, T_i^+} - \, e^{-\hbar \, T_i^-} \;} \over
   {\; q_i^{+1} - \, q_i^{-1} \;}}   \hskip37pt  \\
   \sum\limits_{k = 0}^{1-a_{ij}} (-1)^k {\left[ { 1-a_{ij} \atop k }
\right]}_{\!q_i} q_{ij}^{+k/2\,} q_{ji}^{-k/2} \, E_i^{1-a_{ij}-k} E_j E_i^k  \; = \;  0
\qquad  (\, i \neq j \,)   \hskip25pt  \\
   \sum\limits_{k = 0}^{1-a_{ij}} (-1)^k {\left[ { 1-a_{ij} \atop k }
\right]}_{\!q_i} q_{ij}^{+k/2\,} q_{ji}^{-k/2} \, F_i^k F_j F_i^{1-a_{ij}-k}  \; = \;  0
\qquad  (\, i \neq j \,)   \hskip25pt
 \end{aligned}
\end{equation}
 By \cite[Theorem 4.3.2]{GaGa2},  every FoMpQUEA  $ \uRPhg $
 bears a structure of topological Hopf algebra over\/  $ \kh $
 --- with coproduct taking values into the  $ \hbar $--adically  completed tensor product
 $ \; \uRPhg \!\mathop{\widehat{\otimes}}\limits_\kh\! \uRPhg \; $
 ---   given by  ($ \; \forall \; T \in \lieh \, $,  $ \, \ell \in I $)
%
%%%%%
%
% \begin{equation*}  \label{eq: coprod_x_uPhg}
%    \Delta \big(E_\ell\big) = E_\ell \otimes 1 \, + \, e^{\hbar \, T_\ell^+} \! \otimes E_\ell \; ,  \;\;
%    \Delta\big(T\big) = T \otimes 1 \, + \, 1 \otimes T \; ,  \;\;
%    \Delta\big(F_\ell\big) = F_\ell \otimes e^{-\hbar \, T_\ell^-} + \, 1 \otimes F_\ell
% %
% \end{equation*}
% %
%  \vskip-15pt
% %
% \begin{eqnarray*}
% %
%     \epsilon\big(E_\ell\big) \, = \, 0  \;\;\; ,  \qquad  &   \;\quad
%     \epsilon\big(T\big) \, = \, 0  \;\;\; ,  &
%     \;\qquad  \epsilon\big(F_\ell\big) \, = \, 0  \;\;\;\;  \label{eq: counit_x_uPhg}  \\
% %
%    \cS\big(E_\ell\big)  \, = \,  - e^{-\hbar \, T_\ell^+} E_\ell  \;\; ,  &   \;\;\;
%    \cS\big(T\big)  \, = \,  - T   \;\; ,  &  \;\;\;
%    \cS\big(F_\ell\big)  \, = \,  - F_\ell \, e^{+\hbar \, T_\ell^-}   \qquad
%   \label{eq: antipode_x_uPhg}
% %
% \end{eqnarray*}
% %
%%%%%
%
  $$  \displaylines{
   \Delta \big(E_\ell\big) = E_\ell \otimes 1 + e^{\hbar \, T_\ell^+} \! \otimes E_\ell \; ,  \;\;
   \Delta\big(T\big) = T \otimes 1 + 1 \otimes T \; ,  \;\;
   \Delta\big(F_\ell\big) = F_\ell \otimes e^{-\hbar \, T_\ell^-} \! + 1 \otimes F_\ell  \cr
    \epsilon\big(E_\ell\big) \, = \, 0  \;\;\; ,  \hskip71pt
    \epsilon\big(T\big) \, = \, 0  \;\;\; ,  \hskip71pt
    \epsilon\big(F_\ell\big) \, = \, 0  \cr
   \hskip11pt  \cS\big(E_\ell\big)  \, = \,  - e^{-\hbar \, T_\ell^+} E_\ell  \;\; ,   \hskip41pt
   \cS\big(T\big)  \, = \,  - T  \;\; ,   \hskip41pt
   \cS\big(F_\ell\big)  \, = \,  - F_\ell \, e^{+\hbar \, T_\ell^-}  }  $$
   \indent   Furthermore, by  \cite[Theorem 6.1.4]{GaGa2},  $ \uRPhg $  is a
   \textsl{quantized universal enveloping algebra\/} whose semiclassical limit is
   $ U\big(\lieg^{\bar{\Rpicc}}_{\bar{\Ppicc}}\big) \, $,  \,where
   $ \lieg^{\bar{\Rpicc}}_{\bar{\Ppicc}} $  is a Lie \textit{multiparameter\/}  Lie bialgebra.
   In short, for each pair  $ \, (P,\cR) \, $  as above,
   \textsl{$ U\big(\lieg^{\bar{\Rpicc}}_{\bar{\Ppicc}}\big) $  is the specialization of
   $ \, \uRPhg \, $,  or   --- equivalently ---   $ \uRPhg $  is a quantization of
   $ U\big(\lieg^{\bar{\Rpicc}}_{\bar{\Ppicc}}\big) \, $}   ---
   or also, by a standard abuse of language,
   \textsl{$ \uRPhg $  is a quantization of  $ \lieg^{\bar{\Rpicc}}_{\bar{\Ppicc}} \, $}.
   In particular, writing again  $ T $, $ E_i $  and  $ F_i $  for the
   ``specialized'' images of the generators $ \, T \in \lieh \, $  and
   $ E_i \, $,  $ F_i $  ($ \, i \in I \, $), the Lie algebra structure of
   $ \lieg^{\bar{\Rpicc}}_{\bar{\Ppicc}} $  is given by \eqref{eq: comm-rel's_x_uPhg}
   with the commutator replaced by the (Lie) bracket and the quantum Serre relations
   by the adjoint actions  $ \, {\ad(E_i)}^{1-a_{ij}}(E_j) = 0 \, $  and
   $ \, {\ad(F_i)}^{1-a_{ij}}(F_j) = 0 \, $,
 \,whereas the coalgebra structure is determined by
$$
\delta\big(\,T\big) \, = \, 0  \quad ,   \qquad   \delta\big(E_i\big) \, =
\, 2 \; T^+_i \!\wedge E_i  \quad ,   \qquad  \delta\big(F_i\big) \, = \,
2 \; T^-_i \!\wedge F_i
$$
   \indent   For example, if we take $ \, P := DA \, $, $ \, r := \rk\!\big(DA\big) \, $
   and  $ \, \cR := \big(\, \lieh \, , \Pi \, , \Pi^\vee \big) \, $  a realization of  $ \, DA  \, $,
   where  $ \, \rk(\lieh) = 2\,n-r \, $  and  $ \, T_i^+ = T_i^- \, $ in  $ \Pi^{\vee} \, $,
   for all  $ \, i \in I \, $, one has that  $ U^{\,\cR}_{\!DA,\hbar}(\lieg) $  is
   the ``quantum double version'' of
   the usual Drinfeld's QUEA  $ U_\hbar\big(\lieg_{{}_A}\big) $
   for the Kac-Moody algebra  $ \lieg_{{}_A} $  associated with the Cartan matrix $ A \, $;
   \,in particular, its semiclassical limit is
   $ U\big(\lieg_A^{\scriptscriptstyle \textsl{MD}}\big) \, $,
   \,where  $ \lieg_A^{\scriptscriptstyle \textsl{MD}} $  is the ``Manin double version'' of
   $ \lieg_{{}_A} \, $.
 \vskip5pt
   Now take any  $ \kh $--basis  $ \, {\big\{ H_g \big\}}_{g \in \cG} \, $  of  $ \lieh $
   where  $ \, |\cG| = \rk(\lieh) = t \, $.  For any antisymmetric matrix
   $ \; \Phi = \big( \phi_{i,j} \big)_{1\leq i,j\leq n} \, $  with entries in  $ \kh $  we define
  $$  \JJ_\Phi  \; := \;  {\textstyle \sum_{i,j=1}^n} \phi_{gk} \, H_g \otimes H_k  \; \in \;  \lieh \otimes \lieh  \; \subseteq \;  U^{\,\cR}_{\!P,\hskip0,7pt\hbar}(\lieh) \otimes U^{\,\cR}_{\!P,\hskip0,7pt\hbar}(\lieh)  $$
 By direct check, one sees that the element
  $$  \cF_\Phi  \,\; := \;\,  e^{\,\hbar \, 2^{-1} \JJ_\Phi}  \,\; = \;\,  \exp\Big(\hskip1pt \hbar \, 2^{-1} \, {\textstyle \sum_{g,k=1}^t} \phi_{gk} \, H_g \otimes H_k \Big)  $$
 in  $ \, U^{\,\cR}_{\!P,\hskip0,7pt\hbar}(\lieh) \,\widehat{\otimes}\,
 U^{\,\cR}_{\!P,\hskip0,7pt\hbar}(\lieh) \, $  is actually a  {\sl twist\/}  for  $ \uRPhg \, $.
 For  $ \, i \in I \, $,  define the elements
 $ \,\; \cL_{\Phi, i} := e^{+ \hbar \, 2^{-1} \sum_{g,k=1}^t \alpha_i(H_g) \, \phi_{gk} H_k} \;\, $
 and  $ \,\; \cK_{\Phi,i} := e^{+ \hbar \, 2^{-1} \sum_{g,k=1}^t \alpha_i(H_g) \,
 \phi_{kg} H_k} \; $.  Then, the new coproduct in  $ \big( \uRPhg\big)^{\cF_{\Phi}} $
 is given by
  $$  \displaylines{
   \qquad   \Delta^{\scriptscriptstyle \!\Phi}\big(E_i\big)  \; = \;
   E_i \otimes \cL_{\Phi, i}^{+1} \, +
 \, e^{+\hbar \, T^+_i} \cK_{\Phi, i}^{+1} \otimes E_i   \quad \qquad
 \big(\, \forall \;\; i \in I \,\big)  \cr
   \qquad \qquad \quad   \Delta^{\scriptscriptstyle \!\Phi}\big(T\big)  \; = \;
   T \otimes 1 \, + \, 1 \otimes T
   \quad \quad \qquad \qquad  \big(\, \forall \;\; T \in \lieh \,\big)  \cr
   \qquad   \Delta^{\scriptscriptstyle \!\Phi}\big(F_i\big)  \; = \;
   F_i \otimes \cL_{\Phi, i}^{-1} \, e^{-\hbar \, T^-_i } + \, \cK_{\Phi, i}^{-1} \otimes F_i
   \quad \qquad  \big(\, \forall \;\; i \in I \,\big)  }  $$
 Similarly, the ``twisted'' antipode  $ \cS^{\cF_\Phi} \, $  and the counit
 $ \, \epsilon^{\scriptscriptstyle \Phi} := \epsilon \, $   are given by
  $$  \begin{matrix}
   \qquad   \cS^{\scriptscriptstyle \Phi}\big(E_i\big)  \, = \,
   - e^{-\hbar \, T^+_i} \cK_{\Phi, i}^{-1} \, E_i \, \cL_{\Phi, i}^{-1}  \quad ,
   &  \qquad  \epsilon^{\scriptscriptstyle \Phi}\big(E_i\big) \, = \, 0
     \qquad \qquad  \big(\, \forall \;\; i \in I \,\big)  \\
   \qquad   \cS^{\scriptscriptstyle \Phi}\big(T\big)  \, = \,  -T  \quad ,
 \phantom{\Big|^|}
   &  \qquad  \epsilon^{\scriptscriptstyle \Phi}\big(T\big) \, = \, 0
     \qquad \qquad  \big(\, \forall \;\; T \in \lieh \,\big)  \\
  \qquad   \cS^{\scriptscriptstyle \Phi}\big(F_i\big)  \, = \,
  - \cK_{\Phi, i}^{+1} \, F_i \, \cL_{\Phi, i}^{+1} \, e^{+\hbar \, T^-_i}  \quad ,
   &  \qquad  \epsilon^{\scriptscriptstyle \Phi}\big(F_i\big) \, = \, 0
     \qquad \qquad  \big(\, \forall \;\; i \in I \,\big)
\end{matrix}  $$
   \indent   A key feature of this family of quantum groups is that the
   multiparameter encodes different types of deformations:
   there exists a multiparameter matrix  $ P_\Phi $  and a realization
   $ \cR_{\scriptscriptstyle \Phi} $  such that
   $ \; U_{\!P_\Phi,\,\hbar}^{\cR_{\scriptscriptstyle \Phi}}(\lieg) \cong
   {\big(\, U_{\!P,\,\hbar}^{\,\cR}(\lieg)\big)}^{\cF_\Phi} \; $  as topological Hopf algebras.
   In particular, the class of all FoMpQUEAs of any fixed Cartan type and
   of fixed rank is stable by toral twist deformations, see \cite[Theorem 5.1.4]{GaGa2}.
   Furthermore, it turns out that, under certain restrictions on the realization,
   every FoMpQUEA can be realized as a toral twist deformation of the ``standard''
   FoMpQUEA by Drinfeld   --- see \cite[Theorem 5.1.5]{GaGa2} for further details.
                                                                  \par
   With respect to the semiclassical limit,  $ \, \overline{\JJ}_{\Phi} := \JJ_{\Phi} \pmod{\hbar} \, $  is actually a (toral) twist for the Lie bialgebra  $ \lieg^{\bar{\Rpicc}}_{\bar{\Ppicc}} \, $.
The deformed Lie cobracket is given by the formula
   $$  \delta^{\overline{\JJ}_{\Phi}}(x)  \; := \;  \delta(x) - \ad_x\!\big(\overline{\JJ}_{\Phi}\big)  \; = \;
  \delta(x) - {\textstyle \sum_{g,k=1}^t} \overline{\phi_{gk}} \,
  \big( \big[x,H_g\big] \otimes H_k + H_g \otimes \big[x,H_k\big] \big)  $$
--- for all  $ \, x \in \lieg^{\bar{\Rpicc}}_{\bar{\Ppicc}} \, $,  with  $ \, \overline{\phi_{gk}} := \phi_{gk} \! \pmod{\hbar} \, $  ---   that on generators reads
  $$  \delta^{\overline{\JJ}_{\Phi}}(E_i) = 2 \, T^+_{\Phi,i} \wedge \! E_i  \; ,   \!\quad
  \delta^{\overline{\JJ}_{\Phi}}(T) = 0  \; ,   \!\quad
  \delta^{\overline{\JJ}_{\Phi}}(F_i) = 2 \, T^-_{\Phi,i} \wedge \! F_i  \; ,
     \quad  \forall \; i \in I \, , \, T \!\in \lieh  $$
where  $ \,\; T^{\pm}_{\Phi,i} = T^{\pm}_i \pm {\textstyle \sum_{g,k=1}^t} \, \overline{\phi_{kg}} \,
\alpha_i(H_g) \, H_k \;\, $  for all  $ \, i \in I \, $.
                                                                  \par
   In conclusion, one may consider the deformation  $ \big( \lieg^{\bar{\Rpicc}}_{\bar{\Ppicc}} \big)^{\overline{\JJ}_\Phi} $  of  $ \lieg^{\bar{\Rpicc}}_{\bar{\Ppicc}} $  by the (Lie) twist  $ \overline{\JJ}_\Phi \, $,  \,as well as the deformation  $ {\big(\,\uRPhg\big)}^{\cF_\Phi} $  of  $ \uRPhg $  by the (Hopf) twist  $ \cF_\Phi \, $.  By  \cite[Theorem 6.2.2]{GaGa2},  we know that  $ {\big(\, \uRPhg \big)}^{\cF_\Phi} $  is a QUEA, whose semiclassical limit is  $ \, U\Big(\! {\big( \lieg^{\bar{\Rpicc}}_{\bar{\Ppicc}} \big)}^{\overline{\JJ}_{\Phi}} \Big) \, $:  \,indeed, we have  $ \; {\big(\, \uRPhg \big)}^{\cF_\Phi} \cong \, U_{\!P_{\,\Phi},\hbar}^{\,\cR_\Phi}(\lieg) \; $  and  $ \; {\big( \lieg^{\bar{\Rpicc}}_{\bar{\Ppicc}} \big)}^{\overline{\JJ}_{\Phi}} \cong \, \lieg_{\bar{\Ppicc}_\Phi}^{\bar{\Rpicc}_\Phi} \; $.
\end{exa}

\vskip13pt

\subsection{Deformations by 2-cocycle of QFSHA's}  \label{subsec: 2coc-QFSHAs}  {\ }
 \vskip7pt
   We consider now deformations by 2-cocycle of QFSHA's.
   Due to the (linear) duality between the notions of QUEA and QFSHA,
   and similarly for those of twist and 2-cocycle, the outcome we find is nothing but
   the dual counterpart of  Theorem \ref{thm: twist-deform-QUEA}  above
   (and, consistently, it might be deduced
   from the latter by duality).

\vskip11pt

\begin{theorem}  \label{thm: 2cocycle-deform-QFSHA}  {\ }
 \vskip3pt
   Let  $ \fhg $  be a QFSHA over the Poisson group  $ G $,
   with tangent Lie bialgebra  $ \, \lieg = \big(\, \lieg \, ; \, [\,\ ,\ ] \, , \, \delta \,\big) \, $.
   Let  $ \, \sigma $  be a 2--cocycle for  $ \fhg $  s.t.\
   $ \; \sigma \equiv \epsilon^{\otimes 2} \; \Big(\, \text{\rm mod} \; \hbar \,
   \Big( \fhg^{\,\widetilde{\otimes}\, 2} \Big)^{\!*} \,\Big) \, $;  \,then
   $ \, \varsigma := \hbar^{-1} \log_*(\sigma) \in \Big( \fhg^{\,\widetilde{\otimes}\, 2}
   \Big)^{\!*} \, $,  \,where  ``$ \; \log_* $''
   is the logarithm with respect to the convolution product, and
   $ \, \sigma = \exp_*\!\big( \hbar\,\varsigma \big) \, $.  Last, we set
   $ \, \varsigma_a := \varsigma -\, \varsigma_{2,1} \, $.  Then:
 \vskip5pt
   (a)\;  $ \varsigma $  is antisymmetric, i.e.\
   $ \, \varsigma_{2,1} = -\varsigma \, $,  iff\/
   $ \sigma $  is orthogonal, i.e.\  $ \, \sigma_{2,1} = \sigma^{-1} \, $;
 \vskip5pt
   (b)\;  the element
   $ \, \overline{\varsigma_a} \, := \, \varsigma_a \; \Big( \text{\rm mod} \; \hbar \,
   \Big( \fhg^{\widetilde{\otimes}\, 2} \Big)^{\!*} \,\Big) \, $
   provides a well-defined element
   $ \, \zeta \in {\big( \lieg^* \otimes \lieg^* \big)}^* \! = \lieg \otimes \lieg \, $
   that is an
 \textsl{antisymmetric 2-cocycle} \hbox{for the Lie bialgebra  $ \lieg^* $;}
 \vskip5pt
   (c)\; letting  $ \, \zeta \, $  be as in claim (b),
 the deformation  $ {\big( \fhg \big)}_\sigma $  of  $ \, \fhg $  by the 2--cocycle
 $ \sigma $  is a QFSHA for the formal Poisson group  $ G_\sigma $
 with cotangent Lie bialgebra
  $$
  {\Lie\,(G_\zeta)}^*  \; = \;\,  {(\,\lieg^*)}_\zeta  \,\; = \;\,  \big(\, \lieg^* \, ;
  \, {\big( {[\,\ ,\ ]}_* \big)}_\zeta \, , \, \delta_* \,\big)  $$
 which is the deformation of  $ \, \lieg^* $  by the 2-cocycle  $ \zeta \, $;  in short,
 $ \; {\big( \fhg \big)}_\sigma \cong \, F_\hbar[[G_\zeta\hskip0,5pt]] \; $.
\end{theorem}

\pf

 \textit{(a)}\,  This is obvious, just by construction.
 \vskip7pt
   \textit{(b)}\,  As a first step, we have to prove that
   $ \, \overline{\varsigma_a} \, := \, \varsigma_a \;
   \Big(\, \text{\rm mod} \; \hbar \, \Big( \fhg^{\,\widetilde{\otimes}\, 2} \Big)^{\!*} \,\Big)\, $
   determines a uniquely defined element
   $ \, \zeta \in {\big( \lieg^* \otimes \lieg^* \big)}^* = \lieg \otimes \lieg \, $.
   Hereafter we realize  $ \lieg^* $  as  $ \, \lieg^* = \liem \Big/ \liem^2 \, $  where
   $ \; \liem := \Ker\big( \epsilon_{{}_{F[[G\,]]}} \big) \; $;  \;then we have also
  $$
  \lieg^* \otimes \lieg^*  \,\; = \;\,
  \Big( \liem \Big/ \liem^2 \Big) \otimes \Big( \liem \Big/ \liem^2 \Big)  \,\; \cong \;\,
  \big( \liem \otimes \liem \big) \bigg/ \Big( \liem \otimes \liem^2 +
  \liem^2 \otimes \liem \Big)
  $$
 thus in the end we have to prove that the function
 $ \; \overline{\varsigma_a} \, := \, \varsigma_a \; \big(\, \text{\rm mod} \; \hbar \,\big)
 \; $
 --- defined from
 $ \; {F[[G\,]]}^{\widetilde{\otimes}} = {\fhg}^{\widetilde{\otimes}} \;
 \big(\, \text{\rm mod} \; \hbar \,\big) \; $  to  $ \k $  ---   does vanish onto
 $ \, \liem \otimes \liem^2 + \liem^2 \otimes \liem \, $,  \, hence induces
 $ \, \zeta \, $  defined onto
 $ \, {\big( \lieg^* \otimes \lieg^* \big)}^* =
 \big( \liem \otimes \liem \big) \bigg/ \Big( \liem \otimes \liem^2 +
 \liem^2 \otimes \liem \Big) \, $  by the canonical recipe
 $ \; \zeta\,\big(\, \overline{u \otimes v} \,\big) \, := \,
 \overline{\varsigma_a}\,(u \otimes v) \; $
 for each  $ \, u , v \in \liem \, $.
 In fact, since  $ \, \varsigma_a \, $  is antisymmetric it is enough to prove that
 $ \; \overline{\varsigma_a}\,\big( \liem \otimes \liem^2 \big) = 0 \; $;
 \,in turn, this amounts to showing that   --- writing  $ \varsigma_a $
 as a bilinear function rather than a morphism on a tensor product module ---
 one has
\begin{equation}  \label{eq: varsigma_a(a,bc)}
 \qquad   \varsigma_a\big( a \, , b\,c \,\big) \, \underset{\hbar}{\equiv} \, 0
 \qquad \qquad  \forall \;\; a , b , c \in J_\hbar := \Ker\big( \epsilon_{{}_\fhg} \big)
\end{equation}
   \indent   For the given
   $ \, a , b , c \in J_\hbar := \Ker\big( \epsilon_{{}_\fhg} \big) \, $,
   \,the 2--cocycle nature of  $ \sigma $  gives
\begin{equation}  \label{eq: sigma(a,bc)}
 \qquad   \sigma\big( b_{(1)} \, , c_{(1)} \,\big) \,
 \sigma\big( a \, , b_{(2)} \, c_{(2)} \,\big)  \,\; = \;\,
 \sigma\big( a_{(1)} \, , b_{(1)} \,\big) \, \sigma\big( a_{(2)} \, b_{(2)} \, , c \,\big)
\end{equation}
 Now we expand  $ \sigma $  as
 (cf.\ \S \ref{prel_Q-Groups}  for notation  ``$ \, \cO\big(\hbar^2\big) \, $'')
  $$  \sigma  \,\; = \;\,  \exp_*\big(\, \hbar \, \varsigma \,\big)  \,\; =
  \;\,  \epsilon^{\otimes 2} \, + \, \hbar \, \varsigma \, + \, \cO\big(\hbar^2\big)  \,\;
  = \;\,  \epsilon^{\otimes 2} \, + \,
  \hbar \; {\textstyle \sum\limits_\varsigma} \, \varsigma' \otimes \varsigma'' \, +
  \, \cO\big(\hbar^2\big)  $$
 where we used sort of Sweedler's-like notation
 $ \; \varsigma \, = \, {\textstyle \sum\limits_\varsigma} \, \varsigma' \otimes
 \varsigma'' \; $  to denote  $ \varsigma \, $;  \,
 plugging this into  \eqref{eq: sigma(a,bc)}
 and expanding everything out, we end up with
  $$
  \displaylines{
   \qquad   \epsilon(a) \, \epsilon(b) \, \epsilon(c) \, +
   \, \hbar \, \Big( {\textstyle \sum\limits_\varsigma} \, \varsigma'(a) \,
   \varsigma''(b\,c) \, + \, \epsilon(a) \, {\textstyle \sum\limits_\varsigma} \,
   \varsigma'(b) \, \varsigma''(c) \Big) \, + \, \cO\big(\hbar^2\big)  \,\; =   \hfill  \cr
   \hfill   = \;\,  \epsilon(a) \, \epsilon(b) \, \epsilon(c) \, + \,
   \hbar \, \Big( {\textstyle \sum\limits_\varsigma} \, \varsigma'(a\,b) \, \varsigma''(c) \,
   + \, {\textstyle \sum\limits_\varsigma} \, \varsigma'(a) \, \varsigma''(b) \, \epsilon(c)
   \Big) \, + \, \cO\big(\hbar^2\big)   \qquad }
   $$
 which implies   --- since  $ \, \epsilon(a) = 0 = \epsilon(c) \, $
 by assumption ---   also
  $$
  \hbar \, {\textstyle \sum\limits_\varsigma} \, \varsigma'(a) \, \varsigma''(b\,c) \,
  + \, \cO\big(\hbar^2\big)  \,\; = \;\,
  \hbar \, {\textstyle \sum\limits_\varsigma} \, \varsigma'(a\,b) \, \varsigma''(c) \, +
  \, \cO\big(\hbar^2\big)
  $$
 whence we argue
\begin{equation}  \label{eq: zeta(a,bc)=zeta(ab,c)}
 \qquad   \varsigma(a\,,b\,c\,) \; = \;
 {\textstyle \sum\limits_\varsigma} \, \varsigma'(a) \, \varsigma''(b\,c\,)  \,\;
 \underset{\hbar}{\equiv} \;\,  {\textstyle \sum\limits_\varsigma} \, \varsigma'(a\,b) \,
 \varsigma''(c\,) \; = \; \varsigma(a\,b \, , c\,)
\end{equation}
 Recall also that  $ \fhg $  is commutative modulo  $ \hbar \, $,  \,so that
 $ \; x \, y \, \underset{\hbar}{\equiv} \, y \, x \, $  for all  $ \, x , y \in \fhg \, $.
 Using this along with several instances of  \eqref{eq: zeta(a,bc)=zeta(ab,c)}
 one gets
  $$
  \varsigma(a\,,b\,c\,)  \, \underset{\hbar}{\equiv} \,
  \varsigma(a\,b\,,c\,)  \, \underset{\hbar}{\equiv} \,  \varsigma(b\,a\,,c\,)  \,
  \underset{\hbar}{\equiv} \,  \varsigma(b\,,a\,c\,)  \, \underset{\hbar}{\equiv} \,
  \varsigma(b\,,c\,a)  \, \underset{\hbar}{\equiv} \,  \varsigma(b\,c\,,a)
  $$
 from which we eventually conclude that
  $$
  \varsigma_a(a\,,b\,c\,)  \; := \;  \varsigma(a\,,b\,c\,) \, - \, \varsigma(b\,c,\,a)
  \; \underset{\hbar}{\equiv} \;  \varsigma(b\,c,\,a) \, - \, \varsigma(b\,c,\,a)  \; = \;
  0  \quad ,   \qquad \text{q.e.d.}
  $$
 \vskip3pt
   As a second step, we note that  $ \, \zeta \, $  is antisymmetric, by construction, since
$ \, \varsigma_a \, $  is.
 \vskip7pt
 Third, we need to prove that
 $ \, \zeta : \lieg^* \otimes \lieg^* \relbar\joinrel\longrightarrow \k \, $
 satisfies the remaining condition of \eqref{eq: cocyc-cond_Lie-bialg},
 so that it is indeed a 2--cocycle.  Now, expanding  $ \sigma $  up to order 3, as
  $$
  \sigma  \,\; = \;\,  \exp_*\!\big(\, \hbar \, \varsigma \big)  \,\; = \;\,
  \epsilon^{\otimes 2} \, + \, \hbar \, \varsigma \, + \, \hbar^2 \, \varsigma^{*2} \big/ 2 \, + \,
  \cO\big(\hbar^3\big)
  $$
 and plugging this into  \eqref{eq: sigma(a,bc)},  we find, for all
 $ \, a , b , c \in J_\hbar := \Ker\big( \epsilon_{{}_\fhg} \big) \, $  again,
  $$
  \displaylines{
   \bigg( \epsilon\big(b_{(1)}\big) \, \epsilon\big(c_{(1)}\big) \, +
   \, \hbar \, \varsigma\big(b_{(1)}\,,c_{(1)}\big) \, + \,
   \hbar^2 \, \varsigma\Big({b_{(1)}}_{(1)},{c_{(1)}}_{(1)}\Big) \,
   \varsigma\Big({b_{(1)}}_{(2)},{c_{(1)}}_{(2)}\Big) \Big/ 2 \, + \,
   \cO\big(\hbar^3\big) \bigg) \,\cdot   \hfill  \cr
   \qquad \qquad   \cdot\, \bigg( \epsilon(a) \, \epsilon\big(b_{(2)}\big) \,
   \epsilon\big(c_{(2)}\big) \, + \, \hbar \, \varsigma\big(a\,,b_{(2)}\,c_{(2)}\big) \; +
   \hfill  \cr
   \hfill   + \; \hbar^2 \, \varsigma\Big(a_{(1)} \, , {b_{(2)}}_{(1)} {c_{(2)}}_{(1)}\Big) \,
   \varsigma\Big(a_{(2)} \, , {b_{(2)}}_{(2)} {c_{(2)}}_{(2)}\Big) \Big/ 2 \, + \,
   \cO\big(\hbar^3\big) \bigg)  \,\; =  \cr
   = \;   \bigg(\! \epsilon\big(a_{(1)}\big) \, \epsilon\big(b_{(1)}\big) \, + \, \hbar \,
   \varsigma\big(a_{(1)}\,,b_{(1)}\big) \, + \,
   \hbar^2 \, \varsigma\Big(\! {a_{(1)}}_{(1)},{b_{(1)}}_{(1)} \!\Big) \,
   \varsigma\Big(\! {a_{(1)}}_{(2)},{b_{(1)}}_{(2)} \!\Big) \Big/ 2 \, + \,
   \cO\big(\hbar^3\big) \!\bigg) \,\cdot   \hfill  \cr
   \qquad \qquad   \cdot\, \bigg( \epsilon\big(a_{(2)}\big) \, \epsilon\big(b_{(2)}\big) \,
   \epsilon(c) \, + \, \hbar \, \varsigma\big(a_{(2)}\,b_{(2)} \, , c\big) \; +
   \hfill  \cr
   \hfill   + \; \hbar^2 \, \varsigma\Big({a_{(2)}}_{(1)} {b_{(2)}}_{(1)} , c\Big) \,
   \varsigma\Big({a_{(2)}}_{(2)} {b_{(2)}}_{(2)} , c\Big) \Big/ 2 \, + \,
   \cO\big(\hbar^3\big) \bigg)  }
   $$
 and then performing multiplication and truncating at order 3 we get
  $$  \displaylines{
   \epsilon(a\,b\,c\,) \; + \; \hbar \, \big(\, \varsigma(a\,,b\,c\,) \, + \, \epsilon(a) \,
   \varsigma(b\,,c\,) \big) \; +   \hfill  \cr
   \hfill   + \; \hbar^2 \, \Big( \varsigma^{*2}(a\,,b\,c\,) \big/ 2 \, + \,
   \varsigma\big(b_{(1)} \, , c_{(1)}\big) \, \varsigma\big(a\,,b_{(2)} \, c_{(2)}\big) \, + \,
   \epsilon(a) \, \varsigma^{*2}(b\,,c\,) \big/ 2 \Big) \, + \, \cO\big(\hbar^3\big)  \,\; =
   \cr
   = \;\,  \epsilon(a\,b\,c\,) \; + \; \hbar \, \big(\, \varsigma(a\,b\,,c\,) \, + \,
   \varsigma(a\,,b\,) \, \epsilon(c) \big) \; +   \hfill  \cr
   \hfill   + \; \hbar^2 \, \Big( \varsigma^{*2}(a\,,b\,) \, \epsilon(c) \big/ 2 \, + \,
   \varsigma\big( a_{(1)} \, , b_{(1)}\big) \, \varsigma\big(a_{(2)}\,b_{(2)} \, , c \,\big) \,
   + \, \varsigma^{*2}(a\,b\,,c\,) \big/ 2 \Big) \, + \, \cO\big(\hbar^3\big)  }
   $$
 which in turn   --- since  $ \, \epsilon(a) = 0 =\epsilon(c) \, $  by assumption ---
 simplifies into
  $$  \displaylines{
   \varsigma(a\,,b\,c\,) \; + \; \hbar \, \Big( \varsigma^{*2}(a\,,b\,c\,) \big/ 2 \, + \,
   \varsigma\big(b_{(1)} \, , c_{(1)}\big) \,
   \varsigma\big(a\,,b_{(2)} \, c_{(2)}\big) \Big) \, + \, \cO\big(\hbar^2\big)  \,\; =
   \hfill  \cr
   \hfill   =  \;\,  \varsigma(a\,b\,,c\,) \, + \, \hbar \, \Big(\, \varsigma\big(a_{(1)} \, ,
   b_{(1)}\big) \, \varsigma\big( a_{(2)}\,b_{(2)} \, , c \,\big) \, + \,
   \varsigma^{*2}(a\,b\,,c\,) \big/ 2 \,\Big) \, + \, \cO\big(\hbar^2\big)  }
   $$
 that we eventually we re-write as
\begin{equation}  \label{eq: ident-sigma(a,b,c)}
  \begin{aligned}
     &  \varsigma(a\,,b\,c\,) \; - \; \varsigma(a\,b\,,c\,) \;\; +
 \\
      &  \;\quad   + \; \hbar \; \Big(\, \varsigma^{*2}(a\,,b\,c\,) \big/ 2 \, -
 \, \varsigma^{*2}(a\,b\,,c\,) \big/ 2  \;\; +  \\
%%%%%
%   \;\; \hbar \; \Big(\, \varsigma^{*2}(a\,,b\,c\,) \big/ 2 \, - \,
%   \varsigma^{*2}(a\,b\,,c\,) \big/ 2  \;\; +  \\
%%
   &  \;\quad \;\quad   + \;  \varsigma\big(b_{(1)} \, , c_{(1)}\big) \,
   \varsigma\big(a\,,b_{(2)} \, c_{(2)}\big) \, - \,
  \varsigma\big(a_{(1)} \, , b_{(1)}\big) \,
  \varsigma\big( a_{(2)}\,b_{(2)} \, , c \,\big) \,\Big)  \,\; \underset{\,\hbar^2}{\equiv} \;\,
   0  \\
  \end{aligned}
\end{equation}
 Now let  $ \, \Bbbk\big[\mathbb{S}_3\big] \, $  act onto  $ {\fhg}^{\otimes 3} $
 and consider in particular the action of the antisymmetrizer
 $ \; \textsl{Alt}_{\,3} := \big( \id - (1\,2) - (2\,3) - (3\,1) + (1\,2\,3) + (3\,2\,1) \big) \; $
 onto the equation in  \eqref{eq: ident-sigma(a,b,c)},
 which yields a new equation: denoting equation
 \eqref{eq: ident-sigma(a,b,c)}  by  $ \, \circledast = 0 \, $,  \,we will write
 $ \, \textsl{Alt}_{\,3}(\circledast) = 0 \, $  for the newly found equation.
 To see the latter explicitly, we compute the left-hand member
 $ \, \textsl{Alt}_{\,3}(\circledast) \, $,  \,starting by computing the action of
 $ \, \textsl{Alt}_{\,3} \, $  onto the
 first line
 in  \eqref{eq: ident-sigma(a,b,c)}:
 concretely, we find
\begin{equation}   \label{eq: Alt_3(first-line)}
  \begin{aligned}
     &  \textsl{Alt}_{\,3}\big(\,\textsl{1${}^{\text{\,st}}$  line in\ }
     (\ref{eq: ident-sigma(a,b,c)}) \big)  \,\; =  \\
     &  \;   = \;\,  \varsigma(a\,,b\,c\,) \, - \, \varsigma(b\,,a\,c\,) \, - \,
     \varsigma(a\,,c\,b\,) \, - \, \varsigma(c\,,b\,a\,) \, + \, \varsigma(c\,,a\,b\,) \, + \,
     \varsigma(b\,,c\,a\,)  \; -  \\
     &  \;\;   - \;  \varsigma(a\,b\,,c\,) \, + \, \varsigma(b\,a\,,c\,) \, + \,
     \varsigma(a\,c\,,b\,) \, + \, \varsigma(c\,b\,,a\,) \, - \, \varsigma(c\,a\,,b\,) \, - \,
     \varsigma(b\,c\,,a\,)  \; =  \\
     &  \hskip17pt   = \;\,  \varsigma_a\big(a\,,[b\,,c\,]\big) \; + \;
     \varsigma_a\big(b\,,[c\,,a\,]\big) \; + \; \varsigma_a\big(c\,,[a\,,b\,]\big)  \,\; = \;\,
     \varsigma_a\big(a\,,[b\,,c\,]\big) \; + \; \text{c.p.}
  \end{aligned}
\end{equation}
 where standard notation  $ \, [\,u\,,v\,] := u\,v - v\,u \, $
 is used to denote the usual commutator.
 Modulo  $ \hbar \, $,  \,such a commutator in  $ \fhg $  yields the
 Poisson bracket in  $ F[[G\,]] \, $,  \,hence we can write
 $ \; [\,u\,,v\,] \, = \, \hbar \, \big\{\, \overline{u} \, , \overline{v} \,\big\}' \; $
 where we use notation
 $ \; \overline{z} := \big( z \ \textrm{mod} \ \hbar\,\fhg \big) \; $
 for each $ \, z \in \fhg \, $  and  $ \, \textrm{f\,}' := $  some lift in  $ \fhg $  of any
 $ \, \textrm{f} \in F[[G\,]] \, $,  \,i.e.\  $ \, \overline{\textrm{f\,}'} = \textrm{f} \, $;
 \, note that  $ \textrm{f\,}' $  is only defined up to  $ \,  \hbar^2 \, \fhg \, $,
 \,yet that is enough for our purposes.
 Thanks to this,  \eqref{eq: Alt_3(first-line)}  turns into
\begin{equation}   \label{eq: Alt_3(first-line)-Poisson}
  \textsl{Alt}_{\,3}\big(\textsl{1${}^{\text{\,st}}$  line in\ }
  (\ref{eq: ident-sigma(a,b,c)}) \big)  \,\; = \;\,
\hbar \, \bigg(\, \varsigma_a\Big( a \, ,
\big\{ \overline{b} \, , \overline{c} \,\big\}' \Big) \; + \; \text{c.p.} \,\bigg)
\end{equation}
 Looking at  \eqref{eq: ident-sigma(a,b,c)},  this entails that the  $ \hbar $--adic
 expansion of  $ \, \textsl{Alt}_{\,3}(\circledast) \, $  has zero term at order 0\,,
 while at order 1 it also has a contribution coming from
 \eqref{eq: Alt_3(first-line)-Poisson}.
 \vskip5pt
   Now we go and compute the contribution to
   $ \, \textsl{Alt}_{\,3}(\circledast) \, $  issuing from the third line in
   \eqref{eq: ident-sigma(a,b,c)}.  Again, direct calculations give
\begin{equation}  \label{eq: Alt_3(third-line)}
   \textsl{Alt}_{\,3}\big(\textsl{3${}^{\text{\,rd}}$  line in\ }
   (\ref{eq: ident-sigma(a,b,c)}) \big)  \,\; = \;\,
   \varsigma_a\big(\, a_{(1)} \, , b_{(1)} \big) \,
   \varsigma_a\big(\, c \, , a_{(2)} \, b_{(2)} \big)  \; + \;  \text{c.p.}
\end{equation}
 Since one always has  $ \, x = \epsilon(x) + x_+ \, $  with
 $ \, x_+ := \big( x - \epsilon(x) \big) \in \Ker(\epsilon) \, $,
 \,applying this to each element  $ \, x \in \{a\,,b\,,c\,\} \, $
 occurring in  \eqref{eq: Alt_3(third-line)},
 then expanding everything and taking into account that
%%%
 $ \,\; \varsigma_a\big( J_\hbar \, , J_\hbar^{\,2} \big) \,
 \underset{\hbar}{\equiv} \, 0 \, \underset{\hbar}{\equiv} \,
 \varsigma_a\big( J_\hbar^{\,2} , J_\hbar \big) \;\, $
 --- cf.\ \eqref{eq: varsigma_a(a,bc)}  ---   we obtain
  $$
  \displaylines{
   \textsl{Alt}_{\,3}\big(\textsl{3${}^{\text{\,rd}}$  line in\ }
   (\ref{eq: ident-sigma(a,b,c)}) \big)  \,\; \underset{\hbar}{\equiv} \;\,
   \varsigma_a\big(\, a \, , b_{(1)} \big) \, \varsigma_a\big(\, c \, , b_{(2)} \big) \,
   + \, \varsigma_a\big(\, a_{(1)} \, , b \,\big) \, \varsigma_a\big(\, c \, , a_{(2)} \big)  \;
   + \;  \text{c.p.}   \hfill  \cr
   \hfill   = \;\,  \varsigma_a\big(\, a \, ,
   b^{\,\wedge}_{(1)} \big) \, \varsigma_a\big(\, c \, , b^{\,\wedge}_{(2)} \big)  \;
   + \;  \text{c.p.}  }  $$
 where we make use of short-hand notation
 $ \; x^{\,\wedge}_{(1)} \otimes x^{\,\wedge}_{(2)} :=
 x_{(1)} \otimes x_{(2)} - x_{(2)} \otimes x_{(1)} \; $.
 \vskip5pt
   Finally, we go and compute the contribution to
   $ \, \textsl{Alt}_{\,3}(\circledast) \, $  issuing from the second line in
   \eqref{eq: ident-sigma(a,b,c)}.  Dropping the coefficients
   $ \, \hbar \, $  and  $ \, 1\big/2 \, $  we find the following:
  $$
  \displaylines{
  \textsl{Alt}_{\,3}\big(\textsl{2${}^{\text{\,nd}}$  line in\ }
  (\ref{eq: ident-sigma(a,b,c)}) \big)  \; = \;
  \textsl{Alt}_{\,3}\big(\, \varsigma^{*2}(a\,,b\,c\,) \, - \,
  \varsigma^{*2}(b\,,a\,c\,) \big)  \,\; =   \hfill  \cr
   = \;\,  \varsigma^{*2}(a\,,b\,c\,) \, - \, \varsigma^{*2}(b\,,a\,c\,) \, - \,
   \varsigma^{*2}(a\,,c\,b\,) \, - \, \varsigma^{*2}(c\,,b\,a\,) \, + \,
   \varsigma^{*2}(c\,,a\,b\,) \, + \, \varsigma^{*2}(b\,,c\,a\,)  \; -  \cr
   \qquad   - \;  \varsigma^{*2}(a\,b\,,c\,) \, + \, \varsigma^{*2}(b\,a\,,c\,) \, + \,
   \varsigma^{*2}(a\,c\,,b\,) \, + \, \varsigma^{*2}(c\,b\,,a\,) \, - \,
   \varsigma^{*2}(c\,a\,,b\,) \, - \, \varsigma^{*2}(b\,c\,,a\,)  \; =
  \cr
   \qquad   = \;\,  \varsigma^{*2}\big(a\,,[b\,,c\,]\big) \, + \,
   \varsigma^{*2}\big(b\,,[c\,,a\,]\big) \, + \, \varsigma^{*2}\big(c\,,[a\,,b\,]\big) \; -
   \cr
   \qquad \qquad \qquad   - \;  \varsigma^{*2}([b\,,c\,]\,,a\,\big) \, - \,
   \varsigma^{*2}([c\,,a\,]\,,b\,\big) \, - \, \varsigma^{*2}([a\,,b\,]\,,c\,\big)  }
   $$
 which in turn implies
  $ \,\; \textsl{Alt}_{\,3}\big(\, \varsigma^{*2}(a\,,b\,c\,) \, - \,
  \varsigma^{*2}(b\,,a\,c\,) \big) \; = \; \cO(\hbar) \;\, $
   --- since  $ \, [\,u\,,v\,] = \cO(\hbar) \, $  for all  $ \, u , v \in \fhg \, $.
   The outcome then is that the contribution to  $ \, \textsl{Alt}_{\,3}(\circledast) \, $
   given by the second line in  \eqref{eq: ident-sigma(a,b,c)}  is trivial modulo
   $ \hbar^2 $.
 \vskip7pt
   Summing up, the outcome of the previous analysis is that
  $$
  \varsigma_a\Big( a \, , \big\{ \overline{b} \, , \overline{c} \,\big\}' \Big)  \; + \;
  \text{c.p.}  \; + \;  \varsigma_a\big(\, a \, ,
  b^{\,\wedge}_{(1)} \big) \, \varsigma_a\big(\, c \, , b^{\,\wedge}_{(2)} \big)  \; +
  \;  \text{c.p.}
%%%
   \,\; \underset{\hbar}{\equiv} \;\,  0
   $$
 Taking the latter modulo  $ \, \hbar \, \fhg \, $  we find
  $$  \varsigma_a\big(\, \overline{a} \, , \big\{ \overline{b} \, , \overline{c} \,\big\} \big)
  \; + \;  \text{c.p.}  \; + \;
  \varsigma_a\Big(\, \overline{a} \, , \overline{b}^{\,\wedge}_{(1)} \Big) \,
  \varsigma_a\Big(\, \overline{c} \, , \overline{b}^{\,\wedge}_{(2)} \Big)  \; + \;
  \text{c.p.}  \,\; = \;\,  0  $$
 for the elements  $ \, \overline{a} \, , \overline{b} \, , \overline{c} \in F[[G\,]] \, $.
                                                                   \par
   Now recall that for  $ \, x \in J_\hbar \, $  with
   $ \, \overline{x} := \big(\, x \text{\ mod\ } \hbar\,\fhg \big) \, $  and
   $ \, \text{x} := \big(\, \overline{x} \text{\ mod\ } \liem^2 \big) \, $
   we have
%%%
 $ \; \delta(\text{x}) \, := \,
 \text{x}^{\,\wedge}_{(1)} \otimes \text{x}^{\,\wedge}_{(2)} \; $
%%%
 for the induced Lie cobracket of  $ \, \lieg^* = \liem \big/ \liem^2 \, $
 computed on  $ \text{x} \, $,
 \,by definition; this means that, using our previously established notation
 $ \; \delta(\text{x}) \, := \, \text{x}_{[1]} \otimes \text{x}_{[2]} \, $,
 \,the last formula above yields
\begin{equation}  \label{eq: polar zeta_cocycle}
   \zeta\big(\, \text{a} \, , [\, \text{b} \, , \text{c} \,] \big)  \; + \;  \text{c.p.}  \; +
   \;  \zeta\big(\, \text{a} \, , \text{b}_{[1]} \big) \,
   \zeta\big(\, \text{c} \, , \text{b}_{[2]} \big)  \; + \; \text{c.p.}  \,\; = \;\,  0
\end{equation}
 Finally, the antisymmetry of  $ \zeta $  gives
 $ \,\; \zeta\big(\, \text{a} \, , [\, \text{b} \, , \text{c} \,] \big) \, + \, \text{c.p.} \; = \;
 -\zeta\big(\, [\, \text{a} \, , \text{b} \,] \, , \text{c} \,\big) \, + \, \text{c.p.} \; $,
 \;while a straightforward check shows that
 $ \,\; \zeta\big(\, \text{a} \, , \text{b}_{[1]} \big) \,
 \zeta\big(\, \text{c} \, , \text{b}_{[2]} \big) \, + \, \text{c.p.} \; = \;
 -{[[\,\zeta \, , \zeta \,]]}_* \;\, $.  Therefore,  \eqref{eq: polar zeta_cocycle}
 is equivalent to
%%%%%
  $$
  \zeta\big(\, [\, \text{a} \, , \text{b} \,] \, , \text{c} \,\big)  \; + \;
  \text{c.p.}  \; + \;  {[[\,\zeta \, , \zeta \,]]}_*  \,\; = \;\,  0
  $$
%%%%%
 which means that  $ \, \zeta \, $  is indeed a (strong type of) 2--cocycle for
 $ \lieg^* $,  \;q.e.d.
 \vskip7pt
   \textit{(c)}\,  Let us consider the deformed algebra
   $ {\big( \fhg \big)}_\sigma \, $,  which coincides with  $ \fhg $  as a
   $ \kh $--module but is endowed with the deformed multiplication
   ``\;\raisebox{-5pt}{$ \dot{\scriptstyle \sigma} $}\;''  defined by
\begin{equation}  \label{eq: deform-multipl-fhg}
  a \,\raisebox{-5pt}{$ \dot{\scriptstyle \sigma} $}\, b  \; := \;
  \sigma\big(a_{(1)},b_{(1)}\big) \, a_{(2)} \, b_{(2)} \,
  \sigma^{-1}\big(a_{(3)},b_{(3)}\big)   \qquad \qquad  \forall \; a , b \in \fhg
\end{equation}
 As  $ \sigma $  is of the form  $ \, \sigma = \exp_*\!\big( \hbar \, \varsigma \big) \, $,
 \,in particular it is trivial modulo  $ \hbar \, $,  \,it follows from
 \eqref{eq: deform-multipl-fhg}  that the deformed multiplication
 ``\;\raisebox{-5pt}{$ \dot{\scriptstyle \sigma} $}\;''
 coincides with the old one modulo  $ \hbar \, $,
 that is  $ \; a \,\raisebox{-5pt}{$ \dot{\scriptstyle \sigma} $}\, b \, \equiv \,
 a \, b \;\; \big(\, \text{mod\ } \hbar\,\fhg \big) \; $.
 Therefore,  $ {\big( \fhg \big)}_\sigma $  is again commutative modulo
 $ \hbar \, $,  hence it is (again) a QFSHA, as claimed, say
 $ \, {\big( \fhg \big)}_\sigma = F_\hbar\big[\big[G_{(\sigma)}\big]\big] \, $.
 Then, in order to prove that the newly found Poisson (formal) group  $ G_{(\sigma)} $  is indeed
 $ G_\zeta $  as claimed it is enough to show that the
 Lie bracket induced in  $ \, \liem_\sigma \big/ \liem_\sigma^{\,2} \, $
 --- where  $ \, \liem_\sigma := \Ker \big( \epsilon_{{}_{F[[G_{(\sigma)}]]}} \big) \, $
 ---   is indeed  $ \, {[\ \,,\ ]}_\zeta \, $.
 \vskip3pt
   Let us take  $ \, \text{a} \, , \text{b} \in \liem_\sigma \big/ \liem_\sigma^{\,2} \, $;
   \,then we can pick  $ \, a , b \in J_\hbar := \Ker\big(\epsilon_{{}_\fhg}\big) \, $
   such that
   $ \; \text{a} = a \; \Big(\text{mod\ } \big(\, \hbar\,J_\hbar + J_\hbar^{\,2} \big) \Big) \;
   $  and
   $ \; \text{b} = b \;
   \Big(\text{mod\ } \big(\, \hbar\,J_\hbar + J_\hbar^{\,2} \big) \Big) \; $.
   Now, using the expansion
   $ \,\; \sigma \, = \, \exp_*\!\big( \hbar \, \varsigma \big) \, = \,
   \epsilon^{\otimes 2} \, + \, \hbar \, \varsigma \, + \, \cO\big(\hbar^2\big) \; $,
   \;formula  \eqref{eq: deform-multipl-fhg}  turns into
  $$
  \displaylines{
   a \,\raisebox{-5pt}{$ \dot{\scriptstyle \sigma} $}\, b  \,\; =   \hfill  \cr
%%%
   = \, \big( \epsilon\big(a_{(1)}\big) \epsilon\big(b_{(1)}\big) +
   \hbar \, \varsigma\big(a_{(1)} , b_{(1)}\big) \big) \, a_{(2)} \, b_{(2)} \,
   \big( \epsilon\big(a_{(3)}\big) \epsilon\big(b_{(3)}\big) + \hbar \,
   \varsigma\big(a_{(3)} , b_{(3)}\big) \big) \, + \, \cO\big(\hbar^2\big) \, =  \cr
%%%
   = \;\,  \epsilon\big(a_{(1)}\big) \, \epsilon\big(b_{(1)}\big) \, a_{(2)} \, b_{(2)} \,
   \epsilon\big(a_{(3)}\big) \, \epsilon\big(b_{(3)}\big) \; +   \hfill  \cr
%%%
   \hfill   + \, \hbar \, \Big( \varsigma\big(a_{(1)} , b_{(1)}\big) \, a_{(2)} b_{(2)} \,
   \epsilon\big(a_{(3)}\big) \epsilon\big(b_{(3)}\big) \, - \, \epsilon\big(a_{(1)}\big)
   \epsilon\big(b_{(1)}\big) \, a_{(2)} b_{(2)} \,
   \varsigma\big(a_{(3)} , b_{(3)}\big) \Big) \, + \, \cO\big(\hbar^2\big) \, =  \cr
%%%
   \hfill   = \;\,  a \, b \; + \; \hbar \,
   \Big( \varsigma\big(a_{(1)} \, , b_{(1)}\big) \, a_{(2)} \, b_{(2)} \, -
   \, a_{(1)} \, b_{(1)} \, \varsigma\big(a_{(2)} \, , b_{(2)}\big) \Big) \; +
   \; \cO\big(\hbar^2\big)  }
   $$
 where we took into account coassociativity and counitality properties.
 Therefore, using  ``$ \, {[\ \,,\ ]}_\sigma \, $''  and  ``$ \, [\ \,,\ ] \, $''
 to denote the commutator with respect to the new (deformed) and the old
 (undeformed) multiplication, we also have
  $$
  \displaylines{
   {[a\,,b\,]}_\sigma  \,\; := \;\,  a \,\raisebox{-5pt}{$ \dot{\scriptstyle \sigma} $}\, b \,
   - \, b \,\raisebox{-5pt}{$ \dot{\scriptstyle \sigma} $}\, a  \,\; =   \hfill  \cr
%%%
   = \;\,  a \, b \; + \; \hbar \,
   \Big(\, \varsigma\big(a_{(1)} \, , b_{(1)}\big) \, a_{(2)} \, b_{(2)} \, -
   \, a_{(1)} \, b_{(1)} \, \varsigma\big(a_{(2)} \, , b_{(2)}\big) \Big) \; + \;
   \cO\big(\hbar^2\big) \,\; -   \cr
%%%
   \hfill   - \;\,  b \, a \; - \; \hbar \; \Big(\, \varsigma\big(b_{(1)} \, ,
   a_{(1)}\big) \, b_{(2)} \, a_{(2)} \, - \, b_{(1)} \, a_{(1)} \,
   \varsigma\big(b_{(2)} \, , a_{(2)}\big) \Big) \; + \; \cO\big(\hbar^2\big)  \,\; =  \cr
%%%
   = \;\,  a \, b \, - \, b \, a \,\; +
%%%%%
% \hfill  \cr
%    \qquad \quad   +
%%%%%
 \;\,  \hbar \; \Big( - \varsigma\big(a_{(2)} \, , b_{(2)}\big) \, a_{(1)} \, b_{(1)} \, +
 \, \varsigma\big(b_{(2)} \, , a_{(2)}\big) \, b_{(1)} \, a_{(1)}  \,\; -   \hfill  \cr
   \qquad \quad \qquad \quad   - \;\,
   \varsigma\big(b_{(1)} \, , a_{(1)}\big) \, b_{(2)} \, a_{(2)} \, + \,
   \varsigma\big(a_{(1)} \, , b_{(1)}\big) \, a_{(2)} \, b_{(2)} \, \Big)  \,\; + \;\,
   \cO\big(\hbar^2\big)  \,\; =   \hfill  \cr
%%%
   \hfill   = \;\,  [\,a\,,b\,] \,\; + \;\,  \hbar \;
   \Big( - \varsigma_a\big(a_{(2)} \, , b_{(2)}\big) \, a_{(1)} \, b_{(1)} \, - \,
   \varsigma_a\big(b_{(1)} \, , a_{(1)}\big) \, b_{(2)} \, a_{(2)} \, \Big)  \,\; +
   \;\,  \cO\big(\hbar^2\big)  }
   $$
 where for the last step we used the fact that
 $ \; a_{(s)} \, b_{(s)} \,\underset{\hbar}{\equiv}\, b_{(s)} \, a_{(s)} \; $.
                                                                                 \par
   Recall that
   $ \; [\,a\,,b\,] \, = \, \hbar \, {\big\{ \overline{a} \, , \overline{b} \,\big\}}' \, $,
   \,where hereafter we write
   $ \, \overline{x} := x \; \big(\, \text{mod\ } \hbar\,J_\hbar \big) \, $  and
   $ \, \text{f\,}' \, $  to denote any lift in  $ J_\hbar $  of some given  $ \text{f} $
   in  $ J_\hbar \, $,  \,as we did before; similarly, we have
   $ \; {[\,a\,,b\,]}_\sigma = \,
   \hbar \, {\big\{ \overline{a} \, , \overline{b} \,\big\}}_\sigma' \, $.
   Then modulo  $ \, \hbar \, $  our previous computations give
\begin{equation}  \label{eq: ident-deform-Poisson-bra}
  {\big\{ \overline{a} \, , \overline{b} \,\big\}}_\sigma  \,\; = \;\,
  \big\{ \overline{a} \, , \overline{b} \,\big\}  \; - \;
  \overline{\varsigma}_a\big(\, \overline{a}_{(2)} \, , \overline{b}_{(2)} \big) \,
  \overline{a}_{(1)} \, \overline{b}_{(1)} \; - \;
  \overline{\varsigma}_a\big(\, \overline{b}_{(1)} \, ,
  \overline{a}_{(1)} \big) \, \overline{b}_{(2)} \, \overline{a}_{(2)}
\end{equation}
 For each  $ \, x \in \big\{\, \overline{a}_{(s)} \, ,
 \overline{b}_{(s)} \,\big|\, s=1,2 \,\big\} \, $  we have
 $ \, x = \epsilon(x) + x^+ \, $  with
 $ \, x^+ := \big( x - \epsilon(x) \big) \in J_\hbar \, $.
 Using such expansions in either factor of the products
 $ \, \overline{a}_{(s)} \, \overline{b}_{(s)} \, $  and
 $ \, \overline{b}_{(s)} \, \overline{a}_{(s)} \, $
 occurring within  \eqref{eq: ident-deform-Poisson-bra},
 and noting that
 $ \; \overline{a}_{(s)}^{\,+} \, \overline{b}_{(s)}^{\,+} \,
 \underset{\,\liem^2}{\equiv}\, 0 \,\underset{\,\liem^2}{\equiv}\,
 \overline{a}_{(s)}^{\,+} \, \overline{b}_{(s)}^{\,+} \, $, \,we get
%
%%%%%
% two summands
% %
%   $ \; -\,\overline{\varsigma}_a\big(\, \overline{a}_{(2)} \, , \overline{b}_{(2)} \big) \,
% \epsilon\big(\overline{a}_{(1)}\big) \, \epsilon\big(\overline{b}_{(1)}\big) \, = \, -\,\overline{\varsigma}_a(a\,,b\,) \; $
% %
%  and
% %
%   $ \; -\,\overline{\varsigma}_a\big(\, \overline{b}_{(1)} \, , \overline{a}_{(1)} \big) \,
% \epsilon\big(\overline{b}_{(2)}\big) \, \epsilon\big(\overline{a}_{(2)}\big) \, = \, -\,\overline{\varsigma}_a(b\,,a\,) \; $
% %
%  that add up to zero.  All other terms eventually yield
%%%%%
%
 an equivalence modulo  $ \, \liem^2 = \liem_\sigma^{\,2} \, $
 (noting that  $ \, \liem = \liem_\sigma \, $  as  $ \k $--modules),  namely
  $$
  \displaylines{
   {\big\{ \overline{a} \, , \overline{b} \,\big\}}_\sigma'  \,\; = \;\,
   \big\{ \overline{a} \, , \overline{b} \,\big\}  \; - \;
   \overline{\varsigma}_a\big(\, \overline{a}_{(2)} \, ,
   \overline{b}_{(2)} \big) \, \overline{a}_{(1)} \, \overline{b}_{(1)} \; - \;
   \overline{\varsigma}_a\big(\, \overline{b}_{(1)} \, , \overline{a}_{(1)} \big) \,
   \overline{b}_{(2)} \, \overline{a}_{(2)}  \,\; \underset{\,\liem^2}{\equiv}   \hfill  \cr
%%%
   \underset{\,\liem^2}{\equiv} \;\,  \big\{ \overline{a} \, ,
   \overline{b} \,\big\}  \; - \; \overline{\varsigma}_a\big(\, \overline{a}_{(2)} \, ,
   \overline{b}_{(2)} \big) \, \overline{a}_{(1)} \,
   \epsilon\big(\,\overline{b}_{(1)}\big) \; - \;
   \overline{\varsigma}_a\big(\, \overline{a}_{(2)} \, , \overline{b}_{(2)} \big) \,
   \epsilon\big(\overline{a}_{(1)}\big) \, \overline{b}_{(1)} \; -  \cr
   \hfill   - \; \overline{\varsigma}_a\big(\, \overline{b}_{(1)} \, ,
   \overline{a}_{(1)} \big) \, \overline{b}_{(2)} \,
   \epsilon\big(\overline{a}_{(2)}\big) \; - \;
   \overline{\varsigma}_a\big(\, \overline{b}_{(1)} \, ,
   \overline{a}_{(1)} \big) \, \epsilon\big(\,\overline{b}_{(2)}\big) \,
   \overline{a}_{(2)}  \,\; =  \cr
%%%
   = \;\,  \big\{ \overline{a} \, , \overline{b} \,\big\}  \; - \;
   \overline{\varsigma}_a\big(\, \overline{a}_{(2)} \, ,
   \overline{b} \,\big) \, \overline{a}_{(1)} \; - \;
   \overline{\varsigma}_a\big(\, \overline{a} \, ,
   \overline{b}_{(2)} \big) \, \overline{b}_{(1)} \; - \;
   \overline{\varsigma}_a\big(\, \overline{b}_{(1)} \, ,
   \overline{a} \,\big) \, \overline{b}_{(2)} \; - \;
   \overline{\varsigma}_a\big(\, \overline{b} \, ,
   \overline{a}_{(1)} \big) \, \overline{a}_{(2)}  \,\; =  \cr
%%%
   = \;\,  \big\{ \overline{a} \, , \overline{b} \,\big\}  \; - \;
   \overline{\varsigma}_a\big(\, \overline{a}_{(2)} \, , \overline{b} \,\big) \,
   \overline{a}_{(1)} \; + \; \overline{\varsigma}_a\big(\, \overline{b}_{(2)} \, ,
   \overline{a} \big) \, \overline{b}_{(1)} \; - \;
   \overline{\varsigma}_a\big(\, \overline{b}_{(1)} \, ,
   \overline{a} \,\big) \, \overline{b}_{(2)} \; + \;
   \overline{\varsigma}_a\big(\, \overline{a}_{(1)} \, , \overline{b} \,\big) \,
   \overline{a}_{(2)}  \,\; =  \cr
%%%
   \hfill   = \;\,  \big\{ \overline{a} \, , \overline{b} \,\big\}  \, - \,
   \big(\, \overline{\varsigma}_a\big(\, \overline{a}_{(2)} , \overline{b} \,\big) \,
   \overline{a}_{(1)} - \,
   \overline{\varsigma}_a\big(\, \overline{a}_{(1)} ,
   \overline{b} \,\big) \, \overline{a}_{(2)} \,\big) \, - \,
   \big(\, \overline{\varsigma}_a\big(\, \overline{b}_{(1)} , \overline{a} \,\big) \,
   \overline{b}_{(2)} - \, \overline{\varsigma}_a\big(\, \overline{b}_{(2)} ,
   \overline{a} \,\big) \, \overline{b}_{(1)} \,\big)  }
   $$
 where the element in last line actually belongs to  $ \, \liem = \liem_\sigma \, $.
 When we reduce all this modulo  $ \, \liem^2 = \liem_\sigma^{\,2} \, $,
 \,and recalling the definition of the Lie bracket (for either Lie algebra structure)
 as being induced by the Poisson bracket, and that of the Lie cobracket in
 $ \, \liem \Big/ \liem^2 = \liem_\sigma \Big/ \liem_\sigma^{\,2} \, $
 --- which coincide as Lie coalgebras ---   as being induced by
 $ \, \Delta - \Delta^{\text{op}} \, $,  \,we eventually end up with
  $$  {[\,\text{a} \, , \text{b}\,]}_{(\sigma)}  \,\; = \;\,  {[\,\text{a} \, , \text{b}\,]}_*
   \; - \;  \zeta\big( \text{a}_{[2]} , \text{b} \big) \, \text{a}_{[1]}
   \; - \;  \zeta\big(\, \text{b}_{[1]} , \text{a} \big) \, \text{b}_{[2]}
   \,\; =: \;\,  {\big( {[\,\text{a} \, , \text{b}\,]}_* \big)}_\zeta  $$
   thus(cf.\ Definition \ref{eq: def_cocyc-bracket})
   the Lie bracket we were looking for is just
   $ \, {\big( {[\ \,,\ ]}_* \big)}_\zeta \, $.
\epf

\vskip9pt

\begin{obs}  \label{obs: trivial deform QFSHA}
 We would better point out that the 2-cocycles
 $ \sigma $  considered in
 Theorem \ref{thm: 2cocycle-deform-QFSHA}
 above are those of  ``trivial-modulo-$ \hbar $-type'',
 in that they are the identity modulo  $ \hbar \, $.
 With this assumption, deforming  $ \fhg $
 by such a  $ \sigma $  does not affect the Hopf structure
 of the semiclassical limit;
 in particular, it still reads as
 $ F\big[\big[\widetilde{G}\hskip0,5pt\big]\big] \, $,  with
 $ \widetilde{G} $  being the same formal group as  $ G $
 but with a different Poisson
 structure.  A more general 2-cocycle might be ``unfit'',
 in that the deformed Hopf algebra
 $ {\big( \fhg \big)}_\sigma $  may no longer be a QFSHA,
 in general.
\end{obs}

\begin{exa}  \label{example: toral 2-cocycles for QFSHAs}
 Let  $ \, G := \textit{GL}_{\,n}(\k) \, $  be the general linear
 group over  $ \k \, $,  \,and
 $ \, \lieg := \mathfrak{gl}_n(\k) \, $  its tangent Lie algebra.
 It is well-known   --- cf.\ \cite{Dr}, \cite{CP}  ---
 that a quantization of  $ \lieg $  is provided by the QUEA
 $ \, \uhg = U_\hbar\big(\mathfrak{gl}_n(\k)\hskip-1pt\big) \, $
 defined as follows: it is the unital, associative,
 $ \hbar $--adically  complete  $ \kh $--algebra  with
 generators
  $$  F_1 \, , F_2 \, , \dots, F_{n-1} \, , \varGamma_1 \, ,
  \varGamma_2 \, , \dots , \varGamma_{n-1} \, , \varGamma_n \, , E_1 \, , E_2 \, , \dots , E_{n-1}  $$
 and relations (for all  $ \, i , j \in \{1,\dots,n-1\} \, $,
 $ \, k , \ell \in \{1,\dots,n\} \, $)
  $$
  \displaylines {
   \big[\varGamma_k \, , \varGamma_\ell\big] \, = \, 0
   \;\; ,  \quad  \big[\varGamma_k \, , F_j\big] \, = \,
   - \delta_{k,j} \, F_j \;\; ,
   \quad  [\varGamma_k \, , E_j] \, = \,
   +\delta_{k,j} \, E_j  \cr
   \big[E_i \, , F_j\big] \, = \, \delta_{i,j} \,
   {{\;\; e^{\hbar\,(\varGamma_i - \varGamma_{i+1})} -
   e^{\hbar\,(\varGamma_{i+1} - \varGamma_i)} \;\;} \over
   {\;\; e^{+\hbar} - e^{-\hbar} \;\;}}  \cr
   \hfill   \big[E_i \, , E_j\big] \, = \, 0  \;\; ,
   \qquad  \big[F_i \, , F_j\big] \, = \, 0   \hfill
\forall \;\; i \, , j \;\colon \vert i - j \vert > 1 \,
\phantom{.} \;  \cr
   \hfill   E_i^2 \, E_j - \left( q + q^{-1} \right) E_i \, E_j \, E_i
   + E_j \, E_i^2 \, = \, 0
\hfill  \forall \;\; i \, , j \;\colon \vert i - j \vert = 1 \,
\phantom{.} \;  \cr
   \hfill   F_i^2 \, F_j - \left( q + q^{-1} \right) F_i \, F_j \, F_i +
   F_j \, F_i^2 \, = \, 0
\hfill  \forall \;\; i \, , j \;\colon \vert i - j \vert = 1 \, . \;  \cr }
$$
 where  $ \, [X\,,Y] := X\,Y - Y\,X \, $.
 The (topological) Hopf algebra structure is given by
  $$
  \displaylines {
   \Delta(F_i) \, = \, F_i \otimes e^{\hbar\,(\varGamma_{i+1}
   - \varGamma_i)} + 1 \otimes
F_i \; ,  \; \qquad  S(F_i) \, = \, -F_i \,
e^{\hbar\,(\varGamma_i - \varGamma_{i+1})}
\; ,  \; \qquad  \epsilon(F_i) \, = \, 0  \cr
   \quad \,  \Delta(\varGamma_k) \, = \,
   \varGamma_k \otimes 1 + 1 \otimes \varGamma_k \; ,
   \qquad \qquad \qquad  S(\varGamma_k) \, = \,
   -\varGamma_k \; ,  \quad \quad \qquad
   \epsilon(\varGamma_k) \, = \, 0  \cr
   \Delta(E_i) \, = \, E_i \otimes 1 + e^{\hbar\,(\varGamma_i
   - \varGamma_{i+1})} \otimes
E_i \; ,  \qquad  S(E_i) \, = \,
-e^{\hbar\,(\varGamma_{i+1} - \varGamma_i)} \, E_i \; ,
\qquad
\epsilon(E_i) \, = \, 0  }
$$
 \vskip3pt
   It is also well-known   --- cf.\ \cite{Dr}, \cite{CP}
   ---   that a quantization of  $ \, G := \textit{GL}_{\,n}(\k) \, $
   is provided by the QFSHA
   $ \, \fhg = F_\hbar\big[\big[\textit{GL}_{\,n}(\k)\big]\big] \, $
   defined as follows: it is the unital, associative,
   $ I_\hbar $--adically  complete  $ \kh $--algebra
   generated by the elements of the set
   $ \, \big\{ x_{ij} \,\big|\, i, j = 1, \ldots, n+1 \big\} $
   \textsl{arranged in a  $ q $--matrix,  with
   $ \, q := \exp(\hbar) \, $},  \,with  $ I_\hbar $
   being the ideal generated by
   $ \, \big\{ \hbar \, , x_{1,1} \, , \dots , x_{n,n} \big\} \, $;
   \,this is a quick way to say that the given generators
   obey the relations
\begin{equation}  \label{eq: rel.s_xij}
\begin{aligned}
   \quad   x_{ij} \, x_{ik} \, = \, q \, x_{ik} \, x_{ij} \; ,
   \quad \quad  x_{ik} \, x_{hk} \, = \, q \, x_{hk} \, x_{ik}
   \quad \qquad  \forall \;\; j<k \, , i<h  \qquad  \\
   x_{il} \, x_{jk} \, = \, x_{jk} \, x_{il} \; ,
   \quad  x_{ik} \, x_{jl} - x_{jl} \, x_{ik} \, = \,
   \left( q - q^{-1} \right) \, x_{il} \, x_{jk}  \qquad   \forall \;\; i<j \, ,
   k<l  \;
\end{aligned}
\end{equation}
%
%%%%%
%   $$  \displaylines{
%    \hfill   x_{ij} \, x_{ik} \, = \, q \, x_{ik} \, x_{ij} \; ,
%    \quad \quad  x_{ik} \, x_{hk} \, = \, q \, x_{hk} \, x_{ik}
%    \hfill  \forall \;\; j<k \, , i<h  \quad  \cr
%    \hfill   x_{il} \, x_{jk} \, = \, x_{jk} \, x_{il} \; ,
%    \quad \quad  x_{ik} \, x_{jl} - x_{jl} \, x_{ik} \, = \,
%    \left( q - q^{-1} \right) \, x_{il} \, x_{jk}  \hfill   \forall \;\; i<j \, ,
%    k<l  \quad  }  $$
%%%%%
%
 whereas the comultiplication  $ \Delta \, $,  the counit
 $ \epsilon \, $,
and the antipode  $ S $  are given by matrix formulation
  $$
  \displaylines{
   \Delta\Big(\!{\big(x_{ij}\big)}_{i=1,\dots,n;}^{j=1,\dots,n;}
   \Big) \, := \, {\big(x_{ij}\big)}_{i=1,\dots,n;}^{j=1,\dots,n;}
   \otimes {\big(x_{ij}\big)}_{i=1,\dots,n;}^{j=1,\dots,n;}  \cr
   \epsilon\Big(\!{\big(x_{ij}\big)}_{i=1,\dots,n;}^{j=1,\dots,n;}
   \Big) \, := \, {\big(\delta_{ij}\big)}_{i=1,\dots,n;}^{j=1,\dots,n;}
   \quad ,  \qquad
   S\Big(\!{\big(x_{ij}\big)}_{i=1,\dots,n;}^{j=1,\dots,n;}\Big) \,
   := \, {\Big(\! {\big(x_{ij}\big)}_{i=1,\dots,n;}^{j=1,\dots,n;} \Big)}^{-1}}
   $$
 which in down-to-earth terms read, for all
 $ \, i, j = 1, \dots, n \, $,
  $$
  \Delta\big(x_{ij}\big) \, = \,
  {\textstyle \sum_{k=1}^n} x_{ik} \otimes x_{kj}  \;\; ,
  \qquad  \epsilon\big(x_{ij}\big) \, = \, \delta_{ij}  \;\; ,
  \qquad  S\big(x_{ij}\big) \, =
  \, {(-q)}^{j-i} D_q\!\left(\! {\big(x_{hk}\big)}_{h \neq j}^{k \neq i}
  \right)
  $$
 where  $ D_q $  denotes the so-called
 {\it quantum determinant\/},  defined on any
 $ q $--matrix  of (square) size  $ \ell $  by
  $$
  D_q\Big(\!{\big(x_{ij}\big)}_{i=1,\dots,\ell;}^{j=1,\dots,\ell;}\Big)
  \, := \,  {\textstyle \sum_{\sigma \in S_\ell}}
  {(-q)}^{l(\sigma)} x_{1,\sigma(1)} \, x_{2,\sigma(2)}
  \cdots x_{\ell,\sigma(\ell)}
  $$
 \vskip3pt
   We have also explicit identifications
   $ \, \fhg = {\uhg}^\ast \, $  as well as
   $ \, \uhg = {\fhg}^\star \, $,
   which can be described via the Hopf pairing
   $ \; \langle\,\ ,\ \rangle : \fhg \times \uhg
   \relbar\joinrel\relbar\joinrel\longrightarrow \kh \; $
   uniquely given by the following values on generators:
\begin{equation}   \label{eq: pairing fhg x uhg}
  \big\langle x_{i,j} \, , \varGamma_k \big\rangle \, = \,
  \delta_{i,j} \, \delta_{i,k}  \;\; ,  \quad  \big\langle x_{i,j} \, ,
  E_t \big\rangle \, = \, \delta_{i+1,j} \, \delta_{i,t}  \;\; ,
  \quad  \big\langle x_{i,j} \, , F_t \big\rangle \, = \,
  \delta_{i,j+1} \, \delta_{t,j}
\end{equation}
 \vskip5pt
   Now consider in  $ \, \uhg \,\widehat{\otimes}\, \uhg \, $
   the element
%%%
 $ \, \cF := \exp\Big(\hskip1pt \hbar \, 2^{-1}
 \sum_{k,\ell=1}^n \, \phi_{t,k} \, \varGamma_t \otimes
 \varGamma_k \Big) \, $
%%%
 that is a twist for  $ \uhg \, $   --- just as in
 Example \ref{example: toral twists for FoMpQUEAs},
 this trivially follows from the fact that the  $ \varGamma_t $'s
 are primitive.
 By  Proposition \ref{prop: duality-deforms}\textit{(a)},
 we can see this  $ \cF $  as a 2--cocycle
 $ \sigma_{{}_\cF} $  for  $ \, {\uhg}^* = \fhg \, $,
 simply given by  \textsl{evaluation at  $ \cF \, $},  namely
\begin{equation}   \label{eq: sigma_F}
  \sigma_{{}_\cF} \, : \, \fhg
  \times \fhg \relbar\joinrel\relbar\joinrel\longrightarrow \kh
  \quad ,  \qquad  (\varphi\,,\psi) \,\mapsto\, \big\langle\,
  \varphi \otimes \psi \, , \, \cF \,\big\rangle
\end{equation}
   \indent   Now, from  \eqref{eq: sigma_F}  and  \eqref{eq: pairing fhg x uhg},  direct calculation gives
  $$
  \displaylines{
   \sigma_{{}_\cF}\big( x_{i,\,r} \, , x_{\ell,\,h} \big)  \; = \;  \big\langle\, x_{i,\,r} \otimes x_{\ell,\,h} \, ,
   \, \cF \,\big\rangle  \; =   \hfill  \cr
   \qquad   = \;  {\textstyle \sum\limits_{m=0}^{+\infty}} {{\,\hbar^m\,2^{-m}\,} \over {\,m!\,}} \,
   \left\langle\, x_{i,\,r} \otimes x_{\ell,\,h} \, , \, {\left(\, {\textstyle \sum_{t,\,k=1}^n} \,
   \phi_{t,k} \, \varGamma_t \otimes \varGamma_k \right)}^m \,\right\rangle  \; =   \hfill  \cr
   = \;  {\textstyle \sum\limits_{m=0}^{+\infty}} {{\,\hbar^m\,2^{-m}\,} \over {\,m!\,}} \,
   \left\langle\, \Delta^{(m-1)}\big( x_{i,\,r} \otimes x_{\ell,\,h} \big) \, , \,
   {\left(\, {\textstyle \sum_{t,\,k=1}^n} \, \phi_{t,k} \,
   \varGamma_t \otimes \varGamma_k \right)}^{\otimes m} \,\right\rangle  }
   $$
 Let us stop and consider  $ \; \left\langle\, \Delta^{(m-1)}\big( x_{i,\,r} \otimes x_{\ell,\,h} \big) \, ,
 \, {\left(\, {\textstyle \sum_{t,\,k=1}^n} \, \phi_{t,k} \,
 \varGamma_t \otimes \varGamma_k \right)}^{\otimes m} \,\right\rangle \; $.  Then definitions give
  $$
  \displaylines{
   \left\langle\, \Delta^{(m-1)}\big( x_{i,\,r} \otimes x_{\ell,\,h} \big) \, ,
   \, {\left(\, {\textstyle \sum_{t,\,k=1}^n} \, \phi_{t,k} \,
   \varGamma_t \otimes \varGamma_k \right)}^{\otimes m} \,\right\rangle  \; =   \hfill  \cr
   = \,  {\textstyle \sum\limits_{\substack{s_1,\dots,\,s_{m-1}=1  \\
   e_1,\dots,\,e_{m-1}=1}}^n} \! \left\langle x_{i,\,s_1} \otimes x_{\ell,\,e_1}
   \otimes \cdots \otimes x_{s_{m-1},\,r} \otimes x_{e_{m-1},\,h} \, ,
   {\left( {\textstyle \sum_{t,\,k=1}^n} \, \phi_{t,k} \,
   \varGamma_t \otimes \varGamma_k \right)}^{\otimes m} \right\rangle  \; =  \cr
   \qquad   = \;  {\textstyle \sum\limits_{\substack{s_1,\dots,\,s_{m-1}=1  \\
   e_1,\dots,\,e_{m-1}=1}}^n} \, {\textstyle \prod\limits_{c=1}^m} \left\langle
   x_{s_{c-1},\,s_c} \otimes x_{e_{c-1},\,e_c} \, , \,
   {\textstyle \sum_{t,\,k=1}^n} \, \phi_{t,k} \, \varGamma_t \otimes \varGamma_k \,
   \right\rangle  \; =  \cr
   \hfill   = \;  {\textstyle \sum\limits_{\substack{s_1,\dots,\,s_{m-1}=1  \\
   e_1,\dots,\,e_{m-1}=1}}^n} \, {\textstyle \prod\limits_{c=1}^m} \,
   {\textstyle \sum_{t,\,k=1}^n} \, \phi_{t,k} \, \big\langle x_{s_{c-1},\,s_c} \, ,
   \varGamma_t \big\rangle \, \big\langle x_{e_{c-1},\,e_c} \, , \varGamma_k \, \big\rangle  }
   $$
 where we set  $ \, s_0 := i \, $,  $ \, s_m := r \, $,  $ \, e_0 := \ell \, $,  $ \, e_m := h \, $.
 Now, the formulas in  \eqref{eq: pairing fhg x uhg}  guarantee that
%%%
 $ \; \big\langle x_{s_{c-1},\,s_c} \, , \varGamma_t \big\rangle \, \big\langle x_{e_{c-1},\,e_c} \, ,
 \varGamma_k \, \big\rangle \, = \, 0 \; $
%%%
 whenever  $ \, s_{c-1} \not= s_c \, $  or  $ \, e_{c-1} \not= e_c \, $;
 \,therefore, from the previous computation one eventually gets
  $$
  \displaylines{
   \sigma_{{}_\cF}\big( x_{i,\,r} \, , x_{\ell,\,h} \big)  \; =   \hfill  \cr
   \quad   = \;  \delta_{i,r} \, \delta_{\ell,h} \,
   {\textstyle \sum\limits_{m=0}^{+\infty}} {{\,\hbar^m\,2^{-m}\,} \over {\,m!\,}} \, \left\langle\,
   \Delta^{(m-1)}\big( x_{i,r} \otimes x_{\ell,h} \big) \, , \,
   {\left(\, {\textstyle \sum_{t,\,k=1}^n} \, \phi_{t,k} \,
   \varGamma_t \otimes \varGamma_k \right)}^{\otimes m} \,\right\rangle  \; =   \hfill  \cr
   \quad \quad   = \;  \delta_{i,r} \, \delta_{\ell,h} \,
   {\textstyle \sum\limits_{m=0}^{+\infty}} {{\,\hbar^m\,2^{-m}\,} \over {\,m!\,}} \,
   {\left(\, {\textstyle \sum_{t,\,k=1}^n} \, \phi_{t,k} \, \big\langle x_{i,\,i} \, ,
   \varGamma_t \big\rangle \, \big\langle x_{\ell,\,\ell} \, , \varGamma_k \, \big\rangle \right)}^m  \; =
   \hfill  \cr
   \quad \quad \quad   = \;  \delta_{i,r} \, \delta_{\ell,h} \,
   {\textstyle \sum\limits_{m=0}^{+\infty}} {{\,\hbar^m\,2^{-m}\,} \over {\,m!\,}} \,
   {\left(\, {\textstyle \sum_{t,\,k=1}^n} \, \phi_{t,k} \, \delta_{i,t} \, \delta_{\ell,k} \right)}^m  \; =
   \hfill  \cr
   \quad \quad \quad \quad   = \;  \delta_{i,r} \, \delta_{\ell,h} \,
   {\textstyle \sum\limits_{m=0}^{+\infty}} {{\,\hbar^m\,2^{-m}\,} \over {\,m!\,}} \, {(\phi_{i,\ell})}^m  \; = \;
   \delta_{i,r} \, \delta_{\ell,h} \, \exp\big( \hbar \, 2^{-1} \phi_{i,\ell} \big)  \; = \;
   \delta_{i,r} \, \delta_{\ell,h} \, e^{\hbar \, \phi_{i,\ell} / 2}   \hfill  }
   $$
 in short
%%%
\begin{equation}  \label{eq: formula x sigma_F}
  \sigma_{{}_\cF}\big( x_{i,\,r} \, , x_{\ell,\,h} \big)  \; = \;  \delta_{i,r} \,
  \delta_{\ell,h} \, e^{\hbar \, \phi_{i,\ell} / 2}   \qquad \qquad
  \forall \;\; i \, , r , \ell , h \in \{1,\dots,n\}
\end{equation}
%%%
   \indent   Using this formula, the deformed product in  $ {\fhg}_{\sigma_{{}_\cF}} $
   can be described as follows:   %
  $$
  \displaylines{
   x_{i,\,j} \raise-1pt\hbox{$ \; \scriptstyle \dot\sigma_{{}_\cF} $}\, x_{\ell,\,t}  \; := \;
   \sigma_{{}_\cF}\big( {(x_{i,\,j})}_{(1)} \, , \, {(x_{\ell,\,t})}_{(1)} \big) \;
   {(x_{i,\,j})}_{(2)} \, {(x_{\ell,\,t})}_{(2)} \; \sigma_{{}_\cF}^{-1}\big( {(x_{i,\,j})}_{(3)} \, ,
   \, {(x_{\ell,\,t})}_{(3)} \big)  \; = \hfill  \cr
   = \;  {\textstyle \sum\limits_{r,\,s,\,h,\,k=1}^n} \sigma_{{}_\cF}\big( x_{i,\,r} \, , \, x_{\ell,\,h} \big) \;
   x_{r,\,s} \, x_{h,\,k} \; \sigma_{{}_\cF}^{-1}\big( x_{s,\,j} \, , \, x_{k,\,t} \big)  \; =
   \quad \qquad \qquad  \cr
   \hfill   = \;  \sigma_{{}_\cF}\big( x_{i,\,i} \, , \, x_{\ell,\,\ell} \big) \; x_{i,\,j} \, x_{\ell,\,t} \;
   \sigma_{{}_\cF}^{-1}\big( x_{j,\,j} \, , \, x_{t,\,t} \big)  \; =  \;  e^{\hbar \, (\phi_{i,\ell} \, - \,
   \phi_{j,t}) / 2} \; x_{i,\,j} \, x_{\ell,\,t}  }
   $$
 --- where we used  \eqref{eq: formula x sigma_F};  in short, we get
%%%
\begin{equation}  \label{eq: formula x sigma-deformed-product}
  x_{i,\,j} \raise-1pt\hbox{$ \; \scriptstyle \dot\sigma_{{}_\cF} $}\, x_{\ell,\,t}  \; = \;
  e^{\hbar \, (\phi_{i,\ell} \, - \, \phi_{j,t}) / 2} \; x_{i,\,j} \, x_{\ell,\,t}
  \qquad \qquad  \forall \;\; i \, , j , \ell , t \in \{1,\dots,n\}
\end{equation}
%%%
 \vskip5pt
   Note that this formula shows how the new, deformed product is equivalent modulo
   $ \, \hbar \, $  to the old one: this is a general fact, due to the very construction,
   namely because we are working with 2--cocycles of the form
   $ \, \exp\big( \hbar \, \varsigma \big) \, $  where  $ \, \varsigma \, $
   is some bilinear form on the QFSHA to be deformed.  By this same reason,
   \textit{any set of elements which generate, as an algebra, the QFSHA under exam,
   will also generate it (as an algebra again) w.r.t.\ the new, deformed product}.
   For this reason, the formula  \eqref{eq: formula x sigma-deformed-product}
   is enough to describe the deformed algebra  $ {\fhg}_{\sigma_{{}_\cF}} $
   as the latter is (again) generated   --- w.r.t.\ the new product ---
   by the  $ x_{i,\,j} $'s,  just like  $ \fhg $  was (with the old product).

                                                                                       \par
   More in detail, from the original presentation of  $ \fhg $  by generators
   --- the  $ x_{i,\,j} $'s  ---   and relations   --- namely, those in  \eqref{eq: rel.s_xij}  ---
   using  \eqref{eq: formula x sigma-deformed-product}  above we find a similar presentation of
   $ {\fhg}_{\sigma_{\!{}_\cF}} $  by generators   --- the  $ x_{i,\,j} $'s  again ---
   and relations, where the latter depend on ``multiparameters'' of the form
  $$  q_{r,s}  \, := \,  q \, e^{\hbar\,(\phi_{r,s} - \phi_{s,r})}   \qquad \qquad   \forall \;\; r, s \in \{1,\dots,n\}  $$
 Now, when writing these relations in detail, one finds the following outcome:
 \vskip3pt
   \textit{Our  $ {\fhg}_{\sigma_{\!{}_\cF}} $  is a concrete realization of the quantum function algebra over
   $ \textit{GL}(n) $  introduced in  \cite{JLZ}  (with  $ \, \hat{p} = \hat{q} \, $)   ---
   or also  in  \cite{Ma}  in the totally even case}.
 \vskip3pt
   In particular, this yields an independent construction for the quantum general linear groups
   in  \cite{JLZ}  (for  $ \, \hat{p} = \hat{q} \, $)  or  \cite{Ma}  (in the totally even case).
   Even more, it was proved in  \cite{AST}  that in the totally even case these
   ``multiparameter quantum  $ \textit{GL}_n $''  can all be realized as 2--cocycle twist deformations
   of the uniparameter, ``standard'' one (as considered, e.g., in  \cite{PW}).
   But now, with our realization, we ``recover'' this fact as an intrinsic feature
   of these specific multiparameter quantum groups.
 \vskip3pt
   Finally, this result is also extended to the general  \textsl{super case}
   --- i.e., for any  $ \textit{GL}(m|n) \, $,  not just the totally even case of
   $ \, m=0 \, $  or  $ \, n=0 \, $  ---   in  \cite{GGP}.
 \vskip7pt

   Let us now see how  \eqref{eq: formula x sigma-deformed-product}
   yields a modified Poisson bracket in the semiclassical limit of
   $ \, {\fhg}_{\sigma_{{}_\cF}} \, $.  Using such notation as
   $ \; \overline{x}_{r,\,s} := x_{r,\,s} \ \big(\, \text{mod} \,\
   \hbar\,{\fhg}_{\sigma_{{}_\cF}} \big) \, $,  \,the Poisson bracket inherited from
   $ \, {\fhg}_{\sigma_{{}_\cF}} \, $  is given, by definition, by
  $$  {\big\{ \overline{x}_{i,\,j} \, , \overline{x}_{\ell,\,t} \big\}}_{\sigma_{{}_\cF}}  \,\; := \;\;
  {{\; {\big[{x}_{i,\,j} \, , {x}_{\ell,\,t} \big]}_{\sigma_{{}_\cF}} \;}
  \over {\; \hbar \;}}  \ \ \Big(\, \text{mod} \,\ \hbar \, {\fhg}_{\sigma_{{}_\cF}} \Big)  $$
 Now
  $$  \displaylines{
   {\big[ {x}_{i,\,j} \, , {x}_{\ell,\,t} \big]}_{\sigma_{{}_\cF}}  \, = \;
 x_{i,\,j} \raise-1pt\hbox{$ \; \scriptstyle \dot\sigma_{{}_\cF} $}\, x_{\ell,\,t}  \, - \,
x_{\ell,\,t} \raise-1pt\hbox{$ \; \scriptstyle \dot\sigma_{{}_\cF} $}\, x_{i,\,j}  \; =   \hfill  \cr
   = \;  e^{\hbar \, (\phi_{i,\ell} \, - \, \phi_{j,t}) / 2} \; x_{i,\,j} \, x_{\ell,\,t}  \, - \,
e^{\hbar \, (\phi_{j,t} \, - \, \phi_{i,\ell}) / 2} \; x_{\ell,\,t} \, x_{i,\,j}  \; =
\qquad \qquad \qquad \qquad  \cr
   \hfill   = \;  e^{\hbar \, (\phi_{i,\ell} \, - \, \phi_{j,t}) / 2} \, \big[ x_{i,\,j} \, , x_{\ell,\,t} \big]
   \, + \, \big( e^{\hbar \, (\phi_{i,\ell} \, - \, \phi_{j,t}) / 2} - e^{\hbar \, (\phi_{j,t} \, - \, \phi_{i,\ell}) / 2} \big) \, x_{\ell,\,t} \, x_{i,\,j}  }  $$
 hence expanding the exponentials we get
  $$  {\big[ {x}_{i,\,j} \, , {x}_{\ell,\,t} \big]}_{\sigma_{{}_\cF}}  \, = \;
\big( 1 + \hbar \, (\phi_{i,\ell} \, - \, \phi_{j,t}) / 2 \big) \, \big[ x_{i,\,j} \, , x_{\ell,\,t} \big] \, + \,
\hbar \, \big( \phi_{i,\ell} \, - \, \phi_{j,t} \big) \, x_{\ell,\,t} \, x_{i,\,j} \, + \, \cO\big(\hbar^2\big)  $$
 from which we eventually get
\begin{equation}  \label{eq: deform-Pois-brack}
  {\big\{ \overline{x}_{i,\,j} \, , \overline{x}_{\ell,\,t} \big\}}_{\sigma_{{}_\cF}}  \, := \;
  \big\{ \overline{x}_{i,\,j} \, , \overline{x}_{\ell,\,t} \big\} \, + \,
  \Big(\, \overline{\phi}_{i,\ell} \, - \, \overline{\phi}_{j,t} \Big) \,
  \overline{x}_{\ell,\,t} \, \overline{x}_{i,\,j}
\end{equation}
 where  $ \, \big\{ \overline{x}_{i,\,j} \, , \overline{x}_{\ell,\,t} \big\} \, $
 denotes the old (undeformed) Poisson bracket and we took into account that
 $ \; \big[ \overline{x}_{i,\,j} \, , \overline{x}_{\ell,\,t} \big] \, = \, 0 \; $
 and that the deformed and undeformed product do coincide modulo  $ \hbar \, $.
                                                              \par
   In addition, the formula  \eqref{eq: deform-Pois-brack}
   for the Poisson bracket also induces, following the general recipe,
   a concrete description of the modified Lie bracket in the cotangent Lie bialgebra
   $ \, \lieg^* := \liem \big/ \liem^2 \, $,  \,where  $ \liem $  is the augmentation ideal of
   $ {\fhg}_{\sigma_{{}_\cF}} \, $. Indeed, the latter has as  $ \k $--basis  the set of cosets
   $ \, \big(\text{modulo\ } \liem^2 \,\big) $
  $$
  \Big\{\, \text{x}_{i,\,j} := \big( \overline{x}_{i,\,j} - \delta_{i,\,j} \big)  \,\
  \text{mod} \,\ \liem^2 \,\Big|\, i \, , j = 1, \dots, n \,\Big\}
  $$
 and for these elements from  \eqref{eq: deform-Pois-brack}
 we deduce the deformed Lie bracket as given by
  $$
  \displaylines{
   {\big[ \text{x}_{i,\,j} \, , \text{x}_{\ell,\,t} \big]}_{\sigma_{{}_\cF}}  = \;
   \big[ \text{x}_{i,\,j} \, , \text{x}_{\ell,\,t} \big] \;\;\; ,
   \quad  {\big[ \text{x}_{i,\,i} \, , \text{x}_{\ell,\,\ell} \big]}_{\sigma_{{}_\cF}}  = \;  \big[ \text{x}_{i,\,i} \, ,
   \text{x}_{\ell,\,\ell} \big]   \quad \qquad  \forall \;\; i \not= j \, , \; \ell \not= t  \cr
   {\big[ \text{x}_{i,\,i} \, , \text{x}_{\ell,\,t} \big]}_{\sigma_{{}_\cF}}  \, = \;
   \big[ \text{x}_{i,\,i} \, , \text{x}_{\ell,\,t} \big] \, + \, \Big(\, \overline{\phi}_{i,\ell} \, - \,
   \overline{\phi}_{i,t} \Big) \, \text{x}_{\ell,\,t}   \qquad \qquad  \forall \;\; \ell \not= t  \cr
   {\big[ \text{x}_{i,\,j} \, , \text{x}_{\ell,\,\ell} \big]}_{\sigma_{{}_\cF}}  \, := \;
   \big[ \text{x}_{i,\,i} \, , \text{x}_{\ell,\,t} \big] \, + \, \Big(\, \overline{\phi}_{i,\ell} \, - \,
   \overline{\phi}_{j,\ell} \Big) \, \text{x}_{i,\,j}   \qquad \qquad  \forall \;\; i \not= j  }
   $$
\end{exa}

\vskip13pt

\subsection{Deformations by polar 2--cocycle of QUEA's}  \label{subsec: 2coc-QUEAs}  {\ }
 \vskip7pt
   This subsection is dedicated to deformations by polar 2--cocycle of QUEA's.
   In this case the result that we achieve is somewhat surprising, in that we are indeed
   ``stretching the standard recipe'', as the 2-cocycles that we use to deform our Hopf
   $ \kh $--algebras  are valued in the field  $ \khp $  rather than in our ground ring  $ \kh \, $.
   Therefore,  \textit{a priori\/}  nothing even guarantees that the recipe would just work
   and produce a new Hopf algebra over  $ \kh \, $;  \,nonetheless, we eventually find quite a
   meaningful result, which also says that  \textit{the standard procedure of deformation by
   twist for QUEA's can be extended somewhat beyond its natural borders}.
 \vskip5pt
   We begin with two ancillary results.

\vskip11pt

\begin{lema}  \label{lemma: FACT-on-z-z'}  {\ }
 \vskip3pt
   Let  $ \, U_\hbar := \uhg \, $  be any QUEA, and
   $ \, J_\hbar := \Ker\big(\epsilon_{U_\hbar}\big) \, $.  For every  $ \, z \in U_\hbar \, $,
   \,there exists  $ \, N \in \NN \, $  such that
   $ \; \delta_n(z) \, \in \, \hbar^{\,\max(n,N)-N} J_\hbar^{\,\otimes n} \; $  for every
   $ \, n \in \NN \, $.
\end{lema}

\pf
 As Drinfeld's functors are inverse to each other
 --- cf.\ Theorem \ref{thm: QDP}\textit{(a)}  ---   applying  $ {(\ )}^\vee $  after  $ {(\ )}' $
 to the QUEA  $ U_\hbar $  we get  $ \, U_\hbar = {\big( U_\hbar^{\,\prime} \big)}^{\!\vee} \, $:
 \,letting  $ \, I_\hbar^{\,\prime} := \hbar\,U_\hbar^{\,\prime} +
 \Ker\!\big( \epsilon_{U_\hbar^{\,\prime}} \big) \, $,  this last identity reads
 \vskip7pt
   \centerline{ $ U_\hbar \, = \, \hbar\text{-adic completion of\ }
   \sum\limits_{n \geq 0} \hbar^{-n} {\big( I_\hbar^{\,\prime} \big)}^n \, = \,
   \hbar\text{-adic completion of\ }
   \bigcup\limits_{n \geq 0} \hbar^{-n} {\big( I_\hbar^{\,\prime} \big)}^n $ }
 \vskip5pt
 In particular, this implies that for our
 $ \, z \in \uhg \, $  there exist some  $ \, N \in \NN \, $  and
 $ \, z' \in {\big( I_\hbar^{\,\prime} \big)}^N \, $  such that
 $ \; z \, \equiv \hbar^{-N} z' \, \big(\!\!\! \mod \hbar\,\uhg \big) \; $.
 Now, given  $ \, n \in \NN \, $  we have
 $ \, \delta_n\big(z'\big) \in \hbar^n U_\hbar^{\,\otimes n} \, $
 because
 $ \, z' \in {\big( I_\hbar^{\,\prime} \big)}^N \subseteq U_\hbar^{\,\prime} \, $,
 and also
 $ \, \delta_n\big(z'\big) \in \!\! {\textstyle \sum\limits_{s_1 + \cdots + s_n = N}} \!\!\!
 \otimes_{i=1}^n \! {\big( I_\hbar^{\,\prime} \big)}^{s_i} \, $
 because  $ I_\hbar^{\,\prime} $  is a Hopf ideal; moreover,
 $ \, \Ker\!\big( \epsilon_{U_\hbar^{\,\prime}} \big) \subseteq
 \hbar\,U_\hbar^{\,\prime} \, $  again by construction, hence
 $ \, I_\hbar^{\,\prime} \subseteq \hbar\,U_\hbar^{\,\prime} \, $.
 In the end, all this yields
 $ \; \delta_n\big(z'\big) \in \hbar^{\,\max(n,N)} \, J_\hbar^{\,\otimes n} \, $,
 \,therefore  $ \; \delta_n(z) \, \in \, \hbar^{\,\max(n,N)-N} J_\hbar^{\,\otimes n} \; $
 as claimed.
\epf

\vskip5pt

   For the second, auxiliary result, we
 fix some more notation: namely, hereafter by  ``$ \log_* $''  and  ``$ \exp_* $''
 we denote the logarithm and the exponential with respect to the
 convolution product, whenever defined.

\vskip11pt

\begin{lema}  \label{lemma: properties polar 2-cocycle}
 Let\/  $ \uhg $  be any QUEA, and let\/  $ \chi $  be a  $ \kh $--bilinear
 form on  $ \uhg $  such that  $ \, \chi(z\,,1) = 0 = \chi(1\,,z) \, $  for any
 $ \, z \in \uhg \, $;  \,denote also by the same symbol  $ \chi $
 the scalar extension of  $ \chi $  to a\/  $ \khp $--bilinear  form for the\/
 $ \khp $--vector  space  $ \, \Uhg := \khp \otimes_\kh \uhg \, $.  Then:
 \vskip3pt
   (a)\; the formal expression  $ \, \sigma := \exp_*\!\big( \hbar^{-1} \chi \big) \, $
  uniquely provides a well-defined,  $ \khp $--valued  bilinear form for  $ \, \Uhg \, $;
 \vskip3pt
   (b)\;  $ \, \sigma(z\,,1) = \epsilon(z) = \sigma(1\,,z) \, $  for any  $ \, z \in \uhg \, $;
 \vskip3pt
   (c)\;  $ \sigma $  is orthogonal, i.e.\  $ \, \sigma_{2,1} \! = \sigma^{\,-1} $,
   \,iff\/  $ \chi $  is antisymmetric, i.e.\  $ \, \chi_{2,1} = -\chi \, $.
\end{lema}

\pf
 \textit{(a)}\,  Fix notation  $ \, U_\hbar := \uhg \, $  and
 $ \, J_\hbar := \Ker\big(\epsilon_{U_\hbar}\big) \, $.
 For any  $ \, z \in U_\hbar \, $,  set
\begin{equation}  \label{eq: zhat_z+}
  \hat{z} := \epsilon(z) \, ,  \;\;\; z^+ := z - \epsilon(z) = z - \hat{z} \, \in \, J_\hbar \; ,
  \quad  \text{hence}  \quad  z = z^+ + \hat{z}
\end{equation}
   \indent   The assumption  $ \, \chi(z\,,1) = 0 = \chi(1\,,z) \, $  for  $ \, z \in \uhg \, $
 implies (for all  $ \, u, v \in U_\hbar \, $)
\begin{equation}  \label{eq: expansion chi(u,v)}
  \chi(u,v) \, = \, \chi\big( u^+ \! + \hat{u} \, , v^+ \! + \hat{v} \big) \, = \,
  \chi\big( u^+ , v^+ \big)
\end{equation}
 \vskip5pt
   Now, for any  $ \, a \, , b \in U_\hbar \, $,  the formula
   $ \; \sigma = \exp_*\!\big( \hbar^{-1} \chi \,\big) \, = \,
\sum_{n \,\geq 0} \hbar^{-n} \chi^{*\,n} \big/ n! \; $  gives
\begin{equation}  \label{eq: expansionsigma(a,b)}
  \begin{aligned}
    &  \hskip-5pt  \sigma(a\,,b\,)  \; = \,
    {\textstyle \sum_{n \,\geq 0}} \, \hbar^{-n} \chi^{*\,n}(a\,,b\,) \Big/ n!  \,\; =  \\
    &  \hskip23pt   = \;
    {\textstyle \sum_{n \,\geq 0}} \,
    \hbar^{-n} {\textstyle \prod_{i=1}^n} \chi\big(a_{(i)},b_{(i)}\big) \Big/ n!  \; =
    \;  {\textstyle \sum_{n \,\geq 0}} \,
    \hbar^{-n} {\textstyle \prod_{i=1}^n} \chi\big(a^+_{(i)},b^+_{(i)}\big) \Big/ n!
  \end{aligned}
\end{equation}
 where we took into account that  $ \, \chi^{*\,k}(u\,,v\,) \, = \,
 \prod_{s=1}^k \! \chi\big(u_{(s)},v_{(s)}\big) \, = \,
 \prod_{s=1}^k \! \chi\big(u^+_{(s)},v^+_{(s)}\big) \, $  for each
 $ \, u, v \in U_\hbar \, $,  $ \, k \in \NN \, $,  by definitions along with
 \eqref{eq: expansion chi(u,v)}.  Now we notice that
 $ \, \otimes_{i=1}^n a^+_{(i)} = \delta_n(a) \, $  and
 $ \, \otimes_{i=1}^n b^+_{(i)} = \delta_n(b) \, $,
 hence  Lemma \ref{lemma: FACT-on-z-z'}  guarantees that
  $$  h^{-n} {\textstyle \prod\limits_{i=1}^n} \chi\big(a^+_{(i)},b^+_{(i)}\big)  \; \in \;
  \hbar^{-n+\,\max(n,A)-A+\max(n,B)-B}  \eqno \forall \;\; n \in \NN_+  $$
 whence in particular
\begin{equation}  \label{eq: h-adic expansion chi*n}
  \begin{aligned}
     &  \;\;  h^{-n} {\textstyle \prod\limits_{i=1}^n} \chi\big(a^+_{(i)},b^+_{(i)}\big)  \; \in \;
     \hbar^{-\min(A,B)} \, \kh   \qquad  \forall \;\; n \in \NN_+  \\
     &  h^{-n} {\textstyle \prod\limits_{i=1}^n} \chi\big(a^+_{(i)},b^+_{(i)}\big)  \; \in \;
     \hbar^{n-(A+B)} \, \kh  \qquad  \forall \;\; n \geq A+B  \\
  \end{aligned}
\end{equation}
 where  $ \, A \in \NN \, $,  resp.\  $ \, B \in \NN \, $,
 plays for  $ a \, $,  resp.\ for  $ b \, $,  the role of  $ N $  for  $ z $  in
 Lemma \ref{lemma: FACT-on-z-z'}  above; by this, the formal expansion for
 $ \sigma(a\,,b\,) $  in  \eqref{eq: expansionsigma(a,b)}  yields a well defined
 element in  $ \kh \, $,  hence  $ \sigma $  is a well-defined  $ \khp $--bilinear
 form of  $ \Uhg $  as claimed.
 \vskip3pt
   \textit{(b--c)}\,  Both claims are obvious, by construction,
   as they follow from standard identities for formal exponentials.
\epf

\vskip1pt

   The previous result leads us to introduce the following notion:

\vskip4pt

\begin{definition}  \label{def: polar 2cocycle}
 Let\/  $ \uhg $  be a QUEA, and  $ \, \Uhg := \khp \otimes_\kh \uhg \, $.
 Note that  $ \Uhg $  has a natural ``Hopf algebra structure'' of  $ \Uhg $
 induced by scalar extension from  $ \uhg $
 --- so that, in particular, the ``coproduct'' takes values in
 $ \, \khp \otimes_\kh \! \big( \uhg \,\widehat{\otimes}_\kh \uhg \big) \, $
 rather than in  $ \, \Uhg \otimes_\khp \Uhg \, $.
                                                          \par
%
%%%%%
%    We call  \textit{polar 2-cocycle of}  $ \uhg $  any 2-cocycle of  $ \Uhg $  which has the
% form  $ \, \sigma := \exp_*\!\big( \hbar^{-1} \chi \big) \, $  for some  $ \kh $--bilinear
% form  $ \, \chi \in {\big( \uhg^{\otimes 2} \,\big)}^* \, $  of  $ \uhg $   --- then
% automatically  $ \, \chi(z\,,1) = 0 = \chi(1\,,z) \, $  for  $ \, z \in \uhg \, $  ---
% as in  Lemma \ref{lemma: properties polar 2-cocycle}  above.
%%%%%
%
   We call  \textit{polar 2-cocycle of}  $ \uhg $  any  $ \khp $--bilinear  form
   $ \, \sigma $  of  $ \Uhg $  which has the form
   $ \, \sigma := \exp_*\!\big( \hbar^{-1} \chi \big) \, $  for some  $ \kh $--bilinear
   form  $ \, \chi \in {\Big( \uhg^{\widehat{\otimes} 2} \Big)}^* \, $  of  $ \uhg $  such that
   $ \, \chi(z\,,1) = 0 = \chi(1\,,z) \, $  for all  $ \, z \in \uhg \, $,
   \,and in addition enjoys the 2--cocycle properties with respect to the above
   ``Hopf algebra structure'' of  $ \Uhg \, $.
\end{definition}

\vskip4pt

\begin{rmk}  \label{rmk: equiv-cond x polar 2-cocycle}
 The notion of ``polar 2-cocycle'' for a QUEA  $ \uhg $
 can also be cast in the following, equivalent shape.
 Recall that  $ \, \fhg := {\uhg}^* \, $  is a QFSHA
 (cf.\ \S \ref{equiv-&-(stand)-duality}),  and then
 $ \, {\Big( \uhg^{\widehat{\otimes} 2} \Big)}^* =
 {\uhg}^* \,\widetilde{\otimes}\, {\uhg}^* = \fhg \,\widetilde{\otimes}\, \fhg \, $.
 Given  $ \, \chi \in {\Big( \uhg^{\widehat{\otimes} 2} \Big)}^* \, $  as in
 Definition \ref{def: polar 2cocycle}  above, the condition
 $ \, \chi(z\,,1) = 0 = \chi(1\,,z) \, $  for all  $ \, z \in \uhg \, $
 means that  $ \, \chi \in J_{\fhg} \,\widetilde{\otimes}\, J_{\fhg} \, $,
 with  $ \, J_{\fhg} := \Ker\big( \epsilon_{\fhg} \big) \, $,
 \,hence we have
 $ \; \chi \, \in \, \hbar^2 \, {\big( J_{\fhg}^{\,\vee} \big)}^{\widehat{\otimes}\, 2}
 \subseteq \, \hbar^2 \, {\big( {\fhg}^\vee \big)}^{\widehat{\otimes}\, 2} \; $
 where  $ \, J_{\fhg}^{\,\vee} := \hbar^{-1} J_{\fhg} \, $  and  $ \, {\fhg}^\vee $
 is the QUEA defined in  \S \ref{subsec: QDP}  out of  $ \fhg \, $.
 Thus, it follows that
 $ \; \hbar^{-1} \chi \, \in \, \hbar \, {\big( {\fhg}^\vee \big)}^{\widehat{\otimes}\, 2} $,
 \,so  $ \, \sigma := \exp_*\!\big( \hbar^{-1} \chi \big) \, $
 is a well-defined element in  $ {\big( {\fhg}^\vee \big)}^{\widehat{\otimes}\, 2} \, $.
                                                         \par
   Now,  the requirement that  $ \, \sigma := \exp_*\!\big( \hbar^{-1} \chi \big) \, $
   be a polar 2-cocycle for  $ \uhg $  in the sense of
   Definition \ref{def: polar 2cocycle}  above is equivalent to the property of
   $ \sigma $  being a twist element for  $ {\fhg}^\vee \, $
   --- which makes perfectly sense in sight of
   Proposition \ref{prop: duality-deforms}.
\end{rmk}

\vskip3pt

   Clearly, every 2-cocycle for  $ \uhg $  is a polar 2--cocycle as well;
   the converse, instead, is not true, in general (counterexamples do exist).
   However, the key point is that  \textit{every polar 2--cocycle still provides
   a well-defined deformation by 2-cocycle of  $ \, \uhg $}
   --- in short, the construction of  \textsl{deformations by 2--cocycle\/}
   does properly extend to  \textsl{``deformations by \textit{polar}  2--cocycle''\/}
   as well: this is indeed our next result.

\vskip4pt

\begin{theorem}  \label{thm: polar 2cocycle-deform-QUEA}
 Let  $ \uhg $  be a QUEA, and  $ \, \sigma = \exp_*\!\big( \hbar^{-1} \chi \big) \, $  a polar 2--cocycle for it, as in  Definition \ref{def: polar 2cocycle}.  Then the procedure of 2-cocycle deformation by  $ \,\sigma $  applied to  $ \, \Uhg $  actually restricts to  $ \, \uhg $,  \,making the latter into a new QUEA.
\end{theorem}

\pf
 First of all, we have to explain the statement itself.
 To begin with, note that, by definitions and by
 Lemma \ref{lemma: properties polar 2-cocycle},
 we can perform the deformation by the 2--cocycle  $ \sigma $  onto the
 ``Hopf  $ \, \khp $--algebra''  $ \, \Uhg := \khp \,\widehat{\otimes}\, \uhg \, $.
 Our statement then claims the resulting deformed Hopf structure onto  $ \Uhg $
 actually ``restricts'' to a deformation of  $ \, \uhg $  itself: in turn, this amounts to
 claiming that  $ \uhg $  is closed for the  $ \sigma $--deformed  product in
 $ {\big( \Uhg \big)}_\sigma \, $   --- so we go and tackle this last problem.
 \vskip3pt
   Fix notation  $ \, U_\hbar := \uhg \, $,  $ \, J_\hbar := \Ker\big(U_\hbar\big) \, $,
   $ \, J_\hbar^{\,\prime} := \Ker\big(U_\hbar^{\,\prime}\big) \, $  and
   $ \, \widetilde{J_\hbar} := \hbar^{-1} J_\hbar^{\,\prime} \, $,  where
   $ \, U_\hbar^{\,\prime} := \uhg' \, $  is given in
   Definition \ref{def: Drinfeld's functors}\textit{(a)}.
   As it was mentioned in the proof of  Lemma \ref{lemma: FACT-on-z-z'},
   Theorem \ref{thm: QDP}\textit{(a)\/}  implies that
   $ \, U_\hbar = {\big( U_\hbar^{\,\prime} \big)}^{\!\vee} \, $,  \,that is
 \vskip7pt
   \centerline{ $ U_\hbar \, = \, \hbar\text{-adic completion of\ }
   \sum\limits_{n \geq 0} \hbar^{-n} {\big( I_\hbar^{\,\prime} \big)}^n $ }
\noindent
 where  $ \, I_\hbar^{\,\prime} := \hbar\,U_\hbar^{\,\prime} +
 \Ker\!\big( \epsilon_{U_\hbar^{\,\prime}} \big) = \hbar\,U_\hbar^{\,\prime} +
 J_\hbar^{\,\prime} \, $;  then a moment's thought shows that the previous
 expression of  $ U_\hbar $  reads also
\begin{equation}  \label{eq: uhg = vee-prim(uhg)}
 U_\hbar \, = \, \hbar\text{-adic completion of\ }
 {\textstyle \sum\limits_{n \geq 0}} \hbar^{-n} {\big( J_\hbar^{\,\prime} \big)}^n \, = \,
 \hbar\text{-adic completion of\ }
 {\textstyle \sum\limits_{n \geq 0}} \widetilde{J}_\hbar^{\raisebox{3pt}{$ \;
 \scriptstyle n $}}
\end{equation}
%
%%%%%
%  In a nutshell, this means that the topological, unital  $ \kh $--algebra  $ U_\hbar $
% is generated by  $ \widetilde{J}_\hbar \, $,  \,something we shall loosely refer to by
% writing  $ \, U_\hbar = {\big\langle \widetilde{J}_\hbar \big\rangle}_\kh \, $.
%%%%%
%

Note also that clearly  $ J_\hbar^{\,\prime} $  is a Hopf ideal in
   $ U_\hbar^{\,\prime} $,  and moreover
   $ \; J_\hbar^{\,\prime} \, \subseteq \, \hbar \, J_\hbar \; $
   (by construction); therefore for
   $ \, z' \in J_\hbar^{\,\prime \! \raisebox{1pt}{$ \, \scriptstyle N $}} \, $
   --- with  $ \, N \in \NN \, $  ---
   acting like in the proof of  Lemma \ref{lemma: FACT-on-z-z'},
   one sees that
\begin{equation}  \label{eq: h-adic delta_n(z)}
 \qquad   \delta_n\big(z'\big)  \; \in \; \hbar^n J_\hbar^{\,\otimes n}
 \;{\textstyle \bigcap}\; \Big(\, {\textstyle \sum_{\sum_i N_i = N}
 \bigotimes_{i=1}^n} J_\hbar^{\,\prime \! \raisebox{1pt}{$ \, \scriptstyle N_i $}} \Big)
 \; \subseteq \;  \hbar^{\,\max(n,N)} J_\hbar^{\,\otimes n}
\end{equation}
 \vskip3pt
   Again, for any  $ \, z \in U_\hbar \, $  we retain notation as in  \eqref{eq: zhat_z+}
   above, that is
\begin{equation}  \label{eq: zhat_z+ - BIS}
  \hat{z} := \epsilon(z) \, ,  \;\;\; z^+ := z - \epsilon(z) = z - \hat{z} \, \in \, J_\hbar \; ,
   \quad  \text{hence}  \quad  z = z^+ + \hat{z}
\end{equation}
 and we recall also that for all  $ \, u, v \in U_\hbar \, $  we have
\begin{equation}  \label{eq: expansion chi(u,v) - BIS}
  \chi(u,v) \, = \, \chi\big( u^+ \! + \hat{u} \, , v^+ \! + \hat{v} \big) \, = \,
  \chi\big( u^+ , v^+ \big)
\end{equation}
 \vskip5pt
   Thanks to  \eqref{eq: uhg = vee-prim(uhg)},  in order to prove that
   $ \, \uhg =: U_\hbar \, $  is closed for the  $ \sigma $--deformed  product
   $ \, \raisebox{-5pt}{$ \dot{\scriptstyle \sigma} $} \, $
   it is enough to show that
   $ \; \widetilde{J}_\hbar^{\raisebox{1pt}{$ \; \scriptstyle A $}}
   \,\raisebox{-5pt}{$ \dot{\scriptstyle \sigma} $}\;
   \widetilde{J}_\hbar^{\raisebox{1pt}{$ \; \scriptstyle B $}} \;
 \subseteq {\textstyle \sum\limits_{n \geq 0}}
 \widetilde{J}_\hbar^{\raisebox{3pt}{$ \; \scriptstyle n $}}
 \; $  for any  $ \, A \, , B \in \NN_+ \, $.
                                                                                    \par
   To begin with, we pick
   $ \, a \in \widetilde{J}_\hbar^{\raisebox{1pt}{$ \; \scriptstyle A $}} = \hbar^{-A}
   J_\hbar^{\,\prime \! \raisebox{1pt}{$ \; \scriptstyle A $}} \, $  and
   $ \, b \in \widetilde{J}_\hbar^{\raisebox{1pt}{$ \; \scriptstyle B $}} =
   \hbar^{-B} J_\hbar^{\,\prime \! \raisebox{1pt}{$ \; \scriptstyle B $}} \, $;  \,by definition,
  $$
  a \,\raisebox{-5pt}{$ \dot{\scriptstyle \sigma} $}\, b  \; := \;
  \sigma\big(a_{(1)},b_{(1)}\big) \, a_{(2)} \, b_{(2)} \, \sigma^{-1}\big(a_{(3)},b_{(3)}\big)
  $$
 whence expanding the formal formula
 $ \; \sigma = \exp_*\!\big( \hbar^{-1} \chi \,\big) \, = \, \sum_{n \,\geq 0} \hbar^{-n} \chi^{*\,n}
 \big/ n! \; $   --- much like in the proof of  Lemma \ref{lemma: properties polar 2-cocycle}  ---
 we get
\begin{equation}  \label{eq: expans-sigmaprod}
  \begin{aligned}
    &  \hskip-3pt   a \,\raisebox{-5pt}{$ \dot{\scriptstyle \sigma} $}\, b  \,\; = \;\,
    {\textstyle \sum_{t,\ell \,\geq 0}} \, \hbar^{-(t+\ell)} \, {{(-1)}^\ell} \, {(t!\,\ell!)}^{-1} \,
    \chi^{*\,t}\big(a_{(1)},b_{(1)}\big) \, a_{(2)} \, b_{(2)} \,
    \chi^{*\,\ell}\big(a_{(3)},b_{(3)}\big)  \,\; =  \\
    &  \hskip2pt   = \;\,  \epsilon\big(a'_{(1)}\big) \, \epsilon\big(b'_{(1)}\big) \, a'_{(2)}
    \, b'_{(2)} \, \epsilon\big(a'_{(3)}\big) \, \epsilon\big(b'_{(3)}\big)  \; +  \\
%
%%%%%
%     &  \hskip1pt   + \,  {\textstyle \sum_{t \,> 0}} \, \hbar^{-t} \, {(t!)}^{-1} \, \chi^{*\,t}\big(a'_{(1)},b'_{(1)}\big)
% \, a'_{(2)} \, b'_{(2)} \, \epsilon\big(a'_{(3)}\big) \, \epsilon\big(b'_{(3)}\big)  \; +  \\
%     &  \hskip4pt   + \;  \epsilon\big(a'_{(1)}\big) \, \epsilon\big(b'_{(1)}\big) \, a'_{(2)} \, b'_{(2)}
% \, {\textstyle \sum_{\ell \,> 0}} \, \hbar^{-\ell} \, {(-1)}^\ell \, {(\ell!)}^{-1} \,
% \chi^{*\,\ell}\big(a'_{(3)},b'_{(3)}\big)  \; +  \\
%%%%%
%
    &  \hskip7pt   + \,  {\textstyle \sum_{t+\ell \,> 0}} \,
    \hbar^{-(t+\ell)} \, {(-1)}^\ell \, {(t!)}^{-1} {(\ell!)}^{-1} \,
    \chi^{*\,t}\big(a_{(1)},b_{(1)}\big) \, a_{(2)} \, b_{(2)} \,
    \chi^{*\,\ell}\big(a_{(3)},b_{(3)}\big)  \,\; =  \\
    &  \hskip-3pt  = \;\,  a \cdot b  \,\; +  \\
%
%%%%%
%     &  \hskip3pt  + \,  {\textstyle \sum_{t \,> 0}} \, \hbar^{-t} \, {(t!)}^{-1} \,
% \chi^{*\,t}\big(a'_{(1)},b'_{(1)}\big) \, a'_{(2)} \, b'_{(2)}  \; +  \\
%     &  \hskip9pt  + \,  {\textstyle \sum_{\ell \,> 0}} \, \hbar^{-\ell} \, {(-1)}^\ell \,
% {(\ell!)}^{-1} \, a'_{(1)} \, b'_{(1)} \, \chi^{*\,\ell}\big(a'_{(2)},b'_{(2)}\big)  \; +  \\
%%%%%
%
    &  \hskip15pt  + \,  {\textstyle \sum_{t+\ell \,> 0}} \, \hbar^{-(t+\ell)} \, {(-1)}^\ell
    \, {(t!)}^{-1} {(\ell!)}^{-1} \, \chi^{*\,t}\big(a_{(1)},b_{(1)}\big) \, a_{(2)} \, b_{(2)} \,
    \chi^{*\,\ell}\big(a_{(3)},b_{(3)}\big)
  \end{aligned}
\end{equation}
 where we took into account coassociativity and counitality properties.
 \vskip5pt
   Let us analyze each summand in the very last line in  \eqref{eq: expans-sigmaprod}.
   From the identities
 $ \; \chi^{*\,k}(u\,,v\,) \, = \, \prod_{s=1}^k \chi\big(u_{(s)},v_{(s)}\big) \, = \, \
 \prod_{s=1}^k \chi\big(u^+_{(s)},v^+_{(s)}\big) \; $
%
%%%%%
% for  $ \, u, v \in U_\hbar \, $,  $ \, k \in \NN \, $
%%%%%
   --- cf.\ \eqref{eq: expansion chi(u,v) - BIS}  ---   we get
  $$
  \displaylines{
   \quad   \chi^{*\,t}\big(a_{(1)},b_{(1)}\big) \, a_{(2)} \, b_{(2)} \,
   \chi^{*\,\ell}\big(a_{(3)},b_{(3)}\big)  \; =   \hfill  \cr
   = \;  {\textstyle \prod\limits_{i=1}^t} \, \chi\big(a_{(i)},b_{(i)}\big) \,
   a_{(t+1)} \, b_{(t+1)} \, {\textstyle \prod\limits_{j=1}^\ell} \,
   \chi\big(a_{(t+1+j)},b_{(t+1+j)}\big)  \; =  \cr
   \hfill   = \;  {\textstyle \prod\limits_{i=1}^t} \,
   \chi\big(a^+_{(i)},b^+_{(i)}\big) \, a_{(t+1)} \, b_{(t+1)} \,
   {\textstyle \prod\limits_{j=1}^\ell} \, \chi\big(a^+_{(t+1+j)},b^+_{(t+1+j)}\big)
   \quad  }
   $$
 Consider the expansion of  $ a $  as in  \eqref{eq: zhat_z+ - BIS}.
 Then, letting
 $ \, j_{t+1} : U_\hbar^{\,\otimes\, (t+\ell)} \!\relbar\joinrel\longrightarrow
 U_\hbar^{\,\otimes\,(t+\ell+1)} \, $  be the map given by
 $ \;  \mathop{\otimes}\limits_{s=1}^{t+\ell} x_s \, \mapsto \,
 \Big( \mathop{\otimes}\limits_{s=1}^t x_s \Big) \otimes 1 \otimes \Big(
\mathop{\otimes}\limits_{s=t+1}^{t+\ell} x_s \Big) \; $,  \,we have
  $$
  \displaylines{
   \Big( \mathop{\otimes}\limits_{i=1}^t \! a^+_{(i)} \Big) \otimes
   a^+_{(t+1)} \otimes \Big( \mathop{\otimes}\limits_{j=1}^\ell \! a^+_{(t+1+j)} \Big)  \; = \;
   \delta_{t+\ell+1}(a)  \,\; \in \;\,  \hbar^{\,\max(t+\ell+1,A)\,-A} \,
   U_\hbar^{\,\otimes (t+\ell+1)}  \cr
   \Big( \mathop{\otimes}\limits_{i=1}^t \! a^+_{(i)} \Big) \otimes \widehat{a}_{(t+1)}
   \otimes \Big( \mathop{\otimes}\limits_{j=1}^\ell \! a^+_{(t+1+j)} \Big)  \; =
   \;  j_{t+1}\big(\delta_{t+\ell}(a)\big)  \,\; \in \;\,  \hbar^{\,\max(t+\ell,A)\,-A} \,
   U_\hbar^{\,\otimes (t+\ell+1)}  }
   $$
 so that, summing up,
  $$
  \Big( \mathop{\otimes}\limits_{i=1}^t \! a^+_{(i)} \!\Big) \otimes a_{\,t+1} \otimes
  \Big( \mathop{\otimes}\limits_{j=1}^\ell \! a^+_{(t+1+j)} \!\Big)  =
  \delta_{t+\ell+1}(a) + j_{t+1}\big(\delta_{t+\ell}(a)\!\big)  \in  \hbar^{\,\max(t+\ell,A)\,-A} \,
  U_\hbar^{\,\otimes (t+\ell+1)}
  $$
 --- like in the proof of  Lemma \ref{lemma: FACT-on-z-z'}  ---
 and similarly with  $ b \, $,  resp.\  $ B \, $,  replacing  $ a \, $,  resp.\  $ A \, $.
 Eventually, for all  $ \, t + \ell > 0 \, $  this gives
\begin{equation}  \label{eq: chi*t & chi*ell}
  \begin{aligned}
     &  \hskip-5pt   \chi^{*\,t}\big(a_{(1)},b_{(1)}\big) \, a_{(2)} \, b_{(2)} \,
     \chi^{*\,\ell}\big(a_{(3)},b_{(3)}\big)  \; =   \hfill  \\
     &  \hskip-1pt   = \,  \chi^{*\,t}\big(a_{(1)},b_{(1)}\big) \,
     \widehat{a}_{(2)} \, b_{(2)} \, \chi^{*\,\ell}\big(a_{(3)},b_{(3)}\big)  \, + \,
     \chi^{*\,t}\big(a_{(1)},b_{(1)}\big) \, a^+_{(2)} \, b_{(2)} \,
     \chi^{*\,\ell}\big(a_{(3)},b_{(3)}\big)
%
%%%%%
%   \; =   \hfill  \\
%      &  \qquad   = \;  {\textstyle \prod\limits_{i=1}^t} \, \chi\big(a^+_{(i)},b^+_{(i)}\big)
% \, \widehat{a}_{(t+1)} \, b_{(t+1)} \, {\textstyle \prod\limits_{j=1}^\ell} \,
% \chi\big(a^+_{(t+1+j)},b^+_{(t+1+j)}\big)  \; +  \\
%      &  \qquad \qquad   + \;  {\textstyle \prod\limits_{i=1}^t} \, \chi\big(a^+_{(i)},b^+_{(i)}\big)
% \, a^+_{(t+1)} \, b_{(t+1)} \, {\textstyle \prod\limits_{j=1}^\ell} \,
% \chi\big(a^+_{(t+1+j)},b^+_{(t+1+j)}\big)   \qquad
%%%%%
%
  \end{aligned}
\end{equation}
 where for the two summands in second line, writing  $ \, n := t + \ell \, $,
 we have
  $$  \displaylines{
   \chi^{*\,t}\big(a_{(1)},b_{(1)}\big) \, \widehat{a}_{(2)} \, b_{(2)} \,
   \chi^{*\,\ell}\big(a_{(3)},b_{(3)}\big)  \; =   \hfill  \cr
   \hfill   = \; {\textstyle \prod\limits_{i=1}^t} \, \chi\big(a^+_{(i)},b^+_{(i)}\big) \,
   \widehat{a}_{(t+1)} \, b_{(t+1)} \, {\textstyle \prod\limits_{k=t+2}^{n+1}}
   \chi\big(a^+_{(k)},b^+_{(k)}\big)  \; \in \;  \hbar^{\,\max(n,A)\, - A +
   \,\max(n,B)\, - B} \, U_\hbar^{\,\otimes (n+1)}  \cr
   \chi^{*\,t}\big(a_{(1)},b_{(1)}\big) \, a^+_{(2)} \, b_{(2)} \,
   \chi^{*\,\ell}\big(a_{(3)},b_{(3)}\big)  \; =   \hfill  \cr
   \hfill   = \; {\textstyle \prod\limits_{i=1}^t} \, \chi\big(a^+_{(i)},b^+_{(i)}\big) \,
   a^+_{(t+1)} \, b_{(t+1)} \, {\textstyle \prod\limits_{k=t+2}^{n+1}}
   \chi\big(a^+_{(k)},b^+_{(k)}\big)  \; \in \;  \hbar^{\,\max(n+1,A)\, - A + \,\max(n,B)\,
   - B} \, U_\hbar^{\,\otimes (n+1)}  }  $$
 \vskip3pt
   \textsl{Let us now assume that}  $ \, A := 1 \, $,  \,so that  $ \, n := t + \ell > 0 \, $
   implies  $ \, n := t + \ell \geq 1 = A \, $.  Then the last estimates read
\begin{equation}  \label{eq: h-adic estimates}
  \begin{aligned}
     &  \chi^{*\,t}\big(a_{(1)},b_{(1)}\big) \, \widehat{a}_{(2)} \, b_{(2)} \,
     \chi^{*\,\ell}\big(a_{(3)},b_{(3)}\big)  \; \in \;  \hbar^{\,n - 1 + \,\max(n,B)\, - B} \,
     U_\hbar^{\,\otimes\, (n+1)}  \\
     &  \hskip5pt   \chi^{*\,t}\big(a_{(1)},b_{(1)}\big) \, a^+_{(2)} \, b_{(2)} \,
     \chi^{*\,\ell}\big(a_{(3)},b_{(3)}\big)  \; \in \;  \hbar^{\,n + \,\max(n,B)\, - B} \,
     U_\hbar^{\,\otimes\, (n+1)}
  \end{aligned}
\end{equation}
 The term in the second line, when plugged in  \eqref{eq: chi*t & chi*ell}
 and then in  \eqref{eq: expans-sigmaprod},  yields a contribution of the form
  $$
  {{\,{(-1)}^\ell\,} \over {\,t!\,\ell!\,}} \; \hbar^{-n} \,
\chi^{*\,t}\big(a_{(1)},b_{(1)}\big) \, a^+_{(2)} \, b_{(2)} \,
\chi^{*\,\ell}\big(a_{(3)},b_{(3)}\big)  \; \in \;  \hbar^{\,\max(n,B)\, - B} \,
U_\hbar^{\,\otimes\, (n+1)}
$$
 that belongs to  $ \, \hbar^{\,\max(n,B)\, - B} \, U_\hbar^{\,\otimes\, (n+1)} \, $,
 \,thus for growing  $ n $  these elements sum up to a convergent series in
 $ U_\hbar \, $,  \,and we are done.
                                                         \par
   As to the term in the first line, we split it into
\begin{equation}  \label{eq: split a-hat & b}
  \begin{aligned}
     &  \hskip-5pt   \chi^{*\,t}\big(a_{(1)},b_{(1)}\big) \,
     \widehat{a}_{(2)} \, b_{(2)} \, \chi^{*\,\ell}\big(a_{(3)},b_{(3)}\big)  \; =  \\
   &  \hskip-3pt   = \; \chi^{*\,t}\big(a_{(1)},b_{(1)}\big) \, \widehat{a}_{(2)} \,
   \widehat{b}_{(2)} \, \chi^{*\,\ell}\big(a_{(3)},b_{(3)}\big)  \, + \,
   \chi^{*\,t}\big(a_{(1)},b_{(1)}\big) \, \widehat{a}_{(2)} \, b^+_{(2)} \,
   \chi^{*\,\ell}\big(a_{(3)},b_{(3)}\big)
  \end{aligned}
\end{equation}
 Then for the first summand we have (almost by definition, or acting as before)
  $$
  \chi^{*\,t}\big(a_{(1)},b_{(1)}\big) \, \widehat{a}_{(2)} \,
  \widehat{b}_{(2)} \, \chi^{*\,\ell}\big(a_{(3)},b_{(3)}\big)  \; = \;
  \chi^{*\,(t+\ell\,)}(a\,,b\,)
  $$
 so when we plug every such term in  \eqref{eq: chi*t & chi*ell}
 and then in  \eqref{eq: expans-sigmaprod},  overall they sum up
 to give the contribution
  $$
  \displaylines{
   \quad   {\textstyle \sum_{t+\ell \,> 0}} \, \hbar^{-(t+\ell\,)} \, {(-1)}^\ell \,
   {(t!)}^{-1} {(\ell!)}^{-1} \, \chi^{*\,t}\big(a_{(1)},b_{(1)}\big) \,
   \widehat{a}_{(2)} \, \widehat{b}_{(2)} \, \chi^{*\,\ell}\big(a_{(3)},b_{(3)}\big)  \; =
   \hfill  \cr
   = \;  {\textstyle \sum\limits_{n\,>\,0}} \; {\textstyle \sum_{t+\ell\,=\,n}} \,
   \hbar^{-(t+\ell\,)} \, {(-1)}^\ell \, {(t!)}^{-1} {(\ell!)}^{-1} \, \chi^{*\,(t+\ell)}(a\,,b\,)  \; =
   \cr
   \hfill   = \;  {\textstyle \sum\limits_{n\,>\,0}} \; {{\,1\,} \over {\,n!\,}} \, \hbar^{-n} \,
   \bigg(\, {\textstyle \sum\limits_{t+\ell\,=\,n}} {(-1)}^\ell \, {n \choose \ell} \!\bigg) \;
   \chi^{*\,n}(a\,,b\,)  \; = \;\,  0   \quad  }
   $$
 just because of the combinatorial identity
 $ \,\; {\textstyle \sum\limits_{t+\ell=n}} {(-1)}^\ell \, \displaystyle{n \choose \ell} \; =
 \; 0 \;\, $.
 \vskip5pt
   Finally, we have to dispose of the summands of type
\begin{equation}  \label{eq: a-hat_b+}
  \chi^{*\,t}\big(a_{(1)},b_{(1)}\big) \, \widehat{a}_{(2)} \, b^+_{(2)} \,
  \chi^{*\,\ell}\big(a_{(3)},b_{(3)}\big)
\end{equation}
 for which the analogue of the first identity in  \eqref{eq: h-adic estimates}
 holds true, namely
\begin{equation}  \label{eq: h-adic estimate a-hat & b+}
   \chi^{*\,t}\big(a_{(1)},b_{(1)}\big) \, \widehat{a}_{(2)} \, b^+_{(2)} \,
   \chi^{*\,\ell}\big(a_{(3)},b_{(3)}\big)  \; \in \;
   \hbar^{\,n - 1 + \,\max(n+1,B)\, - B} \, U_\hbar^{\,\otimes\, (n+1)}
\end{equation}
 where  $ \, n := t + \ell \, $,  taking into account that
 $ \, \delta_{n+1}(b) \, \in \, \hbar^{\,\max(n+1,B)\, - B} \,
 U_\hbar^{\,\otimes\, (n+1)} \; $.
                                                                   \par
   Then we have to distinguish two cases, depending on  $ \, n := t + \ell \, $.
 \vskip3pt
   First we assume  $ \; n := t + \ell \geq B \; $.  Then
 $ \; n - 1 + \,\max(n+1,B)\, - B \, \geq \, n \; $,
 \,hence the first identity in  \eqref{eq: h-adic estimates}  yields
 $ \,\; \chi^{*\,t}\big(a_{(1)},b_{(1)}\big) \, \widehat{a}_{(2)} \, b^+_{(2)} \,
 \chi^{*\,\ell}\big(a_{(3)},b_{(3)}\big)  \, \in \,  \hbar^{\,n} \,
 U_\hbar^{\,\otimes\, (n+1)} \; $,
 \;and then, when plugged in  \eqref{eq: split a-hat & b},
 and subsequently in  \eqref{eq: chi*t & chi*ell}  and in
 \eqref{eq: expans-sigmaprod},  this provides to the expansion of
 $ \; a \,\raisebox{-5pt}{$ \dot{\scriptstyle \sigma} $}\, b \; $
 a contribution of the form
  $$
  \hbar^{-n} \, {{\,{(-1)}^\ell\,} \over {\,t!\,\ell!\,}} \,
  \chi^{*\,t}\big(a_{(1)},b_{(1)}\big) \, \widehat{a}_{(2)} \, b^+_{(2)} \,
  \chi^{*\,\ell}\big(a_{(3)},b_{(3)}\big)  \; \in \;
  \hbar^{-n} \, \hbar^{\,n} \, U_\hbar^{\,\otimes\, (n+1)}  \; = \;
  U_\hbar^{\,\otimes\, (n+1)}
  $$
 --- which is fair! ---   hence we are done with it.
 \vskip3pt
   Then we are left with the case  $ \; n := t + \ell \leq B-1 \; $.
   Tracking backwards our construction,
   all these case provide to  \eqref{eq: expans-sigmaprod}
   a contribution of the form
\begin{equation}  \label{eq: contrib x n leq B}
  \begin{aligned}
     &  {\textstyle \sum\limits_{t+\ell\,=\,1}^{B-1}} \, \hbar^{-(t+\ell)} \,
     {{\,{(-1)}^\ell\,} \over {\,t!\,\ell!\,}} \, \chi^{*\,t}\big(a_{(1)},b_{(1)}\big) \,
     \widehat{a}_{(2)} \, b^+_{(2)} \, \chi^{*\,\ell}\big(a_{(3)},b_{(3)}\big)  \; =  \\
     &  \hskip21pt   = \;
     {\textstyle \sum\limits_{n\,=\,1}^{B-1}} \, {{\,1\,} \over {\,n!\,}} \,
     \hbar^{-n} {\textstyle \sum\limits_{t+\ell\,=\,n}} \, {(-1)}^\ell \, {n \choose \ell} \,
     \chi^{*\,t}\big(a_{(1)},b_{(1)}\big) \, \widehat{a}_{(2)} \, b^+_{(2)} \,
     \chi^{*\,\ell}\big(a_{(3)},b_{(3)}\big)
  \end{aligned}
\end{equation}
                                                                      \par
   With no loss of generality,  \textsl{we can assume that}
   $ \; a \not\equiv 0 \, $,  $ \, b \not\equiv 0 \; \big(\, \text{mod}\
   \hbar\,U_\hbar \,\big) \; $.  Then for their corresponding cosets
   $ \; \overline{a} \, , \overline{b} \in U_\hbar \Big/ \hbar\,U_\hbar \cong \ug \; $
   we have  $ \, \overline{a} \in {\ug}_1 \, $  and
   $ \, \overline{b} \in {\ug}_{\!B} \, $,  \,where
   $ \, {\big\{ {\ug}_n \big\}}_{n \in \NN} \, $
   is the standard, coradical filtration of  $ \ug \, $,  \,and also
   $ \, \delta_1(\,\overline{a}\,) \not= 0 \, $  as well as
   $ \, \delta_n\big(\,\overline{b}\,\big) \not= 0 \, $  for  $ \, 1 \leq n \leq B \, $
   --- cf.\ \cite{Ga1},  Lemma 3.3.  Moreover, we recall that
   $ \, U_\hbar := \uhg \, $  \textsl{is cocommutative modulo}
   $ \, \hbar\,U_\hbar \, $,  \,as it is a QUEA: in particular, this implies that
   $ \, \delta_n\big(\,\overline{b}\,\big) \, $  \textsl{is a  \textit{symmetric}  tensor}
   --- for  $ \, 1 \leq n \leq B \, $  ---   hence we can write  $ \delta_n(b\,) $
   in the form
\begin{equation}  \label{eq: delta_n(b) = beta + O(h)}
  \delta_n(b\,)  \; = \;  b^+_{(1)} \otimes \cdots \otimes b^+_{(n)}  \; = \;
  \beta_{\langle 1 \rangle} \otimes \cdots \otimes \! \beta_{\langle n \rangle} \,
  + \, \cO_n\big( \hbar^1 \big)
\end{equation}
 (for  $ \, 1 \leq n \leq B \, $)  where
 $ \, \beta_{\langle 1 \rangle} \otimes \cdots \otimes \beta_{\langle n \rangle} \, $
 --- using some  $ \sigma $--notation  of sort,  as usual
 ---   is some  \textit{symmetric}  tensor in  $ U_\hbar^{\,\otimes\, n} $  and
 \textsl{hereafter  $ \cO_n\big( \hbar^s \big) $  stands for some element in}
 $ \, \hbar^s \, U_\hbar^{\,\otimes\, n} \, $,  \,for every  $ \, s , n \in \NN \, $.
 Then plugging  \eqref{eq: delta_n(b) = beta + O(h)}  in  \eqref{eq: a-hat_b+}
 we find
  $$
  \displaylines{
   \chi^{*\,t}\big(a_{(1)},b_{(1)}\big) \, \widehat{a}_{(2)} \, b^+_{(2)} \,
   \chi^{*\,\ell}\big(a_{(3)},b_{(3)}\big)  \; = \; {\textstyle \prod\limits_{i=1}^t} \,
   \chi\big(a^+_{(i)},b^+_{(i)}\big) \, \widehat{a}_{(t+1)} \, b^+_{(t+1)} \,
   {\textstyle \prod\limits_{k=t+2}^{t+\ell+1}} \chi\big(a^+_{(k)},b^+_{(k)}\big)  \; =
   \hfill  \cr
   \hfill   = \; {\textstyle \prod\limits_{i=1}^t} \,
   \chi\big(a^+_{(i)},\beta_{\langle i \rangle}\big) \, \widehat{a}_{(t+1)} \,
   \beta_{\langle t+1 \rangle} \, {\textstyle \prod\limits_{k=t+2}^{t+\ell+1}}
   \chi\big(a^+_{(k)},\beta_{\langle k \rangle}\big)  \; + \;
   \cO_1\big( \hbar^{\,t+\ell\,} \big)  }
   $$
 for all  $ \, t+\ell \leq B-1 \, $,  \,with
  $$
  {\textstyle \prod\limits_{i=1}^t} \,
  \chi\big(a^+_{(i)},\beta_{\langle i \rangle}\big) \, \widehat{a}_{(t+1)} \,
  \beta_{\langle t+1 \rangle} \, {\textstyle \prod\limits_{k=t+2}^{t+\ell+1}}
  \chi\big(a^+_{(k)},\beta_{\langle k \rangle}\big)  \,\; \in \;\, \hbar^{\,t+\ell-1} \,
  U_\hbar
  $$
   \indent
   Therefore, the contribution to  \eqref{eq: expans-sigmaprod}  given in
   \eqref{eq: contrib x n leq B}  now reads
%
%%%%%
% \begin{equation}  \label{eq: contrib (beta) x n leq B}
%   \begin{aligned}
%      &  {\textstyle \sum\limits_{n\,=\,1}^{B-1}} \, {{\,1\,} \over {\,n!\,}} \,
% \hbar^{-n} {\textstyle \sum\limits_{t+\ell\,=\,n}} \, {(-1)}^\ell \, {n \choose \ell}
% \, \chi^{*\,t}\big(a_{(1)},b_{(1)}\big) \, \widehat{a}_{(2)} \, b^+_{(2)} \,
% \chi^{*\,\ell}\big(a_{(3)},b_{(3)}\big)  \; =  \\
%      &  \hskip3pt   = \;  {\textstyle \sum\limits_{n\,=\,1}^{B-1}} \, {{\,1\,} \over {\,n!\,}}
% \, \hbar^{-n} {\textstyle \sum\limits_{t+\ell\,=\,n}} \, {(-1)}^\ell \, {n \choose \ell} \,
% {\textstyle \prod\limits_{i=1}^t} \, \chi\big(a^+_{(i)},\beta_{\langle i \rangle}\big) \,
% \widehat{a}_{(t+1)} \, \beta_{\langle t+1 \rangle} \, {\textstyle \prod\limits_{k=t+2}^{t+\ell+1}}
% \chi\big(a^+_{(k)},\beta_{\langle k \rangle}\big)  \,\; +  \\
%      &  + \;\,  \cO_1\big( \hbar^0 \big)
% %
%   \end{aligned}
% %
% \end{equation}
%%%%%
%
  $$
  \displaylines{
   {\textstyle \sum\limits_{n\,=\,1}^{B-1}} \, {{\,1\,} \over {\,n!\,}} \,
   \hbar^{-n} {\textstyle \sum\limits_{t+\ell\,=\,n}} \, {(-1)}^\ell \,
   {n \choose \ell} \, \chi^{*\,t}\big(a_{(1)},b_{(1)}\big) \,
   \widehat{a}_{(2)} \, b^+_{(2)} \, \chi^{*\,\ell}\big(a_{(3)},b_{(3)}\big)  \; =   \hfill  \cr
   = \;  {\textstyle \sum\limits_{n\,=\,1}^{B-1}} \, {{\,1\,} \over {\,n!\,}} \,
   \hbar^{-n} {\textstyle \sum\limits_{t+\ell\,=\,n}} \, {(-1)}^\ell \, {n \choose \ell} \,
   {\textstyle \prod\limits_{i=1}^t} \, \chi\big(a^+_{(i)},\beta_{\langle i \rangle}\big) \,
   \widehat{a}_{(t+1)} \, \beta_{\langle t+1 \rangle} \,
   {\textstyle \prod\limits_{k=t+2}^{n+1}}
   \chi\big(a^+_{(k)},\beta_{\langle k \rangle}\big)  \,\; +  \cr
   \hfill   + \;\,  \cO_1\big( \hbar^0 \big)  }
   $$
 where in the last formula we have
  $$
  {\textstyle \sum\limits_{n\,=\,1}^{B-1}} \, {{\,1\,} \over {\,n!\,}} \,
  \hbar^{-n} \! {\textstyle \sum\limits_{t+\ell\,=\,n}} \! {(-1)}^\ell \, {n \choose \ell} \,
  {\textstyle \prod\limits_{i=1}^t} \, \chi\big(a^+_{(i)},\beta_{\langle i \rangle}\big) \,
  \widehat{a}_{(t+1)} \, \beta_{\langle t+1 \rangle} \,
  {\textstyle \prod\limits_{k=t+2}^{n+1}}
  \chi\big(a^+_{(k)},\beta_{\langle k \rangle}\big) \, \in \, \hbar^{-1} \, U_\hbar  $$
   \indent   Now, observe that, setting  $ \, n := t + \ell \, $,  we can re-write
\begin{equation}  \label{eq: chi*t&ell varPhi}
   {\textstyle \prod\limits_{i=1}^t} \, \chi\big(a^+_{(i)},\beta_{\langle i \rangle}\big) \,
   \widehat{a}_{(t+1)} \, \beta_{\langle t+1 \rangle} {\textstyle \prod\limits_{k=t+2}^{n+1}}
   \! \chi\big(a^+_{(k)},\beta_{\langle k \rangle}\big)  \, = \,  \varPhi \big( \delta_n(a)
   \otimes \beta_{\langle 1 \rangle} \otimes \cdots \otimes \beta_{\langle n+1 \rangle} \big)
\end{equation}
 with  $ \; \varPhi : U_\hbar^{\,\otimes\, 2(n+1)} \!\relbar\joinrel\relbar\joinrel\longrightarrow
 U_\hbar \; $  being the map given by the composition
  $$
  \varPhi  \; := \;  \mu \circ \big(\, \chi^{\otimes\, t} \otimes \id_{U_\hbar}^{\,\otimes 2}
  \otimes \chi^{\otimes (n-t)} \big) \circ \varsigma_{\,n+1}
  $$
 where
 \vskip3pt
   \textit{(1)}  $ \; \varsigma_{\,n+1} :
   U_\hbar^{\,\otimes\, (2n+1)} \!\relbar\joinrel\relbar\joinrel\longrightarrow
   U_\hbar^{\,\otimes\, (2n+1)} \; $  is the ``shuffle'' map
  $$
  x_1 \otimes \cdots \otimes x_n \otimes y_1
  \otimes \cdots \otimes y_n \otimes y_{n+1}  \;
  \mapsto \;  x_1 \otimes y_1 \otimes x_2 \otimes y_2 \otimes \cdots \otimes
  x_n \otimes y_n \otimes y_{n+1}
  $$
 and, considering  $ \kh $  as embedded into  $ U_\hbar $  via the unit map,
 \vskip3pt
   \textit{(2)}  $ \; \mu : U_\hbar^{\otimes\, n} \relbar\joinrel\relbar\joinrel\longrightarrow
   U_\hbar \; $  is the obvious  ($ n $--fold  iterated) multiplication by scalars.
 \vskip5pt
   Now recall that  $ \, \beta_{\langle 1 \rangle} \otimes \cdots \otimes
   \beta_{\langle n+1 \rangle} \, $  represents a tensor in  $ \sigma $--notation,
   so more explicitly we might write
  $ \, \beta_{\langle 1 \rangle} \otimes \cdots \otimes \beta_{\langle n+1 \rangle} \, = \,
  \sum_{s=1}^N \beta_{s,1} \otimes \beta_{s,\,n+1} \, $;  \,so in the formula we are dealing with what is written as a product \;
  $ {\textstyle \prod\limits_{i=1}^t} \chi\big(a^+_{(i)},\beta_{\langle i \rangle}\big) \,
  \widehat{a}_{(t+1)} \, \beta_{\langle t+1 \rangle} {\textstyle \prod\limits_{k=t+2}^{n+1}}
  \chi\big(a^+_{(k)},\beta_{\langle k \rangle}\big) $
 \; is actually a sum of several products as
  $ \; {\textstyle \prod\limits_{i=1}^t} \, \chi\big(a^+_{(i)},\beta_{s,\,i}\big) \,
  \widehat{a}_{(t+1)} \, \beta_{s,\,t+1} \, {\textstyle \prod\limits_{k=t+2}^{n+1}}
  \chi\big(a^+_{(k)},\beta_{s,\,k}\big) \; $.
 But then recall that this tensor
  $ \, \beta_{\langle 1 \rangle} \otimes \cdots \otimes \beta_{\langle n+1 \rangle} \, =
  \, \sum_{s=1}^N \beta_{s,1} \otimes \beta_{s,\,n+1} \, $
is  \textsl{symmetric},  therefore, the various products
  $ \; {\textstyle \prod\limits_{i=1}^t} \, \chi\big(a^+_{(i)},\beta_{s,\,i}\big) \,
  \widehat{a}_{(t+1)} \, \beta_{s,\,t+1} \, {\textstyle \prod\limits_{k=t+2}^{n+1}}
  \chi\big(a^+_{(k)},\beta_{s,\,k}\big) \; $
 actually all  \textsl{coincide\/}:
 letting  $ C_n $  be their ``common value'', we deduce that
  $$
  \displaylines{
   {\textstyle \sum\limits_{n\,=\,1}^{B-1}} \, {{\,1\,} \over {\,n!\,}} \, \hbar^{-n} \!
   {\textstyle \sum\limits_{t+\ell\,=\,n}} \! {(-1)}^\ell \, {n \choose \ell} \,
   {\textstyle \prod\limits_{i=1}^t} \, \chi\big(a^+_{(i)},\beta_{\langle i \rangle}\big) \,
   \widehat{a}_{(t+1)} \, \beta_{\langle t+1 \rangle} \,
   {\textstyle \prod\limits_{k=t+2}^{n+1}}
   \chi\big(a^+_{(k)},\beta_{\langle k \rangle}\big)  \; =   \hfill  \cr
   \hfill   = \;\,  {\textstyle \sum\limits_{n\,=\,1}^{B-1}} \, {{\,1\,} \over {\,n!\,}} \,
   \hbar^{-n} \bigg(\, {\textstyle \sum\limits_{t+\ell\,=\,n}} \! {(-1)}^\ell \,
   {n \choose \ell} \!\bigg) \, C_n  \,\; = \;\,  0  }
   $$
 again because of the identity
 $ \,\; {\textstyle \sum\limits_{t+\ell=n}} {(-1)}^\ell \, \displaystyle{n \choose \ell} \; =
 \; 0 \;\, $.
                                                                        \par
   Thus, also the last contributions to  \eqref{eq: expans-sigmaprod}  given in
   \eqref{eq: contrib x n leq B}  actually belong to $ U_\hbar \, $.
 \vskip7pt
   To sum up,  \textsl{we have proved that
\begin{equation}  \label{eq: deform-prod a&b in Uh (FIRST)}
  \begin{aligned}
     &  a \,\raisebox{-5pt}{$ \dot{\scriptstyle \sigma} $}\, b \; = \;  a \cdot b \, + \, z
     \qquad \qquad \textsl{with \ } \; z \in J_\hbar  \\
     &  \forall \;\;\;  a \in \widetilde{J}_\hbar^{\raisebox{1pt}{$ \; \scriptstyle 1 $}} =
     \hbar^{-1} J_\hbar^{\,\prime} \, , \;\; b \in \widetilde{J}_\hbar^{\raisebox{1pt}{$ \;
     \scriptstyle B $}} = \hbar^{-B} J_\hbar^{\,\prime \! \raisebox{1pt}{$ \;
     \scriptstyle B $}}
  \end{aligned}
\end{equation}
 and similarly   --- just switching roles of  $ a $  and  $ b $  ---   also}
\begin{equation}  \label{eq: deform-prod b&a in Uh (FIRST)}
  \begin{aligned}
     &  b \,\raisebox{-5pt}{$ \dot{\scriptstyle \sigma} $}\, a \; = \;  b \cdot a \, + \, x
     \qquad \qquad \textsl{with \ } \; x \in J_\hbar  \\
     &  \forall \;\;\;  a \in \widetilde{J}_\hbar^{\raisebox{1pt}{$ \; \scriptstyle 1 $}} =
     \hbar^{-1} J_\hbar^{\,\prime} \, , \;\; b \in \widetilde{J}_\hbar^{\raisebox{1pt}{$ \;
     \scriptstyle B $}} = \hbar^{-B} J_\hbar^{\,\prime \! \raisebox{1pt}{$ \;
     \scriptstyle B $}}
  \end{aligned}
\end{equation}
 \vskip3pt
   Let  $ \, {\big\langle \widetilde{J}_\hbar \big\rangle}_\kh^\sigma \, $
   the unital  $ \kh $--subalgebra  of
   $ \, {\big( \Uhg \big)}_{\!\sigma} := {\big( \khp \otimes_\kh U_\hbar \big)}_{\!\sigma} \, $
   generated by  $ \widetilde{J}_\hbar \, $.
%%%%%
%  Note that  $ \widetilde{J}_\hbar $  is a
%  $ \kh $--submodule  and a coideal in  $ \Uhg \, $;  since the coalgebra structure is
%  the same in  $ \Uhg $  and  $ {\big( \Uhg \big)}_{\!\sigma} \, $,  it follows that
%  $ \widetilde{J}_\hbar $  is a  $ \kh $--submodule  and a coideal in
%  $ {\big( \Uhg \big)}_{\!\sigma} $  as well, hence
%  $ \, {\big\langle \widetilde{J}_\hbar \big\rangle}_\kh^\sigma \, $
%  \textsl{is a Hopf  $ \kh $--subalgebra  inside\/}  $ {\big( \Uhg \big)}_{\!\sigma} \, $.
%%%%%
 Now recall that  $ \, \widetilde{J}_\hbar := \hbar^{-1} J_\hbar^{\,\prime} \, $  with
 $ \, J_\hbar^{\,\prime} := \Ker\big(U_\hbar^{\,\prime}\big) \, $;  \,then
 $ \, U_\hbar^{\,\prime} = J_\hbar^{\,\prime} \oplus \kh \cdot 1 \, $,  \,which implies
 $ \; \Delta\big(J_\hbar^{\,\prime}\big) \, \subseteq \, J_\hbar^{\,\prime} \otimes 1 +
 J_\hbar^{\,\prime} \otimes J_\hbar^{\,\prime} + 1 \otimes J_\hbar^{\,\prime} \; $.
 Then we get also
  $ \; \Delta\big(\widetilde{J}_\hbar\big) \, \subseteq \, \widetilde{J}_\hbar \otimes 1
  + \widetilde{J}_\hbar \otimes \widetilde{J}_\hbar + 1 \otimes \widetilde{J}_\hbar \; $.
  Since the coalgebra structure is the same in  $ \Uhg $  and
  $ {\big( \Uhg \big)}_{\!\sigma} \, $,  it follows eventually from this that
  $ \, {\big\langle \widetilde{J}_\hbar \big\rangle}_\kh^\sigma \, $
  \textsl{is a Hopf  $ \kh $--subalgebra  inside\/}  $ {\big( \Uhg \big)}_{\!\sigma} \, $.
%%%%%%%%%
                                                         \par
   By repeated use of  \eqref{eq: deform-prod a&b in Uh (FIRST)}  or
   \eqref{eq: deform-prod b&a in Uh (FIRST)}  alike, we find that
   $ \; {\big\langle \widetilde{J}_\hbar \big\rangle}_\kh^\sigma \subseteq U_\hbar =
   {\big\langle \widetilde{J}_\hbar \big\rangle}_\kh \; $.
                                                         \par
   Now observe that the original product  ``$ \, \cdot \, $''  in  $ \, U_\hbar := \uhg \, $
   and  $ \Uhg $  can be obtained from
   ``$ \,\raisebox{-5pt}{$ \dot{\scriptstyle \sigma} $}\, $''
   through deformation via the inverse  $ 2 $--cocycle  $ \sigma^{-1} \, $.
   Thanks to this, we can reverse the roles of
   $ \, U_\hbar = {\big\langle \widetilde{J}_\hbar \big\rangle}_\kh \, $  and
   $ \, {\big\langle \widetilde{J}_\hbar \big\rangle}_\kh^\sigma \, $
   in the previous construction (with some care), thus eventually achieving the
   converse inclusion
   $ \; {\big\langle \widetilde{J}_\hbar \big\rangle}_\kh \subseteq {\big\langle
   \widetilde{J}_\hbar \big\rangle}_\kh^\sigma \; $.
 Therefore  $ \; {\big\langle \widetilde{J}_\hbar \big\rangle}_\kh \subseteq
 {\big\langle \widetilde{J}_\hbar \big\rangle}_\kh^\sigma \; $,  \,which in particular implies
 that  $ \; U_\hbar = {\big\langle \widetilde{J}_\hbar \big\rangle}_\kh \; $  is closed for the
 $ \sigma $--product,  \,q.e.d.
\epf

\vskip5pt

\begin{definition}  \label{def: polar 2-coc deform-QUEA}
 With assumptions as in  Theorem \ref{thm: polar 2cocycle-deform-QUEA},
 the new QUEA obtained from  $ \uhg $  through the process of 2--cocycle
 deformation by  $ \sigma $  of  $ \Uhg $  followed by restriction will be called
 \textsl{the polar 2--cocycle deformation of\/  $ \uhg $  by  $ \sigma $},
 and it will be denoted by  $ {\uhg}_\sigma \, $.
\end{definition}

\vskip7pt

   To complete our analysis, next result sheds light onto the new, polar 2--cocycle deformed QUEA  $ {\uhg}_\sigma \, $,  describing in detail its semiclassical limit:

\vskip9pt

\begin{theorem}  \label{thm: propt.'s qs-2cocycle-deform-QUEA}
 Let  $ \uhg $  be a QUEA over the Lie bialgebra
 $ \, \lieg = \big(\, \lieg \, ; \, [\,\ ,\ ] \, , \, \delta \,\big) \, $.
 Let\/  $ \sigma $  be a polar 2--cocycle for\/  $ \uhg \, $,  so
 $ \, \sigma = \exp_*\!\big( \hbar^{-1} \chi \big) \, $  for some
 $ \, \chi \in {\Big( {\uhg}^{\widehat{\otimes}\, 2} \Big)}^{\!*} \, $  with
 $ \, \chi(z\,,1) = 0 = \chi(1\,,z) \, $  for  $ \, z \in \uhg \, $.  Set also
 $ \; \chi_a := \chi - \chi_{2,1} \; $.
 Then:
 \vskip5pt
   (a)\;  $ \chi $  is antisymmetric, i.e.\  $ \, \chi_{2,1} = -\chi \, $,
   iff\/  $ \sigma $
   is orthogonal, i.e.\  $ \, \sigma_{2,1} = \sigma^{-1} \, $;
 \vskip5pt
%
%%%%%
%     (b)\;  the deformation of the Hopf  $ \, \khp $--algebra  $ \, \Uhg := \khp \,\widehat{\otimes}\,
% \uhg $  by the 2--cocycle  $ \sigma $  restricts to a deformation of  $ \, \uhg $  itself, still
% denoted  $ \, {\big( \uhg \big)}_\sigma \; $;
% %
%  \vskip5pt
% %
%%%%%
%
   (b)\;  the  $ \Bbbk $--linear  map  $ \, \gamma \, := \,
   \chi_a \; \Big(\, \text{\rm mod} \; \hbar \,
   {\big( \uhg^{\widehat{\otimes}\, 2} \big)}^{\!*} \,\Big) \Big|_{\lieg \otimes \lieg} \; $
   from  $ \, \lieg \otimes \lieg \, $  to\/  $ \Bbbk $  is
   \textsl{antisymmetric 2--cocycle}  for the Lie bialgebra  $ \lieg \, $;
 \vskip5pt
   (c)\;  the polar 2--cocycle deformation  $ \, {\big( \uhg \big)}_\sigma $  of
   $ \, \uhg $  is a QUEA for the Lie bialgebra
   $ \, \lieg_\gamma = \big(\, \lieg \, ; \, {[\,\ ,\ ]}_\gamma \, , \, \delta \,\big) \, $
   which is the deformation of  $ \, \lieg $  by the 2--cocycle  $ \gamma \, $;
   in a nutshell, we have
   $ \; {\big( \uhg \big)}_\sigma \cong \, U_\hbar\big(\lieg_\gamma\big) \; $.
 \vskip3pt
   In particular, if\/  $ \sigma $  is\/  $ \kh $--valued
   --- i.e., it is an ordinary 2--cocycle for the Hopf\/  $ \kh $--algebra
   $ \uhg \, $  ---   or equivalently
   $ \, \chi \in \hbar \, {\big( \uhg^{\widehat{\otimes}\,2} \big)}^{\!*} \, $,
   \,then we have just  $ \, \gamma = 0 \, $  and
   $ \; {\big( \uhg \big)}_\sigma \cong \, U_\hbar\big(\lieg_\gamma\big) = \uhg \; $.
\end{theorem}

\pf
 \textit{(a)}\,  This follows from claim  \textit{(c)\/}  in
 Lemma \ref{lemma: properties polar 2-cocycle}.
 \vskip9pt
   \textit{(b)}\,  We are interested in the restriction to  $ \, \lieg \otimes \lieg \, $
   of the specialization of  $ \sigma $  modulo  $ \hbar \, $.  So we start with
   $ \, \text{a} \, , \text{b} \, , \text{c} \in \lieg \, $,  \,that we realize as
   $ \; \text{a} = a \; \big(\,\text{mod} \; \hbar \, U_\hbar \,\big) \; $,
   $ \; \text{b} = b \; \big(\,\text{mod} \; \hbar \, U_\hbar \,\big) \; $  and
   $ \; \text{c} = c \; \big(\,\text{mod} \; \hbar \, U_\hbar \,\big) \; $
   for some ``lifts''  $ \, a , b , c \in U_\hbar \, $.  By the identity
   $ \, U_\hbar = {\big( U_\hbar^{\,\prime} \,\big)}^{\!\vee} \, $
   and by  Lemma 3.3 in  \cite{Ga1},  \textsl{we can choose the lifts
   $ a $,  $ b $  and  $ c $  belong to
   $ \, \widetilde{J}_\hbar := \hbar^{-1} J_\hbar^{\,\prime} \, $,  \,so that
   $ \, a^{\,\prime} := \hbar\,a \, $,  $ \, b^{\,\prime} := \hbar\,b \, $  and
   $ \, c^{\,\prime} := \hbar\,c \, $  belong to  $ J_\hbar^{\,\prime} \, $}.
                                                                           \par
   As  $ \sigma $  is a normalized Hopf 2--cocycle for  $ U_\hbar \, $,
   it must obey the equality
\begin{equation}  \label{eq:Hopf-2-cocycle-proof}
   \sigma\big(b^{\,\prime}_{(1)} , c^{\,\prime}_{(1)}\big) \,
   \sigma\big(a^{\,\prime} , b^{\,\prime}_{(2)} c^{\,\prime}_{(2)}\big)  \; = \;
   \sigma\big(a^{\,\prime}_{(1)} , b^{\,\prime}_{(1)}\big) \,
   \sigma\big(a^{\,\prime}_{(2)} b^{\,\prime}_{(2)} , c^{\,\prime}\,\big)
\end{equation}
   \indent   Let us focus on the left hand side of
   \eqref{eq:Hopf-2-cocycle-proof}.  Expanding the exponential we get
\begin{align*}
   &  \sigma\big(b_{(1)}^{\,\prime},c_{(1)}^{\,\prime}\big) \,
   \sigma\big(a^{\,\prime},b_{(2)}^{\,\prime}c_{(2)}^{\,\prime}\big)  \; = \,
   \sum_{n,m\geq 0} \frac{\hbar^{-(n+m)}}{n!\,m!} \,
   \chi^{\ast\,n}\big(b_{(1)}^{\,\prime},c_{(1)}^{\,\prime}\big) \,
   \chi^{\ast\,m}\big(a^{\,\prime },b_{(2)}^{\,\prime }c_{(2)}^{\,\prime}\big)  \; =  \\
   &  \hskip7pt  = \;  \epsilon\big(a^{\,\prime}\big) \, \epsilon\big(b^{\,\prime}\big) \,
   \epsilon\big(c^{\,\prime}\big) \, + \, \hbar^{-1} \, \chi\big(b^{\,\prime},c^{\,\prime}\big)
   \, \epsilon\big(a^{\,\prime}\big) \, + \, \hbar^{-1} \,
   \chi\big(a^{\,\prime},b^{\,\prime}c^{\,\prime}\big) \, +  \\
  &  \hskip10pt  + \, \hbar^{-2} \, \chi\big(b_{(1)}^{\,\prime},c_{(1)}^{\,\prime}\big) \,
  \chi\big(a^{\,\prime},b_{(2)}^{\,\prime}c_{(2)}^{\,\prime}\big) \, + \,
  \hbar^{-2} \, 2^{-1} \, \chi^{\ast\,2}\big(b^{\,\prime},c^{\,\prime}\big) \,
  \epsilon\big(a^{\,\prime}\big) \, +  \\
   &  \hskip15pt  +  \, \hbar^{-2} \, 2^{-1} \,
   \chi^{\ast\,2}\big(a^{\,\prime},b^{\,\prime}c^{\,\prime }\big) \, +
   \sum_{n+m\geq 3} \frac{\hbar^{-(n+m)}}{n!\,m!} \,
   \chi^{\ast\,n}\big(b_{(1)}^{\,\prime},c_{(1)}^{\,\prime}\big) \,
   \chi^{\ast\,m}\big(a^{\,\prime},b_{(2)}^{\,\prime}c_{(2)}^{\,\prime}\big)
\end{align*}
 Then, noting that  $ \, \epsilon\big(a^{\,\prime}\big) = \epsilon\big(b^{\,\prime}\big) =
 \epsilon\big(c^{\,\prime}\big) = 0 \, $
 by construction, and analyzing all other summands as in the proof of claim
 Theorem \ref{thm: polar 2cocycle-deform-QUEA},  we obtain
\begin{equation}  \label{eq: sigma-primes}
 \begin{aligned}
    &  \sigma\big(b_{(1)}^{\,\prime},c_{(1)}^{\,\prime}\big) \,
    \sigma\big(a^{\,\prime},b_{(2)}^{\,\prime}c_{(2)}^{\,\prime}\big)  \; = \;
    \hbar^{-1} \chi\big(a^{\,\prime},b^{\,\prime}c^{\,\prime}\big) \, +  \\
    &  \hskip50pt   + \, \hbar^{-2} \chi\big(b_{(1)}^{\,\prime},c_{(1)}^{\,\prime}\big) \,
    \chi\big(a^{\,\prime},b_{(2)}^{\,\prime}c_{(2)}^{\,\prime}\big) \, +
    \,\hbar^{-2} \, 2^{-1} \, \chi^{\ast\,2}\big(a^{\,\prime},b^{\,\prime}c^{\,\prime}\big) \, +
    \\
    &  \hskip100pt   + \, \big(\, \textsl{sum of all terms with\ } n+m \geq 3 \,\big)
 \end{aligned}
\end{equation}
 Writing  $ \, z' = z^{\prime +} + \epsilon\big(z'\big) \, $  and using that
 $ \, \chi(z,1) = 0 = \chi(1,z) \, $  and
  $$
  \chi\big(x_{(1)}^{\,\prime},y_{(1)}^{\,\prime}z_{(1)}^{\,\prime}\big) \,
  \chi\big(x_{(2)}^{\,\prime},y_{(2)}^{\,\prime} \,
  \epsilon\big(z_{(2)}^{\,\prime}\big)\big) \, =
  \, \chi\big(x_{(1)}^{\,\prime},y_{(1)}^{\,\prime}z^{\,\prime }\big) \,
  \chi\big(x_{(2)}^{\,\prime},y_{(2)}^{\,\prime}\big)
  $$
 we have that
  $$
  \displaylines{
   \hskip5pt   \chi^{\ast\,2}\big(a^{\,\prime},b^{\,\prime}c^{\,\prime}\big)  \; = \;
   \chi\big(a_{(1)}^{\,\prime},b_{(1)}^{\,\prime}c_{(1)}^{\,\prime }\big) \,
   \chi(a_{(2)}^{\,\prime},b_{(2)}^{\,\prime}c_{(2)}^{\,\prime }\big)  \; =   \hfill  \cr
   \hskip17pt   = \;  \chi\big(a_{(1)}^{\,\prime +},b_{(1)}^{\,\prime +}c_{(1)}^{\,\prime +}\big) \,
   \chi\big(a_{(2)}^{\,\prime +},b_{(2)}^{\,\prime +}c_{(2)}^{\,\prime +}\big) \, + \,
   \chi\big(a_{(1)}^{\,\prime +},b_{(1)}^{\,\prime +}\big) \,
   \chi\big(a_{(2)}^{\,\prime +},b_{(2)}^{\,\prime +}c^{\,\prime}\big) \, +
   \hfill  \cr
   \hskip29pt   + \, \chi\big(a_{(1)}^{\,\prime +},c_{(1)}^{\,\prime +}\big) \,
   \chi\big(a_{(2)}^{\,\prime +},b^{\,\prime}c_{(2)}^{\,\prime +}\big) \, + \,
   \chi\big(a_{(1)}^{\,\prime +},b^{\,\prime}c_{(1)}^{\,\prime +}\big) \,
   \chi\big(a_{(2)}^{\,\prime +},c_{(2)}^{\,\prime +}\big) \, +
   \hfill  \cr
   \hfill   + \, \chi\big(a_{(1)}^{\,\prime +},b^{\,\prime}\big) \,
   \chi\big(a_{(2)}^{\,\prime +},c^{\,\prime}\big) \, + \,
   \chi\big(a_{(1)}^{\,\prime +},b_{(1)}^{\,\prime +}c^{\,\prime }\big) \,
   \chi\big(a_{(2)}^{\,\prime +},b_{(2)}^{\,\prime +}\big) \, + \,
   \chi\big(a_{(1)}^{\,\prime +},c^{\,\prime }\big) \,
   \chi\big(a_{(2)}^{\,\prime +},b^{\,\prime }\big)  }
   $$
 Now, taking into account that  $ \, z' = \hbar\,z \, $  and
 $ \, z_{(i)}^{\prime\ +} \in \hbar \, U_\hbar \, $,  \,
 we may re-write the expression above as
  $$
  \chi^{\ast\,2}\big(a^{\,\prime},b^{\,\prime}c^{\,\prime}\big)  \; = \;\,
  \hbar^4 \, \chi\big(a_{(1)},b\big) \, \chi\big(a_{(2)},c\big) \, + \,
  \hbar^4 \, \chi\big(a_{(1)},c\big) \, \chi\big(a_{(2)},b\big) \, + \, \cO\big(\hbar^5\big)
  $$
 Performing a similar analysis on the term
 $ \, \chi\big(b_{(1)}^{\,\prime},c_{(1)}^{\,\prime}\big) \,
 \chi\big(a^{\,\prime},b_{(2)}^{\,\prime}c_{(2)}^{\,\prime}\big) \, $
 we get
  $$
  \chi\big(b_{(1)}^{\,\prime},c_{(1)}^{\,\prime}\big) \,
  \chi\big(a^{\,\prime},b_{(2)}^{\,\prime}c_{(2)}^{\,\prime}\big)  \; = \;\,
  \hbar^4 \, \chi\big(b,c_{(1)}\big) \, \chi\big(a,c_{(2)}\big) \, +
  \, \hbar^4 \, \chi\big(b_{(1)},c\big) \, \chi\big(a,b_{(2)}\big) \, + \,
  \cO\big(\hbar^5\big)
  $$
 Moreover, with a similar (yet easier) analysis one finds also that
  $$
  \chi^{*n}\big(b'_{(1)},c'_{(1)}\big) \, \chi^{*m}\big(a',b'_{(2)}c'_{(2)}\big) \,
  \in \, \hbar^{+2(n+m)} \, \kh
  $$
 for all  $ \, n+m \geq 3 \, $,  \,so that the (last) summand
 ``$ \big(\, \textsl{sum of all terms with\ } n+m \geq 3 \,\big) $''
 in  \eqref{eq: sigma-primes}  is of type
 $ \, \cO\big( \hbar^{n+m} \big) = \cO\big( \hbar^3 \big) \, $.
%%%%%%
%
 Putting all together in  \eqref{eq: sigma-primes}  we find
$$
 \begin{aligned}
   &  \hskip-5pt   \hbar^3 \, \sigma\big(b_{(1)},c_{(1)}\big) \,
   \sigma\big(a,b_{(2)}c_{(2)}\big)  \; =  \\
   &  \hskip-3pt   = \,  \hbar^2 \, \chi(a,bc) +
   \hbar^2 \, \chi\big(b,c_{(1)}\big) \, \chi\big(a,c_{(2)}\big) +
   \hbar^2 \, \chi\big(b_{(1)},c\big) \, \chi\big(a,b_{(2)}\big) \, +  \\
   &  \hskip5pt   + \, \hbar^2 \, 2^{-1} \, \chi\big(a_{(1)},b\big) \, \chi\big(a_{(2)},c\big) \, + \,
   \hbar^2 \, 2^{-1} \, \chi\big(a_{(1)},c\big) \, \chi\big(a_{(2)},b\big) \, + \,
   \cO\big(\hbar^3\big)
 \end{aligned}
$$
 An analogous treatment of the right hand side of
 \eqref{eq:Hopf-2-cocycle-proof}  yields
$$
 \begin{aligned}
   &  \hskip-5pt   \hbar^3 \, \sigma\big(a_{(1)},b_{(1)}\big) \,
   \sigma\big(a_{(2)}b_{(2)},c\,\big)  \; =  \\
   &  \hskip-3pt   = \,  \hbar^2 \, \chi(ab,c\,) + \hbar^2 \, \chi\big(a,b_{(1)}\big) \,
   \chi\big(b_{(2)},c\,\big) + \hbar^2 \, \chi\big(a_{(1)},b\,\big) \,
   \chi\big(a_{(2)},c\,\big) \, +  \\
   &  \hskip5pt   + \, \hbar^2 \, 2^{-1} \, \chi\big(a,c_{(1)}\big) \,
   \chi\big(b,c_{(2)}\big) \, +
   \, \hbar^2 \, 2^{-1} \, \chi\big(b,c_{(1)}\big) \, \chi\big(a,c_{(2)}\big) \, + \,
   \cO\big(\hbar^3\big)
 \end{aligned}
$$
 Altogether, this implies that
\begin{equation*}
  \begin{aligned}
    &  \hskip-7pt   \chi(a,bc\,) \, + \, \chi\big(b,c_{(1)}\big) \, \chi\big(a,c_{(2)}\big) \,
    + \, \chi\big(b_{(1)},c\,\big) \, \chi\big(a,b_{(2)}\big) \, +  \\
   &  \hskip-7pt \quad   + \, 2^{-1} \, \chi\big(a_{(1)},b\big) \, \chi\big(a_{(2)},c\,\big) \,
   + \, 2^{-1} \, \chi\big(a_{(1)},c\,\big) \, \chi\big(a_{(2)},b\,\big)  \;
   \underset{\hbar}{\equiv}  \\
   &  \hskip-11pt \qquad \qquad   \underset{\hbar}{\equiv} \;  \chi(ab,c\,) \, + \,
   \chi\big(a,b_{(1)}\big) \, \chi\big(b_{(2)},c\,\big) \, + \, \chi\big(a_{(1)},b\,\big) \,
   \chi\big(a_{(2)},c\,\big) \, +  \\
   &  \hskip-11pt \qquad \qquad \quad   + \, 2^{-1} \, \chi\big(a,c_{(1)}\big) \,
   \chi\big(b,c_{(2)}\big) \, + \, 2^{-1} \, \chi\big(b,c_{(1)}\big) \,
   \chi\big(a,c_{(2)}\big)
      \end{aligned}
\end{equation*}
 where (again)  $ \, \underset{\hbar}{\equiv} \, $  stands for ``congruent modulo
 $ \, \hbar \, \kh \, $'',
 that we re-write as
\begin{equation}  \label{eq: sigma-modh}
  \begin{aligned}
    &  \hskip-7pt   \chi(a,bc\,) \, + \, \chi\big(b,c_{(1)}\big) \, \chi\big(a,c_{(2)}\big) \,
    + \, \chi\big(b_{(1)},c\,\big) \, \chi\big(a,b_{(2)}\big) \, +  \\
   &  \hskip-7pt \quad   + \, 2^{-1} \, \chi\big(a_{(1)},b\big) \, \chi\big(a_{(2)},c\,\big) \,
   + \, 2^{-1} \, \chi\big(a_{(1)},c\,\big) \, \chi\big(a_{(2)},b\,\big)  \; -  \\
   &  \hskip-13pt \quad \qquad   - \;  \chi(ab,c\,) \, - \, \chi\big(a,b_{(1)}\big) \,
   \chi\big(b_{(2)},c\,\big) \, - \, \chi\big(a_{(1)},b\,\big) \, \chi\big(a_{(2)},c\,\big) \, -
   \\
   &  \hskip-19pt \qquad \qquad \quad   - \, 2^{-1} \, \chi\big(a,c_{(1)}\big) \,
   \chi\big(b,c_{(2)}\big) \, - \, 2^{-1} \, \chi\big(b,c_{(1)}\big) \, \chi\big(a,c_{(2)}\big)
   \; \underset{\hbar}{\equiv} \;  0
      \end{aligned}
\end{equation}
                                                       \par
   Consider now the action of the group algebra
   $ \, \Bbbk\big[\mathbb{S}_3\big] \, $  of the symmetric group
   $ \mathbb{S}_3 \, $ on  $ \big(U_{\hbar}^{\widehat{\otimes} 3}\big)^{*} \, $
   given by
   $ \, \sigma\,.\,\varphi(a,b,c) := \varphi\big(\sigma^{-1}\,.\,(a,b,c)\big) \, $,
   \,where the action of  $ \, \Bbbk\big[\mathbb{S}_3\big] \, $  on
   $ \, U_\hbar^{\widehat{\otimes} 3} \, $  is the natural one that permutes
   the tensor factors.  Then we let the antisymmetrizer  $ \textsl{Alt}_3 $
   act on both sides of  \eqref{eq: sigma-modh}:  using that
%%%
  $ \; \gamma \, := \, \chi_a \; \Big(\, \text{\rm mod} \; \hbar \,
  {\big( \uhg^{\widehat{\otimes}\, 2} \big)}^{\!*} \,\Big) \Big|_{\lieg \otimes \lieg} \; $
and that
  $ \; a \; \big(\hskip-7pt \mod \hbar\,U_\hbar \big) = \text{a} \; $,
  $ \; b \; \big(\hskip-7pt \mod \hbar\,U_\hbar \big) = \text{b} \; $
and
  $ \; c \; \big(\hskip-7pt \mod \hbar\,U_\hbar \big) = \text{c} \; $,
%%%%%%%%%
   \,a straightforward calculation eventually yields
\begin{equation*}   % \label{eq: 2-cocycle identity x gamma}
  \partial_*(\gamma) \, + \, \text{c.p.} \, + \, {[[ \gamma \, , \gamma \,]]}_*  \,\; = \;\,  0
\end{equation*}
 This means exactly that  $ \, \gamma \, $  is a 2-cocycle for the Lie bialgebra
 $ \lieg $   --- according to  Definition \ref{def: cocyc-deform_Lie-bialg's}  ---
 that is obviously antisymmetric (by construction), q.e.d.

 \vskip9pt
   \textit{(c)}\,  First of all, we start by noting that
   $ \, {\uhg}_\sigma := {\big( \uhg \big)}_\sigma \, $  is equal to  $ \uhg $
   as a counital  $ \kh $--coalgebra  (by construction), but with the new product defined by
  $$
  m_{\sigma}(a,b\,)  \, = \,  a \,\raisebox{-5pt}{$ \dot{\scriptstyle \sigma} $}\, b  \, =
  \,
  \sigma(a_{(1)},b_{(1)}) \, a_{(2)} \, b_{(2)} \, \sigma^{-1}(a_{(3)},b_{(3)})
  \eqno \forall \;\, a, b \in \uhg  \qquad
  $$
 In particular, the  $ \kh $--module  $ \, {\uhg}_\sigma = \uhg \, $
 is still topologically free, so that  $ {\uhg}_\sigma $  is again a Hopf algebra in
 $ \cT_\otimeshat \, $,  cf.\ \S \ref{prel_Q-Groups}.
 Moreover, its semiclassical limit
 $ \, \overline{{\uhg}_\sigma} := {\uhg}_\sigma \Big/ \hbar \, {\uhg}_\sigma \, $
 as a coalgebra is the same as that of  $ \uhg \, $;
 hence it is again cocommutative connected (as these properties are not affected
 by  $ 2 $--cocycle  deformations).  Thus by Milnor-Moore Theorem we have
 $ \, \overline{{\uhg}_\sigma} = U\big( \widehat{\lieg} \big) \, $,  where
 $ \, \widehat{\lieg} = \Prim\big(\,\overline{{\uhg}_\sigma}\,\big) \, $
 is the space of primitive elements in  $ \overline{{\uhg}_\sigma} \, $,
 and as such it coincides with
 $ \, \Prim\big(\,\overline{\uhg}\,\big) = \Prim\big( U(\lieg) \big) = \lieg \, $
 as a Lie coalgebra; its Lie algebra structure, on the other hand, does depend on
 $ \sigma \, $.  Altogether, this shows that  $ {\uhg}_\sigma $  is indeed a QUEA,
 whose semiclassical limit is  $ U\big( \widehat{\lieg} \big) \, $;
 then we are only left to prove that the Lie bracket on  $ \widehat{\lieg} $
 coincides with that of  $ \lieg_\gamma \, $,  \,while also proving that  $ \gamma $
 is an antisymmetric 2--cocycle for the Lie bialgebra  $ \lieg \, $.
                                                        \par
   The Lie bracket in  $ \widehat{\lieg} $  is given by the commutator inside
   $ \, U\big( \widehat{\lieg} \big) = \overline{{\uhg}_\sigma} \, $,
   so we denote it by
   $ \,\; {[\text{a}\,,\text{b}]}_\sigma = \text{a}
   \;\overline{\raisebox{-5pt}{$ \dot{\scriptstyle \,\sigma} $}}\;
   \text{b} - \text{b} \;\overline{\raisebox{-5pt}{$ \dot{\scriptstyle \,\sigma} $}}\;
   \text{a} \;\, $  (for all  $ \, \text{a} \, , \text{b} \in \lieg \, $),  where
   $ \;\overline{\raisebox{-5pt}{$ \dot{\scriptstyle \,\sigma} $}}\; $
   is the product in  $ \, U\big( \widehat{\lieg} \big) = \overline{{\uhg}_\sigma} \, $
   induced by the  ($ \sigma $--deformed)  product in  $ {\uhg}_\sigma \, $.
   Therefore, we will compute such a commutator as the coset modulo
   $ \, \hbar\,U_\hbar \, $  of a commutator in  $ U_\hbar \, $,  \,namely
 $ \,\; {[\text{a}\,,\text{b}]}_\sigma = \text{a}
 \;\overline{\raisebox{-5pt}{$ \dot{\scriptstyle \,\sigma} $}}\; \text{b} -
 \text{b} \;\overline{\raisebox{-5pt}{$ \dot{\scriptstyle \,\sigma} $}}\;
 \text{a} \; = \; a \,\raisebox{-5pt}{$ \dot{\scriptstyle \sigma} $}\, b - b \,
 \raisebox{-5pt}{$ \dot{\scriptstyle \sigma} $}\, a \;\; \big(\, \text{mod} \;
 \hbar\,U_\hbar \,\big) \; $,
 where  \textsl{$ a $  and  $ b \, $,  like in the proof of claim  \textit{(c)},
 are lifts of\/  $ \text{\rm a} $  and  $ \text{\rm b} $   ---
 i.e.,  $ \, a \; \big(\, \text{\rm mod}\; \hbar \, U_\hbar \,\big) = \text{\rm a} \, $
 and  $ \, b \; \big(\, \text{\rm mod}\, \hbar \, U_\hbar \,\big) = \text{\rm b} \, $
 ---   such that  $ \; a^{\,\prime} := \hbar \, a \, \in J_\hbar^{\,\prime} \; $  and
 $ \; b^{\,\prime} := \hbar \, b \, \in J_\hbar^{\,\prime} \; $.}
 \vskip7pt
   We re-start back from  \eqref{eq: expans-sigmaprod},  which now gives
   (taking into account all the analysis carried out there, with  $ \, A = 1 = B \, $)
\begin{align*}
     &  a \,\raisebox{-5pt}{$ \dot{\scriptstyle \sigma} $}\, b - b
     \,\raisebox{-5pt}{$ \dot{\scriptstyle \sigma} $}\, a
     \;\; \underset{\hbar}{\equiv} \;\,  a \cdot b  \; - \;  b \cdot a  \,\; +  \\
     &  \hskip3pt   + \,  \hbar^{-3} \,
     \Big(\, \chi\big(a^{\prime+}_{(1)},b^{\prime+}_{(1)}\big)
     \big(\, \widehat{a}^{\,\prime}_{(2)} \, b^{\prime+}_{(2)} + a^{\prime+}_{(2)} \,
     \widehat{b}^{\,\prime}_{(2)} \big) \,
     - \, \chi\big(b^{\prime+}_{(1)},a^{\prime+}_{(1)}\big)
     \big(\, \widehat{b}^{\,\prime}_{(2)} \, a^{\prime+}_{(2)} + b^{\prime+}_{(2)} \,
     \widehat{a}^{\,\prime}_{(2)} \big) \!\Big)  \; -  \\
     &  \hskip6pt   - \,  \hbar^{-3} \, \Big(\! \big(\, \widehat{a}^{\,\prime}_{(1)} \,
     b^{\prime+}_{(1)} + a^{\prime+}_{(1)} \,
     \widehat{b}^{\,\prime}_{(1)} \big) \,
     \chi\big(a^{\prime+}_{(2)},b^{\prime+}_{(2)}\big) \, - \,
     \big(\, \widehat{b}^{\,\prime}_{(1)}
     \, a^{\prime+}_{(1)} + b^{\prime+}_{(1)} \, \widehat{a}^{\,\prime}_{(1)} \big) \,
     \chi\big(b^{\prime+}_{(2)},a^{\prime+}_{(2)}\big) \!\Big)
\end{align*}
 Second, letting
 $ \; \chi_a := \chi - \chi_{\,2,1} \; $,
 \,the previous formula greatly simplifies into
\begin{align*}
     &  a \,\raisebox{-5pt}{$ \dot{\scriptstyle \sigma} $}\, b - b \,
     \raisebox{-5pt}{$ \dot{\scriptstyle \sigma} $}\, a  \;\; =
     \;\,  a \cdot b  \, - \,  b \cdot a  \; + \;
     \hbar^{-3} \, \chi_a\big(a^{\prime+}_{(1)},b^{\prime+}_{(1)}\big)
     \big(\, \widehat{a}^{\,\prime}_{(2)} \, b^{\prime+}_{(2)} \, + \, a^{\prime+}_{(2)} \,
     \widehat{b}^{\,\prime}_{(2)} \big)  \,\; +  \\
      &  \hskip113pt   + \;\,  \hbar^{-3} \, \big(\, \widehat{b}^{\,\prime}_{(1)}
 \, a^{\prime+}_{(1)} + b^{\prime+}_{(1)} \, \widehat{a}^{\,\prime}_{(1)} \big) \,
 \chi_a\big(b^{\prime+}_{(2)},a^{\prime+}_{(2)}\big)  \; + \;  \cO(\hbar)  \,\; =  \\
% %
%      &  \hskip15pt   = \;\,  a \cdot b  \, - \,  b \cdot a  \; + \;  \hbar^{-3}
% \, \Big(\, \chi_a\big(a',b^{\prime+}_{(1)}\big) \, b^{\prime+}_{(2)} \, + \,
% \chi_a\big(a^{\prime+}_{(1)},b'\big) \, a^{\prime+}_{(2)} \,\Big)  \,\; +  \\
%      &  \hskip111pt   + \,  \hbar^{-3} \, \Big(\, a^{\prime+}_{(1)} \,
% \chi_a\big(b',a^{\prime+}_{(2)}\big) \, + \, b^{\prime+}_{(1)} \,
% \chi_a\big(b^{\prime+}_{(2)},a'\big) \Big)  \; + \;  \cO(\hbar)  \,\; =  \\
%%%%%
%
     &  \hskip30pt   = \;\,  a \cdot b  \, - \,  b \cdot a  \; + \;
     \hbar^{-3} \, \Big(\, \chi_a\big(a^{\prime+}_{(1)},b'\big) \, a^{\prime+}_{(2)} \, - \,
     \chi_a\big(a^{\prime+}_{(2)},b'\big) \, a^{\prime+}_{(1)} \,\Big)  \,\; -  \\
     &  \hskip126pt   - \,  \hbar^{-3} \, \Big(\, \chi_a\big(b^{\prime+}_{(1)},a'\big) \,
     b^{\prime+}_{(2)} \, - \,
     \chi_a\big(b^{\prime+}_{(2)},a'\big) \, b^{\prime+}_{(1)} \Big)  \; + \;
     \cO_{1\!}(\hbar)
\end{align*}
   \indent   Now, let us write  $ \; z' \, = \, \hbar\,z \; $  for all
   $ \; z \, \in \, \{ a \, , b \, \} \; $:  \,then the last formula turns into
\begin{equation}  \label{eq: sigma-comm(a,b)}
  \begin{aligned}
     a \,\raisebox{-5pt}{$ \dot{\scriptstyle \sigma} $}\, b -
     b \,\raisebox{-5pt}{$ \dot{\scriptstyle \sigma} $}\, a  \;\; = \;\,
     a \cdot b  \;  &
     - \;  b \cdot a  \,\; + \; \hbar^{-1} \Big(\, \chi_a\big(a^+_{(1)},b\,\big) \,
     a^+_{(2)} \, - \, \chi_a\big(a^+_{(2)},b\,\big) \, a^+_{(1)} \Big)  \,\; -  \\
     &  - \;  \hbar^{-1} \Big(\, \chi_a\big(b^+_{(1)},a\,\big) \, b^+_{(2)} \, - \,
     \chi_a\big(b^+_{(2)},a\,\big) \, b^+_{(1)} \Big)  \; + \;  \cO_{1\!}(\hbar)
  \end{aligned}
\end{equation}
 Here we recall that, working with a QUEA, for  $ \, c \in \{a\,,b\,\} \, $
 we have
   $$
   \Delta(c\,)  \,\; = \;\,  c \otimes 1 \, + \, 1 \otimes c \, +
   \, c^+_{(1)} \otimes c^+_{(2)} \, + \, \cO_{2\!}\big(\hbar^2\big) \;\; ,
   \qquad   c^+_{(1)} \otimes c^+_{(2)} \, \in \, \hbar \, U_h^{\,\widehat{\otimes} 2}
   $$
and moreover   --- for every  $ \, c \in \{a\,,b\,\} \, $  and
$ \, \text{c} \in \{\text{a}\,,\text{b}\} \, $,
\,so that  $ c $  is a lift of  $ \text{c} $  ---
\begin{equation}  \label{eq: delta_a&b}
  \overline{\hbar^{-1}\big(c^+_{(1)} \! \otimes c^+_{(2)} - c^+_{(2)} \!
  \otimes c^+_{(1)}\big)}  \; = \;  \delta(\text{c})  \; =: \;
  \text{c}_{[1]} \otimes \text{c}_{[2]}
\end{equation}
 where hereafter any ``overlined'' object stands for
 ``its coset modulo  $ \hbar \, $'';
 \,in addition, we recall also that  $ \chi_a $  is antisymmetric.
 Then  \eqref{eq: sigma-comm(a,b)}
 and  \eqref{eq: delta_a&b}  altogether yield
\begin{small}
 \begin{align*}
     &  {[\text{a}\,,\text{b}]}_\sigma = \text{a} \;\overline{\raisebox{-5pt}{$
     \dot{\scriptstyle \,\sigma} $}}\;
     \text{b} - \text{b} \;\overline{\raisebox{-5pt}{$ \dot{\scriptstyle \,\sigma} $}}\;
     \text{a} \; = \;
     \overline{a \,\raisebox{-5pt}{$ \dot{\scriptstyle \sigma} $}\, b - b
     \,\raisebox{-5pt}{$ \dot{\scriptstyle \sigma} $}\, a}  \,\; =  \\
%%%
    &  \hskip35pt   = \,  \overline{a} \cdot \overline{b} -
    \overline{b} \cdot \overline{a}  \,\; + \;
    \overline{\hbar^{-1} \Big(\, \chi_a\big(a^+_{(1)},b\,\big) \, a^+_{(2)} \, - \,
    \chi_a\big(a^+_{(2)},b\,\big) \, a^+_{(1)} \Big)}  \,\; -  \\
     &  \hskip70pt   - \;  \overline{\hbar^{-1} \Big(\, \chi_a\big(b^+_{(1)},a\,\big) \,
     b^+_{(2)} \, - \,
     \chi_a\big(b^+_{(2)},a\,\big) \, b^+_{(1)} \Big)}  \,\; =  \\
%%%
     &  \hskip105pt   = \;\,  [\text{a}\,,\text{b}] \, + \,
     \gamma\big(\text{a}_{[1]},\text{b}\big) \, \text{a}_{[2]} \, - \,
     \gamma\big(\text{b}_{[1]},\text{a}\big) \, \text{b}_{[2]}  \,\; =  \\
     &
 \hskip140pt   = \;\,  [\text{a}\,,\text{b}] \, - \,
 \gamma\big(\text{a}_{[2]},\text{b}\big) \, \text{a}_{[1]} \, - \,
     \gamma\big(\text{b}_{[1]},\text{a}\big) \, \text{b}_{[2]}  \,\; =: \;\,
     {[\text{a}\,,\text{b}]}_\gamma
 \end{align*}
\end{small}
 hence  $ \; {[\text{a}\,,\text{b}]}_\sigma = {[\text{a}\,,\text{b}]}_\gamma \; $
 for all  $ \, \text{a} \, , \text{b} \in \lieg \, $,  \,in the sense of  \eqref{eq: def_cocyc-bracket},
 and we are done.
 \vskip7pt
   Finally, if in particular  $ \sigma $  is\/  $ \kh $--valued
   --- i.e., it is an ordinary 2--cocycle for the Hopf\/  $ \kh $--algebra  $ \uhg \, $  ---
   then we have
   $ \; \chi \, = \, \hbar \, \log_*(\sigma) \, \in \, \hbar \,
   {\big( \uhg^{\widehat{\otimes}\,2} \big)}^{\!*} \, $,  \,hence we have just
   $ \, \gamma = 0 \, $  and
   $ \; {\big( \uhg \big)}_\sigma \cong \, U_\hbar\big(\lieg_\gamma\big) = \uhg \; $.
\epf

\vskip5pt

\begin{exa}  \label{example: qs-2-coc x QUEA}
 Some concrete examples of polar 2--cocycles and deformation by them are treated
 in full depth in \cite[Section 5.2]{GaGa2}:  they concern the wide family of
 \textit{formal multiparameter QUEAs\/}  that we already treated in
 Example \ref{example: toral twists for FoMpQUEAs}.
 We then resume notations and formulas from there.
                                                \par
   Fix  $ \, n \in \NN_+ \, $  and  $ \, I := \{1,\ldots,n\} \, $.
   We choose a multiparameter matrix
   $ \, P := {\big(\, p_{i,j} \big)}_{i, j \in I} \in M_n\big( \kh \big) \, $  of Cartan type,
   with associated Cartan matrix  $ A \, $,  a realization
   $ \; \cR \, := \, \big(\, \lieh \, , \Pi \, , \Pi^\vee \,\big) \; $  of it and the
 (topological) Hopf algebra  $ \uRPhg \, $.
 Let  $ \, {\big\{ H_g \big\}}_{g \in \cG} \, $ be a  $ \kh $--basis in $ \lieh $,
 where  $ \cG $  is an index set with  $ \, |\cG| = \rk(\lieh) = t \, $.
                                                \par
   We consider special polar 2--cocycles of  $ \, \uRPhg \, $,  called ``toral'' as they are induced from the quantum torus.  Fix an antisymmetric,  $ \kh $--bilinear  map  $ \, \chi : \lieh \times \lieh \relbar\joinrel\longrightarrow \kh \, $,  \,that corresponds to  $ \, X = {\big(\, \chi_{g{}\gamma} = \chi(H_g\,,H_\gamma) \big)}_{g , \gamma \in \cG} \in \lieso_t\big(\kh\big) \, $.  Any such map  $ \chi $  also induces uniquely an antisymmetric,  $ \kh $--bilinear  map
  $$  \tilde{\chi}_{\scriptscriptstyle U} \, : \, U_{\!P,\hbar}^{\,\cR}(\lieh) \times U_{\!P,\hbar}^{\,\cR}(\lieh) \relbar\joinrel\relbar\joinrel\relbar\joinrel\longrightarrow \kh  $$
 as follows.  By definition,  $ U_{\!P,\hbar}^{\,\cR}(\lieh) $  is an  $ \hbar $--adically
 complete topologically free Hopf algebra isomorphic to
 $ \, \widehat{S}_\kh(\lieh) := \widehat{\bigoplus\limits_{n \in \NN} S_\kh^{\,n}(\lieh)} \, $,  \,the
 $ \hbar $--adic  completion of the symmetric algebra
 $ \, S_\kh(\lieh) = \bigoplus\limits_{n \in \NN} S_\kh^{\,n}(\lieh) \, $.  Then,
 $ \tilde{\chi}_{\scriptscriptstyle U} $ is defined as the unique  $ \kh $--linear
 (hence  $ \hbar $--adically  continuous) map
 $ \; U_{\!P,\hbar}^{\,\cR}(\lieh) \otimes U_{\!P,\hbar}^{\,\cR}(\lieh)
 \,{\buildrel \tilde{\chi}_{\scriptscriptstyle U} \over {\relbar\joinrel\relbar\joinrel\longrightarrow}}\, \kh \; $
 such that
%
% (with identifications as above)
%
\begin{equation}  \label{eq: chitilde}
  \begin{aligned}
    \tilde{\chi}_{\scriptscriptstyle U}(z,1) := \epsilon(z) =: \tilde{\chi}_{\scriptscriptstyle U}(1,z)
\qquad \qquad \qquad  \forall \;\; z \in \widehat{S}_\kh(\lieh)   \hskip45pt \qquad  \\
    \tilde{\chi}_{\scriptscriptstyle U}(x,y) := \chi(x,y)   \qquad \hskip55pt \qquad  \forall \;\; x, y \in  S_\kh^{\,1}(\lieh)   \qquad \hskip39pt  \\
   \hfill \qquad  \tilde{\chi}_{\scriptscriptstyle U}(x,y) := 0   \quad \qquad   \forall \;\; x \in S_\k^{\,r}(\lieh) \, ,   \; y \in S_\k^{\,s}(\lieh) \, : \, r, s \geq 1 \, , \; r+s > 2   \hskip5pt \hfill
  \end{aligned}
\end{equation}
 By construction,  \textsl{$ \tilde{\chi}_{\scriptscriptstyle U} $  is a normalized Hochschild 2--cocycle on  $ U_{\!P,\hbar}^{\,\cR}(\lieh) \, $},  \,that is
  $$  \epsilon(x) \,\tilde{\chi}_{\scriptscriptstyle U}(y,z) \, - \, \tilde{\chi}_{\scriptscriptstyle U}(xy,z) \, + \, \tilde{\chi}_{\scriptscriptstyle U}(x,yz) \, - \, \tilde{\chi}_{\scriptscriptstyle U}(x,y) \, \epsilon(z) \; = \; 0   \qquad  \forall \;\; x, y, z \in U_{\!P,\hbar}^{\,\cR}(\lieh)  $$
 By \cite[Lemma 5.2.3]{GaGa2}, the convolution powers of  $ \tilde{\chi}_{\scriptscriptstyle U} $  satisfy the following property: for all  $ \, H_+ ,H_- \in \lieh \, $  and  $ \, k, \ell, m \in \NN_+ \, $,  \,we have
  $$  { \tilde{\chi}_{\scriptscriptstyle U}}^{\,\ast\,m}\big(H_+^{\,k},H_-^{\,\ell}\,\big) \; = \;
   \begin{cases}
      \begin{array}{lr}
  \delta_{k,m} \, \delta_{\ell,m} \, {\big(m!\big)}^2 \, \chi(H_+,H_-)^m  &  \quad  \text{for} \;\;\;\; m \geq 1 \; , \\
  \delta_{k,0} \, \delta_{\ell,0} \,  &  \quad  \text{for} \;\;\;\; m = 0 \; .
      \end{array}
   \end{cases}  $$
 This allows one to define a polar 2--cocycle  $ \, \chi_{\scriptscriptstyle U} \, $  as the unique  $ \kh $--linear
%
% (hence  $ \hbar $--adically  continuous)
%
 map from  $ \, U_{\!P,\hbar}^{\,\cR}(\lieh) \!\mathop{\otimes}\limits_\kh\!
 U_{\!P,\hbar}^{\,\cR}(\lieh) \, $  to  $ \khp $  given by the exponentiation of
 $ \, \hbar^{-1} \, 2^{-1} \, \tilde{\chi}_{\scriptscriptstyle U} \, $,  i.e.
 \vskip-5pt
  $$  \chi_{\scriptscriptstyle U}  \; := \;  e^{\hbar^{-1} 2^{-1} \tilde{\chi}_{\scriptscriptstyle U}}  \;
  = \;  {\textstyle \sum_{m \geq 0}} \, \hbar^{-m} \, {\tilde{\chi}_{\scriptscriptstyle U}}^{\,\ast\,m}
  \! \Big/ 2^m\,m!  $$
 By  \cite[Lemma 5.2.2]{GaGa2},  this  $ \, \chi_{\scriptscriptstyle U} $  is in fact a well-defined  \textsl{polar 2--cocycle\/}  for  $ \, U_{\!P,\hbar}^{\,\cR}(\lieh) \, $,  in the sense of  Definition \ref{def: polar 2cocycle}.  Moreover, one has, for all  $ H_+ \, , H_- \in \lieh \, $,  and setting  $ \, K_\pm := e^{\,\hbar\,H_\pm} \, $,
  $$  \chi_{\scriptscriptstyle U}^{\pm 1}(H_+,H_-) \, = \, \pm \hbar^{-1} \, 2^{-1}  \chi(H_+,H_-)
  \quad ,
   \qquad  \chi_{\scriptscriptstyle U}(K_+,K_-) \, = \, e^{\hbar \, 2^{-1} \chi(H_+,H_-)}  $$
 Assume now that  $ \, \chi $  satisfies the additional requirement
  $ \; \chi(S_i \, ,\,-\,) \, = \, 0 \, = \, \chi(\,-\,,S_i) \; $  for all  $ \, i \in I \, $,  \,
 where  $ \, S_i := 2^{-1} \big(\, T^+_i + T^-_i \big) \, $  for all  $ \, i \in I \, $.
 In particular, one has that
  $ \; \chi\big(\, T_i^+ \, , \, T \,\big) \, = \, \chi\big( -T_i^- \, , \, T \,\big) \; $  and
  $ \; \chi\big(\, T \, , T_i^+ \big) \, = \, \chi\big(\, T \, , -T_i^- \big) \; $
 for all  $ \, i \in I \, $  and  $ \, T \in \lieh \, $.  Then  $ \chi $  canonically induces a
 $ \kh $--bilinear  map
 $ \; \overline{\chi} : \overline{\lieh} \times \overline{\lieh} \relbar\joinrel\longrightarrow \kh \; $,
 \;where
$ \, \overline{\lieh} := \lieh \big/ \lies \, $  with
$ \, \lies := \textsl{Span}_\kh\big( {\{\, S_i \,\}}_{i \in I} \big)\, $,  \,given by
  $$
  \overline{\chi}\,\big(\, T' \! + \lies \, , T'' \! + \lies \,\big) \, := \, \chi\big(\,T',T''\big)
  \qquad  \forall \;\; T' , T'' \in \lieh
  $$
   \indent   Now, replaying the construction above with  $ \overline{\lieh} $  and
   $ \overline{\chi} $  replacing  $ \lieh $  and  $ \chi \, $,
   we can construct a normalized Hopf polar  $ 2 $--cocycle
   $ \, \overline{\chi}_{\scriptscriptstyle U} \!: U_{\!P,\hbar}^{\,\cR}\big(\,\overline{\lieh}\, \big)
   \times U_{\!P,\hbar}^{\,\cR}\big(\,\overline{\lieh}\,\big) \!\relbar\joinrel\longrightarrow \khp \, $.
   Moreover, since
   $ \, U_{\!P,\hbar}^{\,\cR}\big(\,\overline{\lieh}\,\big) \cong
   \widehat{S}_\kh\big(\,\overline{\lieh}\,\big) \, $,
   there exists a unique Hopf algebra epimorphism
   $ \; \pi : \uRPhg \relbar\joinrel\relbar\joinrel\twoheadrightarrow
   U_{\!P,\hbar}^{\,\cR}\big(\,\overline{\lieh}\,\big) \; $  given by
   $ \, \pi(E_i) := 0 \, $,  $ \, \pi(F_i) := 0 \, $   --- for  $ \, i \in I \, $  ---  and
   $ \, \pi(T) := (T + \lies) \, \in \,
   \overline{\lieh} \subseteq U_{\!P,\hbar}^{\,\cR}\big(\,\overline{\lieh}\,\big) \, $
   --- for  $ \, T \in \lieh \, $.  Then we consider
  $$
  \sigma_\chi \, := \, \overline{\chi}_{\scriptscriptstyle U} \circ (\pi \times \pi) \, : \,
  \uRPhg \times \uRPhg
  \relbar\joinrel\relbar\joinrel\relbar\joinrel\relbar\joinrel\twoheadrightarrow \khp
  $$
 which is  \textsl{automatically\/}  a normalized,  $ \khp $--valued
 Hopf polar  $ 2 $--cocycle  on  $ \uRPhg \, $.
                                                \par
   By Theorem \ref{thm: propt.'s qs-2cocycle-deform-QUEA},  one may define a
   ``deformed product'' on $ \uRPhg $  using  $ \sigma_\chi $  hereafter denoted by
   $ \, \scriptstyle \dot\sigma_\chi \, $.  Write
   $ \, X^{(n)_{\sigma_{\chi}}} = X\raise-1pt\hbox{$ \, \scriptstyle \dot\sigma_\chi $}
   \cdots \raise-1pt\hbox{$ \,\scriptstyle \dot\sigma_\chi $} X \, $  for the  $ n $--th
   power of any  $ \, X \in \uRPhg \, $  with respect to this deformed product.
                                                \par
   Directly from definitions, sheer computation yields the following formulas,
   relating the deformed product with the old one (for all  $ \; T', T'', T \in \lieh \; $,
   $ \; i \, , \, j \in I \, $):
  $$
  \displaylines{
   T' \raise-1pt\hbox{$ \, \scriptstyle \dot\sigma_\chi $} T''  \, = \;  T' \, T''  \quad ,
     \qquad  E_i \raise-1pt\hbox{$ \, \scriptstyle \dot\sigma_\chi $} F_j  \; = \,
     E_i \, F_j  \quad ,
     \qquad  F_j \raise-1pt\hbox{$ \, \scriptstyle \dot\sigma_\chi $} E_i  \; = \,  F_j \, E_i  \cr
   T \raise-1pt\hbox{$ \, \scriptstyle \dot\sigma_\chi $}  E_j  \; = \;
   T \, E_j \, + \, 2^{-1} \chi\big(T,T_j^+\big) \, E_j \quad ,
     \quad  E_j \raise-1pt\hbox{$ \, \scriptstyle \dot\sigma_\chi $} T  \; = \;
     E_j \, T \, + \, 2^{-1} \chi\big(T_j^+,T\big)\big)  \, E_j  \cr
   T \raise-1pt\hbox{$ \, \scriptstyle \dot\sigma_\chi $} F_j  \; = \;
   T \, F_j \, + \, 2^{-1} \chi\big(T,T_j^-\big)  \, F_j  \quad ,
     \quad  F_j \raise-1pt\hbox{$ \, \scriptstyle \dot\sigma_\chi $} T  \; = \;
     F_j \, T \, + \, 2^{-1} \chi\big(T_j^-,T\big) \, F_j  \cr
%%%%%
   E_i^{\,(m)_{\sigma_{\chi}}}  \, = \;  {\textstyle \prod_{\ell=1}^{m-1}}
   \sigma_\chi\Big( e^{\,+\hbar\,\ell\,T_i^+} \! , \, e^{\,+\hbar\,T_i^+} \Big) \, E_i^{\,m}  \; = \;
   E_i^{\,m}  \cr
   E_i^{\,m} \raise-1pt\hbox{$ \, \scriptstyle \dot\sigma_\chi $} E_j^{\,n}  \; = \;
   \sigma_\chi\Big(
e^{\,+\hbar\,m\,T_i^+} \! , \, e^{\,+\hbar\,n\,T_j^+} \Big) \, E_i^{\,m} \, E_j^{\,n}  \; = \;
e^{\, +\hbar \,m\,n\, 2^{-1} \mathring{\chi}_{ij} } E_i^{\,m} \, E_j^{\,n}  \cr
%
%%%%%%%%%
%    \qquad   E_i^{\,(m)_{\sigma_{\chi}}} \raise-1pt\hbox{$ \, \scriptstyle \dot\sigma_\chi $}\, E_j
% \,\raise-1pt\hbox{$ \, \scriptstyle \dot\sigma_\chi $}\, E_k^{\,(n)_{\sigma_\chi}}  \, =   \hfill  \cr
%    = \;  \Big( {\textstyle \prod_{\ell=1}^{m-1}} \sigma_\chi\Big( e^{\,+\hbar\,\ell\,T_i^+} \! ,
%  \, e^{\,+\hbar\,T_i^+} \Big) \Big) \Big( {\textstyle \prod_{t=1}^{n-1}}
% \sigma_\chi\Big( e^{\,+\hbar\,t\,T_k^+} \! , \, e^{\,+\hbar\,T_k^+} \Big) \Big) \, \cdot  \cr
%   \hfill \cdot \; \sigma_{\chi}\Big( e^{\,+\hbar\,m\,T_i^+} \! , \, e^{\,+\hbar\,T_j^+} \Big) \, \sigma_\chi\Big( e^{\,+\hbar\,(m\,T_i^+ + T_j^+)} \! , \, e^{\,+\hbar\,n\,T_k^+} \Big) \, E_i^{\,m} \, E_j \, E_k^{\,n}   \quad  \cr
%%%%%%%%%
%
   E_i^{\,(m)_{\sigma_{\chi}}} \raise-1pt\hbox{$ \, \scriptstyle \dot\sigma_\chi $}\,
   E_j \,\raise-1pt\hbox{$ \, \scriptstyle \dot\sigma_\chi $}\, E_k^{\,(n)_{\sigma_\chi}}  = \,
   \Big( {\textstyle \prod_{\ell=1}^{m-1}} \sigma_\chi\Big( e^{\,+\hbar\,\ell\,T_i^+} \! ,
   \, e^{\,+\hbar\,T_i^+} \Big) \!\Big) \Big( {\textstyle \prod_{t=1}^{n-1}}
   \sigma_\chi\Big( e^{\,+\hbar\,t\,T_k^+} \! , \, e^{\,+\hbar\,T_k^+} \Big) \!\Big) \, \cdot
   \hfill  \cr
  \hfill \cdot \; \sigma_{\chi}\Big( e^{\,+\hbar\,m\,T_i^+} \! , \, e^{\,+\hbar\,T_j^+} \Big) \,
  \sigma_\chi\Big( e^{\,+\hbar\,(m\,T_i^+ + T_j^+)} \! , \, e^{\,+\hbar\,n\,T_k^+} \Big)
  \, E_i^{\,m} \, E_j \, E_k^{\,n}  \cr
%%%%%%%%%
   F_i^{\,(m)_{\sigma_\chi}}  \, = \;  {\textstyle \prod_{\ell=1}^{m-1}}
   \sigma_\chi^{-1}\Big( e^{\,-\hbar\,\ell\,T_i^-} \! , \,
   e^{\,-\hbar\,T_i^-} \Big) \, F_i^{\,m}  \; = \;  F_i^{\,m}  \cr
%%%
   F_i^{\,m} \raise-1pt\hbox{$ \, \scriptstyle \dot\sigma_\chi $} F_j^{\,n}  \; = \;
   \sigma_\chi^{-1}\Big( e^{\,-\hbar\,m\,T_i^-} \! , \,
   e^{\,-\hbar\,n\,T_j^-} \Big) \, F_i^{\,m} \, F_j^{\,n}  \; = \;
   e^{\, -\hbar \,m\,n\, 2^{-1} \mathring{\chi}_{ij}  } F_i^{\,m} \, F_j^{\,n}  \cr
%%%%%%%%%
   F_i^{\,(m)_{\sigma_{\chi}}} \raise-1pt\hbox{$ \, \scriptstyle \dot\sigma_\chi $}\,
   F_j \,\raise-1pt\hbox{$ \, \scriptstyle \dot\sigma_\chi $}\, F_k^{\,(n)_{\sigma_{\chi}}}  =
   \Big( {\textstyle \prod_{\ell=1}^{m-1}} \sigma_{\chi}^{-1}\Big( e^{\, -\hbar\,\ell\,T_i^-} \! , \,
   e^{\,-\hbar\,T_i^-} \Big) \!\Big) \Big( {\textstyle \prod_{t=1}^{n-1}}
   \sigma_{\chi}^{-1}\Big( e^{\,-\hbar\,t\,T_k^-} \! , \, e^{\,-\hbar\,T_k^-} \Big) \!\Big) \cdot
   \hfill  \cr
   \hfill \cdot \; \sigma_{\chi}^{-1}\Big( e^{\,-\hbar\,m\,T_i^-} \! , \, e^{\,-\hbar\,T_j^-} \Big) \,
   \sigma_{\chi}^{-1}\Big( e^{\,-\hbar\,(m\,T_i^- + T_j^-)} \! , \,
   e^{\,-\hbar\,n\,T_k^-} \Big) \, F_i^{\,m} \, F_j \, F_k^{\,n}  }
   $$
 Fix now  $ \, \mathring{X} := \! {\Big( \mathring{\chi}_{i{}j} =
 \chi\big(\,T_i^+,T_j^+\big) \!\Big)}_{i, j \in I} \, $  and define the multiparameter matrix
  $$  P_{(\chi)}  := \,  P \, + \, \mathring{X}  \, = \,  {\Big(\, p^{(\chi)}_{i{}j} := \,  p_{ij} + \mathring{\chi}_{i{}j} \Big)}_{\! i, j \in I}  \;\, ,   \!\quad
  \Pi_{(\chi)}  := \,  {\Big\{\, \alpha_i^{(\chi)}  := \,  \alpha_i \pm \chi\big(\, \text{--} \, , T_i^\pm \big)  \Big\}}_{i \in I}  $$
\noindent
 It turns out that  $ P_{(\chi)} $  is a matrix of Cartan type   --- the same of  $ P $  indeed ---
 and  $ \, \cR_{(\chi)} \, = \, \big(\, \lieh \, , \Pi_{(\chi)} \, , \Pi^\vee \,\big) \, $
 is a realization of it. Moreover, by \cite[Theorem 5.2.12]{GaGa2},
 there exists an isomorphism of topological Hopf algebras
  $$  {\big( \uRPhg \big)}_{\sigma_\chi}  \; \cong \;\,  U_{\!P_{(\chi)},\,\hbar}^{\,\cR_{(\chi)}}(\lieg)  $$
 which is the identity on generators.
 In short, every toral polar 2--cocycle deformation of a FoMpQUEA is another FoMpQUEA,
 whose multiparameter  $ \, P_{(\chi)} $  and realization  $ \cR_{(\chi)} $
 depend on the original  $ P $  and  $ \cR \, $,  as well as on  $ \chi \, $.
 Moreover, under some mild restrictions on the realizations, one proves that the FoMpQUEA  $ \uRPhg $  is isomorphic to a toral polar 2--cocycle deformation of the Drinfeld's standard double QUEA, see  \cite[Theorem 5.2.14]{GaGa2}.
                                                \par
   About the semiclassical limit, we have the following.  Taking everything modulo  $ \hbar \, $,
   the map  $ \, \chi : \lieh \times \lieh \relbar\joinrel\longrightarrow \kh \,$
   defines a similar antisymmetric,  $ \Bbbk $--bilinear  map
   $ \; \gamma := \big(\, \chi \!\mod \hbar \,\big) :
   \lieh_0 \times \lieh_0 \relbar\joinrel\longrightarrow \Bbbk \; $   --- where
   $ \, \lieh_0 := \lieh \Big/ \hbar \, \lieh \, = \, \overline{\lieh} \, $.
   Out of  $ \gamma $  one constructs a toral  $ 2 $--cocycle  $ \gamma_\lieg $
   for the Lie bialgebra  $ \lieg^{\bar{\Rpicc}}_{\bar{\Ppicc}} \, $,  \,and out of it the
   $ 2 $--cocycle  deformed Lie bialgebra
   $ \, {\big( \lieg^{\bar{\Rpicc}}_{\bar{\Ppicc}} \big)}_{\gamma\raisebox{-1pt} {$ {}_\lieg $}} \, $.
   Similarly as above, out of  $ \gamma $  we get the multiparameter matrix
   $ P_{(\gamma)} $  and its realization  $ \, \cR_{(\gamma)} \, $:
   then by construction  $ \, P_{(\gamma)} = \bar{P}_{(\chi)} \, $  and
   $ \, \cR_{(\gamma)} = \bar{\cR}_{(\chi)} \, $.  Attached to these we have
   $ \, U_{\!P_{(\chi)},\,\hbar}^{\,\cR_{(\chi)}}(\lieg) \, $  and
   $ \, \lieg_{\Ppicc_{(\gamma)}}^{\Rpicc_{(\gamma)}} =
   \lieg_{\bar{\Ppicc}_{(\chi)}}^{\bar{\Rpicc}_{(\chi)}} \, $,
   \,again connected via quantization/specialization, and
   $ \, \lieg_{\Ppicc_{(\gamma)}}^{\Rpicc_{(\gamma)}} \cong \,
   {\big( \lieg^{\bar{\Rpicc}}_{\bar{\Ppicc}} \big)}_{\gamma\raisebox{-1pt} {$ {}_\lieg $}} \,$
   as Lie bialgebras.  Actually, one has that
   \textsl{``deformation by (polar) 2--cocycle commutes with specialization''},
see \cite[Theorem 6.2.4]{GaGa2}:  with assumptions as above, we have that
$ {\big(\, \uRPhg \big)}_{\!\sigma_\chi} $ is a quanti\-zed universal enveloping algebra,
with semiclassical limit
$ \, U\Big(\! { {\big( \lieg^{\bar{\Rpicc}}_{\bar{\Ppicc}} \big)}_{\gamma\raisebox{-1pt}
{$ {}_\lieg $}}} \Big) \, \cong \, U\Big( \lieg_{\Ppicc_{(\gamma)}}^{\Rpicc_{(\gamma)}} \Big) \, $.
\end{exa}

\vskip5pt

\begin{rmk}  \label{rmk: polar 2-coc as formal q-2-coc}
 It is important to stress that  \textsl{our notion of  \textit{polar 2--cocycle}  did\/  \textrm{not}
 come out of the blue},  but rather was suggested by the previous example.
 Indeed, the authors first ``met'' these objects when studying polynomial-type QUEAs
 $ U_q(\lieg) $   --- i.e., QUEAs ``\`a la Jimbo-Lusztig'', defined over
 $ \Bbbk\big[q\,,q^{-1}\big] \, $:  these are standard Hopf algebras (no topology is involved),
 to which one can apply deformation by 2--cocycles and then obtain some
 ``multiparameter QUEAs''   --- cf.\  \cite{GaGa1}, \S 4.2.  Every such  \textsl{polynomial}
 $ U_q(\lieg) $  can be realized as a Hopf subalgebra of a  \textsl{formal}  $ \uhg \, $,
 hence it makes sense to try and extend the 2--cocycle and the associated deformation
 procedure used for  $ U_q(\lieg) $  to the larger Hopf algebra  $ \uhg \, $.
 When we fulfilled this task   --- in  \cite{GaGa2}  ---
 what we actually found was that the unique extension of the 2--cocycle of  $ U_q(\lieg) $  to
 $ \uhg $  actually is a  \textsl{polar\/}  2--cocycle  (and not a 2--cocycle any more),
 yet despite this the deformation procedure does extend from  $ U_q(\lieg) $  to the whole
 $ \uhg \, $.  Thus the very notion of ``polar 2--cocycle'' and the associated deformation
 procedure showed up as something real from this concrete example.
\end{rmk}

\vskip13pt

\subsection{Deformations by polar twist of QFSHA's}  \label{subsec: twist-QFSHAs}  {\ }
 \vskip7pt
   In this subsection we consider deformations by twist of QFSHA's,
   but again ``stretching the standard recipe'',
   much like in  \S \ref{subsec: 2coc-QUEAs}:
   in fact, rather than twists in the usual sense we
   consider some special twist elements belonging to the scalar extension from
   $ \kh $  to  $ \khp $  of our QFSHA.
   For these elements --- that we call ``polar twists'' ---   nothing ensures
   \textit{a priori\/}  that the deformation recipe would properly work on the given
   QFSHA   --- as it is defined on  $ \kh \, $;  \,nevertheless, we eventually
   find that this is indeed the case.  In other words, we prove
   --- with a parallel result to  Theorem \ref{thm: polar 2cocycle-deform-QUEA}
   ---   that  \textit{the standard procedure of deformation by twist for QFSHA's
   can be extended (beyond its natural borders) to the case of polar twist elements}.
 \vskip5pt
   We begin with a couple of technical lemmas:

\vskip11pt

\begin{lema}  \label{lemma: adj-act_F}
 Let  $ \fhg $  be a QFSHA, and  $ \, {\fhg}^\vee $
 the associated QUEA defined in  \S \ref{subsec: QDP}.
 Let  $ \, \varphi \in J_\hbar^{\,2} \, $,  \,with
 $ \, J_\hbar := \Ker\big(\epsilon_{\fhg}\big) \, $.  Then:
 \vskip5pt
   (a)\;  $ F := \exp\big( \hbar^{-1} \varphi \big) \, $
   is a well-defined element in  $ \, {\fhg}^\vee \; $;
 \vskip5pt
   (b)\;  $ \, \Ad(F)\big(f\big) := F \cdot f \cdot F^{-1} \, \in \, \fhg \, $
   for all  $ \, f \in \fhg \, $,  so that the adjoint action of  $ \, F $  onto
   $ \, {\fhg}^\vee $  actually restricts to  $ \, \fhg \; $.
\end{lema}

\begin{proof}
 \textit{(a)}\;  The assumption  $ \, \varphi \in J_\hbar^{\;2} \, $
 implies  $ \; \varphi \, \in \, \hbar^2 \, {\big( J_\hbar^{\,\vee} \big)}^2 \subseteq \,
 \hbar^2 \, {\fhg}^\vee \, $,  \,where
 $ \; J_\hbar^{\,\vee} := \hbar^{-1} J_\hbar \, \subseteq \, {\fhg}^\vee \, $.
 Therefore  $ \; \hbar^{-1} \varphi \, \in \, \hbar \, {\fhg}^\vee \, $,  \,hence
 $ \, F := \exp\big( \hbar^{-1} \varphi \big) \, $  is indeed a well-defined element in
 $ {\fhg}^\vee \, $,  \,q.e.d.
 \vskip5pt
 \textit{(b)}\;  We compute  $ \, \Ad(F)\big(f\big) \, $,  $ \, f \in \fhg \, $:
 \,using the identity  $ \; \Ad\!\big(\!\exp(X)\big)(Y) \, = \,
 \exp\!\big(\!\ad(X)\big)(Y) \; $  and expanding the exponential into a
 power series we get
%%%
  $$
  \Ad(F)\big(f\big)  \; = \;  \Ad\!\big(\!\exp\!\big(\hbar^{-1} \varphi\big)\big)(f)  \; = \;
  \exp\!\big(\ad\!\big(\hbar^{-1} \varphi\big)\big)(f)  \; = \;
  {\textstyle \sum\limits_{n=0}^{+\infty}} \, \frac{\,1\,}{\,n!\,} \;
  {\ad\!\big( \hbar^{-1} \varphi \big)}^n(f)
  $$
 Now  Lemma \ref{lemma: technic-Hopf}\textit{(c)\/}   and the assumption
 $ \, \varphi \in J_{\!F_\hbar}^{\;2} \, $  together guarantee that
  $$  {\ad\!\big( \hbar^{-1} \varphi \big)}^n(f)  \, = \,
  {\ad\!\big( \hbar^{-1} \varphi \big)}^n\big(f_+\big)  \, \in \, (1-\delta_{s,0}) \,
  J_{\!F_\hbar}^{\;n+s}   \eqno \forall \; n \in \NN_+   \qquad
  $$
 with  $ \, s \in \NN \, $  such that  $ \, f \in J_{\!F_\hbar}^{\;s} \, $,  \,hence
 $
\; \Ad(F)\big(f\big) = {\textstyle \sum\limits_{n=0}^{+\infty}} \, \frac{\,1\,}{\,n!\,} \,
{\ad\!\big( \hbar^{-1} \varphi \big)}^n(f) \; $  is indeed a well-defined element
--- a  \textsl{convergent\/}  series! ---   of  $ \fhg \, $.
\end{proof}

\vskip9pt

\begin{lema}  \label{lemma: propt.'s polar twist}
 Let  $ \fhg $  be a QFSHA, and  $ \, {\fhg}^\vee $  the associated
 QUEA defined in  \S \ref{subsec: QDP}.  Let  $ \, \phi \in \fhg^{\widetilde{\otimes}\, 2} \, $
 be such that  $ \, (\id \otimes \epsilon)(\phi) = 0 = (\epsilon \otimes \id)(\phi) \, $.  Then:
 \vskip5pt
   (a)\;  $ \cF = \exp\big( \hbar^{-1} \phi \big) \, $  is a well-defined element in
   $ \, {\big( {\fhg}^\vee \big)}^{\widehat{\otimes}\, 2} \; $;
 \vskip5pt
   (b)\;  $ \, \cF \cdot (x \otimes y) \cdot \cF^{-1} \, \in \, \fhg^{\,\widetilde{\otimes}\, 2} \; $
   for all  $ \, x , y \in \fhg \, $,  so that the adjoint action of  $ \, \cF $  onto
   $ \, {\big( {\fhg}^\vee \big)}^{\widehat{\otimes}\, 2} $  actually restricts to
   $ \, {\fhg}^{\widetilde{\otimes}\, 2} \; $;
 \vskip5pt
   (c)\;  $ (\id \otimes \epsilon)(\cF\,) = 1 = (\epsilon \otimes \id)(\cF\,) \; $;
 \vskip5pt
   (d)\;  $ \cF $  is orthogonal, i.e.\  $ \, \cF_{2,1} = \cF^{-1} \, $,  \,iff\/  $ \phi $
   is antisymmetric, i.e.\  $ \, \phi_{\,2,1} = -\phi \, $;
\end{lema}

\begin{proof}
 \textit{(a)--(b)}\;  The assumption
 $ \, (\id \otimes \epsilon)(\phi) = 0 = (\epsilon \otimes \id)(\phi) \, $  means that
 $ \, \phi \in J_\hbar^{\,\widetilde{\otimes}\, 2} \, $  where
 $ \, J_\hbar := \Ker\big( \epsilon_{\fhg} \big) \, $.
 Now note that
 $ \, F_\hbar[[G \times G\hskip1,5pt]] := {\fhg}^{\,\widetilde{\otimes}\, 2} \, $
 is in turn a QFSHA, for the group  $ \, G \times G \, $;
 \,furthermore, its augmentation ideal is
%
%%%%%
%   $$  \Ker\big(\epsilon_{F_\hbar[[G \times G\hskip1,5pt]]}\big)  \; = \;
% \Ker\big(\epsilon_{\fhg}\big) \otimes \kh \, + \, \kh \otimes \Ker\big(\epsilon_{\fhg}\big)  $$
%%%%%
%
  $$
  \Ker\big(\epsilon_{F_\hbar[[G \times G\hskip1,5pt]]}\big)  \; = \;
\Ker\big(\epsilon_{\fhg}\big) \otimes \fhg \, + \, \fhg \otimes \Ker\big(\epsilon_{\fhg}\big)
$$
 so that  $ \, J_\hbar^{\,\widetilde{\otimes}\, 2}
 \subseteq {\Ker\big(\epsilon_{F_\hbar[[G \times G\hskip1,5pt]]}\big)}^2 \, $,
 and finally, we have also
  $$  {F_\hbar[[G \times G\hskip1,5pt]]}^\vee  \; = \;\,
  {\big( \fhg \,\widetilde{\otimes}\, \fhg \big)}^\vee \; = \;\,
  {\big( {\fhg}^\vee \big)}^{\widehat{\otimes}\, 2}
  $$
 Therefore, applying  Lemma \ref{lemma: adj-act_F}  above to the QFSHA
 $ F_\hbar[[G \times G\hskip1,5pt]] $ and to  $ \, \varphi := \phi \, $
 we get both claims  \textit{(a)\/}  and  \textit{(b)}.
 \vskip5pt
 \textit{(c)--(d)}\;  Both these claims follow at once from definitions,
 along with the assumption that
 $ \, (\id \otimes \epsilon)(\phi) = 0 = (\epsilon \otimes \id)(\phi) \, $.
\end{proof}

\vskip5pt

   The previous result leads us to introduce the notion of ``polar twist'', as follows:

\vskip5pt

\begin{definition}  \label{def: polar twist}
 Let\/  $ \fhg $  be a QFSHA, and  $ \, {\fhg}^\vee \, $  as in
 \ref{def: Drinfeld's functors}\textit{(b)}.  We call  \textit{polar twist (element) of}  $ \fhg $
 any element in
 $ \, \big( {\fhg}^{\widetilde{\otimes}\, 2} \,\big)^{\!\vee} =
 {\big( {\fhg}^\vee \,\big)}^{\widehat{\otimes}\, 2} \, $  of the form
 $ \, \cF := \exp\!\big( \hbar^{-1} \phi \big) \, $   --- for some
 $ \, \phi \in {\fhg}^{\widetilde{\otimes}\, 2} \, $  such that
 $ \, (\id \otimes \epsilon)(\phi) = 0 = (\epsilon \otimes \id)(\phi) \, $  ---
 which have the property of a twist element for the QUEA  $ {\fhg}^\vee \, $.
\end{definition}

\vskip4pt

   Of course, every twist for  $ \fhg $  is a polar twist too; the converse, instead,
   in general, is false (counterexamples do exist).
   However, \textit{every polar twist still provides a well-defined deformation by twist of
   $ \fhg $}   --- in other words, the construction of  \textsl{deformations by twist\/}
   properly extends to  \textsl{``deformations by  \textit{polar}  twist''\/}  as well:

\vskip7pt

\begin{theorem}  \label{thm: polar twist-deform-QFSHA}
 Let  $ \fhg $  be a QFSHA, and  $ \, \cF = \exp\big( \hbar^{-1} \phi \big) \, $
 a polar twist for it, as in  Definition \ref{def: polar twist}.
 Then the procedure of twist deformation by  $ \cF $  applied to the QUEA
 $ {\fhg}^\vee $  restricts to  $ \fhg \, $,  making the latter into a new QFSHA.
\end{theorem}

\begin{proof}
 When deforming  $ {\fhg}^\vee $  by the twist  $ \cF $  one introduces on  $
{\fhg}^\vee $  the new coproduct  $ \Delta^\cF $  given by
$ \, \Delta^\cF := \Ad(\cF\,) \circ \Delta \, $.
Then  Lemma \ref{lemma: propt.'s polar twist}\textit{(b)\/}
ensures that  $ \Delta^\cF $  restricts to  $ \fhg \, $,
\,in that it maps the latter into  $ \, {\fhg}^{\widetilde{\otimes}\, 2} \, $.
The antipode can be dealt with similarly,
whence we conclude that the (deformed) Hopf structure of
$ {\big({\fhg}^{\vee\,}\big)}^\cF $  does restrict to  $ \fhg \, $,  \,q.e.d.
\end{proof}

\vskip7pt

\begin{definition}  \label{def: polar twist deform}
 With assumptions as in  Theorem \ref{thm: polar twist-deform-QFSHA},
 the new QFSHA obtained from  $ \fhg $  through the process of twist deformation
 (by  $ \cF \, $)  of  $ {\fhg}^\vee $  followed by restriction will be called
 \textsl{the polar twist deformation of\/  $ \fhg $  by  $ \cF $},  and denoted by
 $ {\fhg}^\cF \, $.
\end{definition}

\vskip4pt

   Finally, the next result describes in detail what exactly is the nature of the new,
   polar twist deformed QFSHA  $ {\fhg}^\cF $,  shedding light onto its semiclassical limit:

\vskip7pt

\begin{theorem}  \label{thm: propt.'s qs-twist-deform-QFSHA}
 Let  $ \fhg $  be a QFSHA over the Lie bialgebra
 $ \, \lieg = \big(\, \lieg \, ; \, [\,\ ,\ ] \, , \, \delta \,\big) \, $.  Set  $ \, \liem := \Ker\big(\epsilon_{{}_\fg}\big) \, $,
%
%%%%%
%   $ \, \liem_\otimes := \Ker\big(\epsilon_{{}_{F[[G \times G \hskip1,5pt]]}}\big)
% = \liem \otimes \fg + \fg \otimes \liem \, $,
%%%%%
%
 \,so  $ \; \liem \Big/ \liem^2 \, \cong \, \lieg^* \, $
%%%
 and  $ \; \big( \liem \otimes \liem \big) \Big/ \big( \liem^2 \otimes \liem + \liem \otimes \liem^2 \,\big) \, \cong \, \lieg^* \otimes \lieg^* \, $
%%%
 as Lie bialgebras.  Let\/  $ \cF $  be a polar twist for  $ \fhg \, $,  of the form  $ \, \cF = \exp\big( \hbar^{-1} \phi \big) \, $  for some  $ \, \phi \in \fhg^{\widetilde{\otimes}\, 2} \, $,  \, and set also  $ \; \phi_a := \phi - \phi_{\,2,1} \; $.  Then:
 \vskip5pt
   (a)\;  $ \phi $  is antisymmetric, i.e.\  $ \, \phi_{\,2,1} = -\phi \, $,  \,iff\/  $ \cF $  is orthogonal, i.e.\  $ \, \cF_{2,1} = \cF^{-1} \, $;
 \vskip4pt
%
%%%%%
%    (b)\;  $ \, \cF \cdot \Delta(f) \cdot \cF^{-1} \, \in \, \fhg^{\,\widetilde{\otimes}\, 2} \, $
% for every  $ \, f \in \fhg \, $,  so that the twist deformation by  $ \cF $  of the Hopf
% $ \, \khp $--algebra  $ \, \Fhg := \khp \,\widetilde{\otimes}\, \fhg $  actually restricts to
% a deformation of  $ \, \fhg \, $,  that we still denote by  $ \, {\big( \fhg \big)}^\cF \; $;
% %
%  \vskip4pt
% %
%%%%%
%
   (b)\;  the element  $ \, c \, := \, \bigg(\, \phi_a \, \Big(\, \text{\rm mod} \; \hbar \, \fhg^{\widetilde{\otimes}\, 2\,} \Big) \!\!\mod \big( \liem^2 \otimes \liem + \liem \otimes \liem^2 \,\big) \bigg) \; $  in  $ \; \big( \liem \otimes \liem \big) \Big/ \big( \liem^2 \otimes \liem + \liem \otimes \liem^2 \,\big) \, \cong \, \lieg^* \otimes \lieg^* \, $  is an  \textsl{antisymmetric twist}  for  $ \lieg^* \, $;
 \vskip4pt
    (c)\;  the polar twist deformation  $ \, {\big( \fhg \big)}^\cF $  of  $ \, \fhg $  is a QFSHA for the Poisson group  $ G^{\,c} $  whose cotangent  Lie bialgebra is
%%%%%
 $ \,\; {\Lie\,\big(G^{\,c}\big)}^* = \, {\big(\, \lieg^* \big)}^c \, = \,  \big(\, \lieg^* \, ; \, {[\,\ ,\ ]}_* \, , \, \delta_*^{\,c} \,\big) \;\, $
 that is the deformation of  $ \, \lieg^* $  by the twist  $ c \, $;  in short,  $ \; {\big( \fhg \big)}^\cF \cong \, F_\hbar\big[\big[G^{\,c}\big]\big] \; $.
                                                      \par
   In particular, if\/  $ \cF $  is\/  $ \kh $--valued   --- i.e., it is an ordinary twist for the Hopf\/  $ \kh $--algebra  $ \fhg \, $  ---   or equivalently  $ \, \phi \in \hbar \, \fhg^{\widetilde{\otimes}\,2} \, $,  \,then we have just  $ \, c = 0 \, $  and  $ \; {\big( \fhg \big)}^c \cong \, F_\hbar\big[\big[G^{\,c}\big]\big] = \fhg \; $.
\end{theorem}

\pf
 \textit{(a)}\,  This is a special case of  Lemma \ref{lemma: propt.'s polar twist}\textit{(d)}.

\vskip7pt

 \textit{(b)}\,  We start from the twist identity
  $ \,\; \cF_{1{}2} \, \big( \Delta \otimes \text{id} \big)(\cF\,)  \; = \,
  \cF_{2{}3} \, \big( \text{id} \otimes \Delta \big)(\cF\,) \;\, $
 that we re-write in the equivalent form
\begin{equation}   \label{eq: twist-modified}
  \big( \Delta \otimes \text{id} \big)(\cF\,) \, {\big( \text{id} \otimes \Delta \big)(\cF\,)}^{-1}  \;
  = \;\,  \cF_{1{}2}^{\,-1} \, \cF_{2{}3}
\end{equation}
 Replacing  $ \, \cF = \exp\!\big(\hbar^{-1}\phi\big) \, $,  we find
  $$
  \displaylines{
   \quad   \big( \Delta \otimes \text{id} \big)(\cF\,) \big) \cdot
   {\big( \text{id} \otimes \Delta \big)(\cF\,)}^{-1}  \,\; =   \hfill  \cr
   = \;\,  \big( \Delta \otimes \text{id} \big)\big(\exp\!\big(\hbar^{-1}\phi\big)\big) \cdot
   {\big( \text{id} \otimes \Delta \big)\big(\exp\!\big(\hbar^{-1}\phi\big)\big)}^{-1}  \; =  \cr
   \hfill   = \;\,  \exp\big(\,\hbar^{-1} \big( \Delta \otimes \text{id} \big)(\phi) \big) \cdot
   \exp\big(\!-\hbar^{-1} \big( \text{id} \otimes \Delta \big)(\phi) \big)   \quad  }
   $$
   \indent   Now we recall the  \textsl{Baker-Campbell-Hausdorff's formula},
   that is the formal identity
\begin{equation}   \label{eq: BCH-formula}
  \exp(X) \cdot \exp(Y)  \; = \;  \exp\big( \mathcal{B\,C\,H}(X,Y\,) \big)
\end{equation}
 which allows to express the product of two exponential as a single exponential: in it,
 $ \, \mathcal{B\,C\,H}(X,Y\,) \, $  is an explicit formal series given by
\begin{equation}   \label{eq: BCH-expansion}
  \begin{aligned}
    &  \mathcal{B\,C\,H}(X,Y\,)  \,\; := \;\,  \log\big( \exp(X) \, \exp(Y\,) \big)  \,\; =  \\
       &  \qquad \hfill   \,\; = \;\,
  {\textstyle \sum\limits_{n=1}^{+\infty}} \, \frac{\;{(-1)}^{n-1}\;}{\;n\;} \,
  {\textstyle \sum\limits_{\substack{r_i + s_i > 0  \\   1 \leq i \leq n}}}
  \frac{\; \big[\, X^{\bullet r_1} Y^{\bullet s_1} X^{\bullet r_2}
  Y^{\bullet s_2} \cdots X^{\bullet r_n} Y^{\bullet s_n} \big] \;}{\; \big( \sum_{i=1}^n (r_i + s_i)
  \big) \cdot \prod_{j=1}^n r_i!\,s_i! \;}
\end{aligned}
\end{equation}
 where we use notation
  $$
  \displaylines{
   \big[\, X^{\bullet r_1} Y^{\bullet s_1} \cdots X^{\bullet r_n} Y^{\bullet s_n} \big]  \,\; :=
   \hfill  \cr
   \hfill   = \;\,  \big[ \underbrace{X,\big[X,\cdots\big[X}_{r_1} ,
   \big[ \underbrace{Y,\big[Y,\cdots\big[Y}_{s_1} , \cdots ,
   \big[ \underbrace{X,\big[X,\cdots\big[X}_{r_n} ,
   \big[ \underbrace{Y,\big[Y,\cdots , Y}_{s_n} \big] \big] \big]
   \cdots \big] \big] \cdots \big] \big]  }
   $$
 with the silent assumption that the Lie monomial
 $ \, \big[\, X^{\bullet r_1} Y^{\bullet s_1} \cdots X^{\bullet r_n} Y^{\bullet s_n} \big] \, $
 is just  $ X $,  respectively  $ Y $,  when  $ \, n = 1 \, $  and  $ \, s_1 = 0 \, $,
 respectively  $ \, r_1 = 0 \, $,  while it is zero whenever  $ \, s_n > 1 \, $  or
 $ \, s_n = 0 \, $  and  $ \, r_n > 1 \, $.  In words, when
 $ \, S := \sum_{i=1}^n (r_i + s_i) > 1 \, $  the Lie monomial
 $ \, \big[\, X^{\bullet r_1} Y^{\bullet s_1} \cdots X^{\bullet r_n} Y^{\bullet s_n} \big] \, $
 is the composition of several operators  $ {\ad(X)}^{r_i} $  or  $ {\ad(Y)}^{s_i} $  to  $ Y $
 --- when  $ \, s_n = 1 \, $  ---   or to  $ X $   --- when  $ \, s_n = 0 \, $  and  $ \, r_n = 1 \, $.
 Looking up to second order,  \eqref{eq: BCH-expansion}  reads
\begin{equation}   \label{eq: BCH-expansion_truncated}
  \mathcal{B\,C\,H}(X,Y\,)  \,\; := \;\,  X \, + \, Y \, + \, \frac{\,1\,}{\,2\,} \, [X,Y\,] \, + \,
  \cO_{\mathcal{L}}(3)
\end{equation}
 where  $ \, \cO_{\mathcal{L}}(3) \, $  denotes a (formal) infinite linear combination
 of Lie monomials in  $ X $  and  $ Y $  of degree at least 3.
 \vskip3pt
   Setting now  $ \, X := \hbar^{-1} \big( \Delta \otimes \text{id} \big)(\phi) \, $  and
   $ \, Y := -\hbar^{-1} \big( \text{id} \otimes \Delta \big)(\phi) \, $,
   \,the above analysis yields, rewriting  \eqref{eq: BCH-formula},
\begin{equation}   \label{eq: AltDeltaF=CirNablaphi}
 \begin{aligned}
   \big( \Delta \otimes \text{id} \big)(\cF\,)  &
   \cdot {\big( \text{id} \otimes \Delta \big)(\cF\,)}^{-1}  \,\; =  \\
   &  = \;\,  \exp\Big(\, \mathcal{B\,C\, H}\Big(\, \hbar^{-1}
   \big( \Delta \otimes \text{id} \big)(\phi) \, , \, -\hbar^{-1} \big( \text{id} \otimes
   \Delta \big)(\phi) \Big) \Big)
 \end{aligned}
\end{equation}
 where the BCH series has to be expanded as in  \eqref{eq: BCH-expansion}.
 To this end, writing  $ \, \phi = \phi_1 \otimes \phi_2 \, $  (a sum being tacitly intended) with
 $ \, \phi_1 , \phi_2 \in  J_\hbar \, $
 (by  Lemma \ref{lemma: twist-cond's/exp ==> twist-cond's/log})  we have
  $$
  \displaylines{
   \big( \Delta \otimes \text{id} \big)(\phi)  \, = \,
   \Delta\big(\phi_1\big) \otimes \phi_2  \, = \,
   \phi_1 \otimes 1 \otimes \phi_2 \, + \, 1 \otimes \phi_1 \otimes \phi_2 \, +
   \, {\big( {(\phi_1)}_{(1)} \big)}^+ \otimes {\big( {(\phi_1)}_{(2)} \big)}^+ \otimes \phi_2  \, =
   \hfill  \cr
   \hfill   = \;  \phi_{1,3} \, + \, \phi_{2,3} \, + \,
   {\big( {(\phi_1)}_{(1)} \big)}^+ \otimes {\big( {(\phi_1)}_{(2)} \big)}^+ \otimes \phi_2  }
   $$
 where we expanded  $ \, \Delta\big(\phi_1\big) \, $  as in
 Lemma \ref{lemma: technic-Hopf}\textit{(d)\/}  and we used that
 $ \, \epsilon\big(\phi_1\big) = 0 \, $.  Note that in the expansion of
 $ \big( \Delta \otimes \text{id} \big)(\phi) $  we have
\begin{equation}   \label{eq: pieces(Dotimes1(phi)) J-adic}
  \big( \phi_{1,3} \, + \, \phi_{2,3} \big) \, \in \, J_\hbar^{(\otimes 3|2)}
 \;\; ,  \qquad
  {\big( {(\phi_1)}_{(1)} \big)}^{\!+} \otimes {\big( {(\phi_1)}_{(2)} \big)}^{\!+} \otimes \phi_2 \,
  \in \, J_\hbar^{\,\otimes 3}
\end{equation}
 where we introduced the notation
 $ \;  J_\hbar^{(\otimes 3|N)}
 := \sum_{\!\substack{a , b , c \,\geqslant\, 0  \\
 a + b + c \,\geqslant\, N}} \! J_\hbar^{\,a} \,\widetilde{\otimes}\,
 J_\hbar^{\,b} \,\widetilde{\otimes}\, J_\hbar^{\,c} \; $  (for  $ \, N \in \NN \, $).
                                                                                 \par
   A similar analysis for  $ \, \big( \text{id} \otimes \Delta \big)(\phi) \, $,
   \,just switching the roles of  $ \phi_1 $  and  $ \phi_2 \, $,  \,yields
  $$
  \displaylines{
   \big( \text{id} \otimes \Delta \big)(\phi)  \, = \,  \phi_1 \otimes \Delta\big(\phi_2\big)  \, = \,
   \phi_1 \otimes \phi_2 \otimes 1 \, + \, \phi_1 \otimes 1 \otimes \phi_2 \, + \,
   \phi_1 \otimes {\big( {(\phi_2)}_{(1)} \big)}^+ \otimes {\big( {(\phi_2)}_{(2)} \big)}^+  \, =   \hfill  \cr
   \hfill   = \;  \phi_{1,2} \, + \, \phi_{1,3} \, + \, \phi_1 \otimes {\big( {(\phi_2)}_{(1)} \big)}^+
   \otimes {\big( {(\phi_2)}_{(2)} \big)}^+  }
   $$
 with
\begin{equation}   \label{eq: pieces(1otimesD(phi)) J-adic}
  \big( \phi_{1,2} \, + \, \phi_{1,3} \big) \, \in \,
 J_\hbar^{(\otimes 3|2)}
 \;\; ,  \qquad
  \phi_1 \otimes {\big( {(\phi_2)}_{(1)} \big)}^{\!+} \otimes {\big( {(\phi_2)}_{(2)} \big)}^{\!+} \,
  \in \, J_\hbar^{\,\otimes 3}
\end{equation}
 Now, thanks to  Lemma \ref{lemma: technic-Hopf}\textit{(c)},
 from  \eqref{eq: pieces(Dotimes1(phi)) J-adic}  and
 \eqref{eq: pieces(1otimesD(phi)) J-adic}  we get
  $$
  \displaylines{
   \quad   \big[ \hbar\,X , \hbar\,Y \big]  \; = \;
   \big[ \big( \Delta \otimes \text{id} \big)(\phi) \, ,
   \big( \text{id} \otimes \Delta \big)(\phi) \big]  \,\; =   \hfill  \cr
   = \;\,  \big[\, \phi_{1,3} + \phi_{2,3} \, , \, \phi_{1,2} + \phi_{1,3} \,\big]  \; +
   \;  \hbar \cdot \cO\Big(
%%%%%
% J_\hbar^{(\otimes\, 2)}
J_\hbar^{(\otimes 3|4)}
 \Big)  \,\; =   \quad  \cr
   \hfill   = \;\,  \big[\, \phi_{1,3} \, , \, \phi_{1,2} \,\big] + \big[\, \phi_{2,3} \, ,
   \, \phi_{1,2} \,\big] + \big[\, \phi_{2,3} \, , \, \phi_{1,3} \,\big]  \; + \;
   \hbar \cdot \cO\Big(
%%%%%
% J_\hbar^{[\otimes\, 4]}
J_\hbar^{[\otimes 3|4]}  \Big)   \quad  }   $$
 hence
\begin{equation}   \label{eq: [X,Y] = sum(phi)}
  \!\!\!   \big[ X , Y \big]  \, = \,  \hbar^{-1} \Big( \hbar^{-1} \big[ \phi_{1,3} , \phi_{1,2} \big] +
  \hbar^{-1} \big[ \phi_{2,3} , \phi_{1,2} \big] + \hbar^{-1} \big[ \phi_{2,3} , \phi_{1,3} \big]  \,
  + \,  \cO\Big(
%%%%%
% J_\hbar^{[\otimes\, 4]}
 J_\hbar^{[\otimes 3|4]}  \Big) \Big)
\end{equation}
 for some element  $ \, \cO\Big(
%%%%%
% J_\hbar^{[\otimes\, 4]}
J_\hbar^{[\otimes 3|4]}
 \Big) \in
%%%%%
% J_\hbar^{[\otimes\, 4]}
J_\hbar^{[\otimes 3|4]}  \, $,
 \,where hereafter we use notation
\begin{equation}   \label{eq: def-JotimesN}
%
%%%%%
% J_\hbar^{[\otimes\, N]}
J_\hbar^{[\otimes 3|N]}
 \, := \hskip-7pt \sum\limits_{\substack{a, b, c \,\geqslant\, 1  \\
  a+b+c \,\geqslant\, N}} \hskip-11pt J_\hbar^{\otimes a} \otimes J_\hbar^{\otimes b}
  \otimes J_\hbar^{\otimes c}  \; \subseteq \;
  {\Ker\Big( \epsilon_{{}_{{\fhg}^{\otimes 3}}} \Big)}^N   \qquad  \forall \;\; N \in \NN_+
\end{equation}
 Pushing the analysis further on, we find easily that
\begin{equation}   \label{eq: J-adic valuation x Lie monomial}
  \big[\, X^{\bullet r_1} Y^{\bullet s_1} \cdots X^{\bullet r_n} Y^{\bullet s_n} \big]  \; \in \;
  \hbar^{-1} \,
%%%%%
% J_\hbar^{[\otimes\, 4]}
 J_\hbar^{[\otimes 3|S+1]}
   \qquad  \forall \;\;
  S := {\textstyle \sum\limits_{i=1}^n} \, (r_i + s_i) > 1 \,
\end{equation}
 looking at  \eqref{eq: BCH-expansion},  this tells us that the expansion of the
 BCH series occurring in  \eqref{eq: AltDeltaF=CirNablaphi},  when expanded as in
 \eqref{eq: BCH-expansion},  is actually given by  $ \hbar^{-1} $  multiplied by
 \textsl{a truly convergent series inside}  $ {\fhg}^{\otimes 3} \, $.
 In other words, tiding everything up
 --- from  \eqref{eq: AltDeltaF=CirNablaphi},  \eqref{eq: BCH-expansion},
 \eqref{eq: BCH-expansion_truncated},  \eqref{eq: [X,Y] = sum(phi)}  and
 \eqref{eq: J-adic valuation x Lie monomial}  altogether ---
 we find that there exists some
 $ \, \mathcal{Z} \in J_\hbar^{\otimes 3} \subseteq {\fhg}^{\otimes 3} \, $  such that
 $ \; \big( \Delta \otimes \text{id} \big)(\cF\,) \cdot
 {\big( \text{id} \otimes \Delta \big)(\cF\,)}^{-1} = \, \exp\big(\, \hbar^{-1} \mathcal{Z} \,\big) \; $
 Even more, by  \eqref{eq: BCH-expansion}  and  \eqref{eq: BCH-expansion_truncated}
 and the previous analysis we do know the expansion of this  $ \mathcal{Z} $
 up to second order, whence we find
\begin{equation}   \label{eq: DeltaF = exp(hZ) CONTR}
\begin{aligned}
   &  \hskip-6pt   \big( \Delta \otimes \text{id} \big)(\cF\,) \cdot
   {\big( \text{id} \otimes \Delta \big)(\cF\,)}^{-1}  \; =  \\
   &  \hskip-5pt   = \;\,  \exp\Big(\, \hbar^{-1} \, \Big( \big( \Delta \otimes \text{id} \big)(\phi) \,
   - \, \big( \text{id} \otimes \Delta \big)(\phi) \; -  \\
   &  \hskip-5pt   - \hbar^{-1} 2^{-1} \big[ \phi_{1,3} , \phi_{1,2} \big] -
   \hbar^{-1} 2^{-1} \big[ \phi_{2,3} , \phi_{1,2} \big] - \hbar^{-1} 2^{-1} \big[ \phi_{2,3} ,
   \phi_{1,3} \big]  \, + \,  \cO\Big( J_\hbar^{[\otimes\, 3|4]} \Big) \!\Big) \!\Big)
\end{aligned}
\end{equation}
 \vskip5pt
   Now we go and work instead on the right-hand side of  \eqref{eq: twist-modified}.
   Again, replacing  $ \, \cF = \exp\!\big(\hbar^{-1}\phi\big) \, $,  we find
\begin{equation*}   % \label{eq: F12-23=expBCH}
 \begin{aligned}
   &  \cF_{1{}2}^{-1} \cdot \cF_{2{}3}  \; = \;  \exp\!\big(\! -\!\hbar^{-1} \phi \otimes 1 \big) \cdot
   \exp\!\big(\, \hbar^{-1} 1 \otimes \phi \big)  \,\; =
  \\
   &  \;\;\quad   = \;  \exp\!\big(\! -\!\hbar^{-1} \phi_{1,2} \big) \cdot
   \exp\!\big(\, \hbar^{-1} \phi_{2,3} \big)  \; = \;
   \exp\!\Big( \mathcal{B\,C\,H}\big(\! -\!\hbar^{-1} \, \phi_{1,2} \, , \, \hbar^{-1} \,
   \phi_{2,3} \,\big) \Big)
\end{aligned}
\end{equation*}
 Now for the computation
 of  $ \; \mathcal{B\,C\,H}\big(\! -\!\hbar^{-1} \, \phi_{1,2} \, , \hbar^{-1} \, \phi_{2,3} \big) \; $;
 \,to avoid possible confusion, we denote the second, right-hand instance of  $ \phi $  by
 $ \phi' \, $.  We begin noting that
  $$
  \big[\, \phi_{1,2} \, , \phi'_{2,3} \,\big]  \,\; = \;\,  \big[\, \phi_1 \otimes \phi_2 \otimes 1 \, ,
  1 \otimes \phi'_1 \otimes \phi'_2 \,\big]  \,\; = \;\,  \phi_1 \otimes \big[\, \phi_2 \, ,
  \phi'_1 \,\big] \otimes \phi'_2  \,\; \in \;\,  \hbar \,  J_\hbar^{[\otimes 3|3]}
  $$
 so that for  $ \; X := -\hbar^{-1} \, \phi_{1,2} \; $  and  $ \; Y := \hbar^{-1} \, \phi_{2,3} \; $
 we get, using  Lemma \ref{lemma: technic-Hopf}\textit{(c)},
  $$
  \big[ X , Y \big]  \; = \;  \big[ -\!\hbar^{-1} \, \phi_{1,2} \, , \hbar^{-1} \, \phi'_{2,3} \,\big]
  \; = \;  -\!\hbar^{-2} \big[ \, \phi_{1,2} \, , \phi'_{2,3} \,\big]  \; \in \;
  \hbar^{-2} \, \hbar \, J_\hbar^{\otimes 3}  \; = \; \hbar^{-1} J_\hbar^{[\otimes 3|3]}
  $$
 A second, similar step gives
 (with obvious notation  $ \phi \, $,  $ \phi' $  and  $ \phi'' \, $)
  $$  \displaylines{
   \big[\, \phi_{1,2} \, , \big[\, \phi'_{1,2} \, , \phi''_{2,3} \,\big] \big]  \,\; = \;\,
   \big[\, \phi_1 \otimes \phi_2 \otimes 1 \, , \phi'_1 \otimes \big[\, \phi'_2 \, ,
   \phi''_1 \,\big] \otimes \phi''_2 \,\big]  \,\; =   \hfill  \cr
   \hfill   = \;\,  \big[\, \phi_1 \, , \phi'_1 \,\big] \otimes \phi_2 \cdot \big[\, \phi'_2 \, ,
   \phi''_1 \,\big] \otimes \phi''_2 \, + \, \phi_1 \cdot \phi'_1 \otimes \big[\, \phi_2 \, ,
   \big[\, \phi'_2 \, , \phi''_1 \,\big] \big] \otimes \phi''_2  \,\; \in \;\,
   \hbar^2 \, J_\hbar^{[\otimes 3|4]}  }  $$
 so that
$ \; \big[ X , \big[ X , Y \,\big]\big] \, \in \, \hbar^{-3} \, \hbar^2 \, J_\hbar^{[\otimes 3|4]} \,
= \, \hbar^{-1} \, J_\hbar^{[\otimes 3|4]} \; $. More in general, iteration yields
\begin{equation}   \label{eq: iter[X,Y]-in-hJadic_x_[phi,phi]_SECOND}
  \big[\, X^{\bullet r_1} Y^{\bullet s_1} \cdots X^{\bullet r_n} Y^{\bullet s_n} \big]
\; \in \;  \hbar^{-S} \, \hbar^{S-1} \, J_\hbar^{\,[\otimes 3|S+1]}  \; = \;
\hbar^{-1} \, J_\hbar^{\,[\otimes 3|S+1]}
\end{equation}
 with notation as before, still using  Lemma \ref{lemma: technic-Hopf}\textit{(c)}.
 Tiding everything up we find that there exists
 $ \, \mathcal{W} \in J_\hbar^{\otimes 3} \subseteq {\fhg}^{\widetilde{\otimes}\, 3} \, $
 such that
 $ \; \cF_{1{}2}^{-1} \cdot \cF_{2{}3} \, = \, \exp\big(\, \hbar^{-1} \mathcal{W} \,\big) \; $;
 \,moreover, by  \eqref{eq: BCH-expansion}  and  \eqref{eq: BCH-expansion_truncated}
 along with the previous analysis we can write
\begin{equation}   \label{eq: F12-23=exp(...)}
  \cF_{1,2}^{-1} \cdot \cF_{2,3}  \; =  \;  \exp\Big(\, \hbar^{-1} \Big(\! -\phi_{1,2} \, + \,
  \phi_{2,3} \, - \, \hbar^{-1} \, 2^{-1} \big[\, \phi_{1,2} \, , \, \phi_{2,3} \,\big]  \, + \,
  \cO\Big( J_\hbar^{[\otimes\, 3|4]} \Big) \Big)
\end{equation}
   \indent   Finally, comparing  \eqref{eq: DeltaF = exp(hZ) CONTR},
   \eqref{eq: F12-23=exp(...)}  and \eqref{eq: twist-modified}  we get the identity in
   $ {\fhg}^{\widetilde{\otimes}\, 3} $
  $$  \displaylines{
   \big( \Delta \otimes \text{id} \big)(\phi) \, - \, \big( \text{id} \otimes \Delta \big)(\phi) \; -
   \hfill  \cr
   \hskip11pt   - \, \hbar^{-1} 2^{-1} \big[ \phi_{1,3} \, ,\, \phi_{1,2} \big] -
   \hbar^{-1} 2^{-1} \big[ \phi_{2,3} \, , \, \phi_{1,2} \big] - \hbar^{-1} 2^{-1} \big[ \phi_{2,3} \, ,
   \, \phi_{1,3} \big]  \, + \,  \cO\Big( J_\hbar^{[\otimes\, 3|4]} \Big)  \,\; =   \hfill  \cr
  \hfill   = \;\,  - \phi_{1,2} \, + \, \phi_{2,3} \, - \, \hbar^{-1} \, 2^{-1} \big[\, \phi_{1,2} \, ,
  \, \phi_{2,3} \,\big]  \, + \,  \cO\Big( J_\hbar^{[\otimes\, 3|4]} \Big)  }  $$
 that in turn, through simplification and reduction modulo
 $ \, \hbar \, {\fhg}^{\widetilde{\otimes}\, 3} \, $,  \,yields the following identity inside
 $ {\fg}^{\widetilde{\otimes}\, 3} \, $
\begin{equation}  \label{eq: polar twist(Lie) modulo m4}
 \begin{aligned}
    &  \big( \Delta \otimes \text{id} \big)\big(\,\overline{\phi}\,\big) \, - \,
    \big( \text{id} \otimes \Delta \big)\big(\,\overline{\phi}\,\big) \, + \,
    \overline{\phi}_{1,2} \, - \, \overline{\phi}_{2,3} \; +  \\
   &  \hskip41pt   + \; 2^{-1} \big\{\, \overline{\phi}_{1,2} \, ,
   , \overline{\phi}_{1,3} \,\big\} \, + \, \big\{\, \overline{\phi}_{1,2} \, ,
   \, \overline{\phi}_{2,3} \,\big\} \, + \, 2^{-1} \big\{\, \overline{\phi}_{1,3} \, ,
   \, \overline{\phi}_{2,3} \,\big\}  \; \underset{\liem^{[\otimes\, 3\,|4]}}{\equiv} \;  0
 \end{aligned}
\end{equation}
 where hereafter we adopt the notation for which  $ \, \overline{\varphi} \, $
 denotes the coset modulo  $ \, \hbar \, $  of any element
 $ \, \varphi \in {\fhg}^{\widetilde{\otimes}\, 3} \, $  with  $ \, n \in \NN_+ \, $.
 \vskip5pt

   Now let  $ \, \Bbbk\big[\mathbb{S}_3\big] \, $  act onto  $ {\fhg}^{\widetilde{\otimes}\, 3} $
   and consider in particular the action of the antisymmetrizer
 $ \; \textsl{Alt}_{\,3} := \big( \id - (1\,2) - (2\,3) - (3\,1) + (1\,2\,3) + (3\,2\,1) \big) \; $
 onto the above identity: this in turn yields a new identity.
 Within the latter, we have a first contribution of the form
  $$  \displaylines{
   \textsl{Alt}_{\,3\,}.\Big( \big( \Delta \otimes \text{id} \big)\big(\,\overline{\phi}\,\big) \, - \,
   \big( \text{id} \otimes \Delta \big)\big(\,\overline{\phi}\,\big) \Big)  \,\; =   \hfill  \cr
   \hfill   = \;\,  \big(\, \nabla \otimes \text{id} \big)\big(\,\overline{\phi}\,\big) -
   \big(\, \nabla \otimes \text{id} \big)\big(\,\overline{\phi}_{2,1}\big) \; + \; \text{c.p}  \,\; =
   \;\,  \big(\, \nabla \otimes \text{id} \big)\big(\,\overline{\phi_a}\,\big) \; + \; \text{c.p}  }  $$
 and a second contribution of the form
  $$  \textsl{Alt}_{\,3\,}.\big(\, \overline{\phi}_{1,2} \, - \,
  \overline{\phi}_{2,3} \,\big)  \,\; = \;\,  0  $$
   \indent   The third and last contribution is
  $$  \textsl{Alt}_{\,3\,}.\Big(\, 2^{-1} \big\{\, \overline{\phi}_{1,2} \, , \,
  \overline{\phi}_{1,3} \,\big\} \Big) \, + \,
  \textsl{Alt}_{\,3\,}.\Big(\, \big\{\, \overline{\phi}_{1,2} \, ,
  \, \overline{\phi}_{2,3} \,\big\} \Big) \, + \,
  \textsl{Alt}_{\,3\,}.\Big(\, 2^{-1} \big\{\, \overline{\phi}_{1,3} \, ,
  \, \overline{\phi}_{2,3} \,\big\} \Big)  $$
 We go and compute the first summand, as follows:
  $$
  \displaylines{
   \textsl{Alt}_{\,3\,}.\Big(\, 2^{-1} \big\{\, \overline{\phi}_{1,2} \, ,
   \, \overline{\phi}_{1,3} \,\big\} \Big)  \,\; = \;\,
   \big( \id - (2\,3) \big).\Big(\, 2^{-1} \big\{\, \overline{\phi}_{1,2} \, ,
   \, \overline{\phi}_{1,3} \,\big\} \Big) \; + \; \text{c.p.}  \,\; =   \hfill  \cr
   \hfill   = \;\,  \Big(\, 2^{-1} \big\{\, \overline{\phi}_{1,2} \, , \,
   \overline{\phi}_{1,3} \,\big\} \; - \; 2^{-1} \big\{\, \overline{\phi}_{1,3} \, ,
   \, \overline{\phi}_{1,2} \,\big\} \Big) \; + \; \text{c.p.}  \,\; = \;\,
   \big\{\, \overline{\phi}_{1,2} \, , \, \overline{\phi}_{1,3} \,\big\} \; + \; \text{c.p.}  }
   $$
 A similar analysis applies to the third summand, which yields
  $$  \displaylines{
   \textsl{Alt}_{\,3\,}.\Big(\, 2^{-1} \big\{\, \overline{\phi}_{1,3} \, , \, \overline{\phi}_{2,3} \,\big\} \Big)  \,\; = \;\,  \big( \id - (1\,2) \big).\Big(\, 2^{-1} \big\{\, \overline{\phi}_{1,3} \, , \, \overline{\phi}_{2,3} \,\big\} \Big) \; + \; \text{c.p.}  \,\; =   \hfill  \cr
   \hfill   = \;\,  \Big(\, 2^{-1} \big\{\, \overline{\phi}_{1,3} \, , \, \overline{\phi}_{2,3} \,\big\} \; - \; 2^{-1} \big\{\, \overline{\phi}_{2,3} \, , \, \overline{\phi}_{1,3} \,\big\} \Big) \; + \; \text{c.p.}  \,\; = \;\,  \big\{\, \overline{\phi}_{1,3} \, , \, \overline{\phi}_{2,3} \,\big\} \; + \; \text{c.p.}  }  $$
 whereas for the second summand instead we get
  $$
  \displaylines{
   \textsl{Alt}_{\,3\,}.\Big( \big\{\, \overline{\phi}_{1,2} \, , \,
   \overline{\phi}_{2,3} \,\big\} \Big)  \,\; = \;\,
   \big( \id - (1\,3) \big).\Big( \big\{\, \overline{\phi}_{1,2} \, ,
   \, \overline{\phi}_{2,3} \,\big\} \Big) \; + \; \text{c.p.}  \,\; =   \hfill  \cr
   \hfill   = \;\,  \big\{\, \overline{\phi}_{1,2} \, , \,
   \overline{\phi}_{2,3} \,\big\} \; - \; \big\{\, \overline{\phi}_{3,2} \, ,
   \, \overline{\phi}_{2,1} \,\big\} \; + \; \text{c.p.}  }
   $$
 Putting all these together we find
  $$  \displaylines{
   \textsl{Alt}_{\,3\,}.\Big(\, 2^{-1} \big\{\, \overline{\phi}_{1,2} \, , \,
   \overline{\phi}_{1,3} \,\big\} \Big) \, + \,
   \textsl{Alt}_{\,3\,}.\Big(\, \big\{\, \overline{\phi}_{1,2} \, , \,
   \overline{\phi}_{2,3} \,\big\} \Big) \, + \,
   \textsl{Alt}_{\,3\,}.\Big(\, 2^{-1} \big\{\, \overline{\phi}_{1,3} \, ,
   \, \overline{\phi}_{2,3} \,\big\} \Big)  \,\; =   \hfill  \cr
   = \;  \big\{\, \overline{\phi}_{1,2} \, , \, \overline{\phi}_{1,3} \big\} \, + \,
   \big\{\, \overline{\phi}_{1,2} \, , \, \overline{\phi}_{2,3} \big\} \, - \,
   \big\{\, \overline{\phi}_{3,2} \, , \, \overline{\phi}_{2,1} \big\} \, + \,
\big\{\, \overline{\phi}_{1,3} \, , \, \overline{\phi}_{2,3} \big\} \, + \, \text{c.p.}  \; = \;
\big\{\!\big\{\, \overline{\phi_a} \, , \, \overline{\phi_a} \,\big\}\!\big\}  }  $$
 where the very last identity follows from a routine calculation.
 Joint with the previously found identities,
 the latter gives yet the following, last one, which is the result of letting
 $ \, \textsl{Alt}_{\,3} \, $  act onto the congruence  \eqref{eq: polar twist(Lie) modulo m4}:
  $$  \Big(\! \big(\, \nabla \otimes \id \,\big)\big(\, \overline{\phi_a} \,\big) \, + \,
  \text{c.p.} \Big) \, + \, \big\{\!\big\{\, \overline{\phi_a} \, , \, \overline{\phi_a} \,\big\}\!\big\}
  \,\; \underset{\liem^{[\otimes\, 4]}}{\equiv} \;\,  0  $$
 At last, recalling that  $ \, c := \overline{\phi_a} \, \big(\, \text{mod\ } \liem^2 \big) \, $  in  $ \, \liem \Big/ \liem^2 = \lieg^* \, $,  and that in the latter Lie bialgebra the Lie cobracket, resp.\ the Lie bracket, is given by  $ \, \nabla \, $,  resp.\ by  $ \, [\,\ ,\ ] \, $,  reduced modulo  $ \liem^2 \, $,  the last formula above   --- in  $ \liem^{\widehat{\otimes}\, 3} $  ---   implies
  $$  \Big(\! \big(\, \delta \otimes \id \,\big)(c\,) \, + \, \text{c.p.} \Big) \; + \; [[\, c \, , c \,]]  \,\; = \;\,  0  $$
 within  $ {\big( \lieg^* \big)}^{\otimes 3} $,  \,which implies exactly that
 $ \, c \, $   --- which is antisymmetric by construction ---
 is an antisymmetric twist for the Lie bialgebra
 $ \, \lieg^* = \liem \Big/ \liem^2 \, $,  \;q.e.d.
 \vskip9pt
\textit{(c)}\,  We adopt the following notational convention: any element in  $ \fhg $  will be denoted by an italic letter, say  $ \, f \in \fhg \, $;  then its coset modulo  $ \, \hbar\,\fhg \, $  will be denoted with a line over that letter, say  $ \, \overline{f} := \big(\, f \ \text{mod} \ \hbar\,\fhg \,\big) \, $,  and finally the coset of the latter modulo  $ \, \liem^2 \, $  will be denoted by the corresponding letter in roman font, say  $ \, \text{f} := \big(\, \overline{f} \ \text{mod} \ \liem^2 \,\big) \, $.  Note also that every element in  $ \, \lieg^* = \liem \Big/ \liem^2 \, $  can be written as such an  $ \; \text{f} = \big(\, \overline{f} \ \text{mod} \ \liem^2 \,\big) \, $  for some  $ \, f \in J_\hbar \in \Ker\big( \epsilon_{{}_\fhg} \big) \, $.
                                                                  \par
   Similar notation will be used for elements in  $ {\fhg}^{\widetilde{\otimes}\,2} $  and their coset modulo  $ \hbar $  and (further on) modulo  $ \, \liem^{[\otimes 2\,|3]} := \liem \otimes \liem^2 + \liem^2 \otimes \liem \, $.

\vskip7pt

   Recall that the Lie cobracket induced on  $ \, \liem \Big/ \liem^2 = \lieg^* \, $  by the deformed quantization is defined by
\begin{equation*}  \label{eq: deltac(f)}
   \delta^{\,\cF}(\,\text{f}\,)  \; := \;
   \Big( \Delta^\cF - {\big( \Delta^\cF \big)}^{2{}1} \Big)\big(\,\overline{f}\,\big)
   \ \ \text{mod} \  \liem^{[\otimes 2\,|3]}  \; = \;
   \overline{\Delta^\cF(f)} - \overline{{\big( \Delta^\cF \big)}^{\text{op}}(f)}
   \,\ \ \text{mod} \  \liem^{[\otimes 2\,|3]}
\end{equation*}
 so we start computing  $ \, \Delta^\cF(f) \, $.  Definitions give
  $$  \displaylines{
   \Delta^\cF(f)  \,\; = \;\,  \Ad(\,\cF\,)\big(\Delta(f)\big)  \, = \,
   \Ad(\cF\,)\big(\, f_{(1)} \otimes f_{(2)} \big)  \,\; =   \hfill  \cr
   \hfill   = \;\,  \Ad(\cF\,)\Big(\! \big(\, f_{(1)} \otimes 1 \,\big)
   \cdot \big(\, 1 \otimes f_{(2)} \big) \!\Big)  \,\; = \;\,
   \Ad(\cF\,)\big(\, f_{(1)} \otimes 1 \,\big) \cdot \Ad(\cF\,)\big(\, 1 \otimes f_{(2)} \big)  }  $$
 In the last product, we focus on the first factor: thus we get
  $$  \displaylines{
   \Ad(\cF\,)\big(\, f_{(1)} \otimes 1 \big)  \,\; = \;\,
   \Ad\big(\exp\!\big( \hbar^{-1} \phi \big)\big)\big(\, f_{(1)} \otimes 1 \,\big)  \,\; = \;\,
   \exp\big(\ad\!\big( \hbar^{-1} \phi \big)\big)\big(\, f_{(1)} \otimes 1 \,\big)  \,\; =   \hfill  \cr
   = \;\,  {\textstyle \sum\limits_{n=0}^{+\infty}} \, \frac{\,1\,}{\,n!\,} \;
   {\ad\!\big( \hbar^{-1} \phi \big)}^n \big(\, f_{(1)} \otimes 1 \,\big)  \,\; = \;\,
   {\textstyle \sum\limits_{n=0}^{+\infty}} \, \frac{\,1\,}{\,n!\,} \;
   {\ad\!\big( \hbar^{-1} \phi_1 \otimes \phi_2 \big)}^n \big(\, f_{(1)} \otimes 1 \,\big)  \,\; =  \cr
   \hfill   = \;\,  {\textstyle \sum\limits_{n=0}^{+\infty}} \,
   \frac{\,1\,}{\,n!\,} \; {\ad\!\big( \hbar^{-1} \phi_1 \big)}^n
   \big(\, f_{(1)} \big) \otimes \phi_2^{\;n}  \,\; = \;\,
   f_{(1)} \otimes 1 \, + \, \big[\, \hbar^{-1} \phi_1 \, , f_{(1)} \,\big] \otimes \phi_2 \, + \,
   \cO(2)  }  $$
 that is in short
\begin{equation*}   \label{eq: Ad(F)(f(1)otimes1)}
  \Ad(\cF\,)\big(\, f_{(1)} \otimes 1 \big)  \,\; = \;\,  f_{(1)} \otimes 1 \, + \,
  \big[\, \hbar^{-1} \phi_1 \, , f_{(1)} \,\big] \otimes \phi_2 \, + \, \cO(2)
\end{equation*}
 where hereafter  $ \, \cO(2) \, $  denotes any element in  $ \, J_\hbar^{[\otimes 2\,|3]} \, $.  A similar calculation yields
\begin{equation*}   \label{eq: Ad(F)(1otimesf(2))}
  \Ad(\cF\,)\big(\, 1 \otimes f_{(2)} \big)  \,\; = \;\,  1 \otimes f_{(2)} \, +
  \, \phi_1 \otimes \big[\, \hbar^{-1} \phi_2 \, , f_{(2)} \,\big] \, + \, \cO(2)
\end{equation*}
 Pasting together the last two identities we find
  $$  \displaylines{
   \Delta^\cF(f)  \,\; = \;\,  \Ad(\cF\,)\big(\, f_{(1)} \otimes f_{(2)} \big)  \,\; = \;\,
   \Ad(\cF\,)\big(\, f_{(1)} \otimes 1 \,\big) \cdot \Ad(\cF\,)\big(\, 1 \otimes f_{(2)} \big)  \,\; =
   \hfill  \cr
   = \,  \big(\, f_{(1)} \otimes 1 + \big[\, \hbar^{-1} \phi_1 \, , f_{(1)} \,\big] \otimes \phi_2 +
   \cO(2) \big) \cdot \big(\, 1 \otimes f_{(2)} + \phi_1 \otimes \big[\, \hbar^{-1} \phi_2 \, ,
   f_{(2)} \,\big] + \cO(2) \big)  \, =  \cr
   \quad   = \;  f_{(1)} \otimes f_{(2)} \, + \, f_{(1)} \, \phi_1 \otimes \big[\, \hbar^{-1} \phi_2 \, ,
   f_{(2)} \,\big] \, + \, \big[\, \hbar^{-1} \phi_1 \, , f_{(1)} \,\big] \otimes \phi_2 \, f_{(2)} \, + \,
   \cO(2)  \,\; =   \hfill  \cr
   \qquad   = \;  f_{(1)} \otimes f_{(2)} \, + \, \epsilon\big(f_{(1)}\big) \,
   \phi_1 \otimes \big[\, \hbar^{-1} \phi_2 \, , f_{(2)} \,\big] \, + \, \big[\, \hbar^{-1} \phi_1 \, ,
   f_{(1)} \,\big] \otimes \phi_2 \, \epsilon\big(f_{(2)}\big) \, + \, \cO(2)  \,\; =   \hfill  \cr
   \hfill   = \;  f_{(1)} \otimes f_{(2)} \, + \, \phi_1 \otimes \big[\, \hbar^{-1} \phi_2 \, ,
   f \,\big] \, + \, \big[\, \hbar^{-1} \phi_1 \, , f \,\big] \otimes \phi_2 \, + \, \cO(2)  \,\; =  \cr
   \hfill   = \;  \Delta(f) \, + \, \phi_1 \otimes \big[\, \hbar^{-1} \phi_2 \, , f \,\big] \, + \,
   \big[\, \hbar^{-1} \phi_1 \, , f \,\big] \otimes \phi_2 \, + \, \cO(2)  }  $$
 so that, eventually, we get in short
\begin{equation*}   \label{eq: DeltaF(f)}
  \Delta^\cF(f)  \,\; \underset{J_\hbar^{[\otimes 2\,|3]}}{\equiv} \;\,  \Delta(f) \, + \,
  \phi_1 \otimes \big[\, \hbar^{-1} \phi_2 \, , f \,\big] \, + \,
  \big[\, \hbar^{-1} \phi_1 \, , f \,\big] \otimes \phi_2
\end{equation*}
 Therefore, for  $ \, \nabla_{\!\cF} := \Delta^\cF - {\big( \Delta^\cF \big)}^{2{}1} \, $  we get
\begin{equation}   \label{eq: NablaF(f)}
 \begin{aligned}
   \nabla_{\!\cF}(f)  &  \,\; \underset{J_\hbar^{[\otimes 2\,|3]}}{\equiv} \;\,
   \Delta(f) \, + \, \phi_1 \otimes \big[\, \hbar^{-1} \phi_2 \, , f \,\big] \, +
   \, \big[\, \hbar^{-1} \phi_1 \, , f \,\big] \otimes \phi_2 \, -  \\
   &  \hskip21pt   - \, \Delta^{\text{op}}\!(f) \, - \, \big[\, \hbar^{-1} \phi_2 \, ,
   f \,\big] \otimes \phi_1 \, - \, \phi_2 \otimes \Big[\, \hbar^{-1} \phi_1 \, , f \,\Big]  \,\; =  \\
   &  \hskip33pt   = \;\,  \nabla(f) \, + \,
   \phi_1^{(a)} \otimes \hbar^{-1} \Big[\, \phi_2^{(a)} , \, f \,\Big] \, + \,
   \hbar^{-1} \Big[\, \phi_1^{(a)} , \, f \,\Big] \otimes \phi_2^{(a)}  \; =  \\
   &  \hskip33pt   = \;\,  \nabla(f) \, - \, \phi_1^{(a)} \otimes \hbar^{-1}
   \Big[\, f \, , \, \phi_2^{(a)} \Big] \, - \,
   \hbar^{-1} \Big[\, f \, , \, \phi_1^{(a)} \Big] \otimes \phi_2^{(a)}
 \end{aligned}
\end{equation}
 where we used notation  $ \; \phi_a := \phi - \phi_{\,2{}1} =
 \phi_1^{(a)} \otimes \phi_2^{(a)} \; $.
 When we reduce the last identity in  \eqref{eq: NablaF(f)}  modulo
 $ \, \hbar \, J_\hbar^{\otimes 2} \, $  we end up with
  $$
  \nabla_{\!\cF}\big(\,\overline{f}\,\big)  \, \underset{\liem^{[\otimes 2\,|3]}}{\equiv} \,
  \nabla\big(\,\overline{f}\,\big) \, - \, \overline{\phi_1^{(a)}} \otimes \Big\{\, \overline{f} \, , \,
  \overline{\phi_2^{(a)}} \,\Big\} \, - \, \Big\{\, \overline{f} \, , \, \overline{\phi_1^{(a)}} \,\Big\}
  \otimes \overline{\phi_2^{(a)}}
  $$
 hence reducing the latter modulo  $ \liem^{[\otimes 2\,|3]} $  we find in
 $ \; \liem^{\otimes 2} \! \Big/ \liem^{[\otimes 3]} = \lieg^* \otimes \lieg^* \, $
 the identity
  $$
  \displaylines{
   \quad   \Big(\, \nabla_{\!\cF} \text{\ mod\ } \liem^{[\otimes 2\,|3]} \Big)(\,\text{f}\,)  \,\; =
   \;\,  \delta(\,\text{f}\,) \, - \, c_{\,1} \otimes \big[\, \text{f} \, , \, c_{\,2} \big] \, -
   \, \big[\, \text{f} \, , \, c_{\,1} \big] \otimes c_{\,2}  \,\; =   \hfill  \cr
   \hfill   = \;\,  \delta(\,\text{f}\,) \, - \, \big(\,\ad(\,\text{f}\,)\big)(c)  \,\; = \;\,
   \big(\, \delta - \partial_c \big)(\,\text{f}\,)  \,\; = \;\,  \delta^{\,c}(\,\text{f}\,)   \quad  }
   $$
 which means that the induced Lie cobracket on
 $ \, \liem \Big/ \liem^2 = \lieg^* \, $  is just  $ \delta^{\,c} $,  \,q.e.d.
\epf

\begin{exa}  \label{example: qs-tw x QFSHA}
 Let  $ \, G := \textit{GL}_{\,n}(\k) \, $  be the general linear group over  $ \k \, $,
 and  $ \, \lieg := \mathfrak{gl}_n(\k) \, $.  We consider the QUEA
 $ \, \uhg = U_\hbar\big(\,\mathfrak{gl}_n(\k)\hskip-1pt\big) \, $  and the QFSHA
 $ \, \fhg = F_\hbar\big[\big[ \textit{GL}_{\,n}(\k) \big]\big] \, $  introduced in
 Example \ref{example: toral 2-cocycles for QFSHAs}.  Letting  $ \lieb^- $  and
 $ \lieb^+ $  be the Borel Lie subalgebras in  $ \lieg $  of lower triangular and upper
 triangular matrices, respectively, the subalgebra  $ \uhbm $  of  $ \uhg $
 generated by the  $ F_i $'s  and the  $ \varGamma_k $'s  is a QUEA for  $ \lieb^- \, $,
 while the subalgebra  $ \uhbp $  generated by the  $ E_i $'s  and the
 $ \varGamma_k $'s  is a QUEA for  $ \lieb^+ \, $
 --- both being also Hopf subalgebras of  $ \uhg \, $, indeed.  Dually, the QFSHA
 $ \, \fhbm = {\uhbm}^* \, $  identifies with the Hopf algebra quotient of  $ \fhg $
 obtained by modding out the ideal generated by the  $ x_{i,\,j} $'s  with  $ \, i < j \, $;
 similarly, the QFSHA  $ \, \fhbp = {\uhbp}^* \, $  identifies with the Hopf algebra quotient of
 $ \fhg $  obtained by modding out the ideal generated by the  $ x_{i,\,j} $'s  with
 $ \, i > j \, $.  Therefore, from the presentation of  $ \fhg $  in
 Example \ref{example: toral 2-cocycles for QFSHAs}  one deduces the following
 presentations for these quotient Hopf algebras:  $ \fhbm $  is generated by the entries
 of the ``lower triangular  $ q $--matrix''
 $ \, {\big( x^-_{i,\,j} \big)}_{i=1,\dots,n;}^{j=1,\dots,n;} \, $  with  $ \, x^-_{i,\,j} := x_{i,\,j} \, $
 for all  $ \, i \geq j \, $  and  $ \, x^-_{i,\,j} := 0 \, $  for all  $ \, i < j \, $,  \,and similarly
 $ \fhbp $  is generated by the entries of the ``upper triangular  $ q $--matrix''
 $ \, {\big( x^+_{i,\,j} \big)}_{i=1,\dots,n;}^{j=1,\dots,n;} \, $  where
 $ \, x^+_{i,\,j} := x_{i,\,j} \, $  when  $ \, i \leq j \, $  and  $ \, x^+_{i,\,j} := 0 \, $  for
 $ \, i > j \, $.
 \vskip5pt
   Now we consider a new group  $ \, G \, $  which is ``double version'' of
   $ \textit{GL}_{\,n}(\k) \, $,  in that it is a Manin double of  $ B^- $  and  $ B^+ \, $;
   its tangent Lie algebra  $ \lieg $  then is the Manin double of  $ \lieb^- $  and
   $ \lieb^+ \, $;  in particular,  $ \, G = B^- \times B^+ \, $  as algebraic varieties
   (not as groups), with  $ B^- $  and  $ B^+ $  being embedded as subgroups,
   whereas  $ \, \lieg = \lieb^- \oplus \lieb^+ \, $  as vector spaces, with  $ \lieb^- $  and
   $ \lieb^+ $  being embedded as Lie subalgebras (this case is explained in detail in
   \cite{GaGa2}  when  $ G $  is a ``double version'' of a semisimple (connected) group:
   $ \, G = \textit{SL}_{\,n}(\k) \, $  is such an example, and  $ \textit{GL}_{\,n}(\k) $
   is just a very slight variation of that).
                                                                       \par
   For these new  $ G $  and  $ \lieg $,  a QUEA  $ \uhg $ is defined as follows:
   it is the unital, associative,  $ \hbar $--adically  complete  $ \kh $--algebra  with generators
  $$
  F_1 \, , F_2 \, , \dots, F_{n-1} \, , \varGamma^-_1 \, , \varGamma^-_2 \, , \dots ,
  \varGamma^-_{n-1} \, , \varGamma^-_n \, , \varGamma^+_1 \, , \varGamma^+_2 \, ,
  \dots , \varGamma^+_{n-1} \, , \varGamma^+_n \, ,
  E_1 \, , E_2 \, , \dots , E_{n-1}
  $$
 and relations
  $$
  \displaylines {
   \big[\varGamma^\pm_k \, , \varGamma^\pm_\ell\big] = 0  \; ,  \quad
   \big[\varGamma^\pm_k \, , F_j\big] = - \delta_{k,j} \, F_j \; ,
   \quad  [\varGamma^\pm_k \, , E_j] = +\delta_{k,j} \, E_j  \; ,
   \quad  \big[\varGamma^\pm_k \, , \varGamma^\mp_\ell\big] = 0  \cr
   \big[E_i \, , F_j\big] \, = \, \delta_{i,j} \, {{\;\; e^{\hbar\,(\varGamma^+_i -
   \varGamma^+_{i+1})} - e^{\hbar\,(\varGamma^-_{i+1} - \varGamma^-_i)} \;\;} \over
   {\;\; e^{+\hbar} - e^{-\hbar} \;\;}}  \cr
   \hfill   \big[E_i \, , E_j\big] \, = \, 0  \;\; ,  \qquad  \big[F_i \, , F_j\big] \, = \, 0   \hfill
\forall \;\; i \, , j \;\colon \vert i - j \vert > 1 \, \phantom{.} \;  \cr
   \hfill   E_i^2 \, E_j - \left( q + q^{-1} \right) E_i \, E_j \, E_i + E_j \, E_i^2 \, = \, 0
\hfill  \forall \;\; i \, , j \;\colon \vert i - j \vert = 1 \, \phantom{.} \;  \cr
   \hfill   F_i^2 \, F_j - \left( q + q^{-1} \right) F_i \, F_j \, F_i + F_j \, F_i^2 \, = \, 0
\hfill  \forall \;\; i \, , j \;\colon \vert i - j \vert = 1 \, . \;  \cr }
$$
 where  $ \, [X\,,Y] := X\,Y - Y\,X \, $  again.
 The Hopf structure then is given by
  $$
  \displaylines {
   \Delta(F_i) \, = \, F_i \otimes e^{\hbar\,(\varGamma^-_{i+1} - \varGamma^-_i)} + 1 \otimes
F_i \; ,  \;\;\quad  S(F_i) \, = \, -F_i \, e^{\hbar\,(\varGamma^-_i - \varGamma^-_{i+1})}
\; ,  \;\;\;\quad  \epsilon(F_i) \, = \, 0  \cr
   \quad \,  \Delta\big(\varGamma^\pm_k\big) \, = \, \varGamma^\pm_k \otimes 1 + 1
   \otimes \varGamma^\pm_k \; ,  \;\quad \qquad \qquad  S\big(\varGamma^\pm_k\big) \, =
   \, -\varGamma^\pm_k \; ,  \;\; \quad \qquad  \epsilon\big(\varGamma^\pm_k\big) \, = \, 0  \cr
   \Delta(E_i) \, = \, E_i \otimes 1 + e^{\hbar\,(\varGamma^+_i - \varGamma^+_{i+1})}
   \otimes
E_i \; ,  \;\;\quad  S(E_i) \, = \, -e^{\hbar\,(\varGamma^+_{i+1} - \varGamma^+_i)} \, E_i \; ,
\;\;\quad
\epsilon(E_i) \, = \, 0  }
$$
 \vskip3pt
   In fact, this  $ \uhg $  can be realized as a  \textsl{quantum double\/}  of  $ \uhbm $  and
   $ \uhbp \, $:  \,in particular, then, we have a decomposition
   $ \, \uhg = \uhbm \,\widehat{\otimes}\, \uhbp \, $  as coalgebras.
   Dually, the latter implies that for the QFSHA  $ \, \fhg := {\uhg}^* \, $
   we have an identification  \textsl{as algebras}
  $$
  \fhg  \, = \,  {\big(\uhbm \,\widehat{\otimes}\, \uhbp\big)}^*  \, =
  \,  {\uhbm}^* \,\widetilde{\otimes}\, {\uhbp}^*  \, = \,  \fhbm \,\widetilde{\otimes}\, \fhbp
  $$
 As a consequence, exploiting the presentations of  $ \fhbm $  and  $ \fhbp $
 given above, we find in a nutshell a presentation for  $ \fhg $
 as being the algebra generated by the entries of the  ``$ q $--matrix  in blocks''
  $$
  \begin{pmatrix}
   \!\; X^+  &  \!\! \mathbf{0}_{n \times n} \,  \\
   \ \mathbf{0}_{n \times n}  &  \!\! X^- \,
      \end{pmatrix}  \qquad  \text{with}  \quad  X^\pm :=
      {\big( x^\pm_{i,\,j} \big)}_{i=1,\dots,n;}^{j=1,\dots,n;}
      \text{\ \ \ as defined above (triangular).}
      $$
 Moreover, explicit identifications  $ \, \fhg = {\uhg}^\ast \, $  and  $ \, \uhg = {\fhg}^\star \, $
 can be encoded in the Hopf pairing
 $ \; \langle\,\ ,\ \rangle : \fhg \times \uhg \relbar\joinrel\relbar\joinrel\longrightarrow \kh \; $
 given on generators by the following values:
\begin{equation}   \label{eq: pairing x fhg uhg double}
 \hskip-3pt
  \begin{aligned}
   \hskip-11pt   \Big\langle x^-_{i,\,j} \, , {\textstyle \prod_{k=1}^n}
   {\big( \varGamma^+_k \big)}^{g_k} \Big\rangle = 0 = \big\langle x^-_{i,\,j} \, ,
   E_t \big\rangle  \;\; ,  \!\!\quad  \big\langle x^+_{i,\,j} \, , F_t \big\rangle = 0 =
   \Big\langle x^+_{i,\,j} \, , {\textstyle \prod_{k=1}^n} {\big( \varGamma^-_k \big)}^{g_k}
   \Big\rangle   \hskip-15pt  \\
%%%
   \big\langle x^-_{i,\,j} \, , F_t \big\rangle = \delta_{i,j+1} \, \delta_{t,j}  \qquad ,
   \qquad \qquad  \big\langle x^+_{i,\,j} \, , E_t \big\rangle = \delta_{i+1,j} \, \delta_{i,t}
   \hskip39pt  \\
%%%
   \Big\langle x^+_{i,\,j} \, , {\textstyle \prod_{k=1}^n} {\big( \varGamma^+_k \big)}^{g_k}
   \Big\rangle \, = \, \delta_{i,j} \, (1- \delta_{g_i,0}) \, {\textstyle \prod_{k\not=i}} \delta_{g_k,0} \,
   = \, \Big\langle x^-_{i,\,j} \, , {\textstyle \prod_{k=1}^n} {\big( \varGamma^-_k \big)}^{g_k}
   \Big\rangle
%%%%%
  \end{aligned}
\end{equation}
 In particular, from the first line in  \eqref{eq: pairing x fhg uhg double}  note that if
 $ \underline{\varGamma}_1 $  and  $ \underline{\varGamma}_2 $
 are two monomials in the  $ \varGamma^\pm_k $'s,  then for all  $ \, i = 1 , \dots , n \, $
 we have
\begin{equation}  \label{eq: x_ii as as characters}
  \Big\langle x^\pm_{i,\,i} \, , \,
  \underline{\varGamma}_1 \cdot \underline{\varGamma}_2 \Big\rangle  \; = \;
  \Big\langle x^\pm_{i,\,i} \, , \,
  \underline{\varGamma}_1 \Big\rangle \cdot \Big\langle x^\pm_{i,\,i} \, , \,
  \underline{\varGamma}_2 \Big\rangle
\end{equation}
 \vskip5pt
   Thanks to  Proposition \ref{prop: duality-"polar"deforms}  later on, any polar twist for
   $ \fhg $  can be seen as a polar 2--cocycle for  $ \, \uhg = {\fhg}^\star \, $.
   Now, some examples of the latter were introduced in
   Example \ref{example: qs-2-coc x QUEA}  above for a large class of QUEA,
   including that for  $ \, \lieg = \mathfrak{sl}_n(\k) \, $.  The same procedure can be
   applied to the present case, which is a slight variation of that case applied to
   $ \mathfrak{gl}_n(\k) $  instead of  $ \mathfrak{sl}_n(\k) \, $,  as follows.
 \vskip3pt
   Let  $ \lieh $  be the  $ \kh $--span  of
   $ \, B_\varGamma := \big\{\, \varGamma^+_k , \varGamma^-_k \,\big|\, k=1,
   \dots,n \,\big\} \, $.
   Then fix an antisymmetric,  $ \kh $--bilinear  map
   $ \, \chi : \lieh \times \lieh \relbar\joinrel\longrightarrow \kh \, $
   whose matrix of values on pairs of elements in the  $ \kh $--basis  $ B_\varGamma $  is
   $ \, X = {\Big(\, \chi^{\varepsilon,\,\eta}_{k,\,t} =
   \chi\big( \varGamma^\varepsilon_k \, ,
   \varGamma^\eta_t \,\big) \!\Big)}_{k,t=1,\dots,n;}^{\varepsilon , \,
   \eta \in \{+\,,-\}} \in
   \lieso_{2n}\big(\kh\big) \, $.  Any such map  $ \chi $  also induces uniquely an
   antisymmetric,
   $ \kh $--bilinear  map  $ \, \tilde{\chi}_{\scriptscriptstyle U} \, $  on
   $ \, U_\hbar(\lieh) = \widehat{S}_\kh(\lieh) := \widehat{\bigoplus\limits_{n \in \NN}
   S_\kh^{\,n}(\lieh)} \, $  with values in  $ \kh \, $,
%%%%%
 by setting
\begin{equation}  \label{eq: chitilde-BIS}
  \begin{aligned}
    \tilde{\chi}_{\scriptscriptstyle U}(z,1) := \epsilon(z) =:
    \tilde{\chi}_{\scriptscriptstyle U}(1,z)
\qquad \qquad \qquad  \forall \;\; z \in \widehat{S}_\kh(\lieh)   \hskip45pt \qquad  \\
    \tilde{\chi}_{\scriptscriptstyle U}(x,y) := \chi(x,y)   \qquad \hskip55pt \qquad
\forall \;\; x, y \in  S_\kh^{\,1}(\lieh)   \qquad \hskip39pt  \\
   \hfill \qquad  \tilde{\chi}_{\scriptscriptstyle U}(x,y) := 0   \quad \qquad
\forall \;\; x \in S_\k^{\,r}(\lieh) \, ,   \; y \in S_\k^{\,s}(\lieh) \, :
\, r, s \geq 1 \, , \; r+s > 2   \hskip5pt \hfill
  \end{aligned}
\end{equation}
 Then we can define a  $ \kh $--linear  map
%%%%%
 $ \; \chi_{\scriptscriptstyle U} := e^{\hbar^{-1} 2^{-1} \tilde{\chi}_{{}_U}} =
 {\textstyle \sum_{m \geq 0}} \, \hbar^{-m} \, {\tilde{\chi}_{{}_U}}^{\,\ast\,m} \!
 \Big/ 2^m\,m! \; $
%%%%%
 from  $ \, U_\hbar(\lieh) \mathop{\widehat{\otimes}}\limits_\kh U_\hbar(\lieh) \, $
 to  $ \khp $,  \,which   --- like in  \cite[Lemma 5.2.2]{GaGa2}  ---
 happens to be a well-defined  \textsl{polar 2--cocycle\/}  for
 $ \, U_\hbar(\lieh) \, $.
                                                           \par
   Assume now that  $ \, \chi $  satisfies the additional constraint
   $ \; \chi(S_i \, ,\,-\,) \, = \, 0 \, = \, \chi(\,-\,,S_i) $  \; for all
   $ \, i \in I := \{1,\dots,n-1\} \, $,  \,where
   $ \, S_i := 2^{-1} \big(\, \varGamma^+_i - \varGamma^+_{i+1} +
   \varGamma^-_i - \varGamma^-_{i+1} \big) \, $  for all  $ \, i \in I \, $;
   note that requiring  $ \; \chi(S_i \, ,\,-\,) \, = \, 0 \; $  for all  $ \, i \in I \, $
   is enough (by antisymmetry), and the latter in turn is equivalent to requiring
   $ \; \chi^{+,\,\eta}_{i,\,t} - \chi^{+,\,\eta}_{i+1,\,t} + \chi^{-,\,\eta}_{i,\,t} -
   \chi^{-,\,\eta}_{i+1,\,t} = 0 \; $  for all  $ \, i \in I \, $,  all  $ \, t = 1, \dots, n-1 \,$
   and all $ \, \eta \in \{+\,,-\} \, $.  Then  $ \chi $  induces a unique  $ \kh $--bilinear
   map  $ \; \overline{\chi} : \overline{\lieh} \times \overline{\lieh}
   \relbar\joinrel\longrightarrow \kh \; $,  \;where
   $ \, \overline{\lieh} := \lieh \big/ \lies \, $  with
   $ \, \lies := \textsl{Span}_\kh\big( {\{\, S_i \,\}}_{i = 1, \dots, n-1;} \big)\, $,
   \,given by
  $ \; \overline{\chi}\,\big( H' \! + \lies \, , H'' \! + \lies \,\big) \, := \,
  \chi\big(H',H''\big) \; $
 for all  $ \, H' , H'' \in \lieh \, $.
                                                           \par
   Now repeat the above construction with  $ \overline{\lieh} $  and
   $ \overline{\chi} $  replacing  $ \lieh $  and  $ \chi \, $:  \,this yields a
   polar  $ 2 $--cocycle  $ \overline{\chi}_{\scriptscriptstyle U} $  for
   $ U_\hbar\big(\,\overline{\lieh}\,\big) \, $.  But now the additional assumption on
   $ \chi $  implies that there exists a  \textsl{Hopf algebra epimorphism}
%%%
 $ \; \pi : \uhg \relbar\joinrel\relbar\joinrel\twoheadrightarrow
 \widehat{S}_\kh\big(\,\overline{\lieh}\,\big) \cong U_\hbar\big(\,\overline{\lieh}\,\big) \; $
%%%
 given by  $ \, \pi(E_i) := 0 \, $,  $ \, \pi(F_i) := 0 \, $   --- for  $ \, i = 1, \dots, n-1 \, $
 ---  and  $ \, \pi(T) := (T + \lies) \, \in \, \overline{\lieh} \subseteq
 U_\hbar\big(\,\overline{\lieh}\,\big) \, $   --- for  $ \, T \in \lieh \, $.  Finally, we set
 $ \; \sigma_\chi := \overline{\chi}_{{}_U} \circ (\pi \times \pi) \, $,
 \,which is a well-defined polar  $ 2 $--cocycle  for  $ \uhg \, $,
 \,again in the sense of  Definition \ref{def: polar 2cocycle}.  Note that
  $$  \sigma_\chi := \overline{\chi}_{{}_U} \circ (\pi \times \pi) = \exp\big( \hbar^{-1} \, 2^{-1} \,
  \widetilde{\overline{\chi}}_{{}_U} \big) \circ (\pi \times \pi) =
  \exp\big( \hbar^{-1} \, 2^{-1} \, \widetilde{\overline{\chi}}_{{}_U} \circ (\pi \times \pi) \big)
  $$
 \vskip5pt
   Now let us re-think the polar 2--cocycle  $ \sigma_\chi $  for  $ \uhg $
   as a polar twist for  $ \fhg \, $.  First of all,
%
%%%%%
% for  $ \lieh $  we have the QFSHA  $ \, F_\hbar[[H]] := {U_\hbar(\lieh)}^* \, $  which
% can be described as the algebra of formal series in  $ \hbar $  and in the variables
% $ \dot{x}^{\,\varepsilon}_i $'s  dual to the  $ \varGamma^\varepsilon_i $'s.  It is clear
% then that  $ F_\hbar[[H]] $  is a quotient of  $ \fhg \, $,  with projection  $ \, p \colon
% \fhg \relbar\joinrel\relbar\joinrel\twoheadrightarrow F_\hbar[[H]] \, $  being dual to the
% natural embedding  $ \, U_\hbar(\lieh) \lhook\joinrel\relbar\joinrel\longrightarrow \uhg \, $;
% \,explicitly, this quotient map is given by  $ \, p\big(x^{\,\varepsilon}_{i,i}\big) =
% \dot{x}^{\,\varepsilon}_i \, $  for all  $ \, i = 1 , \dots , n \, $.
%                                                                         \par
%    Second, the identification  $ \, F_\hbar[[H]] := {U_\hbar(\lieh)}^* \, $  induces
% $ \, {F_\hbar[[H]]}^{\widetilde{\otimes}\, 2} = {\Big( {U_\hbar(\lieh)}^{\widehat{\otimes}\, 2}
% \Big)}^* \, $.  Therefore, since by construction we have  $ \; \big\langle\, \dot{x}^{\,\varepsilon}_i
% \, , \varGamma^{\,\eta}_j \,\big\rangle = \delta_{\varepsilon,\,\eta} \, \delta_{i,\,j} \, $,  \,
%%%%%
%
 comparing  \eqref{eq: x_ii as as characters}  and  \eqref{eq: chitilde-BIS}
 we deduce that the form  $ \, \widetilde{\overline{\chi}}_{{}_U} \circ (\pi \times \pi) \, $
 in  $ {\Big( {\uhg}^{\widehat{\otimes}\, 2} \Big)}^* $  identifies with
%%%%%
 $ \; \varPhi_\chi \, := \,
%
%%%
% \sum\limits_{\substack{k,\,t=1  \\  \varepsilon,\,\eta \in \{+,-\}}}^n
%%%
%
 \sum\limits_{k,\,t=1}^n \sum\limits_{\varepsilon,\,\eta \in \{+,-\}} \!
 \chi_{k,\,t}^{\,\varepsilon,\,\eta} \, y^{\,\varepsilon}_{k,\,k} \otimes y^{\,\eta}_{t,\,t} \; $
%%%%%
 in  $ {\fhg}^{\widetilde{\otimes}\, 2} \, $,  \,where
 $ \, y^{\,\varsigma}_{\ell,\ell} := \log\big( x^{\,\varsigma}_{\ell,\ell} \big) \, $
 is a well-defined element in  $ \fhg \, $.

 Then, exponentiating yields
  $$
  \cF_\chi  \; := \;  \sigma_\chi  \; = \;  \exp\big( \hbar^{-1} \, 2^{-1} \,
  \widetilde{\overline{\chi}}_{{}_U} \circ (\pi \times \pi) \big)  \; = \;
  \exp\big( \hbar^{-1} \, 2^{-1} \, \varPhi_\chi \big)
  $$
 which is exactly the polar twist of  $ \fhg $  we were looking for.
 \vskip5pt
   We can also check directly that this  $ \cF_\chi $
   is indeed a polar twist by an explicit computation.
   We see this in the simplest case, when  $ \, n = 2 \, $;
   \,the other cases are quite similar,
   but require quite another shot of calculations.
                                                                             \par
   We need to compute the coproduct of the  $ x^{\,\varepsilon}_{t,\,t} $'s  in  $ \fhg \, $,
   which is defined (by construction) by the condition
   $ \; \big\langle\, \Delta\big( x^{\,\varepsilon}_{t,\,t} \big) , A \otimes Z \,\big\rangle \, = \,
   \big\langle\, x^{\,\varepsilon}_{t,\,t} \, , A \cdot Z \,\big\rangle \; $  for all  $ \, A \, , Z \in \uhg \, $;  \,since  $ \uhg $  admits the PBW-type basis
  $$  \mathcal{B}  \; := \;  \Big\{\, F^f \, {\big(\varGamma_1^-\big)}^{g_1^-} {\big(\varGamma_2^-\big)}^{g_2^-} {\big(\varGamma_1^+\big)}^{g_1^+} {\big(\varGamma_2^+\big)}^{g_2^+} E^e \,\Big|\, f, g_1^- , g_2^- , g_1^+ , g_2^+ , e \in \NN \,\Big\}  $$
 we can replace  $ A $ and  $ Z $  with any two PBW monomials from  $ \mathcal{B} \, $.  Now, let us say that a PBW monomial of the form  $ \, \mathcal{M} = F^f \, {\big( \varGamma_1^- \big)}^{g_1^-} {\big( \varGamma_2^- \big)}^{g_2^-} {\big( \varGamma_1^+ \big)}^{g_1^+} {\big( \varGamma_2^- \big)}^{g_2^+} E^e \, $  belongs to the root space  $ \, (e-\!f) \, \alpha \, $.  Then root/weight considerations easily show that  $ \; \big\langle\, x^{\,\varepsilon}_{t,\,t} \, , \mathcal{M} \,\big\rangle \neq 0 \; $  only for PBW monomials  $ \, \mathcal{M} \, $  in the root space 0, i.e.\ such that  $ \, e = f \, $.  A straightforward computation gives
  $$  E^e \cdot F^f  \; = \;  \sum\limits_{s=0}^{e \wedge f} {\big([s]_q!\big)}^2 \, {\bigg[{e \atop s}\bigg]}_{\!q} \, {\bigg[{f \atop s}\bigg]}_{\!q} \, F^{f-s} \, K_{e,f}(s) \, E^{e-s}  $$
 where  $ \, q := \exp(\hbar) \, $,  $ \, [r]_q := \displaystyle{{\,q^r - q^{-r}\,} \over {\,q - q^{-1}\,}} \, $,  $ \, [m]_q! := \prod\limits_{r=1}^m [r]_q \, $,  $ \, \displaystyle{{\bigg[{\ell \atop s}\bigg]}_{\!q} := {{\,[\ell]_q!\,}
 \over {\,[s]_q! \, [\ell-s]_q!\,}}} \, $  and
%%%%%
  $$  K_{e,f}(s)
%%%
% \bigg[{{K \, ; \, c\,} \atop s}\bigg]}_{\!q}
%%%
 := \prod\limits_{r=1}^s {{\,\;q^{\,2s-e-f+1-r} K_+^{+1} - q^{\,r-1-2s+e+f\,} K_-^{-1}\;\,}
 \over {\,q^{\,r} - q^{\,-r}\,}}  $$
 with  $ \, K_+^{+1} := 1 \otimes {\exp\big(\! + \! \hbar \,
 \big( \varGamma_i^+ - \varGamma_{i+1}^+ \big) \big)}^{\pm 1} \, $  and
 $ \, K_-^{-1} := {\exp\big(\! - \! \hbar \, \big( \varGamma_i^- -
 \varGamma_{i+1}^- \big) \big)}^{\pm 1} \otimes 1 \, $.
 Then the product of two PBW monomials expands into
  $$
  \displaylines{
   \mathcal{M}' \cdot \mathcal{M}''  \; =   \hfill  \cr
   \quad   = \;  F^{f'} {\big(\varGamma_1^-\big)}^{\dot{g}_1^-}
   {\big(\varGamma_2^-\big)}^{\dot{g}_2^-} {\big(\varGamma_1^+\big)}^{\dot{g}_1^+}
   {\big(\varGamma_2^+\big)}^{\dot{g}_2^+} E^{e'} \, \cdot \,
   F^{f''} {\big(\varGamma_1^-\big)}^{\ddot{g}_1^-}
   {\big(\varGamma_2^-\big)}^{\ddot{g}_2^-} {\big(\varGamma_1^+\big)}^{\ddot{g}_1^+}
   {\big(\varGamma_2^+\big)}^{\ddot{g}_2^+} E^{e''}  \; =   \hfill  \cr
   = \;  \sum\limits_{s=0}^{e' \wedge f''} {\big([s]_q!\big)}^2 \,
   {\bigg[{e \atop s}\bigg]}_{\!q} \, {\bigg[{f \atop s}\bigg]}_{\!q} \,
   F^{f'} \underline{\varGamma}^{\,\underline{\dot{g}}} \, F^{f''-s} \,
   K_{e',f''}(s) \, E^{e'-s} \, \underline{\varGamma}^{\,\underline{\ddot{g}}} \, E^{e''}  \; =
   \quad \qquad  \cr
   \hfill   = \;  \sum\limits_{s=0}^{e' \wedge f''}
   q^{\,\mathcal{E}(\,\dot{g},\,f'',\,s\,,e',\,\ddot{g})} \, {\big([s]_q!\big)}^2 \,
   {\bigg[{e' \atop s}\bigg]}_{\!q} \, {\bigg[{f'' \atop s}\bigg]}_{\!q} \, F^{f'+f''-s}
   \underline{\varGamma}^{\,\underline{\dot{g}}} \, K_{e',f''}(s) \,
   \underline{\varGamma}^{\,\underline{\ddot{g}}} \, E^{e'-s+e''}  }
   $$
 where
%%%
  $ \, \underline{\varGamma}^{\,\underline{\dot{g}}} :=
  {\big(\varGamma_1^-\big)}^{\dot{g}_1^-} {\big(\varGamma_2^-\big)}^{\dot{g}_2^-}
  {\big(\varGamma_1^+\big)}^{\dot{g}_1^+}
{\big(\varGamma_2^+\big)}^{\dot{g}_2^+} \, $,
%%%
  $ \, \underline{\varGamma}^{\,\underline{\ddot{g}}} :=
  {\big(\varGamma_1^-\big)}^{\ddot{g}_1^-} {\big(\varGamma_2^-\big)}^{\ddot{g}_2^-}
  {\big(\varGamma_1^+\big)}^{\ddot{g}_1^+}
{\big(\varGamma_2^+\big)}^{\ddot{g}_2^+} \, $
 while  $ \, \mathcal{E}\big(\,\dot{g}\,,f'',s\,,e',\ddot{g}\big) \in \ZZ \, $
 is a suitable exponent.  When we expand the last expression w.r.t.\ the PBW basis
 $ \mathcal{B} \, $,  the part given by a linear combination of PBW monomials
 in the root space 0 is
  $$  {\big( \mathcal{M}' \cdot \mathcal{M}'' \big)}_0  \; = \;  \delta_{f',\,0} \, \delta_{f'',\,e'} \,
  \delta_{e'',\,0} \, q^{\,\mathcal{E}(\,\dot{g},\,e',\,e',e',\,\ddot{g})} \, {\big([e']_q!\big)}^2 \,
  \underline{\varGamma}^{\,\underline{\dot{g}}} \, K_{e',\,e'}\big(e'\big) \,
  \underline{\varGamma}^{\,\underline{\ddot{g}}}  $$
 Eventually, tiding everything up we find
  $$
  \displaylines{
   \quad   \Big\langle \Delta\big( x^{\,\varepsilon}_{t,\,t} \big) \, ,
   \mathcal{M}' \otimes \mathcal{M}'' \Big\rangle \; = \;
   \Big\langle\, x^{\,\varepsilon}_{t,\,t} \, , \mathcal{M}' \cdot \mathcal{M}'' \Big\rangle \; =
   \;  \Big\langle\, x^{\,\varepsilon}_{t,\,t} \, ,
   {\big( \mathcal{M}' \cdot \mathcal{M}'' \big)}_0 \,\Big\rangle  \; =   \hfill  \cr
   \hfill   = \;  \delta_{f',\,0} \, \delta_{f'',\,e'} \,
   \delta_{e'',\,0} \, q^{\,\mathcal{E}(\,\dot{g},\,e',\,e',\,e',\,\ddot{g})} \,
   {\big([e']_q!\big)}^2 \, \Big\langle\, x^{\,\varepsilon}_{t,\,t} \, , \,
   \underline{\varGamma}^{\,\underline{\dot{g}}} \, K_{e',\,e'}\big(e'\big) \,
   \underline{\varGamma}^{\,\underline{\ddot{g}}} \,\Big\rangle   \quad  }
   $$
 Now, a similar analysis yields (notation being obvious, hopefully)
  $$
  \Big\langle\, x^{\,\varepsilon}_{t,\,t} \, , \, \underline{\varGamma}^{\,\underline{\dot{g}}} \,
  K_{e',\,e'}\big(e'\big) \, \underline{\varGamma}^{\,\underline{\ddot{g}}} \,\Big\rangle  \; = \;
   {\Big\langle\, x^{\,\varepsilon}_{t,\,t} \, , \, \underline{\varGamma} \,
   \Big\rangle}^{\,\underline{\dot{g}}} \cdot
   \Big\langle\, x^{\,\varepsilon}_{t,\,t} \, , \, K_{e',\,e'}\big(e'\big) \Big\rangle \cdot
   {\Big\langle\, x^{\,\varepsilon}_{t,\,t} \, , \,
   \underline{\varGamma} \,\Big\rangle}^{\,\underline{\ddot{g}}}
   $$
 and finally direct computation gives
 $ \; \Big\langle\, x^{\,\varepsilon}_{t,\,t} \, , \, K_{e',\,e'}\big(e'\big) \Big\rangle \,
 = \, \delta_{e',0} \; $,  \,which is a direct consequence of the assumption
 $ \; \chi(S_i \, ,\,-\,) \, = \, 0 \, = \, \chi(\,-\,,S_i) \; $  for  $ \, i = 1,\dots,n-1 \, $.
 Therefore, we end up with
  $$
  \Big\langle \Delta\big( x^{\,\varepsilon}_{t,\,t} \big) \, ,
  \mathcal{M}' \otimes \mathcal{M}'' \Big\rangle \; \not= \;  0
  $$
 only for monomials with  $ \, \big(f',e'\big) = (0,0) = \big(f'',e''\big) \, $,
 \,and for them we have
  $$
  \displaylines{
   \quad   \Big\langle \Delta\big( x^{\,\varepsilon}_{t,\,t} \big) \, ,
   \mathcal{M}' \otimes \mathcal{M}'' \Big\rangle  \, = \,
   \Big\langle \Delta\big( x^{\,\varepsilon}_{t,\,t} \big) \, ,
   \underline{\varGamma}^{\,\underline{\dot{g}}} \otimes
   \underline{\varGamma}^{\,\underline{\ddot{g}}} \Big\rangle  \, =   \hfill  \cr
   \hfill   = \,  \Big\langle x^{\,\varepsilon}_{t,\,t} \, ,
   \underline{\varGamma}^{\,\underline{\dot{g}}} \cdot
   \underline{\varGamma}^{\,\underline{\ddot{g}}} \Big\rangle  \, = \,
   \Big\langle x^{\,\varepsilon}_{t,\,t} \, ,
   \underline{\varGamma}^{\,\underline{\dot{g}}} \Big\rangle \cdot
   \Big\langle x^{\,\varepsilon}_{t,\,t} \, , \underline{\varGamma}^{\,\underline{\ddot{g}}}
   \Big\rangle  \, = \,  \Big\langle x^{\,\varepsilon}_{t,\,t} \, , \mathcal{M}'
   \Big\rangle \cdot \Big\langle x^{\,\varepsilon}_{t,\,t} \, , \mathcal{M}''
   \Big\rangle   \quad  }
   $$
 which in short means  $ \; \Delta\big( x^{\,\varepsilon}_{t,\,t} \big) \, =
 \, x^{\,\varepsilon}_{t,\,t} \otimes x^{\,\varepsilon}_{t,\,t} \; $,  \,i.e.\
 $ x^{\,\varepsilon}_{t,\,t} $  is group-like.  Therefore, for
 $ \, y^{\,\varepsilon}_{t,\,t} := \log\big( x^{\,\varepsilon}_{t,\,t} \big) \, $
 instead we have
%%%
 $ \; \Delta\big( y^{\,\varepsilon}_{t,\,t} \big) \, = \,
 y^{\,\varepsilon}_{t,\,t} \otimes 1 + 1 \otimes y^{\,\varepsilon}_{t,\,t} \; $,
 \,i.e.\  $ y^{\,\varepsilon}_{t,\,t} $  is primitive.
%%%
                                                     \par
   Eventually, as all the  $ y^{\,\varepsilon}_{t,\,t} $'s  are primitive,
 a trivial computation shows that  $ \cF_\chi $  does obey condition
   \eqref{eq: twist-cond.'s},  hence it is indeed a polar twist, as claimed.
\end{exa}

\vskip13pt

\subsection{Duality issues}  \label{subsec: duality issues}  {\ }
 \vskip7pt
   The procedures of deformation by twist or by  $ 2 $--cocycle,
   both for Lie bialgebras and for Hopf algebras, are dual
   to each other, in the sense of  Proposition \ref{prop: duality-deforms x LbA's}
   and  Proposition \ref{prop: duality-deforms}.
   Because of this, we are lead to compare these two procedures \textsl{before\/}
   and  \textsl{after\/}  specialization, as follows.

\begin{prop}  \label{prop: F=sigma before-&-after}
 Let  $ \uhg $  and  $ \fhg $  be a QUEA and a QFSHA which are dual to each other,
 that is  $ \, \fhg = {\uhg}^* \, $  and  $ \, \uhg = {\fhg}^\star \, $.
 Then let  $ \cF $  be a twist for  $ \uhg \, $,  and  $ \sigma $  be a
 $ 2 $--cocycle  for
 $ \fhg \, $.  Assume that both  $ \cF $  and  $ \sigma $  are trivial modulo
 $ \hbar \, $,  \,so that there exists a corresponding twist  $ \, c_{{}_\cF} $
 for
 $ \, \lieg $
 (induced by  $ \cF $  via  Theorem \ref{thm: twist-deform-QUEA})  and a corresponding
 $ 2 $--cocycle  $ \, \zeta_\sigma $  for  $ \, \lieg^* $  (induced by  $ \sigma $  via
 Theorem \ref{thm: 2cocycle-deform-QFSHA}).  Finally, we identify twists for  $ \uhg $
 and  $ 2 $--cocycles  for  $ \fhg $ via  Proposition \ref{prop: duality-deforms},
 and similarly twists for  $ \, \lieg $  and  $ 2 $--cocycles  for  $ \, \lieg^* $
%%%%%%%%%
   \hbox{via  Proposition \ref{prop: duality-deforms x LbA's}.}
%%%%%%%%%
%%%%%%%%%
                                                                              \par
   Then the following holds: if  $ \, \cF = \sigma \, $,  \,then
   $ \, c_{{}_\cF} = \zeta_\sigma \, $.   \hfill  $ \square $
\end{prop}

   The  \textit{proof\/}  of the above claim is trivial
   --- just track the whole construction of both  $ \, c_{{}_\cF} $  and
   $ \, \zeta_\sigma \, $,  \,and compare the outcomes.

\vskip7pt

   A similar duality result holds for deformations by polar  $ 2 $--cocycles  and deformations by polar twists.
   Indeed, let us first notice that, by the very definitions, if  $ \uhg $  and  $ \fhg $  are a QUEA and a QFSHA in duality   --- i.e.,  $ \, \fhg = {\uhg}^* \, $  and  $ \, \uhg = {\fhg}^\star \, $  ---   then  \textit{any polar  $ 2 $--cocycle  for  $ \uhg $  is automatically a polar twist for  $ \fhg \, $,  \,and viceversa}   --- see  Proposition \ref{prop: duality-"polar"deforms}  later on.
%%%%%
% \footnote{\color{blue} \, \dbend  \ --- \
%    Here this statement is dramatically vague...!  BUT later on, in  \S \ref{subsec: deform.'s-QDP},
% we will explain better this fact, showing that  \, ``polar 2--cocycle for the QUEA  $ \uhg $ =
% 2-cocycle for the QFSHA  $ {\uhg}' \, $  trivial modulo  $ \hbar \, $'' \,  and  \, ``polar twist
% for the QFSHA  $ \fhg $ = twist for the QUEA  $ {\fhg}^\vee $  trivial modulo  $ \hbar \, $''
%    --- hence the statement that we are presenting here is just a consequence of
% Proposition \ref{prop: duality-deforms}  for 2-cocycles and twists which are trivial
% modulo  $ \hbar \, $.  ---  \textbf{15 February 2024:} I do confirm, but I have to
% modify a bit the presentation of that part.}
%%%%%%%%%
%%%%%%%%%
                                                                            \par
   Once this is settled, next result (which mirrors  Proposition \ref{prop: F=sigma before-&-after}  above) holds too, whose proof again follows by direct comparison of the two deformation procedures (just tracking the whole construction of  $ \, \gamma_\sigma $  and  $ \, c_{{}_\cF} \, $,  \,and comparing the outcomes):

\vskip9pt

\begin{prop}  \label{prop: polar F=sigma before-&-after}
 Let  $ \uhg $  and  $ \fhg $  be a QUEA and a QFSHA which are dual to each other,
 that is
 $ \, \fhg = {\uhg}^* \, $  and  $ \, \uhg = {\fhg}^\star \, $.
 Then let  $ \sigma $  be a polar  $ 2 $--cocycle  for  $ \uhg \, $,  and  $ \cF $  be a polar twist for
 $ \fhg \, $.
 Let  $ \, \gamma_\sigma $  be the  $ 2 $--cocycle  for  $ \, \lieg $
 induced by  $ \sigma $
 via  Theorem \ref{thm: propt.'s qs-2cocycle-deform-QUEA},  and let
 $ \, c_{{}_\cF} $
 be the twist for  $ \, \lieg^* $  induced by  $ \cF $  via
 Theorem \ref{thm: propt.'s qs-twist-deform-QFSHA}.
 Finally, we identify polar  $ 2 $--cocycles  for  $ \uhg $  and polar twists
for  $ \fhg $  as mentioned above, and similarly we identify twists for
$ \, \lieg $  and
$ 2 $--cocycles  for  $ \, \lieg^* $
%%%%%%%%%
   \hbox{via  Proposition \ref{prop: duality-deforms x LbA's}.}
%%%%%%%%%
Then the following holds: if  $ \, \sigma = \cF \, $,  \,then
   $ \, \gamma_\sigma = c_{{}_\cF} \, $.   \hfill  $ \square $
\end{prop}

\bigbreak

\section{Deformations vs.\ QDP}  \label{sec: deform.'s-QDP}

   When we interchange (contravariantly) QUEA's and QFSHA's via linear duality,
   the interaction of such a process with deformation processes (either by twist
   or by 2-cocycle) is clear.  By  Proposition \ref{prop: duality-deforms},
   any twist, resp.\ 2-cocycle, for a given quantum algebra
   (a QUEA or QFSHA, say) is automatically a 2-cocycle, resp.\ a twist,
   for its dual (a QFSHA or a QUEA, respectively),
   and the corresponding deformation processes ``commute'' with dualization.
                                                                           \par
   On the other hand, in this section we investigate how the deformation procedures
   interact when we interchange (covariantly) QUEA's and QFSHA's via Drinfeld's
   functors, as in  Theorem \ref{thm: QDP}.  In other words, we analyze how deformation
   procedures behave with respect to the Quantum Duality Principle.
%
%%%%%
% %
%  \vskip5pt
% %
%    We begin analyzing the case of the functor  $ \, \fhg \mapsto {\fhg}^\vee \, $,
% as it is definitely the easier to work with, due its more explicit description.
%%%%%
%

\subsection{Some auxiliary results}  \label{subsec: auxiliary results}  {\ }
 \vskip7pt
   We begin with a key observation, which shows a crucial fact:
   our ``polar 2--cocycles'' for any QUEA and ``polar twists''
   for any QFSHA are actually  \textsl{standard\/}
   2--cocycles and twists, respectively, for the QFSHA and
   for the QUEA that are associated with the original quantum
   group via Drinfeld's functors from the QDP.
                                                                           \par
   Here is the precise result:

\vskip9pt

\begin{lema}
\label{lemma: polar 2-cocycle = 2-cocycle & polar twist x Fhg = twist}{\ }
 \vskip3pt
   \textit{(a)\,  Let  $ \uhg $  be a QUEA, and let  $ {\uhg}' $
   be its associated QFSHA following  Theorem \ref{thm: QDP}.
   Let  $ \, \sigma := \exp_*\!\big( \hbar^{-1} \chi \big) \, $  be a
   \textsl{polar 2--cocycle}  for  $ \uhg $  as in  Definition \ref{def: polar 2cocycle}.
   Then the restriction  $ \sigma{\big|}_{{\uhg}' \times {\uhg}'} $  of
   $ \, \sigma $  to  $ \, {\uhg}' \times {\uhg}' \, $  is a well-defined,
   $ \kh $--valued  bilinear form on  $ \, {\uhg}' \, $,  of the form
   $ \; \sigma' = \exp_*\!\big( \hbar^{+1} \chi' \big) \, $  with
   $ \, \chi' := {\big( \hbar^{-2} \, \chi \big)}{\big|}_{{\uhg}' \times {\uhg}'} \, $,
   \,and this  $ \, \sigma' := \exp_*\!\big( \hbar^{+1} \chi' \big) \, $  is a
   \textsl{2-cocycle}  for  $ \, {\uhg}' \, $.}
 \vskip3pt
   \textit{(b)\,  Let  $ \fhg $  be a QFSHA, and let  $ {\fhg}^\vee $
   be its associated QUEA following  Theorem \ref{thm: QDP}.  Let
   $ \, \cF := \exp\!\big( \hbar^{-1} \phi \big) \, $  be a
   \textsl{polar twist}  for  $ \fhg $
   as in  Definition \ref{def: polar twist}.
   Then  $ \, \cF := \exp\!\big( \hbar^{-1} \phi \big) =
   \exp\!\big( \hbar^{+1} \phi^\vee \big) \, $  with
   $ \, \phi^\vee := \hbar^{-2} \phi
   \in {\big( {\fhg}^\vee \big)}^{\widehat{\otimes}\,2} \, $,
   \,and, in these terms,
   $ \, \cF^\vee := \exp\!\big( \hbar^{+1} \phi^\vee \big) \, $
   is a  \textsl{twist}  for  $ \, {\fhg}^\vee \, $.}
\end{lema}

\begin{proof}
   \textit{(a)}\,  We retain notation from the proof of
   Lemma \ref{lemma: properties polar 2-cocycle},
   and we proceed along the same lines.
   Thus we set  $ \, U_\hbar := \uhg \, $  and
   $ \, J_\hbar := \Ker\big(U_\hbar\big) \, $,  \,and we write
\begin{equation*}   %  \label{eq: zhat_z+ - BIS}
  \hat{z} := \epsilon(z) \, ,  \;\;\; z^+ := z -
  \epsilon(z) = z - \hat{z} \, \in \, J_\hbar \; ,
  \quad  \text{hence}  \quad  z = z^+ + \hat{z}
  \qquad  \forall \;\; z \in U_\hbar \;\;
\end{equation*}
 We already saw that
 $ \; \chi(u,v) \, = \, \chi\big( u^+ , v^+ \big) \; $  for all  $ \, u , v \in U_\hbar \, $,
 and then for any  $ \, a \, , b \in U_\hbar \, $  we have
\begin{equation}  \label{eq: expansionsigma(a,b) - BIS}
  \sigma(a\,,b\,)  \; = \;  {\textstyle \sum_{n \,\geq 0}} \,
  \hbar^{-n} \, {\textstyle \prod_{i=1}^n}
  \chi\big(a^+_{(i)},b^+_{(i)}\big) \Big/ n!
\end{equation}
 where  $ \, \otimes_{i=1}^n a^+_{(i)} = \delta_n(a) \, $  and
 $ \, \otimes_{i=1}^n b^+_{(i)} = \delta_n(b) \, $.
                                             \par
   Now, restricting to  $ U_\hbar^{\;\prime} $  we get that
   $ \, a', b' \in U_\hbar^{\;\prime} \, $  yields
   $ \, \delta_n\big(a'\big) , \delta_n\big(b'\big) \in \hbar^n \,
   {U_\hbar}^{\!\!\widehat{\otimes}\,n} \, $;
   \,also, in the sequel we can clearly assume
   $ \, \epsilon\big(a'\big) = 0 = \epsilon\big(b'\big) \, $.  Then we get
  $$
  {\textstyle \prod_{i=1}^n}
  \chi'\big(a^{\prime\,+}_{(i)},b^{\prime\,+}_{(i)}\big) \, =
  \, {\textstyle \prod_{i=1}^n} \hbar^{-2}
  \chi\big(a^{\prime\,+}_{(i)},b^{\prime\,+}_{(i)}\big)
  \; \in \;  \hbar^{-2n} \, \hbar^{2n} \, \kh  \; = \;  \kh  \;\;\;\;
  \eqno \forall \;\; n \in \NN_+
  $$
 whence, just like in  \eqref{eq: expansionsigma(a,b) - BIS},
 we eventually get
\begin{equation*}  % \label{eq: expansionsigma'(a,b)}
  \sigma'\big(a',b'\,\big)  \; = \;
  {\textstyle \sum_{n \,\geq 0}} \, \hbar^{+n} \, {\textstyle \prod_{i=1}^n}
  \chi'\big(a^{\prime\,+}_{(i)},b^{\prime\,+}_{(i)}\big) \Big/ n!
  \; \in \;  \kh   \qquad  \forall \; a' , b' \in U_\hbar^{\;\prime}
\end{equation*}
 which proves the claim.
 \vskip5pt
   \textit{(b)}\,  This follows directly from
   Definition \ref{def: polar twist}.
\end{proof}

\vskip3pt

   As a direct consequence, we have the following significant result:

\vskip5pt

\begin{prop}  \label{prop: duality-"polar"deforms}
 Let\/  $ U_\hbar $  be a QUEA and  $ F_\hbar $  be a QFSHA that are dual to each other, i.e.\ such that  $ \, F_\hbar = {(U_\hbar)}^\ast \, $  and  $ \, U_\hbar = {(F_\hbar)}^\star \, $.  Then:
 \vskip2pt
   (a)\;\;  $ \sigma $  is a polar 2--cocycle for\/
   $ U_\hbar $  $ \, \iff \, $  $ \sigma $
   is a polar twist for\/  $ F_\hbar \; $;
 \vskip2pt
   (b)\;\;  $ \cF $  is a polar twist for\/
   $ F_\hbar $  $ \, \iff \, $  $ \cF $
   is a polar 2--cocycle for\/  $ U_\hbar \; $.
\end{prop}

\begin{proof}
 The proof follows directly from the very definitions
 of polar 2--cocycle and polar twist, along with the observation that
 $ \, F_\hbar = {(U_\hbar)}^\ast \, $  and
 $ \, U_\hbar = {(F_\hbar)}^\star \, $  imply
 $ \, F_\hbar^{\,\vee} = {(U_\hbar^{\,\prime})}^\star \, $
 and  $ \, U_\hbar^{\,\prime} = {\big(F_\hbar^{\,\vee}\big)}^\ast \, $,
 \,by  \eqref{eq: QDP-duality}.
\end{proof}

\vskip13pt

\subsection{Drinfeld's functors and ``polar deformations''}
\label{subsec: Drinf-funct.'s & polar deform.'s}  {\ }
 \vskip7pt
   In this subsection we analyze the interaction between the
   process of ``polar deformation'' and the action of a Drinfeld's functor on
   some quantum group.
                                                           \par
   We begin with deformations by polar twist of a QFSHA and the application
   to the latter of Drinfeld's functor  $ {(\ )}^\vee \, $;
   \,then we shall look at deformation by polar 2--cocycle of a QUEA and the
   application to the latter of Drinfeld's functor  $ {(\ )}'\, $.

\vskip11pt

\begin{free text}  \label{vee-functor & polar twist}
 \textbf{Deformations by polar twist under $ \, \fhg \mapsto \fhg^\vee \, $.}
 We look now what happens with deformations by polar twist for a QFSHA
 when the latter is acted upon by the functor  $ {(\ )}^\vee $  which associates with it a
 QUEA.  Here is our result:
\end{free text}

\vskip9pt

\begin{theorem}  \label{thm: polar twist-deform vs Drinfeld functor x QFSHA}  {\ }
 \vskip3pt
   Let\/  $ \fhg $  be a QFSHA.  Let  $ \, \cF = \exp\!\big( \hbar^{-1} \phi \big) \, $
   be a polar twist for  $ \fhg \, $,  with  $ \, \phi \in {\fhg}^{\widetilde{\otimes}\, 2} \, $
   (cf.\ Definition \ref{def: polar twist}).
%%%%
 Set  $ \, \phi^\vee := \hbar^{-1} \, \log(\cF\,) = \hbar^{-2} \phi \, $,  $ \, \phi_a :=
 \phi - \phi_{\,2,1} \; $  and  $ \; \phi^\vee_a := \phi^\vee - \phi^\vee_{\,2,1} \; $.
 Then we have:
 \vskip5pt
   (a)\;  $ \phi $  is antisymmetric, i.e.\  $ \, \phi_{\,2,1} = -\phi \, $,  \,iff\/  $ \cF $
   is orthogonal, i.e.\  $ \, \cF_{2,1} = \cF^{-1} \, $,  \,iff\/  $ \phi^\vee $
   is antisymmetric, i.e.\  $ \, \phi^\vee_{\,2,1} = -\phi^\vee \, $;
 \vskip4pt
   (b)\;  $ \, \cF = \exp\!\big( \hbar \, \phi^\vee \big) \, $
   is a twist element for the QUEA  $ \, \uhgs := \fhg^\vee \, $.
 \vskip4pt
   (c)\;  Let  $ \, c \, $  be the antisymmetric twist of\/  $ \lieg^* $
   corresponding to  $ \cF $  as provided by
   Theorem \ref{thm: propt.'s qs-twist-deform-QFSHA},  and let  $ \, c^\vee $
   be the similar twist provided by  Theorem \ref{thm: twist-deform-QUEA}
   along with claim  \textit{(b)\/}  above.  Then  $ \; c = c^\vee \; $.
\end{theorem}

\pf
 \textit{(a)}\,  This follows directly by construction.
 \vskip9pt
   \textit{(b)}\,  This is granted by
   Lemma \ref{lemma: polar 2-cocycle = 2-cocycle & polar twist x Fhg = twist}\textit{(b)}.
 \vskip9pt
   \textit{(c)}\,  This follows by a careful   --- yet entirely straightforward ---
   check, just tracking both constructions involved (of  $ c $
   and of  $ c^\vee $  alike).
\epf

\vskip9pt

\begin{free text}  \label{prime-functor & polar 2-cocycles}
 \textbf{Deformations by polar 2--cocycle under  $ \, \uhg \mapsto \uhg' \, $.}
 Given a QUEA, we can apply on it Drinfeld's functor  $ {(\ )}' \, $;
 \,we now see what happens when a deformation by polar 2--cocycle is performed.
 Our result reads as follows:
\end{free text}

\vskip7pt

\begin{theorem}  \label{thm: polar 2cocycle-deform vs Drinfeld functor x QUEA}  {\ }
 \vskip3pt
   Let\/  $ \uhg $  be a QUEA.  Let  $ \, \sigma = \exp_*\!\big( \hbar^{-1} \chi \big) \, $
   be a polar 2--cocycle for\/  $ \uhg \, $,  with
   $ \, \chi \in {\big( \uhg^{\widehat{\otimes}\, 2} \,\big)}^* \, $
   (cf.\ Definition \ref{def: polar 2cocycle}).
%%%%
 Set  $ \, \chi' := \hbar^{-1} \, \log_*(\sigma) = \hbar^{-2} \, \chi \, $,
 $ \, \chi_a := \chi - \chi_{\,2,1} \; $  and  $ \; \chi'_a := \chi' - \chi'_{2,1} \; $.
 Then we have:
 \vskip3pt
   (a)\;  $ \chi $  is antisymmetric, i.e.\  $ \, \chi_{\,2,1} = -\chi \, $,  \,iff\/  $ \sigma $
   is orthogonal, i.e.\  $ \, \sigma_{2,1} = \sigma^{-1} \, $,  \,iff\/  $ \chi' $
   is antisymmetric, i.e.\  $ \, \chi'_{2,1} = -\chi' \, $;
 \vskip1pt
   (b)\;  $ \, \sigma = \exp_*\!\big( \hbar \, \chi' \big) \, $  is a 2-cocycle for the QFSHA
   $ \, F_\hbar\big[\big[G^*\big]\big] := \uhg' \, $.
 \vskip1pt
   (c)\;  Let  $ \, \gamma \, $  be the antisymmetric 2-cocycle of\/
   $ \lieg $  corresponding to  $ \sigma $  as provided by
   Theorem \ref{thm: polar 2cocycle-deform-QUEA},  and let
   $ \, \gamma' $  be the similar 2-cocycle provided by
   Theorem \ref{thm: 2cocycle-deform-QFSHA}  along with claim  \textit{(b)\/}  above
   --- using the identification  $ \, \lieg^{**} = \lieg \, $.
   Then  $ \; \gamma = \gamma' \; $.
\end{theorem}

\pf
 \textit{(a)}\,  This follows directly by construction.
 \vskip3pt
   \textit{(b)}\,  This is true because of
   Lemma \ref{lemma: polar 2-cocycle = 2-cocycle & polar twist x Fhg = twist}\textit{(a)}.
 \vskip3pt
   \textit{(c)}\,  Much like for
   Theorem \ref{thm: polar twist-deform vs Drinfeld functor x QFSHA}\textit{(c)},
   this follows again by a straightforward check, just carefully tracking both
   constructions involved (of  $ \gamma $
   and of  $ \gamma $  alike).
\epf

\vskip5pt

\begin{rmk}  \label{rmk: duality Drinf-fctr.s-vs-polar-deform.'s}
 Recall that the notions of twist and that of 2-cocycle are dual to each other
 (cf.\ Proposition \ref{prop: duality-deforms}),
 and the same holds for those of polar twist and polar 2--cocycle
 (cf.\ \S \ref{subsec: duality issues}).  Moreover, Drinfeld's functors are also
 dual to each other, in the sense of  \eqref{eq: QDP-duality}.
 Taking all this into account, it turns out easily that
 Theorem \ref{thm: polar twist-deform vs Drinfeld functor x QFSHA}  and
 Theorem \ref{thm: polar 2cocycle-deform vs Drinfeld functor x QUEA}
 above are also ``dual to each other'', in that either one of these two statements
 can be deduced from the other by a duality argument.
\end{rmk}

\vskip13pt

\subsection{Drinfeld's functors and (standard) deformations}
\label{subsec: Drinf-funct.'s & standard deform.'s}  {\ }
 \vskip7pt
   We analyze now the interaction between the process of deformation
   --- in the standard sense ---   and the action of a Drinfeld's functor on some
   quantum group.
                                                           \par
   We begin with deformations by twist of a QUEA and the application to the latter of
   Drinfeld's functor  $ {(\ )}' \, $;  \,then we shall look at deformation by 2-cocycle of any
   QUEA, to which we apply Drinfeld's functor  $ {(\ )}'\, $.  In both cases (obviously
   related by duality) we have to adopt a slightly stronger assumption, namely that the
   given twist, resp.\ 2-cocycle, is additionally an  $ R $--matrix,  resp.\ a
   $ \varrho $--comatrix.

\vskip11pt

\begin{free text}  \label{prime-functor & twist}
 \textbf{Deformations by twist under  $ \, \uhg \mapsto \uhg' \, $.}
 As a first step, we look what happens for deformations by twist of a QUEA
 when the latter is acted upon by the functor  $ {(\ )}' $  which associates with
 it a QFSHA.  It turns out that we find a relevant result when we make the stronger
 assumption that the given twist is in fact a  \textsl{(quantum)  $ R $--matrix  twist},
 as in  Definition \ref{def: twist & R-mat.'s}\textit{(d)}.  Here is our result:
\end{free text}

\vskip3pt

\begin{theorem}  \label{thm: twist-deform vs Drinfeld functor x QUEA}  {\ }
 \vskip3pt
   Let\/  $ \uhg $  be a QUEA, and let\/  $ \cF $  be a twist for\/  $ \uhg $  s.t.\
   $ \; \phi \equiv 1 \; \big(\, \text{\rm mod} \; \hbar \, \uhg^{\widehat{\otimes}\, 2} \,\big) $;
   \,then  $ \, \phi := \hbar^{-1} \log(\cF\,) \in \uhg^{\widehat{\otimes}\, 2} \, $,  \,and
   $ \, \cF = \exp(\hbar\,\phi) \, $.  Set also
   $ \, \phi' := \hbar^{-2} \log(\cF\,) = \hbar^{-2} \phi \, $,  $ \; \phi_a := \phi - \phi_{2,1} \; $
   and  $ \; \phi'_a := \phi' - \phi'_{2,1} \; $.
%%%%
 Assume in addition that  $ \cF $  is indeed a  \textsl{(quantum)  $ R $--matrix  twist},
 as in  Definition \ref{def: twist & R-mat.'s}\textit{(d)}.
%%%%
%
 Then we have:
 \vskip5pt
   (a)\;  $ \phi $  is antisymmetric, i.e.\  $ \, \phi_{2,1} = -\phi \, $,  \,iff\/  $ \cF $
   is orthogonal, i.e.\  $ \, \cF_{2,1} = \cF^{-1} \, $,  \,iff\/  $ \phi' $  is antisymmetric,
   i.e.\  $ \, \phi'_{2,1} = -\phi' \, $;
 \vskip3pt
   (b)\;  $ \, \cF = \exp\big( \hbar^{-1} \phi' \big) \, $  is a polar twist for the QFSHA
   $ \, F_\hbar\big[\big[G^*\big]\big] := \uhg' \, $.
 \vskip3pt
   (c)\;  Let  $ \, c \, $  be the antisymmetric twist of\/  $ \lieg $  corresponding to
   $ \cF $  as provided by  Theorem \ref{thm: twist-deform-QUEA},  and let
   $ \, c^{\,\prime} $  be the similar twist provided by
   Theorem \ref{thm: propt.'s qs-twist-deform-QFSHA}  along with claim
   \textit{(b)\/}  above   --- using the identification  $ \, \lieg^{**} = \lieg \, $.
   Then  $ \; c = c^{\,\prime} \; $.
\end{theorem}

\pf
 \textit{(a)}\,  This follows directly by construction.
 \vskip3pt
   \textit{(b)}\,  This is proved in  \cite[Theorem 0.1]{EH}.
 Note that the overall assumption there is that  $ \cR $  be an  $ R $--matrix,
 in the standard sense   --- so that  $ \uhg $  is quasitriangular.
 \textsl{Nevertheless},  \textsl{all the arguments used there to
 prove the main result only apply the defining properties of an
 ``$ R $--matrix''  in the sense of  Definition \ref{def: twist & R-mat.'s},
 namely  \eqref{eq: R-mat_properties}  and the right-hand side of
 \eqref{eq: twist-cond.'s};  the assumption  \eqref{eq: quasi-cocommutative}, instead,
 is  \textit{never}  used}.  Therefore, the same arguments,
 and the whole proof, used in  \cite{EH}  to prove Theorem 0.1
 actually do prove also the present statement, that is actually stronger.
 \vskip3pt
   \textit{(c)}\,  Here again, the proof follows from a straightforward,
   careful checking procedure, keeping track of both constructions involved
   (of  $ c $  and of  $ c^{\,\prime} $  alike), much like for
   Theorem \ref{thm: polar twist-deform vs Drinfeld functor x QFSHA}\textit{(c)\/}
   and for
   Theorem \ref{thm: polar 2cocycle-deform vs Drinfeld functor x QUEA}\textit{(c)}.
\epf

\vskip7pt

\begin{free text}  \label{vee-functor & 2-cocycles}
 \textbf{Deformations by 2-cocycle under  $ \, \fhg \mapsto \fhg^\vee \, $.}
 As a second step, we look what happens to deformations of a QFSHA by 2-cocycle when we apply Drinfeld's functor  $ {(\ )}^\vee \, $.  Here again, we get a relevant result under the stronger assumption that the given 2-cocycle is in fact a  \textsl{(quantum)  $ \varrho $--comatrix  2-cocycle},  as in  Definition \ref{def: 2-cocyc.'s & rho-comat.'s}\textit{(d)}.  Our result reads as follows:
\end{free text}

\vskip3pt

\begin{theorem}  \label{thm: 2cocycle-deform vs Drinfeld functor x QFSHA}  {\ }
 \vskip3pt
   Let\/  $ \fhg $  be any QFSHA, and let  $ \sigma $  be a 2-cocycle for  $ \fhg $
   such that  $ \; \sigma \equiv 1 $\allowbreak
$ \; \Big(\, \text{\rm mod} \; \hbar \,
{\big( \fhg^{\widetilde{\otimes}\, 2} \,\big)}^{\!\star} \,\Big) \, $;  \,then
$ \, \varsigma := \hbar^{-1} \log(\sigma) \in
{\big( \fhg^{\widetilde{\otimes}\, 2} \,\big)}^{\!\star} \, $,  \,and
$ \, \sigma = \exp\!\big( \hbar\,\varsigma \big) \, $.  Set also
$ \, \varsigma^\vee := \hbar \, \log(\sigma) = \hbar^2 \, \varsigma \, $,
$ \; \varsigma_a := \varsigma - \varsigma_{\,2,1} \; $  and
$ \; \varsigma^\vee_a := \varsigma^\vee - \varsigma^\vee_{\,2,1} \; $.
%%%%
 Assume in addition that  $ \sigma $  is a  \textsl{(quantum)  $ \varrho $--comatrix
 2-cocycle},  as in  Definition \ref{def: 2-cocyc.'s & rho-comat.'s}\textit{(d)}.
%%%%
                                                          \par
   Then the following holds true:
 \vskip5pt
   (a)\;  $ \varsigma $  is antisymmetric, i.e.\  $ \, \varsigma_{\,2,1} = -\varsigma \, $,
   \,iff\/  $ \sigma $  is orthogonal, i.e.\  $ \, \sigma_{2,1} = \sigma^{-1} \, $,
   \,iff\/  $ \varsigma^\vee $  is antisymmetric, i.e.\
   $ \, \varsigma^\vee_{\,2,1} = -\varsigma^\vee \, $;
 \vskip3pt
   (b)\;  $ \, \sigma = \exp\!\big( \hbar^{-1} \varsigma^\vee \big) \, $
   is a polar 2--cocycle for the QUEA  $ \, \uhgs := \fhg^\vee \, $;
 \vskip3pt
   (c)\;  Let  $ \, \gamma \, $  be the antisymmetric 2-cocycle of\/  $ \lieg^* $  corresponding to  $ \sigma $  as provided by  Theorem \ref{thm: 2cocycle-deform-QFSHA},  and let  $ \, \gamma^\vee $  be the similar 2-cocycle provided by  \ref{thm: polar 2cocycle-deform-QUEA}  along with claim  \textit{(b)\/}  above.  Then  $ \; \gamma = \gamma^\vee \; $.
\end{theorem}

\pf
 \textit{(a)}\,  This is obvious, by standard identities for formal exponentials.
 \vskip3pt
   \textit{(b)}\,
 This claim is the dual to
 Theorem \ref{thm: twist-deform vs Drinfeld functor x QUEA}\textit{(b)},
 so it follows from that one via a duality argument   --- involving the results in
 \S \ref{subsec: auxiliary results}, in particular
 Proposition \ref{prop: duality-"polar"deforms}.
 \vskip3pt
   \textit{(c)}\,  Once more, as in previous cases, the claim follows from direct
   checking, keeping track of the two involved 2-cocycles   ---  $ \gamma $
   and of  $ \gamma^\vee $  ---   were constructed.
\epf

\vskip7pt

\begin{rmk}  \label{rmk: duality Drinf-fctr.s-vs-deform.'s}
 Much like as we did in  Remark \ref{rmk: duality Drinf-fctr.s-vs-polar-deform.'s},
 we notice here as well that   --- by the same reasons as before ---
 Theorem \ref{thm: twist-deform vs Drinfeld functor x QUEA}  and
 Theorem \ref{thm: 2cocycle-deform vs Drinfeld functor x QFSHA}
 above are once more ``dual to each other'', in that either one of these two statements
 can be deduced from the other by a duality argument.
\end{rmk}

\vskip15pt

\section{Morphisms in the ``(co)quasitriangular'' case}  \label{sec: (co)quasitriangular}  {\ }
   In this section we focus onto  \textsl{$ R $--matrices\/}  and  \textsl{$ \varrho $--comatrices}.
                                                                           \par
   First, we investigate what happens with  \textsl{$ R $--matrices\/}  and
   \textsl{$ \varrho $--comatrices\/}  with respect to the Quantum Duality Principle,
   namely when Drinfeld's functors are applied.  This leads us to introduce the weaker notions of ``polar  $ R $--matrix''  and ``polar  $ \rho $--comatrix''  which eventually will prove quite significant.
                                                                           \par
   Second, we consider the standard constructions providing morphisms between a
   Hopf algebra  $ H $  and its dual when an  $ R $--matrix  or a  $ \varrho $--comatrix
   is available, looking at how this applies to QUEAs and to QFSHAs.
   It turns out that the construction then improves, in a way that involves Drinfeld's
   functors, again; moreover, we prove that it can also be extended to the case when
   $ R $--matrices  or  $ \rho $--comatrices  are replaced with their ``polar'' counterparts.

\vskip11pt

 \subsection{$ R $--matrices  and  $ \varrho $--comatrices  w.r.t.\ QDP:
 polar (co)matrices}  \label{subsec: R-mat&rho-comat vs. QDP}  {\ }
 \vskip7pt
   In next two results, we explain how  $ R $--matrices  and  $ \rho $--comatrices
   ``behave well'' with respect to Drinfeld's functors and the Quantum Duality Principle.
   In fact, this leads us to introduce the notions of  \textsl{``polar  $ R $--matrix''\/}  and
   of  \textsl{``polar  $ \rho $--comatrix''},  which are straight analogue of the notions of
   ``polar twist''and of ``polar 2--cocycle''.
 \vskip3pt
   We begin introducing some more bare definitions:

\begin{definition}  \label{def: polar R-matrix}
 {\ }
 \vskip3pt
    \textit{(a)}\,  Let  $ \fhg $  be a QFSHA.  We call  \textsl{``polar  $ R $--matrix''}  of
    $ \fhg $  any  $ R $--matrix  $ \cR $  for  $ \, \uhgs := {\fhg}^\vee \, $  such that
    $ \; \cR \, \equiv \, 1^{\otimes 2} \;\; \text{mod} \,
    \Big( \hbar \, {\fhg}^\vee \,\widehat{\otimes}\, {\fhg}^\vee \Big) \; $.
 \vskip3pt
    \textit{(b)}\,  Let  $ \uhg $  be a QUEA.  We call  \textsl{``polar  $ \rho $--comatrix''}
    of  $ \uhg $  any  $ \varrho $--comatrix  $ \rho \, $  for  $ \, \fhgs := {\uhg}' \, $
    such that  $ \; \rho \, \equiv \, \epsilon^{\otimes 2} \;\; \text{mod} \,
    \Big( \hbar \, {\big( {\uhg}' \,\widetilde{\otimes}\, {\uhg}' \,\big)}^\star \Big) \; $.
\end{definition}

\vskip3pt

\begin{rmk}  \label{rmk: duality polar R-mat/polar rho-comat}
 In the same spirit of  Proposition \ref{prop: duality-deforms}  and \eqref{eq: QDP-duality},  it is clear that the two notions of ``polar  $ R $--matrix''  and of ``polar  $ \varrho $--comatrix''  are dual to each other.  Namely, let  $ \uhg $  and  $ \fhg $  be a QUEA and a QFSHA which are dual to each other, i.e.\  $ \, \fhg = {\uhg}^* \, $  and  $ \, \uhg = {\fhg}^\star \, $.  Then, through the identification  $ \, \fhg \,\widetilde{\otimes}\, \fhg = {\big( \uhg \,\widehat{\otimes}\, \uhg \big)}^* \, $,  \,we have that  \textit{$ \, \cR = \rho \, $  is a polar  $ R $--matrix  of  $ \fhg $  if and only if it is a polar  $ \varrho $--comatrix  of  $ \, \uhg \, $}.
\end{rmk}

\vskip3pt

\begin{obs's}  \label{obs's: about-def_polar-gad}  {\ }
 \vskip1pt
   \textit{(a)}\,  With assumptions as in  Definition \ref{def: polar R-matrix}\textit{(a)\/}
   above, let  $ \cR $  be any polar  $ R $--matrix  for a QFSHA  $ \fhg \, $:  \,since
   $ \; \cR \, \equiv \, 1^{\otimes 2} \;\; \text{mod} \, \hbar \; $,  \,we can write  $ \cR $
   in the form  $ \; \cR \, = \, \exp\!\big( \hbar^{+1} \, \theta \big) \; $  for some
   $ \, \theta \in {\fhg}^\vee \,\widehat{\otimes}\, {\fhg}^\vee \, $.
                                                                 \par
   Similarly, if  $ \rho $  is any polar  $ \varrho $--comatrix  for a QUEA  $ \uhg \, $,
   \,then we can write it in the form
   $ \; \rho \, = \, \exp_*\!\big( \hbar^{+1} \, \varsigma \big) \; $  for some
   $ \, \varsigma \in {\big( {\uhg}' \,\widetilde{\otimes}\, {\uhg}' \,\big)}^\star \, $.
 \vskip1pt
   \textit{(b)}\,  Note that in the very definition of ``polar  $ R $--matrix'',
   resp.\ of ``polar  $ \varrho $--comatrix'',  we assume a condition which is quite close,
   yet  \textsl{weaker},  than the one demanded for the definition of
   ``polar twist'', resp.\ of ``polar 2--cocycle'', in  Definition \ref{def: polar 2cocycle},
   resp.\ in  Definition \ref{def: polar twist}.  In fact, our choice for these definitions about
   $ R $--matrices  and  $ \varrho $--comatrices  is motivated by
   Proposition \ref{prop: R(rho)-(co)matrices}  below,
   which eventually implies that  \textsl{when the two setups overlap, the stronger
   condition for twists/2--cocycles actually holds true}
   --- cf.\  Theorem \ref{thm: twist-deform vs Drinfeld functor x QUEA}  and
   Theorem \ref{thm: 2cocycle-deform vs Drinfeld functor x QFSHA}.
\end{obs's}

\vskip3pt

   The key result about polar  $ R $--matrices  and polar  $ \varrho $--comatrices  is the following:

\vskip3pt

\begin{prop}  \label{prop: R(rho)-(co)matrices}  {\ }
 \vskip5pt
   (a)\,  Let  $ \, \uhg $  be a QUEA, and let  $ \, {\uhg}' $  be the QFSHA associated to it by the Quantum Duality Principle, as in  Theorem \ref{thm: QDP}.  Then for any  $ R $--matrix  of  $ \, \uhg $  of the form  $ \, \cR = \exp\!\big( \hbar\,\theta \big) \, $,  \,with  $ \, \theta = \hbar^{-1} \log\,(\cR) \in {\uhg}^{\widehat{\otimes}\, 2} \, $,  \,we have
  $$  \vartheta  \, := \,  \hbar^2\,\theta  \, = \,  \hbar^{+1} \log\,(\cR)  \; \in \;  {\big( {\uhg}' \,\big)}^{\widetilde{\otimes}\, 2}  $$
 \vskip5pt
   (b)\,  Let  $ \, \fhg $  be a QFSHA, and let  $ \, {\fhg}^\vee $  be the QUEA associated to it by the Quantum Duality Principle, as in  Theorem \ref{thm: QDP}.  Then for any  $ \varrho $--comatrix  of  $ \, \fhg $  of the form  $ \, \rho = \exp_*\!\big( \hbar\,\varsigma \big) \, $,  \,with  $ \, \varsigma = \hbar^{-1} \log_*(\rho) \in {\Big( {\fhg}^{\widehat{\otimes}\, 2} \Big)}^{\!*} \, $,  \,we have
  $$  \zeta  \, := \,  \hbar^2\,\varsigma  \, = \,  \hbar^{+1} \log_*(\rho)  \; \in \;  {\Big( {\big( {\fhg}^\vee \,\big)}^{\widehat{\otimes}\, 2} \Big)}^{\!*}  $$
\end{prop}

\begin{proof}
 \textit{(a)}\,  This is proved in  \cite[Theorem 0.1]{EH}.  Indeed, the overall assumption there is that  $ \cR $  be an  $ R $--matrix,  in the standard sense   --- so that  $ \uhg $  is quasitriangular.  \textsl{Nevertheless},  \textsl{all the arguments used there to prove the main result only apply the defining properties of an  ``$ R $--matrix''  in the sense of  Definition \ref{def: twist & R-mat.'s},  namely  \eqref{eq: R-mat_properties}  and the right-hand side of  \eqref{eq: twist-cond.'s};  the assumption  \eqref{eq: quasi-cocommutative}, instead,  is  \textit{never}  used}.  Therefore, the same arguments, and the whole proof, used in  \cite{EH}  to prove Theorem 0.1 actually do prove also the present, stronger statement.
 \vskip3pt
   \textit{(b)}\,  This follows from claim  \textit{(a)},
   by duality, using the duality relation  \eqref{eq: QDP-duality}, the fact that
 $ \, {\fhg}^\vee \, $  is a QUEA when  $ \fhg $  is a QFSHA, and
 Proposition \ref{prop: duality-deforms}.
\end{proof}

   The previous result has the following, important consequence:

\begin{theorem}  \label{thm: (co)mat vs. polar (co)mat}  {\ }
 Let  $ \, \uhg $  be a QUEA and  $ \, \fhg $  be a QFSHA; let  $ \, \fhgs := {\uhg}' $
 be the QFSHA and  $ \, \uhgs := {\fhg}^\vee $  be the QUEA provided by the
 Quantum Duality Principle, as in  Theorem \ref{thm: QDP}.  Then the following holds:
 \vskip3pt
   (a)\,  every  $ R $--matrix  for  $ \, \uhg $  which is congruent to  $ 1^{\otimes 2} $
   modulo  $ \hbar $  is a polar  $ R $--matrix  for  $ \, {\uhg}' \, $;
 \vskip3pt
   (b)\,  every  $ \varrho $--comatrix  for  $ \, \fhg $  which is congruent to
   $ \epsilon^{\otimes 2} $  modulo  $ \hbar $  is a polar  $ \varrho $--comatrix  for
   $ \, \fhg^\vee \, $;
 \vskip3pt
   (c)\,  every  $ \varrho $--comatrix  for  $ \, {\uhg}' $  which is congruent to
   $ \epsilon^{\otimes 2} $  modulo  $ \hbar $  is a polar  $ \varrho $--comatrix  for
   $ \, \uhg $  itself;
 \vskip3pt
   (d)\,  every  $ R $--matrix  for  $ \, {\fhg}^\vee $  which is congruent to
   $ 1^{\otimes 2} $  modulo  $ \hbar $  is a polar  $ R $--matrix  for  $ \, \fhg $  itself.
\end{theorem}

\begin{proof}
 Claims  \textit{(a)\/}  and  \textit{(b)\/}  follow directly from
 Proposition \ref{prop: R(rho)-(co)matrices}, claims  \textit{(a)\/}  and  \textit{(b)},
 respectively.  Now claim  \textit{(c)\/}  follows from claim  \textit{(b)\/}
 applied to the QFSHA  $ \, F_\hbar\big[\big[G^*\big]\big] := {\uhg}' \, $,
 \,since  $ \, {F_\hbar\big[\big[G^*\big]\big]}^\vee \! = {\big( {\uhg}' \,\big)}^\vee \! = \uhg \, $
 by  Theorem \ref{thm: QDP}\textit{(a)}.  Similarly, claim  \textit{(d)\/}
 follows from  \textit{(a)\/}  applied to the QUEA
 $ \, U_\hbar\big(\lieg^*\big) := {\fhg}^\vee \, $,  \,since
 $ \, {U_\hbar\big(\lieg^*\big)}' = {\big( {\fhg}^\vee \,\big)}' = \fhg \, $
 by  Theorem \ref{thm: QDP}\textit{(a)\/}  again.
\end{proof}

\vskip13pt

\subsection{Morphisms from  $ R \, $--matrices  and  $ \varrho \, $--comatrices}
\label{subsec: morph.'s_R-mat/rho-comat}
  {\ }
 \vskip11pt
   We shall now explore what happens if we apply the general constructions leading to
   Proposition \ref{prop: morph.'s from R-mat},  resp.\ to
   Proposition \ref{prop: morph.'s from rho-comat},
   is (tentatively) applied to a QUEA, resp.\ a QFSHA, as the Hopf algebra $ H $  to start with,
   and a suitable  $ R $--matrix,  resp.\  $ \varrho $--comatrix,  for it.
   To begin with, we check that  Proposition \ref{prop: morph.'s from R-mat}
   still makes sense when  $ \, H := \uhg \, $  is a QUEA:

\vskip9pt

\begin{prop}  \label{prop: Phi from R x QUEA}
 Let  $ \, \uhg \, $  be a QUEA, and  $ \, \fhg := {\uhg}^* \, $  be its dual QFSHA, as in
 \S \ref{equiv-&-(stand)-duality}.  Let  $ \; \cR = \cR_1 \otimes \cR_2 \, $  (in Sweedler's notation) be
 an  $ R $--matrix for  $ \uhg \, $.  Then there exist two morphisms of topological Hopf algebras
  $$
  \displaylines{
   \overleftarrow{\Phi}_{\!\cR} : \fhg := {\uhg}^* \!\relbar\joinrel\relbar\joinrel\longrightarrow
   \uhg^{\text{\rm cop}}  \;\; ,  \qquad  \eta \, \mapsto \, \overleftarrow{\Phi}_{\!\cR}(\eta)
   := \eta(\cR_1) \, \cR_2  \cr
   \overrightarrow{\Phi}_{\!\cR} : \fhg :=
   {\uhg}^* \!\relbar\joinrel\relbar\joinrel\longrightarrow \uhg^{\text{\rm op}}  \;\; ,
   \qquad  \eta \, \mapsto \, \overrightarrow{\Phi}_{\!\cR}(\eta) := \cR_1 \, \eta(\cR_2)  }
   $$
\end{prop}

\begin{proof}
 This is straightforward.  Let us write the given  $ R $--matrix  as an
 $ \hbar $--adically  convergent series
 $ \; \cR = \sum_{n \geq 0} \hbar^n \, \cR'_n \otimes \cR''_n \; $
 (where each  $ \, \cR'_n \otimes \cR''_n \, $  is written in Sweedler's notation).  Then
%%%
 $ \,\; \overleftarrow{\Phi}_{\!\cR}(\eta) \, := \, \eta(\cR_1) \, \cR_2 \, = \,
 \sum_{n \geq 0} \hbar^n \, \eta\big(\cR'_n\big) \otimes \cR''_n \;\, $
%%%
 for every  $ \, \eta \in {\uhg}^* \, $,  where the last term is a well-defined,  $ \hbar $--adically
 convergent series in  $ \uhg \, $,  \,so that  $ \, \overleftarrow{\Phi}_{\!\cR}(\eta) \in \uhg \, $
 and the map  $ \, \overleftarrow{\Phi}_{\!\cR} \, $  is well-defined.
 The properties making this map  $ \overleftarrow{\Phi}_{\!\cR} $
 into a (topological) Hopf morphism then are proved via a formal check, just like for
 Proposition \ref{prop: morph.'s from R-mat}.  The proof for  $ \overrightarrow{\Phi}_{\!\cR} $
 is pretty similar.
\end{proof}

   Dually,  Proposition \ref{prop: morph.'s from rho-comat}
   still makes sense when  $ \, H := \fhg \, $  is a QFSHA:

\begin{prop}  \label{prop: Psi from rho x QFSHA}
 Let  $ \, \fhg \, $  be a QFSHA, and  $ \, \uhg := {\fhg}^\star \, $
 be its dual QUEA, as in  \S \ref{equiv-&-(stand)-duality}.
 Let  $ \, \rho \, $  be a  $ \varrho $--comatrix for  $ \fhg \, $.
 Then there exist two morphisms of topological Hopf algebras
  $$
  \displaylines{
   \overleftarrow{\Psi}_{\!\rho} : \fhg \!\relbar\joinrel\relbar\joinrel\longrightarrow
   {\big( {\fhg}^\star \big)}^{\text{\rm cop}} = \uhg^{\text{\rm cop}}  \;\; ,
   \qquad  \ell \, \mapsto \, \overleftarrow{\Psi}_{\!\rho}(\ell) := \rho(\ell \, , - )  \cr
   \overrightarrow{\Psi}_{\!\rho} :
   \fhg \!\relbar\joinrel\relbar\joinrel\longrightarrow{\big( {\fhg}^\star \big)}^{\text{\rm op}} =
   \uhg^{\text{\rm op}}  \;\; ,  \qquad
   \ell \, \mapsto \, \overrightarrow{\Psi}_{\!\rho}(\ell) := \rho( - , \ell\,)  }
   $$
\end{prop}

\begin{proof}
 Again, this is straightforward, the main point being to show that the given maps
 are indeed well-defined.  We consider  $ \overleftarrow{\Psi}_{\!\rho} \, $,
 the case of  $ \overrightarrow{\Psi}_{\!\rho} $  is similar.
                                                                     \par
   By definition,  $ \, \rho \in {\big( \fhg \,\widetilde{\otimes}\, \fhg \big)}^\star \, $,
   the latter being the space of all continuous  $ \kh $--linear  maps from
   $ \, \fhg \,\widetilde{\otimes}\, \fhg \, $  to  $ \, \kh \, $.
   Then for any  $ \, \ell \in \fhg \, $  the ``left-hand evaluation'' of  $ \rho $  in  $ \ell $
   yields clearly a continuous  $ \kh $--linear  map
   $ \; \overleftarrow{\Psi}_{\!\rho} := \rho(\ell \, , - ) : \fhg \longrightarrow \kh \, , \,
   f \mapsto \rho(\ell \, , f) \; $;  \;thus  $ \overleftarrow{\Psi}_{\!\rho} $
   is well defined, q.e.d.
                                                                     \par
   The Hopf properties of this map  $ \overleftarrow{\Psi}_{\!\rho} $
   are then proved by direct check.
\end{proof}

\vskip9pt

   Note that  $ \, {\uhg}^{\text{cop}} \, $  is obviously a QUEA for the Lie bialgebra
   $ \lieg^{\text{cop}} $  (the ``co-opposite'' to  $ \lieg \, $,  i.e.\ one reverses the
   Lie cobracket),  and similarly  $ \, {\uhg}^{\text{op}} \, $  is a QUEA for the Lie bialgebra
   $ \lieg^{\text{op}} $  (the ``opposite'' to  $ \lieg \, $,  using reversed Lie bracket).
   Then both previous results provide morphisms between quantum groups of
   \textsl{different\/}  nature, namely a QFSHA as domain and a QUEA as codomain.
                                                           \par
   Nevertheless, we shall now show that both previous results can be refined,
   eventually yielding morphisms that connect quantum groups of the  \textsl{same\/}
   nature, namely both QFSHA'a in one case and both QUEA's in the other case.

\vskip9pt

\begin{theorem}  \label{thm: refined R-morph.'s}
 Let  $ \, \uhg \, $  be a QUEA, let  $ \, \fhg := {\uhg}^* \, $  be its dual QFSHA, as in
 \S \ref{equiv-&-(stand)-duality},  and let  $ \, \fhgs := {\uhg}' \, $,  resp.\
 $ \, \uhgs := {\fhg}^\vee \, $,  be the QFSHA, resp.\ the QUEA, introduced in  \S \ref{subsec: QDP}.  Let  $ \; \cR = \cR^i \otimes \cR_i \, $  (sum, possibly infinite, over repeated indices) be an  $ R $--matrix for  $ \, \uhg \, $,
which is congruent to  $ 1^{\otimes 2} $  modulo  $ \hbar \, $.
                                                               \par
   Then, for the two morphisms
%%%%%
  $ \; \fhg \,{\buildrel \overleftarrow{\Phi}_{\!\cR} \over
  {\relbar\joinrel\relbar\joinrel\longrightarrow}}\, {\uhg}^{\text{\rm cop}} \; $
%%%
 and
%%%
  $ \; \fhg \,{\buildrel \overrightarrow{\Phi}_{\!\cR} \over
  {\relbar\joinrel\relbar\joinrel\longrightarrow}}\, {\uhg}^{\text{\rm op}} \; $
%%%%%
 in  Proposition \ref{prop: Phi from R x QUEA},  the following holds:
 \vskip5pt
   \textit{(a)}\;  they take values inside  $ \, {\uhg}' \, $,
   \,and so they corestrict to morphisms
%%%
  $$
  \begin{matrix}
   \overleftarrow{\Phi}'_{\!\cR} : \fhg
   \relbar\joinrel\relbar\joinrel\relbar\joinrel\relbar\joinrel\longrightarrow
   {\big(\uhg'\,\big)}^{\text{\rm cop}} = {\fhgs}^{\text{\rm cop}}   \phantom{{}_{\big|}}  \\
   \overrightarrow{\Phi}'_{\!\cR} : \fhg
   \relbar\joinrel\relbar\joinrel\relbar\joinrel\relbar\joinrel\longrightarrow
   {\big(\uhg'\,\big)}^{\text{\rm op}} = {\fhgs}^{\text{\rm op}}   \phantom{{}^{\big|}}
   \end{matrix}   \leqno \text{and}
   $$
%%%
\noindent
 between QFSHA's for mutually dual (formal) Poisson groups;
 \vskip5pt
   \textit{(b)}\;  they uniquely extend to
   $ \, \fhg^\vee = {\big( {\uhg}^* \big)}^\vee \, $,  \,i.e.\ they extend to morphisms
%%%
  $$  \begin{matrix}
   \overleftarrow{\Phi}^\vee_{\!\cR} : \uhgs :=
   {\fhg}^\vee \relbar\joinrel\relbar\joinrel\relbar\joinrel\relbar\joinrel\longrightarrow
   {\uhg}^{\text{\rm cop}}   \phantom{{}_{\big|}}  \\
   \overrightarrow{\Phi}^\vee_{\!\cR} : \uhgs :=
   {\fhg}^\vee \relbar\joinrel\relbar\joinrel\relbar\joinrel\relbar\joinrel\longrightarrow
   {\uhg}^{\text{\rm op}}   \phantom{{}^{\big|}}
   \end{matrix}   \leqno \text{and}
   $$
%%%
\noindent
 between QUEA's for mutually dual Lie bialgebras.
\end{theorem}

\begin{proof}
 \textit{(a)}\;  Recall that
 $ \, \Ker\big(\epsilon_{{\uhg}'}\big) =: J_{{\uhg}'} \subseteq \big( J_{{\uhg}'}
 + \kh \, 1_{{\uhg}'} \big) =: I_{{\uhg}'} \, $,  \,and  $ \uhg $  is a topological
 Hopf algebra with respect to the  $ I_{{\uhg}'} $--adic  topology.
                                                                 \par
   Recall also that, by  Proposition \ref{prop: R(rho)-(co)matrices}\textit{(a)},
   we can write  $ \; \cR \, := \, \exp\big( \hbar^{-1} \, \vartheta \big) \; $  with
   $ \, \vartheta \in {\uhg}' \,\widetilde{\otimes}\, {\uhg}' \, $;  \,we write the latter as
   $ \, \vartheta = \vartheta^i \otimes \vartheta_i \, $  (sum over repeated indices),
   where  $ \, \vartheta^i , \vartheta_i \in J_{{\uhg}'} \, $.  Then we have
  $$
  \hbar^{-1} \, \vartheta  \; = \;  \hbar^{-1} \, \vartheta^i \otimes \vartheta_i  \; = \;
  \big( \hbar^{-1} \, \vartheta^i \big) \otimes \vartheta_i  \; = \;  \theta^i \otimes \vartheta_i
  $$
 with  $ \, \theta^i := \hbar^{-1} \vartheta^i \in \hbar^{-1} J_{{\uhg}'} \subseteq J_{\uhg}
 := \Ker\big(\epsilon_{\uhg}\big) \, $,  \,where the latter inclusion follows by
 the basic properties of  $ {\uhg}' $,  cf.\  \cite{Ga1}.  Now writing
 $ \, {\big( \theta^i \otimes \vartheta_i \big)}^n = \theta^j_{[n]} \otimes \vartheta_{[n],j} \, $
 for each  $ \, n \in \NN \, $,  we have in particular  $ \, \theta^j_{[n]} \in J_{\uhg}^{\;n} \, $
 and  $ \, \vartheta_{[n],j} \in J_{{\uhg}'}^{\;n} \, $,  for every  $ \, n \in \NN \, $.
 When we expand  $ \cR \, $,  by all this we find
\begin{equation*}  \label{eq: R-mat U' expan-1}
 \cR  \; = \;  \exp\big( \hbar^{-1} \, \vartheta \big)  \, = \,
 \exp\big( \theta^i \otimes \vartheta_i \big)  \, = \,
 {\textstyle \sum\limits_{n \geq 0}} {{\,1\,} \over {\,n!\,}} \,
 {\big( \theta^i \otimes \vartheta_i \big)}^n  \, = \,
 {\textstyle \sum\limits_{n \geq 0}} {{\,1\,} \over {\,n!\,}} \,
 \big( \theta^j_{[n]} \otimes \vartheta_{[n],j} \big)
\end{equation*}
 Therefore, for every  $ \, \eta \in \fhg := {\uhg}^* \, $  we have
  $$
  \overleftarrow{\Phi}_{\!\cR}(\eta)  \; = \;  \eta\big(\cR^s\big) \,\cR_s  \; = \;
  {\textstyle \sum\limits_{n \geq 0}} {{\,1\,} \over {\,n!\,}} \, \eta\big(\theta^j_{[n]}\big) \,
  \vartheta_{[n],j}  $$
 which describes a well-defined element (a convergent series, in the relevant topology!) of
 $ {\big( {\uhg}' \,\big)}^{\text{cop}} $   --- equal to  $ {\uhg}' $  as a  $ \kh $--module
 ---   exactly because  $ \, \vartheta_{[n],j} \in J_{{\uhg}'}^{\;n} \, $  for each
 $ \, n \in \NN \, $.  Thus  $ \Phi_{\!\cR}^\leftarrow $  corestricts to
 $ \, {\big( {\uhg}' \,\big)}^{\text{cop}} = {\fhgs}^{\text{cop}} \, $  as claimed, q.e.d.
 \vskip5pt
   The proof for  $ \overrightarrow{\Phi}_{\!\cR} $  goes exactly the same,
   just switching left and right.
 \vskip7pt
   \textit{(b)}\;  We begin acting as in the proof of  \textit{(a)\/}  above,
   but switching the roles of left and right hand sides.  Namely, we write
%%%%%
 $ \; \hbar \, \vartheta \, = \, \hbar \, \big( \vartheta^i \otimes \vartheta_i \big) = \,
 \theta^i \otimes \vartheta_i \; $
%%%%%
 where  $ \, \theta_i := \hbar^{-1} \vartheta_i \in \hbar^{-1} J_{{\uhg}'} \subseteq J_{\uhg}
 := \Ker\big(\epsilon_{\uhg}\big) \, $,  and also
 $ \, {\big( \vartheta^i \otimes \theta_i \big)}^n = \vartheta^j_{[n]} \otimes \theta_{[n],j} \, $,
 with  $ \, \vartheta^j_{[n]} \in J_{{\uhg}'}^{\;n} \, $  and  $ \, \theta_{[n],j} \in J_{\uhg}^{\;n} \, $,
 for all  $ \, n \in \NN \, $.  Then expanding  $ \cR $  yields
\begin{equation*}  \label{eq: R-mat U' expan-2}
 \cR  \; = \;  \exp\big( \hbar^{-1} \, \vartheta \big)  \, = \,
 \exp\big( \vartheta^i \otimes \theta_i \big)  \, = \,
 {\textstyle \sum\limits_{n \geq 0}} {{\,1\,} \over {\,n!\,}} \,
 {\big( \vartheta^i \otimes \theta_i \big)}^n  \, = \,
 {\textstyle \sum\limits_{n \geq 0}} {{\,1\,} \over {\,n!\,}} \,
 \big( \vartheta^j_{[n]} \otimes \theta_{[n],j} \big)
\end{equation*}
 \vskip-5pt
\noindent
 hence for every  $ \, \mu \in {\uhg}^* \, $  we have
 \vskip-5pt
\begin{equation}  \label{eq: R-mat U' expan-3}
  \overleftarrow{\Phi}_{\!\cR}(\mu)  \; := \;  \mu\big(\cR^s\big) \,\cR_s  \; = \;
  {\textstyle \sum\limits_{n \geq 0}} {{\,1\,} \over {\,n!\,}} \,
  \mu\big(\vartheta^j_{[n]}\big) \, \theta_{[n],j}
\end{equation}
 \vskip-0pt
   Now, recall that  $ \, {\big( {\uhg}^* \big)}^\vee = {\big( {\uhg}' \,\big)}^\star \, $,  \,by  \eqref{eq: QDP-duality}.  Then we consider the formula  \eqref{eq: R-mat U' expan-3}  for any  $ \, \mu \in {\big( {\uhg}^* \big)}^\vee = {\big( {\uhg}' \,\big)}^\star \, $   --- which contains  $ {\uhg}^* \, $.  As all coefficients  $ \mu\big(\vartheta'_n) $  belong to  $ \kh \, $,  every partial sum in the right-hand side formal series is a well-defined element in  $ {\uhg}^{\text{cop}} $   --- equal to  $ \uhg $  as a  $ \kh $--module.  In addition, since  $ \, \vartheta^j_{[n]} \in J_{{\uhg}'}^{\;n} \subseteq I_{{\uhg}'}^{\;n} \, $   --- for each  $ \, n \in \NN \, $  ---   and  $ \, \mu : {\uhg}' \!\longrightarrow \kh \, $  is  \textsl{continuous\/}  (with respect to the  $ I_{{\uhg}'}^{\;n} $--adic  topology on the left and the  $ \hbar $--adic  topology on the right), for every  $ \, s \in \NN \, $  there exist  $ \, n_s \, $  such that  $ \, \mu\big(\vartheta^j_{[n]}) \in \hbar^{n_s} \kh \, $  for all  $ \, n \geq n_s \, $.  This ensures that the formal series in  \eqref{eq: R-mat U' expan-3}  is actually convergent in the  $ \hbar $--adic  topology of  $ \uhg \, $,  thus describing a well-defined element in  $ \uhg \, $.  Letting  $ \mu $  range freely inside  $ {\big( {\uhg}^* \big)}^\vee \, $,  this proves that  $ \overleftarrow{\Phi}_{\!\cR} $  does indeed extend from  $ {\uhg}^* $  to  $ \, {\big( {\uhg}^* \big)}^\vee = {\fhgs}^\vee =: \uhgs \, $,  \;q.e.d.
 \vskip3pt
   Switching left and right in the arguments above we get the proof for
   $ \overrightarrow{\Phi}_{\!\cR} $  too.
\end{proof}

\vskip5pt

\begin{rmk}
 Claim  \textit{(a)\/}  of  Theorem \ref{thm: refined R-morph.'s}
 above appears also in  \cite{EK}, \S 4.5.
\end{rmk}

\vskip3pt

   In the dual framework, the parallel result holds true as well:

\vskip9pt

\begin{theorem}  \label{thm: refined rho-morph.'s}
 Let  $ \, \fhg \, $  be a QFSHA, let  $ \, \uhg := {\fhg}^\star \, $  be its dual QUEA, as in
 \S \ref{equiv-&-(stand)-duality},  and let also  $ \, \uhgs := {\fhg}^\vee \, $,  resp.\
 $ \, \fhgs := {\uhg}' \, $,  be the QUEA, resp.\ the QFSHA, introduced in
 \S \ref{subsec: QDP}.  Let  $ \, \rho \, $  be a  $ \varrho $--comatrix for  $ \, \fhg \, $,
which is congruent to  $ \epsilon^{\otimes 2} $  modulo  $ \hbar \, $.
                                                               \par
   Then, for the two morphisms
%%%%%
  $ \; \fhg \,{\buildrel \overleftarrow{\Psi}_{\!\rho} \over
  {\relbar\joinrel\relbar\joinrel\longrightarrow}}\, {\uhg}^{\text{\rm cop}} \; $
%%%
 and
%%%
  $ \; \fhg \,{\buildrel \overrightarrow{\Psi}_{\!\rho} \over
  {\relbar\joinrel\relbar\joinrel\longrightarrow}}\, {\uhg}^{\text{\rm op}} \; $
%%%%%
 in  Proposition \ref{prop: Psi from rho x QFSHA},  the following holds:
 \vskip3pt
   \textit{(a)}\;  they take values inside  $ \, {\uhg}' \, $,  \,so they corestrict to morphisms
%%%
  $$
  \begin{matrix}
  \overleftarrow{\Psi}'_{\!\rho} : \fhg
  \relbar\joinrel\relbar\joinrel\relbar\joinrel\relbar\joinrel\longrightarrow
  {\big(\uhg'\,\big)}^{\text{\rm cop}} = {\fhgs}^{\text{cop}}   \phantom{{}_{\big|}}  \\
   \overrightarrow{\Psi}'_{\!\rho} :
   \fhg \relbar\joinrel\relbar\joinrel\relbar\joinrel\relbar\joinrel\longrightarrow
   {\big(\uhg'\,\big)}^{\text{\rm op}} = {\fhgs}^{\text{op}}   \phantom{{}^{\big|}}
   \end{matrix}   \leqno \text{and}
   $$
%%%
\noindent
 between QFSHA's for mutually dual (formal) Poisson groups.
 \vskip3pt
   \textit{(b)}\;  they uniquely extend to  $ \, \uhgs = {\fhg}^\vee \, $,  \,i.e.\ they
   extend to morphisms
%%%
  $$
  \begin{matrix}
   \overleftarrow{\Psi}^\vee_{\!\rho} : \uhgs =
   \fhg^\vee \relbar\joinrel\relbar\joinrel\relbar\joinrel\relbar\joinrel\longrightarrow
   {\uhg}^{\text{\rm cop}}   \phantom{{}_{\big|}}  \\
   \overrightarrow{\Psi}^\vee_{\!\rho} : \uhgs =
   \fhg^\vee \relbar\joinrel\relbar\joinrel\relbar\joinrel\relbar\joinrel\longrightarrow
   {\uhg}^{\text{\rm op}}   \phantom{{}^{\big|}}
   \end{matrix}   \leqno \text{and}
   $$
%%%
\noindent
 between QUEA's for mutually dual Lie bialgebras.
\end{theorem}

\begin{proof}
 \textit{(a)}\,  By the assumption  $ \, \rho \equiv \epsilon^{\otimes 2} \pmod \hbar \, $  and by
 Proposition \ref{prop: R(rho)-(co)matrices}\textit{(b)},  we can write  $ \rho $  in the form
%%%%%
 $ \; \rho \, = \, \exp_*\!\big(\, \hbar^{-1} \, \zeta \,\big) \; $
%%%%%
 for some  $ \; \zeta \in {\Big(\! {\big( {\fhg}^\vee \big)}^{\widehat{\otimes} \, 2} \,\Big)}^{\!*} \, $.
 Then
  $$
  \displaylines{
   \zeta \, \in \, {\Big(\! {\big( {\fhg}^\vee \big)}^{\widehat{\otimes} \, 2} \,\Big)}^{\!*}  \, = \,
   {\big( {\fhg}^\vee \,\widehat{\otimes}\, {\fhg}^\vee \,\big)}^*  \, = \,
   {\big( {\fhg}^\vee \,\big)}^* \,\widetilde{\otimes}\, {\big( {\fhg}^\vee \,\big)}^*  \, =   \hfill  \cr
   \hfill   = \,  {\big( {\fhg}^\star \,\big)}' \,\widetilde{\otimes}\, {\big( {\fhg}^\star \,\big)}'  \, = \,
   {\uhg}' \,\widetilde{\otimes}\, {\uhg}'}
 $$
 --- thanks to  \eqref{eq: QDP-duality}  ---   hence
 $ \; \hbar^{-1} \, \zeta \, \in \, \hbar^{-1} \, {\uhgs}' \,\widetilde{\otimes}\, {\uhgs}' \; $.
 Now, the right-hand side of  \eqref{eq: 2-cocyc-cond.'s}  for  $ \, \sigma := \rho \, $  implies
 $ \; \zeta(1\,,a) = 0 \; $  for all  $ \, a \in \fhg \, $,  \,hence
 $ \, \zeta(-,a) \in \Ker\big( \epsilon_{{\uhg}'} \big) \, $  and so
 $ \; \hbar^{-1} \, \zeta(-,a) \, \in \, {\big( {\uhg}' \big)}^\vee = \uhg = {\fhg}^\star \; $  for all
 $ \, a \in \fhg \, $.  This implies that
\begin{equation}  \label{eq: hmeno1Zeta in fhg*_ot_uhg'}
  \hbar^{-1} \, \zeta(-\,,-) \; \in \; {\big( {\uhg}' \,\big)}^{\!\vee} \!\otimes {\uhg}' \, = \,
  {\fhg}^\star \otimes {\uhg}'
\end{equation}
 where hereafter we are being temporarily sloppy with the tensor product
 --- we fix this later on.  Clearly,  \eqref{eq: hmeno1Zeta in fhg*_ot_uhg'}  implies
%%%%%
  $ \; \rho \, = \, \exp_*\!\big(\hbar^{-1} \, \zeta\big) \, \in \, {\fhg}^\star \otimes {\uhg}' \; $
%%%%%
 as well.  Therefore we get at once  $ \; \Psi_\rho^\leftarrow(\ell) := \rho(\ell\,,-) \, \in \, {\uhg}' \; $
 for all  $ \, \ell \in \fhg \, $,  \;q.e.d.  This proves the claim about  $ \overleftarrow{\Psi}_{\!\rho} \, $,
 \,and that concerning  $ \overrightarrow{\Psi}_{\!\rho} $  is entirely similar.
 \vskip3pt
   It remains to ``dot your  \textsl{i\/}'s''  about the tensor product in
   \eqref{eq: hmeno1Zeta in fhg*_ot_uhg'}.  In fact,  \textit{a priori\/}
   we have  $ \, \rho \in {\fhg}^\star \,\widehat{\otimes}\, {\fhg}^\star =
   \uhg \,\widehat{\otimes}\, \uhg \, $,  \,hence also
  $$
  \hbar^{-1} \, \zeta  \; \in \;  {\fhg}^\star \,\widehat{\otimes}\, {\fhg}^\star  = \,
  \uhg \,\widehat{\otimes}\, \uhg
  $$
 --- where the (completed, topological) tensor product
 ``$ \,\widehat{\otimes}\, $''  is considered.
 On the other hand, we have found that
  $$
  \zeta  \; \in \;  {\big( {\fhg}^\star \big)}' \,
  \widetilde{\otimes}\, {\big( {\fhg}^\star \big)}'  \, = \,  {\uhg}' \,\widetilde{\otimes}\, {\uhg}'
  $$
 --- where the (completed, topological) tensor product
 ``$ \,\widetilde{\otimes}\, $''  is used.  Then the critical point is:
 what kind of tensor product  ``$ \, \otimes \, $''  is taken in
 \eqref{eq: hmeno1Zeta in fhg*_ot_uhg'}?
                                                                 \par
   Instead of giving a direct, formal answer to this question, we point out the following.
   First observe that  $ \, {\uhg}' \,\widetilde{\otimes}\, {\uhg}' = {\big( \uhg \,\widehat{\otimes}\,
   \uhg \big)}' \, $  naturally embeds into  $ \, \uhg \,\widehat{\otimes}\, \uhg \, $.
   Then, when  $ \; \hbar^{-1} \, \zeta \, \in \, \uhg \,\widehat{\otimes}\, \uhg \; $
   is expanded into some (suitably convergent) series
   $ \, \hbar^{-1} \, \zeta = \beta^i \otimes \beta_i \, $
   (summing over repeated indices) with  $ \, \beta^i , \beta_i \in \uhg \, $  for all  $ i \, $,
   \,what we proved above is that we actually have
   $ \, \beta_i \in {\uhg}' \; \big(\! \subseteq \uhg \big) \, $  for all indices  $ i \, $.
   This is what we loosely wrote as
%%%%%
 $ \; \hbar^{-1} \, \zeta \, \in \, \uhg \otimes {\uhg}' \, = \, {\fhg}^\star \otimes {\uhg}' \; $
%%%%%
 in  \eqref{eq: hmeno1Zeta in fhg*_ot_uhg'}  above.
%%%%%%%%%
%   \footnote{\color{blue} All this is somewhat ``rough'', but I do believe it is
% definitely enough  --- no need to spend more words on that!  \textbf{What do you
% think, Gast{\'o}n?} \color{red} I think that this is ok!}
%%%%%%%%%
%%%%%%%%%
%
 \vskip7pt
   \textit{(b)}\,  Acting as in part  \textit{(a)},  we find
   $ \, \rho = \exp_*\!\big( \hbar^{-1} \, \zeta \,\big) \, $  with
   $ \; \zeta \in {\big( {\uhg}' \,\big)}^{\widetilde{\otimes}\, 2} \; $  and
   $ \; \zeta \in {\big( \Ker\big(\epsilon_{\uhg}\big) \big)}^{\widehat{\otimes}\, 2} \, $
   too, so  $ \; \zeta \in {\big( \Ker\big(\epsilon_{{\uhg}'}\big) \big)}^{\widetilde{\otimes}\, 2}
   \; $.  Since  $ \, \Ker\big(\epsilon_{{\uhg}'}\big) \subseteq \hbar \,
   \Ker\big(\epsilon_{\uhg}\big) \, $,  \,this implies that, expanding
   $ \, \hbar^{-1} \, \zeta \, $  as a (convergent) series
   $ \; \hbar^{-1} \, \zeta \, = \, \beta^i \otimes \beta_i \, $,  \,we can assume
   $ \, \beta^i \in {\uhg}' \, $.  As  $ \, {\uhg}' = {\big( {\fhg}^\star \big)}' =
   {\big( {\fhg}^\vee \,\big)}^* \, $,  we end up with
\begin{equation}  \label{eq: h-1Zeta in U ot U'}
  \hbar^{-1} \, \zeta \, = \, \beta^i \otimes \beta_i  \; \in \;
  {\big( {\fhg}^\vee \,\big)}^* \otimes \uhg
\end{equation}
 where again the meaning of the tensor product  ``$ \, \otimes \, $''
 considered in this formula (along with the corresponding convergence issues) is
 handled just as in part  \textit{(a)}.  Finally, from \eqref{eq: h-1Zeta in U ot U'}
 it follows at once that  $ \overleftarrow{\Psi}_{\!\rho} $  extends from  $ \fhg $  to
 $ {\fhg}^\vee $  as claimed.  This proves our statement for
 $ \overleftarrow{\Psi}_{\!\rho} \, $,  and the case of  $ \overrightarrow{\Psi}_{\!\rho} $
 is entirely similar.
\end{proof}

\vskip9pt

\begin{free text}  \label{duality-propt.'s}
 \textbf{Duality properties.}  When we deal with a QUEA and a QFSHA which are dual to each other, it makes sense to compare the previous results.  The outcome is that  Proposition \ref{prop: duality-morph.'s_R-mat/rho-comat} turns to an enhanced version (with trivial proof), as follows:
\end{free text}

\vskip9pt

\begin{theorem}  \label{thm: duality x refined R-morph.'s/rho-morph.'s}
 Let  $ \, \uhg $  be a QUEA,  $ \fhg \, $  a QFSHA, which are dual to each other, i.e.\  $ \, \fhg = {\uhg}^* \, $  and  $ \, \uhg = {\fhg}^\star \, $.  Let  $ \, \cR = \rho \, $  be an  $ R $--matrix  for  $ \, \uhg $  and a  $ \varrho $--comatrix for  $ \, \fhg \, $,  which is trivial modulo  $ \hbar \, $,  i.e.\  congruent to  $ 1^{\otimes 2} $,  resp.\ to  $ \epsilon^{\otimes 2} $,  modulo  $ \hbar \, $.
%%%%%
 Then, for the morphisms in  Proposition \ref{prop: Phi from R x QUEA},  Proposition \ref{prop: Psi from rho x QFSHA},  Theorem \ref{thm: refined R-morph.'s}  and  Theorem \ref{thm: refined rho-morph.'s}  we have the following identifications
%%%%%
  $$  \overleftarrow{\Phi}_{\!\cR} = \overleftarrow{\Psi}_{\!\rho} \; ,  \,\; \overleftarrow{\Phi}'_{\!\cR} = \overleftarrow{\Psi}'_{\!\rho} \; ,  \,\; \overleftarrow{\Phi}^\vee_{\!\cR} = \overleftarrow{\Psi}^\vee_{\!\rho}
   \quad \text{and} \quad
      \overrightarrow{\Phi}_{\!\cR} = \overrightarrow{\Psi}_{\!\rho} \; ,  \,\; \overrightarrow{\Phi}'_{\!\cR} = \overrightarrow{\Psi}'_{\!\rho} \; ,  \,\; \overrightarrow{\Phi}^\vee_{\!\cR} = \overrightarrow{\Psi}^\vee_{\!\rho} \eqno \hskip4pt \square  $$
\end{theorem}

\vskip9pt

\begin{free text}  \label{compar-morph.'s}
 \textbf{Comparing morphisms (1).}  Let us fix assumptions as in
 Theorem \ref{thm: refined R-morph.'s}:  $ \uhg $  is a given QUEA,
 $ \fhg $  its dual
 QFSHA, and  $ \, \cR = \cR^s \otimes \cR_s \, $  is a (quantum)
 $ R $--matrix  of  $ \uhg \, $.
 Then from  Theorem \ref{thm: refined R-morph.'s}  we have Hopf algebra
 morphisms
%%%
\begin{equation}  \label{eq: morph.'s Phi'(R)}
   \fhg \,{\buildrel {\overleftarrow{\Phi}'_{\!\cR}} \over {\relbar\joinrel\relbar\joinrel\relbar\joinrel\relbar\joinrel\longrightarrow}}\,
%
% {\big(\uhg'\,\big)}^{\text{\rm cop}} \! =
%
 {\fhgs}^{\text{\rm cop}}
     \;\; ,  \;\;\qquad
   \fhg \,{\buildrel {\overrightarrow{\Phi}'_{\!\cR}} \over {\relbar\joinrel\relbar\joinrel\relbar\joinrel\relbar\joinrel\longrightarrow}}\,
%
% {\big(\uhg'\,\big)}^{\text{\rm op}} \! =
%
 {\fhgs}^{\text{\rm op}}
\end{equation}
\noindent
 between QFSHA's for mutually dual (formal) Poisson groups, and
%%%
\begin{equation}  \label{eq: morph.'s Phi^v(R)}
   \uhgs
%
% := {\fhg}^\vee
%
 \,{\buildrel {\overleftarrow{\Phi}^\vee_{\!\cR}} \over {\relbar\joinrel\relbar\joinrel\relbar\joinrel\relbar\joinrel
 \relbar\joinrel\longrightarrow}}\, {\uhg}^{\text{\rm cop}}
     \;\;\; ,  \quad \qquad
   \uhgs
%
% := {\fhg}^\vee
%
 \,{\buildrel {\overrightarrow{\Phi}^\vee_{\!\cR}} \over {\relbar\joinrel\relbar\joinrel\relbar\joinrel\relbar\joinrel
 \relbar\joinrel\longrightarrow}}\, {\uhg}^{\text{\rm op}}
\end{equation}
\noindent
 between QUEA's for mutually dual Lie bialgebras, which we re-write in the form
%%%
\begin{equation}  \label{eq: morph.'s Phi^v(R) - ALT}
   {\uhgs}^{\text{\rm cop}}
%
% := {\fhg}^\vee
%
 \,{\buildrel {\overleftarrow{\Phi}^\vee_{\!\cR}} \over {\relbar\joinrel\relbar\joinrel\relbar\joinrel\relbar
 \joinrel\relbar\joinrel\longrightarrow}}\, \uhg
     \;\;\; ,  \quad \qquad
   {\uhgs}^{\text{\rm op}}
%
% := {\fhg}^\vee
%
 \,{\buildrel {\overrightarrow{\Phi}^\vee_{\!\cR}} \over {\relbar\joinrel\relbar\joinrel\relbar\joinrel\relbar\joinrel
 \relbar\joinrel\longrightarrow}}\, \uhg
\end{equation}
 that is entirely equivalent.  We now go and compare  \eqref{eq: morph.'s Phi'(R)}  and  \eqref{eq: morph.'s Phi^v(R) - ALT}.
                                                                  \par
   Recall that  $ \, \fhgs := {\uhg}' \, $  and  $ \, \uhgs := {\fhg}^\vee \, $,
   \,which are in duality because  $ \uhg $  and  $ \fhg $  are in duality (by construction)
   and we can apply  \eqref{eq: QDP-duality}.  Then also  $ {\fhgs}^{\text{\rm cop}} $
   and  $ {\uhgs}^{\text{\rm op}} $  are in duality, as well as  $ {\fhgs}^{\text{\rm op}} $
   and  $ {\uhgs}^{\text{\rm cop}} \, $.
 \vskip3pt
   We are now ready to compare the morphisms in  \eqref{eq: morph.'s Phi'(R)}
   with those in  \eqref{eq: morph.'s Phi^v(R) - ALT}.
   Namely, we have a couple of diagrams
 \vskip-13pt
%
%%%%%
\begin{equation}  \label{eq: morph.'s SIN-DES/DES-SIN (R)}
  \xymatrix{
    \fhg \ar@{->}[rr]^{\hskip-1pt \overleftarrow{\Phi}'_{\!\cR}} \ar@{~}[d]  &
    &  \ar@{~}[d] {\fhgs}^{\text{\rm cop}}
   &   \fhg \ar@{->}[rr]^{\hskip-1pt \overrightarrow{\Phi}'_{\!\cR}} \ar@{~}[d]  &
   &  \ar@{~}[d] {\fhgs}^{\text{\rm op}}  \\
    \uhg \ar@{<-}[rr]_{\hskip1pt \overrightarrow{\Phi}^\vee_{\!\cR}}  &
    &  {\uhgs}^{\text{\rm op}}   &   \uhg \ar@{<-}[rr]_{\hskip1pt \overleftarrow{\Phi}^\vee_{\!\cR}}  &  &  {\uhgs}^{\text{\rm cop}}  }
\end{equation}
 where the vertical, twisting lines denote a relationship of mutual (Hopf) duality.
 Next result tells us that the link between the morphisms on top row and those underneath
 is indeed ``the best possible one'':
\end{free text}

\vskip9pt

\begin{theorem}  \label{thm: adj-morph.'s (R)}
 The two morphisms in left-hand side, resp.\ in right-hand side, of
 \eqref{eq: morph.'s SIN-DES/DES-SIN (R)}  are adjoint to each other, that is for all
 $ \, \eta \in \uhgs \, $  and  $ \, f \in \fhg \, $  we have
  $$
  \Big\langle\, \overrightarrow{\Phi}^\vee_{\!\cR}(\eta) \, ,
  f \Big\rangle  \; = \;  \Big\langle \eta \, ,
  \overleftarrow{\Phi}'_{\!\cR}(f) \Big\rangle
   \qquad  \text{and}  \qquad
      \Big\langle\, \overleftarrow{\Phi}^\vee_{\!\cR}(\eta) \, ,
      f \Big\rangle  \; = \;  \Big\langle \eta \, ,
      \overrightarrow{\Phi}'_{\!\cR}(f) \Big\rangle
      $$
 where by  ``$ \, \big\langle\,\ ,\ \big\rangle $''
 we denote the pairing between any two Hopf algebras in duality.
\end{theorem}

\pf
 It is enough to prove half of the claim   --- the other one being entirely similar ---   say the right-hand side.  Direct computation yields
  $$
  \displaylines{
   \quad   \Big\langle\, \overleftarrow{\Phi}^\vee_{\!\cR}(\eta) \, ,
   f \Big\rangle  \; = \;
   \Big\langle \big\langle \eta\,,\cR^s \big\rangle \, \cR_s \, ,
   f \Big\rangle  \; = \;
   \big\langle \eta\,,\cR^s \big\rangle \, \big\langle\, \cR_s \, ,
   f \,\big\rangle  \; =   \hfill  \cr
   \hfill   = \;
   \big\langle\, \cR_s \, , f \,\big\rangle \, \big\langle \eta\,,\cR^s
   \big\rangle  \; = \;  \Big\langle \eta \, , \cR^s \big\langle\, \cR_s \, ,
   f \,\big\rangle \!\Big\rangle  \; = \;  \Big\langle \eta \, , \overrightarrow{\Phi}'_{\!\cR}(f) \Big\rangle   \quad  }
   $$
 for all  $ \, \eta \in \uhgs \, $  and  $ \, f \in \fhg \, $,
 \,hence we are done.
\epf

\vskip7pt

   As a second step, let now start with assumptions as in
   Theorem \ref{thm: refined rho-morph.'s}:  $ \fhg $  is a given QFSHA,
   $ \uhg $  its dual QUEA, and  $ \, \rho \, $  is a (quantum)
   $ \varrho $--comatrix  of  $ \fhg \, $.  Then
   Theorem \ref{thm: refined rho-morph.'s}  provides Hopf algebra morphisms
%%%
\begin{equation}  \label{eq: morph.'s Psi'(rho)}
   \fhg \,{\buildrel {\overleftarrow{\Psi}'_{\!\rho}} \over {\relbar\joinrel\relbar\joinrel\relbar\joinrel\relbar\joinrel\longrightarrow}}\,
   {\fhgs}^{\text{\rm cop}}
     \;\; ,  \;\;\qquad
   \fhg \,{\buildrel {\overrightarrow{\Psi}'_{\!\rho}} \over
   {\relbar\joinrel\relbar\joinrel\relbar\joinrel\relbar\joinrel\longrightarrow}}\,
   {\fhgs}^{\text{\rm op}}
\end{equation}
\noindent
 between QFSHA's for mutually dual (formal) Poisson groups, and
%%%
\begin{equation}  \label{eq: morph.'s Psi^v(rho)}
   \uhgs \,{\buildrel {\overleftarrow{\Psi}^\vee_{\!\rho}} \over {\relbar\joinrel\relbar\joinrel\relbar\joinrel\relbar\joinrel
   \relbar\joinrel\longrightarrow}}\, {\uhg}^{\text{\rm cop}}
     \;\;\; ,  \quad \qquad
   \uhgs \,{\buildrel {\overrightarrow{\Psi}^\vee_{\!\rho}} \over {\relbar\joinrel\relbar\joinrel\relbar\joinrel\relbar\joinrel
   \relbar\joinrel\longrightarrow}}\, {\uhg}^{\text{\rm op}}
\end{equation}
\noindent
 between QUEA's for mutually dual Lie bialgebras;
 we re-write the latter as
%%%
\begin{equation}  \label{eq: morph.'s Psi^v(rho) - ALT}
   {\uhgs}^{\text{\rm cop}} \,{\buildrel
   {\overleftarrow{\Psi}^\vee_{\!\rho}} \over {\relbar\joinrel\relbar\joinrel\relbar\joinrel\relbar\joinrel
   \relbar\joinrel\longrightarrow}}\, \uhg
     \;\;\; ,  \quad \qquad
   {\uhgs}^{\text{\rm op}} \,{\buildrel
   {\overrightarrow{\Psi}^\vee_{\!\rho}} \over {\relbar\joinrel\relbar\joinrel\relbar\joinrel
   \relbar\joinrel\relbar\joinrel\longrightarrow}}\, \uhg
\end{equation}
 that is entirely equivalent.  We now go and compare
 \eqref{eq: morph.'s Psi'(rho)}  and  \eqref{eq: morph.'s Psi^v(rho) - ALT}.
                                                                  \par
%
%%%%%%%
%    Recall that  $ \, \fhgs := {\uhg}' \, $  and  $ \, \uhgs := {\fhg}^\vee \, $,
% \,which are in duality because  $ \uhg $  and  $ \fhg $  are in duality (by construction)
% and we can apply  \eqref{eq: QDP-duality}.  Then also  $ {\fhgs}^{\text{\rm cop}} $  and
% $ {\uhgs}^{\text{\rm op}} $  are in duality, as well as  $ {\fhgs}^{\text{\rm op}} $  and
% $ {\uhgs}^{\text{\rm cop}} \, $.
% %
%  \vskip3pt
% %
%   We are now ready to compare the morphisms in  \eqref{eq: morph.'s Phi'(R)}
% with those in  \eqref{eq: morph.'s Phi^v(R) - ALT}.  Namely, we have a couple of diagrams
%%%%%%%
%
 Acting as before (for the morphisms induced by an  $ R $--matrix),
 we find diagrams
 \vskip-13pt
%
%%%%%
\begin{equation}  \label{eq: morph.'s SIN-DES/DES-SIN (rho)}
 \vbox{
  \xymatrix{
    \fhg \ar@{->}[rr]^{\hskip-1pt \overleftarrow{\Psi}'_{\!\rho}} \ar@{~}[d]  &
    &  \ar@{~}[d] {\fhgs}^{\text{\rm cop}}
   &   \fhg \ar@{->}[rr]^{\hskip-1pt \overrightarrow{\Psi}'_{\!\rho}} \ar@{~}[d]  &
   &  \ar@{~}[d] {\fhgs}^{\text{\rm op}}  \\
    \uhg \ar@{<-}[rr]_{\hskip1pt \overrightarrow{\Psi}^\vee_{\!\rho}}  &  &
    {\uhgs}^{\text{\rm op}}   &   \uhg \ar@{<-}[rr]_{\hskip1pt
    \overleftarrow{\Psi}^\vee_{\!\rho}}  &  &  {\uhgs}^{\text{\rm cop}}  }
 }
\end{equation}
 where the vertical, twisting lines denote a relationship of mutual (Hopf) duality.
 Again, the link between the morphisms on top row and those underneath turns
 out to be ``the best possible one'', as the following result claims:

\vskip9pt

\begin{theorem}  \label{thm: adj-morph.'s (rho)}
 The two morphisms in left-hand side, resp.\ in right-hand side, of
 \eqref{eq: morph.'s SIN-DES/DES-SIN (R)}  are adjoint to each other,
 that is for all  $ \, \eta \in \uhgs \, $
   \hbox{and  $ \, f \in \fhg \, $  we have}
  $$  \Big\langle\, \overrightarrow{\Psi}^\vee_{\!\rho}(\eta) \, ,
  f \Big\rangle  \; = \;  \Big\langle \eta \, ,
  \overleftarrow{\Psi}'_{\!\rho}(f) \Big\rangle
   \qquad  \text{and}  \qquad
      \Big\langle\, \overleftarrow{\Psi}^\vee_{\!\rho}(\eta) \, ,
      f \Big\rangle  \; = \;  \Big\langle \eta \, ,
      \overrightarrow{\Psi}'_{\!\rho}(f) \Big\rangle
      $$
 where by  ``$ \, \big\langle\,\ ,\ \big\rangle $''
 we denote the pairing between any two Hopf algebras in duality.
\end{theorem}

\pf
 We prove the left-hand side of the claim, the other side being similar.
 By direct computation we find
  $$  \Big\langle\, \overrightarrow{\Psi}^\vee_{\!\rho}(\eta) \, ,
  f \Big\rangle  \; = \;  \big\langle\, \rho\,(-, \eta) \, ,
  f \,\big\rangle  \; = \;  \rho\big(\,f,\eta\big)  \; = \;
  \big\langle\, \eta \, , \rho\,(\,f,-) \,\big\rangle  \; = \;
  \Big\langle \eta \, ,\overleftarrow{\Psi}'_{\!\rho}(f) \Big\rangle  $$
 for all  $ \, \eta \in \uhgs \, $  and  $ \, f \in \fhg \, $,  \,as requested.
\epf

\vskip13pt

\subsection{Morphisms from polar  $ R \, $--matrices  and polar  $ \varrho \,
$--comatrices}  \label{subsec: morph.'s_polar R-mat/rho-comat}
  {\ }
 \vskip7pt
   We shall now explore what happens when the constructions leading to
   Proposition \ref{prop: morph.'s from R-mat}  or
   Proposition \ref{prop: morph.'s from rho-comat},  respectively, is (tentatively)
   applied to a QFSHA and a polar  $ R $--comatrix  for it, or to a QUEA and a polar  $\varrho $--comatrix  for it, respectively.  Much like for deformations by polar twists or polar  $ 2 $--cocycles,
   we eventually achieve a nice result by ``watching through the lens'' of the QDP.
%
% \vskip3pt
%
                                                         \par
   As a first result, we find that the construction of Hopf morphisms as in
   Proposition \ref{prop: morph.'s from R-mat}  can be applied again
   (though it might not be done through direct application of the abstract,
   general recipe) when the Hopf algebra under scrutiny is a QFSHA and its
   $ R $--matrix  is indeed only (!) a polar  $ R $--matrix.

\vskip11pt

\begin{prop}  \label{prop: morph.'s from polar R-mat}
 Let  $ \fhg $  be a QFSHA, and  $ \cR $  a polar  $ R $--matrix  for it.
 Then the recipes in  Proposition \ref{prop: morph.'s from R-mat}
 provide two well-defined morphisms
%%%
  $$  \begin{matrix}
   \overleftarrow{\underline{\Phi}}_{\cR} : \fhgs := {\big( \uhgs \big)}^*
   = {\big( \fhg^\vee \big)}^* \relbar\joinrel\relbar\joinrel\longrightarrow
   {\big( {\fhg}^\vee \big)}^{\text{\rm cop}} \! = {\uhgs}^{\text{\rm cop}}   \phantom{{}_{\big|}}  \\
   \overrightarrow{\underline{\Phi}}_{\cR} : \fhgs := {\big( \uhgs \big)}^* =
   {\big( \fhg^\vee \big)}^* \relbar\joinrel\relbar\joinrel\longrightarrow
   {\big( \fhg^\vee \big)}^{\text{\rm op}} \! = {\uhgs}^{\text{\rm op}}
   \phantom{{}^{\big|}}
      \end{matrix}  $$
%%%
%
\end{prop}

\begin{proof}
 This follows from a direct application of  Proposition \ref{prop: Phi from R x QUEA}
 to the QUEA  $ \, \uhgs :=  \fhg^\vee \, $  and its  $ R $--matrix  $ \cR \, $.
\end{proof}

\vskip3pt

   The previous result provide morphisms from a QFSHA to a QUEA, just like
   Proposition \ref{prop: Phi from R x QUEA}  did.  Following a similar path,
   we shall now improve such a result   --- much like we did in
   \S \ref{subsec: morph.'s_R-mat/rho-comat}  ---
   finding a couple of morphisms between QFSHA's and another couple
   between QUEA's.

\vskip11pt

\begin{theorem}  \label{thm: refined morph.'s from polar R-mat}
 Assume that  $ \, \cR $  is a  \textsl{polar  $ R $--matrix}  for the QFSHA
 $ \, \fhg \, $,  \,i.e.\ an $ R $--matrix  for the QUEA
 $ \, {\fhg}^\vee =: \uhgs \, $,  \,of the form
 $ \; \cR \, = \, \exp\big( \hbar^{-1} r \big) \; $  for some
 $ \, r \in {\fhg}^{\widetilde{\otimes}\, 2} \, $.  Then, for the two morphisms
 $ \, \overleftarrow{\underline{\Phi}}_{\cR} \, $  and
 $ \, \overrightarrow{\underline{\Phi}}_{\cR} \, $  in
 Proposition \ref{prop: morph.'s from polar R-mat}  above, the following holds:
 \vskip2pt
   \textit{(a)}\,  they corestrict to morphisms
%%%
  $$  \begin{matrix}
   \overleftarrow{\underline{\Phi}}'_{\cR} \! : \fhgs = \!
   {\big( \fhg^\vee \big)}^{\!*} \!\longrightarrow\!
   {\Big(\!\hskip-1pt {\big( {\fhg}^\vee \big)}^{\text{\rm cop}} \Big)}^{\!\prime} \! =
   \! {\Big(\!\hskip-1pt {\big( {\fhg}^\vee \big)}' \Big)}^{\!\text{\rm cop}} \!\! =
   {\fhg}^{\text{\rm cop}}   \phantom{{}_{\big|}}  \\
   \overrightarrow{\underline{\Phi}}'_{\cR} : \fhgs = \!
   {\big( \fhg^\vee \big)}^{\!*} \!\longrightarrow\!
   {\Big(\! {\big( {\fhg}^\vee \big)}^{\text{\rm op}} \Big)}^{\!\prime} \! =
   \! {\Big(\! {\big( {\fhg}^\vee \big)}' \Big)}^{\!\text{\rm op}} \! =
   {\fhg}^{\text{\rm op}}   \phantom{{}^{\big|}}
      \end{matrix}   \leqno \text{and}  $$
%%%
%
% %%%
%  shorter formulation...
%
%   $$  \widetilde{\Phi}_{\!\cR}^\leftarrow : \, \fhgs \relbar\joinrel\relbar\joinrel\longrightarrow
% {\fhg}^{\text{\rm cop}}
% %%%
%  \quad ,  \qquad
% %%%
%    \widetilde{\Phi}_{\!\cR}^\rightarrow : \, \fhgs \relbar\joinrel\relbar\joinrel\longrightarrow
% {\fhg}^{\text{\rm op}}  $$
% %%%
%
between QFSHA's for mutually dual (formal) Poisson groups;
 \vskip2pt
   \textit{(b)}\,  they extend to morphisms
%%%
  $$  \begin{matrix}
   \overleftarrow{\underline{\Phi}}^\vee_{\cR} : \, \uhg = {\fhg}^\star =
   {\Big(\! {\big( \fhg^\vee \big)}^* \Big)}^{\!\vee} \relbar\joinrel\relbar\joinrel\relbar\joinrel\longrightarrow
   {\big( {\fhg}^\vee \big)}^{\text{\rm cop}} = {\uhgs}^{\text{\rm cop}}   \phantom{{}_{\big|}}  \\
   \overrightarrow{\underline{\Phi}}^\vee_{\cR} : \,
   \uhg = {\fhg}^\star = {\Big(\! {\big( \fhg^\vee \big)}^* \Big)}^{\!\vee}
   \relbar\joinrel\relbar\joinrel\relbar\joinrel\relbar\joinrel\longrightarrow
   {\big( {\fhg}^\vee \big)}^{\text{\rm op}} = {\uhgs}^{\text{\rm op}}
   \phantom{{}^{\big|}}
      \end{matrix}   \leqno \text{and}  $$
%%%
%
% %%%
%  shorter formulation...
%
%   $$  \widetilde{\Phi}_{\!\cR}^\leftarrow : \, \uhg \relbar\joinrel\relbar\joinrel\longrightarrow
% {\uhgs}^{\text{\rm cop}}
% %%%
%  \quad ,  \qquad
% %%%
%    \widetilde{\Phi}_{\!\cR}^\rightarrow : \, \uhg \relbar\joinrel\relbar\joinrel\longrightarrow
% {\uhgs}^{\text{\rm op}}  $$
% %%%
%
between QUEA's for mutually dual Lie bialgebras.
\end{theorem}

\begin{proof}
 First of all, note that the chain of identities
%%%
  $$  {\Big(\! {\big( {\fhg}^\vee \big)}^{\text{\rm cop}} \Big)}'  \, = \,
  {\Big(\! {\big( {\fhg}^\vee \big)}' \Big)}^{\!\text{\rm cop}}  \, = \,
  {\fhg}^{\text{\rm cop}}  $$
%%%
 --- and similarly with superscript  ``op''  instead of ``cop'' throughout ---
 is obvious from definitions along with the fact that Drinfeld's functors
 $ \, {(\ )}' \, $  and  $ \, {(\ )}^\vee \, $  are inverse to each other.
 Similarly, it is also obviously true the chain of identities
%%%
  $$  {\Big(\! {\big( \fhg^\vee \big)}^* \Big)}^{\!\vee}  \, = \;
  {\Big(\! {\big( \fhg^\vee \big)}' \Big)}^{\!\star}  \, = \;
  {\fhg}^\star  \, = \;  \uhg  $$
%%%
                                                                   \par
   As to the rest of the claim, everything follows from  Theorem \ref{thm: refined R-morph.'s}  applied to the QUEA  $ \, \uhgs := \fhg^\vee \, $  along with its  $ R $--matrix  $ \cR \, $.
\end{proof}

\vskip7pt

   Now we go for the dual constructions, concerning polar  $ \rho $--comatrices  for a QUEA:

\vskip9pt

\begin{prop}  \label{prop: morph.'s from polar rho-comat}
 Assume that  $ \rho $  is a  \textsl{polar  $ \varrho $--comatrix}
 for the QUEA  $ \, \uhg \, $,  \,i.e.\ an element of the form
 $ \; \rho \, = \, \exp_*\!\big( \hbar^{-1} \varrho \big) \; $  for some
 $ \, \varrho \in \Big( {\uhg}^{\widehat{\otimes}\, 2} \Big)^* \, $
 --- taking into account  Lemma \ref{lemma: properties polar 2-cocycle} ---
 which obeys  \eqref{eq: rho-comat_prop.'s}.
                                                      \par
   Then the recipes in  Proposition \ref{prop: morph.'s from rho-comat}
   provide two well-defined morphisms
%%%
  $$  \begin{matrix}
   \overleftarrow{\underline{\Psi}}_{\,\rho} \; : \; \fhgs := \uhg' \relbar\joinrel\relbar\joinrel\longrightarrow
   \Big(\! {\big( \uhg' \,\big)}^\star \Big)^{\text{\rm cop}} =
   \big( {\fhgs}^\star \big)^{\text{\rm cop}} = \, {\uhgs}^{\text{\rm cop}}
   \phantom{{}_{\big|}}  \\
   \overrightarrow{\underline{\Psi}}_{\,\rho} \; : \;
   \fhgs := \uhg' \relbar\joinrel\relbar\joinrel\longrightarrow
   \Big(\! {\big( \uhg' \,\big)}^\star \Big)^{\text{\rm op}} =
   \big( {\fhgs}^\star \big)^{\text{\rm op}} = \, {\uhgs}^{\text{\rm op}}
   \phantom{{}^{\big|}}
      \end{matrix}   \leqno \text{and}  $$
%%%
%
% %%%
%  shorter formulation...
%
%   $$  \widehat{\Psi}_{\!\rho}^\leftarrow : \, \Fhgs \relbar\joinrel\relbar\joinrel\longrightarrow
% {\uhgs}^{\text{\rm cop}}
% %%%
%  \quad ,  \qquad
% %%%
%    \widehat{\Psi}_{\!\rho}^\rightarrow : \, \fhgs \relbar\joinrel\relbar\joinrel\longrightarrow
% {\uhgs}^{\text{\rm op}}  $$
% %%%
%
\end{prop}

\begin{proof}
 Everything follows from definitions once we apply
 Proposition \ref{prop: Psi from rho x QFSHA}  to the QFSHA
 $ \, F_\hbar[[G^*]] := \uhg' \, $,
 \,also taking into the account the chain of (obvious) identities
 $ \; {\big( \uhg' \,\big)}^\star = {\fgs}^\star = \uhgs \; $.
\end{proof}

\vskip7pt

   Once again, the previous result provides morphisms from a QFSHA to a QUEA,
   and now we ``enhance'' it   --- like we did with
   Theorem \ref{thm: refined morph.'s from polar R-mat}  ---
   finding morphisms between QFSHA's and morphisms between QUEA's:

\vskip9pt

\begin{theorem}  \label{thm: refined morph.'s from polar rho-comat}
 Assume that  $ \rho $  is a  \textsl{polar  $ \varrho $--comatrix}  for the QUEA
 $ \, \uhg \, $,  \,i.e.\ an element of the form
 $ \; \rho \, = \, \exp_*\!\big( \hbar^{-1} \varrho \big) \; $  for some
 $ \, \varrho \in \Big( {\uhg}^{\widehat{\otimes}\, 2} \Big)^* \, $
 --- taking into account  Lemma \ref{lemma: properties polar 2-cocycle} ---
 which obeys  \eqref{eq: rho-comat_prop.'s}.
%                                                       \par
  Then, for the two morphisms  $ \, \overleftarrow{\underline{\Psi}}_{\,\rho} \, $  and
  $ \, \overrightarrow{\underline{\Psi}}_{\,\rho} \, $
  in  Proposition \ref{prop: morph.'s from polar rho-comat}  above,
  the following holds:
 \vskip2pt
   \textit{(a)}\,  they corestrict to morphisms
  $$  \begin{matrix}
   \overleftarrow{\underline{\Psi}}'_{\,\rho} \; : \; F_\hbar[[G^*]] = \uhg'
   \relbar\joinrel\longrightarrow {\Big(\! \Big(\! {\big( \uhg' \,\big)}^\star \Big)^{\text{\rm cop\,}} \Big)}'
   = \big( {\uhg}^* \big)^{\text{\rm cop}} =: {\fhg}^{\text{\rm cop}}
   \phantom{{}_{\big|}}  \\
   \overrightarrow{\underline{\Psi}}'_{\,\rho} \; : \; F_\hbar[[G^*]] = \uhg'
   \relbar\joinrel\longrightarrow
   {\Big(\! \Big(\! {\big( \uhg' \,\big)}^\star \Big)^{\text{\rm op\,}} \Big)}' =
   \big( {\uhg}^* \big)^{\text{\rm op}} =: {\fhg}^{\text{\rm op}}   \phantom{{}^{\big|}}
      \end{matrix}   \leqno \text{and}
      $$
%%%
%
% %%%
%  shorter formulation...
%
%   $$  \widehat{\Psi}_{\!\rho}^\leftarrow : \, \fhgs \relbar\joinrel\relbar\joinrel\longrightarrow
% {\fhg}^{\text{\rm cop}}
% %%%
%  \quad ,  \qquad
% %%%
%    \widehat{\Psi}_{\!\rho}^\rightarrow : \, \fhgs \relbar\joinrel\relbar\joinrel\longrightarrow
% {\fhg}^{\text{\rm op}}  $$
% %%%
%
between QFSHA's for mutually dual (formal) Poisson groups;
 \vskip2pt
   \textit{(b)}\,  they extend to morphisms
  $$  \begin{matrix}
   \overleftarrow{\underline{\Psi}}^\vee_{\,\rho} \; : \; \uhg =
   {\big( \uhg' \,\big)}^{\!\vee} \relbar\joinrel\longrightarrow
   \Big(\! {\big( \uhg' \,\big)}^\star \Big)^{\text{\rm cop}} =
   \big( {\fhgs}^\star \big)^{\text{\rm cop}} = \, {\uhgs}^{\text{\rm cop}}
   \phantom{{}_{\big|}}  \\
   \overrightarrow{\underline{\Psi}}^\vee_{\,\rho} \; : \; \uhg =
   {\big( \uhg' \,\big)}^{\!\vee} \relbar\joinrel\longrightarrow
   \Big(\! {\big( \uhg' \,\big)}^\star \Big)^{\text{\rm op}} =
   \big( {\fhgs}^\star \big)^{\text{\rm op}} = \, {\uhgs}^{\text{\rm op}}
   \phantom{{}^{\big|}}
      \end{matrix}   \leqno \text{and}
      $$
%%%
%
% %%%
%  shorter formulation...
%
%   $$  \widehat{\Psi}_{\!\rho}^\leftarrow : \, \uhg \relbar\joinrel\relbar\joinrel\longrightarrow
% {\uhgs}^{\text{\rm cop}}
% %%%
%  \quad ,  \qquad
% %%%
%    \widehat{\Psi}_{\!\rho}^\rightarrow : \, \uhg \relbar\joinrel\relbar\joinrel\longrightarrow
% {\uhgs}^{\text{\rm op}}  $$
% %%%
%
 between QUEA's for mutually dual Lie bialgebras.
\end{theorem}

\begin{proof}
 As the functors  $ \, {(\ )}' \, $  and  $ \, {(\ )}^\vee \, $
 are inverse to each other, definitions yield
%%%
  $$  {\Big(\! \Big(\! {\big( \uhg' \,\big)}^{\!\star} \Big)^{\!\!\text{\rm cop\,}}
  \Big)}^{\!\prime} \! = \! {\Big(\! \Big(\! {\big( \uhg' \,\big)}^\star \Big)'
  \,\Big)}^{\!\!\text{\rm cop}} \!\! = {\Big(\! \Big(\! {\big( \uhg' \,\big)}^{\!\vee}
  \Big)^{\!*} \,\Big)}^{\!\!\text{\rm cop}} \!\! = \big( {\uhg}^* \big)^{\text{\rm cop}}
  =: {\fhg}^{\text{\rm cop}}  $$
%%%
 --- and similarly with superscript  ``op''  instead of ``cop'' throughout ---
 also thanks to  $ \, {(\ )}' \circ {(\ )}^\star = {(\ )}^* \circ {(\ )}^\vee \, $.
 Basing on this, the entire claim follows at once from
 Theorem \ref{thm: refined rho-morph.'s}  applied to the QFSHA
 $ \, \fhgs := {\uhg}' \, $  and to its  $ \varrho $--comatrix  $ \rho \, $.
\end{proof}

\vskip9pt

\begin{free text}  \label{duality-propt.'s x polar}
 \textbf{Duality properties.}  If we consider a QUEA and a QFSHA which
 are dual to each other, we can compare the previous results:
 thus we find the following ``polar analogue''
 --- whose proof is trivial again ---   of
 Theorem \ref{thm: duality x refined R-morph.'s/rho-morph.'s}:
\end{free text}

\vskip9pt

\begin{theorem}  \label{thm: duality x refined polar R-morph.'s/polar rho-morph.'s}
 Let  $ \, \uhg $  be a QUEA,  $ \fhg \, $
 a QFSHA, which are dual to each other, i.e.\
 $ \, \fhg = {\uhg}^* \, $  and  $ \, \uhg = {\fhg}^\star \, $.
 Let  $ \, \rho = \cR \, $  be a polar  $ \rho $--comatrix  for
 $ \, \uhg $  and a polar  $ R $--matrix for  $ \, \fhg \, $.
                                                                  \par
   Then, for the morphisms in  Proposition \ref{prop: morph.'s from polar R-mat},
   Proposition \ref{prop: morph.'s from polar rho-comat},
   Theorem \ref{thm: refined morph.'s from polar R-mat}  and
   Theorem \ref{thm: refined morph.'s from polar rho-comat}
   we have the following identifications
%%%%%
  $$  \overleftarrow{\underline{\Phi}}_{\!\cR} =
  \overleftarrow{\underline{\Psi}}_\rho \; ,  \,\; \overleftarrow{\underline{\Phi}}'_{\!\cR} =
  \overleftarrow{\underline{\Psi}}'_\rho \; ,  \,\; \overleftarrow{\underline{\Phi}}^\vee_{\!\cR} = \overleftarrow{\underline{\Psi}}^\vee_\rho
   \hskip9pt \text{and} \hskip9pt
      \overrightarrow{\underline{\Phi}}_{\!\cR} =
      \overrightarrow{\underline{\Psi}}_\rho \; ,  \,\; \overrightarrow{\underline{\Phi}}'_{\!\cR} =
      \overrightarrow{\underline{\Psi}}'_\rho \; ,  \,\; \overrightarrow{\underline{\Phi}}^\vee_{\!\cR} = \overrightarrow{\underline{\Psi}}^\vee_\rho
      \eqno \hskip4pt \square  $$
\end{theorem}

\vskip7pt

\begin{free text}  \label{compar-(polar)morph.'s}
 \textbf{Comparing morphisms (2).}  We shall now compare
 morphisms among quantum groups provided by a polar  $ R $--matrix
 as above; the analysis is pretty similar to what we did with  $ R $--matrices,
 so we can be quicker.
                                                                      \par
   Let us start with a QFSHA  $ \fhg \, $,  \,with dual QUEA denoted by
   $ \uhg \, $,  \,and a polar  $ R $--matrix  $ \cR $  for  $ \fhg \, $.
   Then  Theorem \ref{thm: refined morph.'s from polar R-mat}
   gives a couple of diagrams
 \vskip-15pt
%
%%%%%
\begin{equation}  \label{eq: morph.'s SIN-DES/DES-SIN (polar R)}
 \vbox{
  \xymatrix{
    \fhgs \ar@{->}[rr]^{\hskip-1pt \overleftarrow{\underline{\Phi}}'_{\cR}} \ar@{~}[d]  &
    &  \ar@{~}[d] {\fhg}^{\text{\rm cop}}
   &   \fhgs \ar@{->}[rr]^{\hskip-1pt \overrightarrow{\underline{\Phi}}'_{\cR}}
   \ar@{~}[d]  &  &  \ar@{~}[d] {\fhg}^{\text{\rm op}}  \\
    \uhgs \ar@{<-}[rr]_{\hskip1pt \overrightarrow{\underline{\Phi}}^\vee_{\cR}}  &
    &  {\uhg}^{\text{\rm op}}   &   \uhgs \ar@{<-}[rr]_{\hskip1pt
    \overleftarrow{\underline{\Phi}}^\vee_{\cR}}  &  &  {\uhg}^{\text{\rm cop}}  }
 }
\end{equation}
 where the vertical, twisting lines denote a relationship of mutual (Hopf)
 duality while the horizontal arrows are Hopf algebra morphisms.
 Next result is the ``polar analogue'' of  Theorem \ref{thm: adj-morph.'s (R)},
 telling us that the morphisms on top row and those underneath are
 ``as close as possible'':
\end{free text}

\vskip7pt

\begin{theorem}  \label{thm: adj-morph.'s (polar R)}
 The two morphisms in left-hand side, resp.\ in right-hand side, of
 \eqref{eq: morph.'s SIN-DES/DES-SIN (polar R)}  are adjoint to each other,
 that is for all  $ \, \eta \in \uhg \, $  and  $ \, f \in \fhgs \, $  we have
  $$  \Big\langle\, \overrightarrow{\underline{\Phi}}^\vee_\cR(\eta) \, ,
  f \Big\rangle  \; = \;  \Big\langle \eta \, ,
  \overleftarrow{\underline{\Phi}}'_\cR(f) \Big\rangle
   \qquad  \text{and}  \qquad
      \Big\langle\, \overleftarrow{\underline{\Phi}}^\vee_\cR(\eta) \, ,
      f \Big\rangle  \; = \;  \Big\langle \eta \, ,\overrightarrow{\underline{\Phi}}'_\cR(f) \Big\rangle  $$
 where by  ``$ \, \big\langle\,\ ,\ \big\rangle $''
 we denote the pairing between any two Hopf algebras in duality.
\end{theorem}

\begin{proof}
 The proof follows from  Theorem \ref{thm: adj-morph.'s (R)}
 along with  Theorem \ref{thm: (co)mat vs. polar (co)mat}.
\end{proof}

   Similarly, let  $ \uhg $  be a QUEA, with dual QFSHA denoted by
   $ \fhg \, $,  \,and let  $ \rho $  be a polar  $ \varrho $--comatrix
   $ \uhg \, $.  Then  Theorem \ref{thm: refined morph.'s from polar rho-comat}
   yields a couple of diagrams
 \vskip-15pt
%
%%%%%
\begin{equation}  \label{eq: morph.'s SIN-DES/DES-SIN (polar rho)}
 \vbox{
  \xymatrix{
    \fhgs \ar@{->}[rr]^{\hskip-1pt \overleftarrow{\underline{\Psi}}'_{\rho}} \ar@{~}[d]
    &  &  \ar@{~}[d] {\fhg}^{\text{\rm cop}}
   &   \fhgs \ar@{->}[rr]^{\hskip-1pt \overrightarrow{\underline{\Psi}}'_{\rho}}
   \ar@{~}[d]  &  &  \ar@{~}[d] {\fhg}^{\text{\rm op}}  \\
    \uhgs \ar@{<-}[rr]_{\hskip1pt \overrightarrow{\underline{\Psi}}^\vee_{\rho}}
    &  &  {\uhg}^{\text{\rm op}}   &
    \uhgs \ar@{<-}[rr]_{\hskip1pt \overleftarrow{\underline{\Psi}}^\vee_{\rho}}  &  &
    {\uhg}^{\text{\rm cop}}  }
 }
\end{equation}
 where the vertical, twisting lines denote a relationship of mutual (Hopf) duality
 while the horizontal arrows are Hopf algebra morphisms.  We get now the
 ``polar analogue'' of  Theorem \ref{thm: adj-morph.'s (rho)},
 which claims that the morphisms on top row of
 \eqref{eq: morph.'s SIN-DES/DES-SIN (polar rho)}  and those underneath are
 ``as close as possible'':

\vskip9pt

\begin{theorem}  \label{thm: adj-morph.'s (polar rho)}
 The two morphisms in left-hand side, resp.\ in right-hand side, of
 \eqref{eq: morph.'s SIN-DES/DES-SIN (polar rho)}  are adjoint to each other,
 that is for all  $ \, \eta \in \uhg \, $  and  $ \, f \in \fhgs \, $  we have
  $$
  \Big\langle\, \overrightarrow{\underline{\Psi}}^\vee_\rho(\eta) \, ,
  f \Big\rangle  \; = \;  \Big\langle \eta \, ,
  \overleftarrow{\underline{\Psi}}'_\rho(f) \Big\rangle
   \qquad  \text{and}  \qquad
      \Big\langle\, \overleftarrow{\underline{\Psi}}^\vee_\rho(\eta) \, ,
      f \Big\rangle  \; = \;  \Big\langle \eta \, ,\overrightarrow{\underline{\Psi}}'_\rho(f) \Big\rangle
      $$
 where  ``$ \, \big\langle\,\ ,\ \big\rangle $''  denotes the pairing
 between any two Hopf algebras in duality.
\end{theorem}

\begin{proof}
 Here again, the proof follows from  Theorem \ref{thm: adj-morph.'s (R)}
 and  Theorem \ref{thm: (co)mat vs. polar (co)mat}.
\end{proof}

\vskip13pt

\subsection{Semiclassical morphisms induced by specialization}
\label{subsec: semiclass-morph.'s by special}
  {\ }
 \vskip9pt
   We will now go and study the semiclassical limit of the various morphisms
   among quantum groups, considered in \S\S \ref{subsec: morph.'s_R-mat/rho-comat}
   and  \ref{subsec: morph.'s_polar R-mat/rho-comat}  above, induced by  $ R $--matrices,
   $ \varrho $--comatrices  and their  ``\textsl{polar\/}  counterparts''.
 \vskip3pt
   First we consider the case of an  $ R $--matrix  $ \cR $  for a given QUEA  $ \uhg \, $,
   \,whose dual QFSHA is  $ \fhg \, $.  With this assumptions, we recall the existence of
   the Hopf algebra morphisms in  \eqref{eq: morph.'s SIN-DES/DES-SIN (R)},
   which by  Theorem \ref{thm: adj-morph.'s (R)}  are pairwise mutually adjoint.
                                                                        \par
   Specialising  $ \hbar $  to 0, the left-hand side of
   \eqref{eq: morph.'s SIN-DES/DES-SIN (R)}  provides two mutually adjoint morphisms
%%%
 $ \; \fg \,{\buildrel \overleftarrow{\Phi}'_{\!\cR}{\big|}_{\hbar=0} \over {\relbar\joinrel\relbar\joinrel\relbar\joinrel\relbar\joinrel\longrightarrow}}\, {\fgs}^{\text{\rm cop}} \; $
%%%
 and
%%%
 $ \; {\ugs}^{\text{\rm op}} \,{\buildrel
 \overrightarrow{\Phi}^\vee_{\!\cR}{\big|}_{\hbar=0} \over {\relbar\joinrel\relbar\joinrel\relbar\joinrel
 \relbar\joinrel\longrightarrow}}\, \ug \, $,
%%%
 \;the first being a morphism of  \textsl{Poisson\/}  Hopf algebras, the second one of
 \textsl{co-Poisson\/}  Hopf algebras.  As they are mutually adjoint,
 each one of them defines one and the same morphism of formal Poisson groups
%
%%%%%
%   $$  \phi^+_{{}_\cR} : G^*_{\text{\rm op}} \!\relbar\joinrel\longrightarrow G  $$
%%%
 $ \; \phi^+_{{}_\cR} : G^*_{\text{\rm op}} \!\relbar\joinrel\longrightarrow G \; $
 where  $ \, G^*_{\text{\rm op}} \, $  is the  \textsl{opposite\/}
 (i.e., with opposite product) formal Poisson group to  $ G^* \, $.  Note that
%%%
 $ \, \phi^+_{{}_\cR} : G^*_{\text{\rm op}} \!\relbar\joinrel\longrightarrow G \, $
 is directly defined by  $ \overleftarrow{\Phi}'_{\!\cR}{\big|}_{\hbar=0} \, $,
 \,while the morphism of Lie bialgebras
%%%
 $ \; d\phi^+_{{}_\cR} : \lieg^*_{\text{\rm op}}
 \!\relbar\joinrel\longrightarrow \lieg \; $
 can be deduced directly from  $ \, \overrightarrow{\Phi}^\vee_{\!\cR}{\big|}_{\hbar=0} \, $,
 \,by restriction to  $ \lieg^*_{\text{\rm op}} $  and corestriction to  $ \lieg \, $.
                                                                        \par
   Similarly, specialising  $ \hbar $  to 0  the right-hand side of
   \eqref{eq: morph.'s SIN-DES/DES-SIN (R)}  yields two mutually adjoint morphisms
%%%
 $ \; \fg \,{\buildrel \overrightarrow{\Phi}'_{\!\cR}{\big|}_{\hbar=0} \over
 {\relbar\joinrel\relbar\joinrel\relbar\joinrel\relbar\joinrel\longrightarrow}}\,
 {\fgs}^{\text{\rm op}} \; $
%%%
 and
%%%
 $ \; {\ugs}^{\text{\rm cop}} \,{\buildrel
 \overleftarrow{\Phi}^\vee_{\!\cR}{\big|}_{\hbar=0} \over {\relbar\joinrel\relbar\joinrel\relbar\joinrel
 \relbar\joinrel\longrightarrow}}\, \ug \; $
%%%
 which in turn defines one single morphism of formal Poisson groups
%
%%%%%
%   $$  \phi^-_{{}_\cR} : G^*_{\text{\rm cop}} \!\relbar\joinrel\longrightarrow G  $$
%%%
 $ \; \phi^-_{{}_\cR} : G^*_{\text{\rm cop}} \!\relbar\joinrel\longrightarrow G \; $
 where now  $ \, G^*_{\text{\rm cop}} \, $  denotes the  \textsl{co-opposite\/}
 formal Poisson group to  $ G^* \, $
 --- i.e., with same product but opposite Poisson structure.
 This goes along with its associated morphism of Lie bialgebras
%%%
 $ \; d\phi^-_{{}_\cR} : \lieg^*_{\text{\rm cop}}
 \!\relbar\joinrel\longrightarrow \lieg \; $.
 In short, we have pairs of morphisms
\begin{equation}  \label{eq: morph.'s form-Pois-grps (R)}
  G^*_{\text{\rm op}} {\buildrel \phi^+_{{}_\cR} \over {\relbar\joinrel\relbar\joinrel\longrightarrow}}\, G
     \; ,   \quad
  G^*_{\text{\rm cop}} {\buildrel \phi^-_{{}_\cR} \over {\relbar\joinrel\relbar\joinrel\longrightarrow}}\, G
     \qquad \text{and} \qquad
  \lieg^*_{\text{\rm op}} {\buildrel d\phi^+_{{}_\cR} \over {\relbar\joinrel\relbar\joinrel\longrightarrow}}\, \lieg
     \; ,   \quad
  \lieg^*_{\text{\rm cop}} {\buildrel d\phi^-_{{}_\cR} \over {\relbar\joinrel\relbar\joinrel\longrightarrow}}\, \lieg
\end{equation}
 of formal Poisson groups and of Lie  bialgebras,
 respectively.
 \vskip7pt
   Second, we consider the case of a  $ \varrho $--comatrix  $ \rho $
   for a given QFSHA  $ \fhg \, $,  \,with dual QUEA  $ \uhg \, $.
   In this case, there exist Hopf algebra morphisms as in
   \eqref{eq: morph.'s SIN-DES/DES-SIN (rho)},
   which are pairwise mutually adjoint due to  Theorem \ref{thm: adj-morph.'s (rho)}.
                                                                        \par
   Acting as before, specialising  $ \hbar $  to 0 we find that the
   semiclassical limits of these (quantum) morphisms eventually
   define two pairs of morphisms
\begin{equation}  \label{eq: morph.'s form-Pois-grps (rho)}
  G^*_{\text{\rm op}} {\buildrel \psi^+_\rho \over {\relbar\joinrel\relbar\joinrel\longrightarrow}}\, G
     \; ,   \quad
  G^*_{\text{\rm cop}} {\buildrel \psi^-_\rho \over {\relbar\joinrel\relbar\joinrel\longrightarrow}}\, G
     \qquad \text{and} \qquad
  \lieg^*_{\text{\rm op}} {\buildrel d\psi^+_\rho \over {\relbar\joinrel\relbar\joinrel\longrightarrow}}\, \lieg
     \; ,   \quad
  \lieg^*_{\text{\rm cop}} {\buildrel d\psi^-_\rho \over {\relbar\joinrel\relbar\joinrel\longrightarrow}}\, \lieg
\end{equation}
 of formal Poisson groups and of Lie  bialgebras, respectively.
 \vskip7pt
   Third, to compare the two constructions, assume that,
   given mutually dual quantum groups  $ \uhg $  and  $ \fhg \, $,
   \,we pick a single element  $ \, \cR = \rho \, $,  \,thought of simultaneously
   as an  $ R $--matrix  for  $ \uhg $  and as a  $ \varrho $--comatrix  for
   $ \fhg \, $,  \,much in the spirit of  Proposition \ref{prop: duality-deforms}
   and  Theorem \ref{prop: duality-morph.'s_R-mat/rho-comat}.
   Then morphisms as in  \eqref{eq: morph.'s form-Pois-grps (R)}  and
   \eqref{eq: morph.'s form-Pois-grps (rho)}  are defined: but in addition,
   directly by  Theorem \ref{thm: duality x refined R-morph.'s/rho-morph.'s}
   we get at once that
  $$
  \phi^+_{{}_\cR}  \, = \,  \psi^+_\rho  \;\; ,  \quad   \phi^-_{{}_\cR}  \, = \,
  \psi^-_\rho
   \quad \qquad \text{and} \qquad \quad
      d\phi^+_{{}_\cR}  \, = \,  d\psi^+_\rho  \;\; ,  \quad   d\phi^-_{{}_\cR}  \, = \,
      d\psi^-_\rho  $$
 \vskip9pt
   If one works instead with polar  $ R $--matrices  and polar  $ \varrho $--comatrices,
   the roles of  $ G $  and  $ G^* $  are reversed, but for the rest the analysis is
   entirely similar (so we may be more sketchy).  Therefore, assume we have dual
   quantum groups  $ \uhg $  and  $ \fhg \, $.
                                                                               \par
   Given a polar  $ R $--matrix  $ \cR $  for  $ \fhg \, $,  \,the Hopf algebra
   morphisms in  Theorem \ref{thm: refined morph.'s from polar R-mat}  give rise
   (through their semiclassical limit) to two pairs of morphisms
\begin{equation}  \label{eq: morph.'s form-Pois-grps (polar R)}
  G_{\text{\rm op}} \,{\buildrel \underline{\phi}^+_{\,\cR} \over {\relbar\joinrel\relbar\joinrel\longrightarrow}}\, G^*
     \; ,   \quad
  G_{\text{\rm cop}} \,{\buildrel \underline{\phi}^-_{\,\cR} \over {\relbar\joinrel\relbar\joinrel\longrightarrow}}\, G^*
     \qquad \text{and} \qquad
  \lieg_{\text{\rm op}} \,{\buildrel d\underline{\phi}^+_{\,\cR} \over {\relbar\joinrel\relbar\joinrel\longrightarrow}}\, \lieg^*
     \; ,   \quad
  \lieg_{\text{\rm cop}} \,{\buildrel d\underline{\phi}^-_{\,\cR} \over {\relbar\joinrel\relbar\joinrel\longrightarrow}}\, \lieg^*
\end{equation}
 of formal Poisson groups and of Lie  bialgebras, respectively.
 \vskip7pt
   Similarly, if  $ \rho $  is a polar $ \varrho $--comatrix  for  $ \uhg \, $,
   \,the Hopf algebra morphisms in
   Theorem \ref{thm: refined morph.'s from polar rho-comat}  define
   (via their semiclassical limit) two pairs of morphisms
\begin{equation}  \label{eq: morph.'s form-Pois-grps (polar rho)}
  G_{\text{\rm op}} \,{\buildrel \underline{\psi}^+_{\,\rho} \over {\relbar\joinrel\relbar\joinrel\longrightarrow}}\, G^*
     \; ,   \quad
  G_{\text{\rm cop}} \,{\buildrel \underline{\psi}^-_{\,\rho} \over {\relbar\joinrel\relbar\joinrel\longrightarrow}}\, G^*
     \qquad \text{and} \qquad
  \lieg_{\text{\rm op}} \,{\buildrel d\underline{\psi}^+_{\,\rho} \over {\relbar\joinrel\relbar\joinrel\longrightarrow}}\, \lieg^*
     \; ,   \quad
  \lieg_{\text{\rm cop}} \,{\buildrel d\underline{\psi}^-_{\,\rho} \over {\relbar\joinrel\relbar\joinrel\longrightarrow}}\, \lieg^*
\end{equation}
 of formal Poisson groups and of Lie  bialgebras, respectively.
 \vskip7pt
   Finally, if  $ \, \cR = \rho \, $   --- in the spirit of
   Proposition \ref{prop: duality-deforms},  more precisely like in
   Theorem \ref{thm: duality x refined polar R-morph.'s/polar rho-morph.'s}
   ---   then  Theorem \ref{thm: duality x refined polar R-morph.'s/polar rho-morph.'s}
   gives at once
  $$
  \underline{\phi}^+_{\,\cR}  \, = \,  \underline{\psi}^+_{\,\rho}  \;\; ,
  \quad   \underline{\phi}^-_{\,\cR}  \, = \,  \underline{\psi}^-_{\,\rho}
   \quad \qquad \text{and} \qquad \quad
      d\underline{\phi}^+_{\,\cR}  \, = \,  d\underline{\psi}^+_{\,\rho}
      \;\; ,  \quad   d\underline{\phi}^-_{\,\cR}  \, = \,  d\underline{\psi}^-_{\,\rho}
      $$
 \vskip9pt
   Studying in depth all the morphisms introduced above seems to be quite an
   interesting problem; we cannot, however, cope with in the present paper
   --- we just finish with a comparison with previous results.
                                                             \par
   Assume we have an  $ R $--matrix  $ \cR $  for a given QUEA  $ \uhg \, $,
   \,whose dual QFSHA is  $ \, \fhg := {\uhg}^* \, $.  It is well-known that
   the ``semiclassical limit'' of  $ \cR \, $,  that is
   $ \, \displaystyle{r := {{\,\cR - 1^{\otimes 2}\,} \over \hbar} \pmod{\hbar}} \, $,
   is in turn a ``classical  $ r $--matrix''  for the Lie bialgebra  $ \lieg \, $.
   Then Lie bialgebra morphisms
   $ \; \lieg^*_{\text{\rm op}} {\buildrel \varphi^+_r \over {\relbar\joinrel\relbar\joinrel\longrightarrow}}\, \lieg \; $  and
   $ \; \lieg^*_{\text{\rm cop}} {\buildrel \varphi^-_r \over {\relbar\joinrel\relbar\joinrel\longrightarrow}}\, \lieg \; $  are
   defined directly through  $ r $  itself   --- with no need of $ \cR \, $,
   nor of  $ \uhg \, $,  nor  $ \fhg \, $,  \,cf.\  \cite{CP}, \S 2.1,  or
   \cite{Mj}, \S 8.1.  Tracking the various constructions involved
   --- in particular, the functor  $ \, \fhg \mapsto {\fhg}^\vee =: \uhgs \, $
   ---   by direct comparison one immediately sees that
%%%%%
  $$  \varphi^+_r  \, = \,  d\phi^+_{{}_\cR}
   \quad \qquad \text{and} \qquad \quad
      \varphi^-_r  \, = \,  d\phi^-_{{}_\cR}  $$
%%%%%
 In particular, we get that  \textit{the morphisms  $ d\phi^\pm_{{}_\cR} $
 depend on  $ r $  alone, rather than on  $ \cR \, $,
 hence the same is true for the morphisms  $ \phi^\pm_{{}_\cR} \, $};
 indeed, both facts can also be easily proved by direct inspection.
 Similarly, one can prove, via direct analysis again, or by a duality argument
 from the previous result, that  \textit{the morphisms  $ \psi^\pm_\rho $  and
 $ d\psi^\pm_\rho $  depend only on the \textsl{``classical  $ \varrho $--comatrix''}
 $\displaystyle{\rho_0 := {{\,\rho - \epsilon^{\otimes 2}\,} \over \hbar} \pmod{\hbar}} $
 alone},  rather than on  $ \rho \, $.

\bigskip

\end{document}